\documentclass{CUP-JNL-FMS}
\usepackage{amsmath,amssymb,amsthm,amsfonts,amsopn}
\usepackage{dsfont}
\usepackage{enumitem}
\usepackage[utf8]{inputenc}
\usepackage{mathrsfs}
\usepackage[numbers]{natbib}
\usepackage{newtxtext}
\usepackage[varvw]{newtxmath}
\usepackage{textcomp}
\usepackage[colorlinks,allcolors=weblmcolor,bookmarksnumbered]{hyperref}
\usepackage{todonotes}

\usepackage{color,url,setspace,wasysym}
\usepackage{lscape}

\makeatletter
\newcommand{\odagger}{\mathbin{\mathpalette\make@circled\dagger}}
\newcommand{\make@circled}[2]{%
  \ooalign{$\m@th#1\smallbigcirc{#1}$\cr\hidewidth$\m@th#1#2$\hidewidth\cr}%
}
\newcommand{\smallbigcirc}[1]{%
  \vcenter{\hbox{\scalebox{1.2}{$\m@th#1\bigcirc$}}}%
}
\makeatother

\usepackage{etoolbox}

\usepackage{cleveref}

\theoremstyle{plain}
\newtheorem{theorem}{Theorem}[section]
\newtheorem{proposition}[theorem]{Proposition}
\newtheorem{lemma}[theorem]{Lemma}
\newtheorem{corollary}[theorem]{Corollary}
\theoremstyle{definition}
\newtheorem{definition}[theorem]{Definition}
\newtheorem{assumption}[theorem]{Assumption}
\newtheorem{notanddef}[theorem]{Notation \& Definition}
\theoremstyle{remark}
\newtheorem{remark}[theorem]{Remark}
\newtheorem{example}[theorem]{Example}
\theoremstyle{remark}
\newtheorem*{rem*}{Remark}
\crefname{assumption}{Assumption}{Assumptions}

\numberwithin{equation}{section}

\newcommand{\bd}{\mathbf}

\newcommand{\LtRd}{\mathbf L^2(\mathbb{R}^d)}

\newcommand{\LiRd}{\mathbf L^1(\mathbb{R}^d)}
\newcommand{\Ltw}{\bd L^2_{\sqrt{w}}}
\newcommand{\Ltv}{\bd L^2_{\sqrt{w_0}}}
\newcommand{\AAPlain}{\mathcal{A}}
\newcommand{\AAi}{\mathcal A_1}
\newcommand{\AAm}{\mathcal A_m}
\newcommand{\AAmY}{{\mathcal A_{m,Y}}}

\newcommand{\BBi}{\mathcal B_1}
\newcommand{\BBm}{\mathcal B_m}
\newcommand{\BBmo}{\mathcal B_{m_0}}

\newcommand{\BB}{\mathcal B}

\newcommand{\UU}{\mathcal{U}}
\newcommand{\VV}{\mathcal{V}}

\newcommand{\loc}{\mathrm{loc}}
\newcommand{\RR}{\mathbb{R}}
\newcommand{\R}{\mathbb{R}}
\newcommand{\ZZ}{\mathbb{Z}}
\newcommand{\Z}{\mathbb{Z}}
\newcommand{\NN}{\mathbb{N}}
\newcommand{\N}{\mathbb{N}}
\newcommand{\CC}{\mathbb{C}}

\newcommand{\Eye}{\mathbf{\operatorname{E}}_{\tau,\varepsilon}}

\def\esssup{\mathop{\operatorname{ess~sup}}}

\def\supp{\mathop{\operatorname{supp}}}
\def\sgn{\mathop{\operatorname{sgn}}}
\newcommand{\inner}[2]{\langle #1,#2 \rangle}

\makeatletter
\newcommand{\LeftEqNo}{\let\veqno\@@leqno}

\DeclareFontFamily{U}{mathx}{\hyphenchar\font45}
\DeclareFontShape{U}{mathx}{m}{n}{
      <5> <6> <7> <8> <9> <10>
      <10.95> <12> <14.4> <17.28> <20.74> <24.88>
      mathx10
      }{}
\DeclareSymbolFont{mathx}{U}{mathx}{m}{n}
\DeclareFontSubstitution{U}{mathx}{m}{n}
\DeclareMathAccent{\widecheck}{0}{mathx}{"71}
\DeclareMathAccent{\wideparen}{0}{mathx}{"75}
\newcommand{\LtD}{\mathbf L^2(D)}
\newcommand{\LtDF}{\bd L^{2,\mathcal F}(D)}
\newcommand{\Hil}{\mathcal{H}}
\newcommand{\oscVG}{{\mathrm{osc}}_{\VV,\Gamma}}
\newcommand{\oscVGd}{{\mathrm{osc}}_{\VV^\delta,\Gamma}}
\newcommand{\oscVFGd}{{\mathrm{osc}}_{\VV^\delta_\Phi,\Gamma}}

\newcommand{\Indicator}{{\mathds{1}}}

\newcommand{\translation}{\bd{T}}
\newcommand{\lebesgue}{\bd{L}}

\newcommand{\Fourier}{\mathcal{F}}

\newcommand{\CalQ}{\mathcal{Q}}

\newcommand{\CalP}{\mathcal{P}}
\newcommand{\CalR}{\mathcal{R}}
\newcommand{\CalU}{\mathcal{U}}
\newcommand{\CalO}{\mathcal{O}}

\newcommand{\CalG}{\mathcal{G}}
\newcommand{\CalH}{\mathcal{H}}
\newcommand{\CalC}{\mathcal{C}}
\newcommand{\CalV}{\mathcal{V}}
\newcommand{\CalW}{\mathcal{W}}
\newcommand{\CalN}{\mathcal{N}}
\newcommand{\GL}{\mathrm{GL}}

\newcommand{\Schwartz}{\mathcal{S}}
\newcommand{\Co}{\operatorname{Co}}

\newcommand{\dimension}{d}
\newcommand{\identity}{\mathrm{id}}

\newcommand{\eps}{\varepsilon}
\newcommand{\with}{\,:\,}
\newcommand{\BanachOne}{Y}
\newcommand{\BanachTD}{B}
\newcommand{\PhVar}{\lambda}
\newcommand{\PhVarA}{\rho}

\newcommand{\PhVarC}{{\nu}}
\newcommand{\PhSpace}{\Lambda}
\newcommand{\GoodVectors}{\mathcal{H}_v^1}
\newcommand{\Reservoir}{(\GoodVectors)^\urcorner}

\binoppenalty=\maxdimen
\relpenalty=\maxdimen

\newcommand{\vrho}{\varrho}
\newcommand{\vrhoinv}{\vrho_{\ast}}
\newcommand{\vsig}{\varsigma}
\newcommand{\vsiginv}{\vsig_{\ast}}

\newcommand{\MaxKernel}[1]{\mathrm{M}_{#1}}
\newcommand{\vertiii}[1]{{\left\vert \kern-0.25ex  \left\vert \kern-0.25ex  \left\vert #1\right\vert \kern-0.25ex  \right\vert \kern-0.25ex  \right\vert}}

\usepackage{scalerel}
\renewcommand\upsilon{{\scaleobj{0.65}{\Upsilon}}}

\let\emptyset\varnothing

\newcommand{\nicki}[1]{#1}  

\articletype{RESEARCH ARTICLE}
\jyear{2022}
\jdoi{xx}

\begin{document}

\begin{Frontmatter}

\title[Coorbit theory of warped time-frequency systems]
       {Coorbit theory of warped time-frequency systems in $\RR^d$}

\author[1]{Nicki Holighaus}\orcid{0000-0003-3837-2865}
\author[2]{Felix Voigtlaender}\orcid{0000-0002-5061-2756}

\authormark{Nicki Holighaus and Felix Voigtlaender}

\address[1]{\orgname{Acoustics Research Institute, Austrian Academy of Sciences}, \orgaddress{\street{Wohllebengasse 12--14}, \postcode{1040} \city{Vienna}, \country{Austria}};\email{nicki.holighaus@oeaw.ac.at}}
\address[2]{\orgdiv{Lehrstuhl \emph{Reliable Machine Learning}}, \orgname{Katholische Universität Eichstätt-Ingolstadt}, \orgaddress{\street{Ostenstraße 26}, \postcode{85072} \city{Eichstätt}, \country{Germany}};\email{felix.voigtlaender@ku.de}}
\received{02 August 2022}



\abstract{Warped time-frequency systems have recently been introduced as a
  class of structured continuous frames for functions on the real line.
  Herein, we generalize this framework to the setting of functions of
  arbitrary dimensionality.
  After showing that the basic properties of warped time-frequency
  representations carry over to higher dimensions, we determine
  conditions on the warping function which guarantee that
  the associated Gramian is well-localized, so that
  associated families of coorbit spaces can be constructed.
  We then show that discrete
  Banach frame decompositions for these coorbit spaces can be obtained
  by sampling the continuous warped time-frequency systems.
  In particular, this implies that sparsity of a given function $f$
  in the discrete warped time-frequency dictionary is equivalent
  to membership of $f$ in the coorbit space.
  We put special emphasis on the case of radial warping functions,
  for which the relevant assumptions simplify considerably.}

\end{Frontmatter}

\localtableofcontents

\vspace*{14pt}
\section{Introduction}

Time-frequency representations%
\footnote{The term \emph{time-frequency representation} is used in a wide sense here,
also covering time-scale representations like wavelets.}
(TF representations) are versatile tools for the analysis and decomposition of general functions
(or signals) with respect to simpler, structured building blocks.
They provide rich and intuitive information about a function's time-varying spectral
behavior in settings where both time-series and stationary Fourier transforms are insufficient.

Important fields relying on time-frequency representations include
signal processing~\cite{portnoff1981time,adler2012audio,ma08-3,boashash2015time}
and image processing~\cite{CAI2008131,chan2006total,ma08-3,taubman2012jpeg2000},
medical imaging~\cite{li1995multisensor,zhao2000x},
the numerical treatment of PDEs~\cite{harbrecht2015adaptive,cohen2004adaptive},
and quantum mechanics~\cite{perelomov2012generalized}.
In particular, short-time Fourier transforms \cite{gr01}
and wavelet transforms~\cite{da92} are widely and successfully used in these fields.

Yet, the limitations of such rigid schemes, considering only translations and modulations
(resp.~simple scalar dilations) of a single prototype function, are often considered detrimental
to their representation performance.
Therefore, numerous more flexible time-frequency representations
have been proposed and studied in the last decades.
As the most prominent of such systems,
we mention curvelets~\cite{candes2005continuous,CandesSecondGenerationCurvelets},
shearlets~\cite{kula12-alt,dakustte09}, ridgelets~\cite{candes1999ridgelets},
and $\alpha$-modulation systems
\cite{fefo06,cofe78,hona03,daforastte08,han2014alpha,speckbacher2016alpha}.

In the present article, we consider a more flexible scheme for constructing time-frequency representations,
namely the framework of \emph{warped time-frequency systems}
that was recently introduced for dimension $d=1$ in \cite{howi14,bahowi15}.
To motivate this construction, note that the systems mentioned above
are all examples of so-called \emph{generalized translation-invariant (GTI) systems}
\cite{helawe02,jakobsen2016reproducing,rosh04},
i.e., each of these systems is of the form
$(\translation_{x}\psi_{i})_{i\in I, x\in Z_{i}}$
for certain generators $\psi_{i} \in \bd L^{2}(\RR^{d})$ and
subgroups ${Z_{i} \subset \RR^{d}}$.
Here, $\translation_x \psi (y) = \psi(y - x)$ denotes the translation of $\psi$ by $x$.
Although it is not required that the $Z_i$ are discrete, they are often taken to be lattices,
i.e., $Z_i = T_{i} \ZZ^{d}$, with $T_{i} \in{\rm GL} (\RR^{d})$. 
The various systems differ in the way in which the generators $\psi_{i}$
and the lattices $Z_{i}$ are chosen.
But in each case there is a \emph{finite} set of prototypes, often a single prototype, such that each $\psi_{i}$ is a
certain dilated and/or modulated version of one of the prototypes.
Here, the dilations might be anisotropic, as is the case for shearlets.

As two canonical examples, we note that for a Gabor system, we have
$\psi_{k}(x) = e^{2\pi i \alpha \langle k,x \rangle } \cdot \psi(x)$
for $k \in I = \ZZ^{d}$, while for a (homogeneous) wavelet system, we
have $\psi_{j}(x) = 2^{dj / 2} \cdot \psi(2^{j} x)$ for $j \in I = \ZZ$.
Thus, the two systems differ with respect to the frequency localization
of the generators $\psi_{i}$:
For a Gabor system, the (essential) frequency supports of the generators
$\psi_{k}$ form a \emph{uniform} covering of the frequency space $\RR^{d}$---%
in contrast to the case of wavelets,
where the (essential) frequency supports form a \emph{dyadic} covering.

Warped time-frequency systems are motivated by
the crucial observation that the dyadic covering corresponds to
a uniform covering \emph{with respect to a logarithmic scaling of the frequency space}.
This suggests the following general construction:
Starting from a \emph{warping function} $\Phi$---i.e., a diffeomorphism
$\Phi : D \subset \RR^{d} \to \RR^{d}$---and a \emph{prototype function}
$\theta \in \bd L^{2}(\RR^{d})$,
we consider the associated \emph{warped time-frequency system}
$\mathcal{G}(\theta,\Phi) = \left(g_{y,\omega}\right)_{y\in\RR^{d},\omega\in D}$ given by
\begin{equation}
    g_{y,\omega}=\translation_{y}\left[\Fourier^{-1}g_{\omega}\right]
  \quad\text{ with }\quad
  g_{\omega}= c_\omega \cdot \left(\translation_{\Phi\left(\omega\right)}\theta\right)\circ\Phi
  \qquad \text{ for } \qquad (y,\omega) \in \RR^d \times D.
    \label{eq:IntroductionWarpedSystemIntuition}
\end{equation}
Here, the function
\(
  c_\omega \cdot \left(\translation_{\Phi(\omega)}\theta\right)\circ\Phi :
  D \subset \RR^{d} \to \CC
\)
is extended trivially to a map defined on all of $\RR^{d}$ before
applying the (inverse) Fourier transform $\Fourier^{-1}$ to it,
and the constant $c_\omega > 0$ is chosen such that
the resulting family $(g_{y,\omega})_{y \in \RR^d ,\omega \in D}$
forms a \emph{tight} frame for the space $ \LtDF = \Fourier^{-1} (\lebesgue^2 (D))$
of all $\bd L^{2}$ functions with Fourier transform vanishing outside of $D$.

At first sight, this construction might seem intimidating, but it
can be unraveled as follows:
The warping function $\Phi$ provides a map from the frequency space $D$
to the warped frequency space $\R^{d}$.
Thus, $\theta$ serves as a prototype for the Fourier transform of the GTI generators
$\mathcal{F}^{-1}g_{\omega}$, but in warped coordinates.
In that sense, $g_{\omega}$ can be understood as a shifted version of $\theta$,
but the shift is performed in warped (frequency) coordinates. 
In order to build further intuition for this construction, it is helpful to consider
the case in which $\theta$ is (essentially) concentrated at $0$,
so that $\translation_{\Phi(\omega)}\theta$ is concentrated at $\Phi(\omega)$,
whence $g_{\omega}$ is concentrated at $\omega$.
Put briefly, the warping function $\Phi$ determines the frequency scale and,
with it, the frequency-bandwidth relationship of the resulting warped time-frequency system.

As a further illustration, let us explain how wavelet systems fit into the above construction.
Define $D := (0,\infty)$ and $\Phi : D \to \R, x \mapsto \ln (x)$.
Then
\[
    \big(
      [\, \translation_{\Phi(\omega)} \theta \,]    \circ \Phi
    \big)(\xi)
  = \theta \big( \ln(\xi) - \ln(\omega) \big)  
  = [\theta\circ\ln] \bigl(\xi / \omega\bigr) ,
\]
and hence, with $\psi = \mathcal F^{-1} (\theta\circ\ln)$, it holds that
\(
  \mathcal{F}^{-1} g_\omega
  = c_\omega \cdot \omega \cdot \left[\mathcal{F}^{-1} (\theta\circ\ln)\right](\omega \bullet)
  = c_\omega \, \omega \cdot \psi (\omega \bullet) ,
\)
so that
\(
  (g_{y,\omega})_{y \in \R, \omega \in D}
  = \bigl(
      c_\omega \, \omega \cdot \translation_y [\psi (\omega \bullet)]
    \bigr)_{y \in \R, \omega \in D}
\)
is a continuous wavelet system, for an appropriate choice of $c_\omega$.
Finally, since translations in frequency domain correspond to modulations
in the time domain, continuous Gabor systems can be obtained by choosing $\Phi : \R^d \to \R^d$
to be the identity function.

\subsection{Contribution}\label{subsec:contributions}

The overall goal of the present article is to start an in-depth study
of the properties of warped time-frequency systems on $\RR^d$.
Some essential, basic characteristics of warped time-frequency systems on $\RR^d$
are obtained analogously to the one-dimensional case treated in \cite{bahowi15}.
In particular, $\mathcal{G}(\theta,\Phi)$ forms, under mild assumptions on $\theta$ and $\Phi$,
a \emph{continuous} tight frame for $\LtDF$.
However, our main objective, verifying the applicability of \emph{general coorbit theory}
(as developed in \cite{fora05,rauhut2011generalized,kempka2015general})
to the continuous frame $\mathcal{G}(\theta,\Phi)$ is \emph{decidedly} more
involved in the higher-dimensional case that we consider here than in the case $d = 1$.
Therefore, the main results presented in this work are concerned with establishing
a set of assumptions on $\theta$ and $\Phi$,
such that the rich discretization theory for the coorbit spaces
associated with $\mathcal{G}(\theta,\Phi)$ is accessible.
%
%

To make the latter point more precise, let us briefly recall the main points
of coorbit theory related to the present setting.
The main tenet of coorbit theory is to quantify the regularity of a
function $f$ using a certain norm
\(
  \left\Vert f\right\Vert_{{\rm Co} \left( \BanachOne \right)}
  := \left\Vert V_{\theta,\Phi} f \right\Vert_{\BanachOne}
\)
of the \emph{voice transform} $V_{\theta,\Phi}f(y,\omega) = \langle f,\,g_{y,\omega}\rangle$.
The \emph{coorbit space} associated with a Banach space
${\BanachOne \subset \lebesgue_{{\rm loc}}^{1}(\RR^d \times D)}$ is then given by
\({
  {\rm Co}_{\theta,\Phi} \left(\BanachOne\right)
  = \big\{
      f
      \,:\,
      V_{\theta,\Phi} f \in \BanachOne
    \big\}
  .
}\)

Of course, the general theory of coorbit spaces as developed
in \cite{fora05,rauhut2011generalized,kempka2015general}
does not consider the special frame $\mathcal{G}(\theta,\Phi)$,
but a general continuous frame $\Psi = (\psi_{\PhVar})_{\PhVar \in \PhSpace}$.
Coorbit theory then provides (quite technical) conditions concerning
the frame $\Psi$ which ensure that the associated coorbit spaces
${\rm Co}_{\Psi}(\BanachOne)$ are indeed well-defined Banach spaces.
We will verify these conditions in the setting of the warped time-frequency systems
$\mathcal{G}(\theta,\Phi)$.
Precisely, we shall derive verifiable conditions concerning $\theta$ and $\Phi$ which ensure
that coorbit theory is applicable.

Additionally, coorbit spaces come with a powerful \emph{discretization theory}: 
Under suitable conditions on the frame $\Psi = (\psi_{\PhVar})_{\PhVar\in\PhSpace}$,
taken from an appropriate test function space,
and on the discrete set $\PhSpace_{d} \subset \PhSpace$, coorbit theory shows that
the sampled frame $\Psi_{d} = (\psi_{\PhVar})_{\PhVar\in\PhSpace_{d}}$
forms a \emph{Banach frame decomposition} for the coorbit space ${\rm Co}_{\Psi} (\BanachOne)$.
The precise definition of this concept will be given later.
Here, we just note that it implies the existence of sequence spaces
$\BanachOne_d^\flat \subset \CC^{\PhSpace_d}$ and $\BanachOne_d^\sharp \subset \CC^{\PhSpace_d}$
such that
\[
    \left\Vert f \right\Vert_{{\rm Co}\left(\BanachOne\right)}
    \asymp \big\Vert
              \bigl(
                \left\langle f,\psi_{\PhVar}\right\rangle
              \bigr)_{\PhVar \in \PhSpace_{d}}
           \big\Vert_{\BanachOne_d^\flat}
    \asymp \inf
           \bigg\{
              \big\Vert
                  (c_{\PhVar})_{\PhVar\in\PhSpace_{d}}
              \big\Vert_{\BanachOne_d^{\sharp}}
              \,:\,
              f = \sum_{\PhVar\in\PhSpace_{d}}
                    c_{\PhVar} \cdot \psi_{\PhVar}
           \bigg\} .
\]
Hence, for a generalized notion of \emph{sparsity}, membership of $f$
in ${\rm Co}_{\Psi}(\BanachOne)$ is \emph{simultaneously} equivalent to \emph{analysis sparsity}
and \emph{synthesis sparsity} of $f$ with respect to the discretized frame $\Psi_{d}$.
Specifically, a sequence $c$ is considered sparse if $c \in \BanachOne_d^\flat$
or $c \in \BanachOne_d^\sharp$.
This is most closely related to classical sparsity if $\BanachOne_d^{\flat}$ and
$\BanachOne_d^\sharp$ coincide with certain (weighted) $\ell^p$ spaces.

We indeed show under suitable conditions concerning $\theta$ and $\Phi$
that the discretization theory applies to $\mathcal{G}\left(\theta,\Phi\right)$.
Therefore, the coorbit spaces ${\rm Co}_{\theta,\Phi}\left(\BanachOne\right)$
characterize sparsity with respect to the (suitably discretized) warped
time-frequency system $\mathcal{G}\left(\theta,\Phi\right)$.
As a byproduct, we also show that the space ${\rm Co}_{\theta,\Phi}(\BanachOne)$
is essentially independent of the choice of \emph{appropriate (sufficiently regular)} $\theta$.

\subsection{Related work: Warped time-frequency systems}%
\label{sub:RelatedWorkWarpedTF}

Warped time-frequency systems have already been considered before,
though only for the one-dimensional case $d=1$. 
In particular, in \cite{bahowi15}, the authors essentially obtain the results
that we just outlined, i.e., that warped time-frequency systems form tight frames
and that the assumptions of generalized coorbit theory can be satisfied,
at least for coorbit spaces associated to the (weighted) Lebesgue spaces
${\BanachOne = \lebesgue_{\kappa}^{p}(\R \times D)}$.
We generalize these results to higher dimensions $d > 1$ and to
the weighted \emph{mixed} Lebesgue spaces $\lebesgue_{\kappa}^{p,q}(\R^d \times D)$,
equipped with the norm
\(
  \Vert F \Vert_{\lebesgue_{\kappa}^{p,q}}
  = \big\Vert \,
      \omega \mapsto \Vert
                       (\kappa \cdot F)(\bullet, \omega)
                     \Vert_{\lebesgue^{p}(\RR^{d})}
    \big\Vert_{\lebesgue^{q}(D)}
  .
\)
Furthermore, we relax some of the assumptions imposed in \cite{bahowi15}.
The generalization to higher dimensions is, as we will see, by no means trivial.
The extension to the spaces $\lebesgue_{\kappa}^{p,q}(\R^d \times D)$
relies on our recent work~\cite{hovo20_algebra}.

Hilbert space frames obtained by sampling warped time-frequency systems were
examined in~\cite{howi14}, where different necessary or sufficient 
frame conditions similar to those for Gabor and wavelet frames were obtained.
In the same paper, the authors also derive readily verifiable conditions
under which the sampled warped time-frequency system satisfies the local integrability condition,
thereby providing access to useful results from the theory of GTI systems.

\subsection{Related work: GTI systems}
\label{subsec:RelatedWorkGTI}

Warped time-frequency representations are GTI systems~\cite{rosh04,helawe02,jakobsen2016reproducing}, 
and they could be analyzed within this abstract framework.
However, fully general GTI systems include a considerable number of---usually undesired---%
pathological cases \cite{FuehrLemvigSystemBandwidth,VelthovenLCI};
these can be excluded by imposing additional structure-enforcing conditions.
The most general and well-known such condition is the \emph{local integrability condition} (LIC)
of Hernandez et al.~\cite{helawe02}, further investigated in~\cite{jakobsen2016reproducing,VelthovenLCI}. 

In practice, GTI systems are mostly generated from one (or few) prototype
functions through the application of a family of operators---like modulations or dilations---%
that promote a given frequency-bandwidth relationship,
such as the constant frequency/bandwidth ratio for classical wavelet systems.
Naturally, such systems are well suited for representing functions
with certain frequency-domain properties.

In our case, structure is imposed by the choice of the prototype and warping function
that determine the frequency-bandwidth relation and the distribution of GTI generators
in the frequency domain.
In this sense, warped time-frequency systems provide a unified framework for studying structured
time-frequency representations.
We will see that warped time-frequency systems, despite their generality,
satisfy many beneficial properties that are not simply trivial consequences of them being GTI systems.

As other related time-frequency systems, we mention dictionaries obtained by
combining multiple TF dictionaries, either globally~\cite{akch97,zezi97,bachkrokro14},
or locally in \emph{weaved} phase space covers~\cite{do11,doro14,ro11}.
Furthermore, nonstationary Gabor systems~\cite{badohojave11,doma14-1,doma14-2,ho14-1}
are closely related to GTI systems via the Fourier transform.

\subsection{Related work: Function space theory}
\label{subsec:RelatedWorkFST}

The joint study of integral transforms and appropriate (generalized) function spaces
is a classical topic in Fourier- and harmonic analysis.
In particular, localization and smoothness properties
of functions and their Fourier transforms have received much attention.
Indeed, from the distribution theory of Laurent Schwartz~\cite{sc57,sc57-1}
to Paley-Wiener spaces~\cite{boasentire}, Sobolev spaces
\cite{adams2003sobolev,LeoniSobolevSpaces,TriebelTheoryOfFunctionSpaces}
and Besov spaces~\cite{TriebelTheoryOfFunctionSpaces,triebel2010theory,besov59},
a large number of classical function spaces can be meaningfully characterized through their
Fourier transform properties.
Other examples include the family of modulation spaces~\cite{gr01,feichtinger1983modulation}---%
defined through the short-time Fourier transform~\cite{ga46,gr01}---as well as
spaces of (poly-)analytic functions~\cite{balk1991polyanalytic,abreu2010sampling}
and the Bargmann~\cite{bargmann1961hilbert,bargmann1967hilbert}
and Bergman transforms~\cite{abreu2012super}.

A powerful general framework for studying function spaces associated with a certain transform
is provided by coorbit theory, originally introduced by Feichtinger
and Gröchenig~\cite{fegr89,fegr89-1,gr91}.
As described above, the underlying idea for this theory
is to measure the regularity of a function or distribution in terms of
growth or decay properties of an abstract \emph{voice transform}.
In the original approach of Feichtinger and Gröchenig, the voice transform is defined through
an integrable group representation acting on a suitable prototype function.
Prime examples of different transforms and the associated coorbit spaces are
the short-time Fourier transform~\cite{ga46,gr01}
and modulation spaces,
associated with the (reduced) Heisenberg group,
and the wavelet transform~\cite{da92} and (homogeneous) Besov spaces~\cite{triebel2010theory,besov59},
associated with the $ax+b$ group.

Fornasier and Rauhut~\cite{fora05} realized that the group structure on which classical
coorbit theory relies can be discarded completely.
Instead, one can consider the voice transform associated with a general
\emph{continuous frame}~\cite{alanga93,alanga00},
the Gramian kernel of which is required to satisfy certain integrability and oscillation conditions.
Since the introduction of this \emph{general coorbit theory}, these results have been
improved and expanded~\cite{rauhut2011generalized,kempka2015general,ournote},
as well as successfully applied, e.g., to Besov
and Triebel-Lizorkin spaces~\cite{TriebelTheoryOfFunctionSpaces,triebel2010theory,tr06}
or $\alpha$-modulation spaces~\cite{gr92-2};
see e.g.~\cite{rauhut2011generalized,ullrich2012continuous}
and \cite{daforastte08,speckbacher2016alpha}.

\subsection{Structure of the paper}
\label{subsec:structure}

We begin with a brief introduction to general coorbit theory in Section~\ref{sec:prelims}.
We then formally introduce warped time-frequency systems in Section~\ref{sec:warpedsystems},
in which we also discuss several concrete examples.
Section~\ref{sec:coorbits} is concerned with conditions on the warping function $\Phi$
and the prototype $\theta$ which ensure that the continuous frame $\mathcal{G}(\theta, \Phi)$
satisfies the assumptions of (general) coorbit theory.

To show that the continuous frame $\mathcal{G}(\theta, \Phi)$ can be sampled to obtain discrete
Banach frame decompositions of the associated coorbit spaces, we will need certain coverings
of the phase space $\PhSpace = \RR^d \times D$ associated with the warping function $\Phi$.
These coverings are studied in Section~\ref{sec:coverings}.
In Section~\ref{sec:discretewarped}, we prove the existence of
discrete Banach frame decompositions for the coorbit spaces $\Co_{\theta,\Phi} (\BanachOne)$.
Finally, in Section~\ref{sec:RadialWarping} we investigate warped time-frequency systems generated
by \emph{radial} warping functions on $\RR^d$.
In particular, we show that admissible symmetric warping functions on $\RR$ give rise to
admissible radial warping functions on $\RR^d$.

\subsection{Notation and fundamental definitions}
\label{subsec:notation}

We use the notation $\underline{n}:=\{1,\ldots,n\}$ for $n\in\NN$.
We write $\R^+ = (0,\infty)$ for the set of positive real numbers,
and $S^1 := \{ z \in \CC \colon |z| = 1 \}$.
For the composition of functions $f$ and $g$ we use the notation $f \circ g$
defined by $f \circ g (x)= f(g(x))$.
For a subset $M \subset X$ of a fixed base set $X$
(which is usually understood from the context), we use the indicator function
$\Indicator_{M}$ of the set $M$, where $\Indicator_M (x) = 1$ if $x \in M$
and $\Indicator_M (x) = 0$ otherwise.

The (topological) dual space of a (complex) topological vector space $X$
(i.e., the space of all continuous linear functions $\varphi : X \to \CC$)
is denoted by $X'$, while the (topological) \emph{anti}-dual of a Banach space $X$
(i.e., the space of all \emph{anti-linear} continuous functionals on $X$)
is denoted by $X^\urcorner$.
A superscript asterisk ($^*$) is used to denote the adjoint of an operator between Hilbert spaces.

We use the convenient short-hand notations $\lesssim$ and $\asymp$,
where $A \lesssim B$ means $A \leq C \cdot B$, for some constant $C > 0$ that depends
on quantities that are either explicitly mentioned or clear from the context.
$A \asymp B$ means $A \lesssim B$ and $B \lesssim A$.

\subsubsection{Norms and related notation}

We write $|x|$ for the Euclidean norm of a vector $x \in \RR^d$,
and we denote the operator norm of a linear operator $T : X \to Y$
by $\| T \|_{X \to Y}$, or by $\| T \|$, if $X,Y$ are clear from the context.
In the expression $\| A \|$, a matrix $A \in \RR^{n \times d}$ is interpreted
as a linear map $(\RR^d,|\bullet|) \to (\RR^n,|\bullet|)$.
The open (Euclidean) ball around $x \in \R^d$ of radius $r > 0$ is denoted by $B_r (x)$.

\subsubsection{Fourier-analytic notation}

The Lebesgue measure of a (measurable) subset $M \subset \RR^d$ is denoted by $\mu(M)$.
The Fourier transform is given by
$\widehat{f}(\xi) = \Fourier f(\xi) = \int_{\RR^d} f(x) \, e^{-2 \pi i \langle x, \xi\rangle} \, dx$,
for all $f \in \LiRd$.
It is well-known that $\Fourier$ extends to a unitary automorphism of $\LtRd$.
The inverse Fourier transform is denoted by $\widecheck{f} := \Fourier^{-1}f$.
We write $\bd L^{2,\mathcal F}(D) := \mathcal F^{-1}(\bd L^2(D))$ for the space of
square-integrable functions whose Fourier transform vanishes (a.e.) outside of $D \subset \R^d$.
In addition to the Fourier transform, the \emph{modulation} and \emph{translation operators}
$\bd{M}_\omega f = f\cdot e^{2\pi i\langle\omega, \bullet\rangle}$ and $\bd T_y f = f(\bullet - y)$,
will be used frequently.

\subsubsection{Matrix notation}

For matrix-valued functions $A : U \to \RR^{d \times d}$, the notation
$A(x)\langle y\rangle := A(x)\cdot y$ denotes the multiplication of the matrix $A(x)$, $x \in U$,
with the vector $y \in \R^d$ in the usual sense.
Likewise, for a set $M\subset \RR^d$, we write
\[
  A(x)\langle M\rangle := \big\{ A(x)\langle y\rangle~:~ y\in M \big\}.
\]
Moreover, we define $A^{-1} (\tau) := [A(\tau)]^{-1}$
and similarly $A^{\pm T} (\tau) := [A(\tau)]^{\pm T}$.
Here, as in the remainder of the paper, the notation $A^T$
denotes the transpose of a matrix $A$.
We will denote the elements of the standard basis of $\R^d$ by $e_1,\dots,e_d$.

\subsubsection{Convention for variables}

Throughout this article, $x,y,z \in \R^d$ will be used to denote variables in time/position space,
$\xi,\omega,\eta \in D$ in frequency space, $\PhVar,\PhVarA,\PhVarC \in \R^d \times D$ in phase space,
and finally $\sigma, \tau,\upsilon \in \R^d$ denote variables in warped frequency space.
Unless otherwise stated, this also holds for subscript-indexed variants;
precisely, subscript indices (i.e., $x_i$) are used to denote the $i$-th element
of a vector $x \in \RR^d$.
In some cases, we also use subscripts to enumerate multiple vectors,
e.g., $x_1,\dots,x_n \in \R^d$.
In this case, we denote the components of $x_i$ by $(x_i)_j$.

\subsubsection{Solid spaces, integral kernels, and mixed Lebesgue spaces}
\label{subsub:IntegralKernels}

Unless noted otherwise, we will always consider $\PhSpace = \R^d \times D$
(with an open set $D \subset \R^d$), equipped with the Lebesgue measure $\mu$.
A Banach space $\BanachOne \subset \lebesgue_{\loc}^1 (\PhSpace)$ will be called \emph{solid}
if it satisfies the following:
whenever $F, G :\PhSpace \to \CC$ are measurable with $|F| \leq |G|$ almost everywhere
and with $G \in \BanachOne$, then $F \in \BanachOne$
and $\| F \|_{\BanachOne} \leq \| G \|_{\BanachOne}$.
$\BanachOne$ is \emph{rich}, if it contains all (locally) integrable, compactly supported functions.
The analogous definitions apply for general locally compact measure spaces,
and in particular to sequence spaces (where the index set is equipped with the discrete topology).

A \emph{kernel} on $\PhSpace$ is a (measurable) function
$K : \PhSpace \times \PhSpace \rightarrow \CC$.
Its application to a (measurable) function $F : \PhSpace \to \CC$ is denoted by
\begin{equation}
  K(F)(\PhVar):= \int_\PhSpace K(\PhVar,\PhVarA)F(\PhVarA)~d\mu(\PhVarA),
  \qquad \text{ whenever the integral exists}.
  \label{eq:IntegralOperator}
\end{equation}
We will identify two kernels if they agree almost everywhere.
As usual, $K^\ast$ denotes the \emph{adjoint kernel}
$K^\ast (\PhVar,\PhVarA) = \overline{K(\PhVarA,\PhVar)}$,
and $K^T$ denotes the \emph{transposed kernel}, given by $K^T (\PhVar,\PhVarA) = K(\PhVarA, \PhVar)$.

Since $\PhSpace = \R^d \times D$ has a product structure, it is natural to
consider the weighted, \emph{mixed} Lebesgue spaces $\bd L^{p,q}_\kappa(\PhSpace)$,
for $1\leq p,q\leq \infty$, that consist of all
(equivalence classes of almost everywhere equal) measurable functions $F: \PhSpace \to \CC$
for which
\begin{equation}\label{eq:mixedLebesgue}
  \|F \|_{\bd L_\kappa^{p,q}}
  := \left\|
          \vphantom{\sum}
          \PhVar_2 \mapsto \left\|
                              (\kappa\cdot F)(\bullet, \PhVar_2)
                            \right\|_{\bd L^p (\R^d)}
     \right\|_{\bd L^q (D)}
     <  \infty.
\end{equation}
Here, $\kappa : \PhSpace \to \RR^+$ is a (measurable) weight function.

\section{Frames, coverings and coorbit spaces}
\label{sec:prelims}

In this section, we prepare our investigation of warped time-frequency systems
by recalling several notions and results related to the theory of continuous frames
and general coorbit theory.

A collection $\Psi = (\psi_{\PhVar})_{\PhVar\in\PhSpace}$ of elements $\psi_{\PhVar} \in \CalH$
of a separable Hilbert space $\CalH$ is called a \emph{tight continuous frame (for $\CalH$)},
if there exists $A \in \RR^+$ such that
\begin{equation}\label{eq:frameeq}
   A \cdot \|f\|_{\Hil}^2
   = \int_{\PhSpace}
       |\langle f,\psi_{\PhVar}\rangle_{\Hil}|^2
     d\mu(\PhVar)
  \quad \text{for all } f \in \CalH,
\end{equation}
and if furthermore the map $\PhVar \mapsto \psi_{\PhVar}$ is weakly measurable,
meaning that $\lambda \mapsto \langle f, \psi_\lambda \rangle$ is measurable for each $f \in \CalH$.
For the warped time-frequency systems considered later, we will see that
$\PhVar \mapsto \psi_{\PhVar}$ is in fact continuous (see \Cref{prop:WarpedSystemWeaklyContinuous}).
We say that $\Psi$ is a \emph{Parseval frame} if $A = 1$ in Equation~\eqref{eq:frameeq}.

The \emph{voice transform} with respect to a tight continuous frame
$\Psi = (\psi_\PhVar)_{\PhVar \in \PhSpace}$ is given by
\begin{equation}
  V_\Psi~:~ \Hil \rightarrow \bd L^2(\PhSpace),
  \quad \text{defined by}\quad
  V_\Psi f(\PhVar):= \langle f, \psi_\PhVar \rangle_{\Hil} \text{ for all } \PhVar\in\PhSpace.
\end{equation}
The adjoint of the voice transform is given by
\begin{equation}
  V_\Psi^\ast~:~ \bd L^2(\PhSpace) \rightarrow \Hil,\quad V_\Psi^\ast G
  = \int_\PhSpace G(\PhVar) \, \psi_\PhVar d\mu(\PhVar),
\end{equation}
with the integral understood in the weak sense (see \cite[Page~43]{gr01}).
Finally, the \emph{frame operator} of $\Psi$ is given by
$\bd{S}_\Psi := V_\Psi^\ast \circ V_\Psi : \Hil \to \Hil$, so that
\[
  \bd S_\Psi f = \int_{\PhSpace}
                   \langle f,\psi_{\PhVar} \rangle_{\Hil} \psi_{\PhVar}
                 d\mu(\PhVar).
\]
It follows from \eqref{eq:frameeq} that $\bd{S}_{\Psi} f = A \cdot f$ for all $f \in \CalH$;
see \cite{ch16}.

Essentially all of coorbit theory is based on certain regularity properties of the
\emph{reproducing kernel} $K_{\Psi}$ associated to the continuous frame $\Psi$.
It is given by
\nicki{\begin{equation}
  K_\Psi : \quad
  \PhSpace \times \PhSpace \to \CC, \quad
  (\PhVar, \PhVarA) \mapsto A^{-1}\langle \psi_{\PhVarA}, \psi_{\PhVar} \rangle_{\Hil} .
  \label{eq:ReproducingKernel}
\end{equation}
Without loss of generality, we will henceforth assume that $A=1$, i.e., $\Psi$ is a Parseval frame.}
We remark that $K_\Psi$ is measurable with respect to the product $\sigma$-algebra.
Indeed, since $\Hil$ is separable, we can choose a countable orthonormal basis
$(\eta_j)_{j \in J} \subset \Hil$, so that
\(
  K_\Psi (\PhVar,\PhVarA)
  = \sum_{j \in J}
      \langle \psi_{\PhVarA}, \eta_j \rangle_{\Hil}
      \langle \eta_j, \psi_{\PhVar} \rangle_{\Hil}
\)
is seen to be measurable as a convergent, countable series of measurable functions.

A \emph{(discrete) frame} for $\Hil$ is a countable family $\Psi_d = (\psi_j)_{j \in J} \subset \Hil$
for which there exist $0 \!<\! A \leq B \!<\! \infty$ such that
\begin{equation}
  A \cdot \| f \|_{\Hil}^2
  \leq \sum_{j \in J}
         |\langle f, \psi_j \rangle_{\Hil}|^2
  \leq B \cdot \| f \|_{\Hil}^2
  \quad \text{for all } f \in \CalH.
  \label{eq:IntroductionDiscreteFrameDefinition}
\end{equation}
This implies (cf.~\cite{ch16} for details) that every $f \in \Hil$
can be expanded with respect to $\Psi_d$;
that is, for each $f \in \Hil$ there exists a sequence $(c_j)_{j\in J}\in\ell^2(J)$ such that
\begin{equation}
   f = \sum_{j \in J} c_j \, \psi_j.
  \label{eq:IntroductionFrameDecomposition}
\end{equation}

\subsection{Banach frame decompositions}
\label{ssec:discframes}

When the Hilbert space $\Hil$ is exchanged for a Banach space $(\BanachTD,\|\bullet\|_\BanachTD)$,
and $\ell^2 (J)$ is replaced by a suitable sequence space $B^{\flat} \subset \CC^J$,
then validity of the (modified) frame inequality
\(
  \big\| \bigl(\langle f, \psi_j \rangle_{\BanachTD, \BanachTD'}\bigr)_{j \in J} \big\|_{B^{\flat}}
  \asymp \| f \|_{\BanachTD}
\)
does \emph{not} necessarily imply a statement similar to \eqref{eq:IntroductionFrameDecomposition}
(among other things because in general $\psi_j \in B'$ and not $\psi_j \in B$).
Therefore, the dual concepts of Banach frames and atomic decompositions have been introduced;
see \cite{gr91,fegr89,fegr89-1}.
To reduce the number of required definitions, in this article we only consider the
combined concept of a \emph{Banach frame decomposition},
which unifies both concepts, under some mild assumptions on $B$ that allow one to make sense
of the intersection $B \cap B'$; see Appendix~\ref{sec:GelfandTripleAppendix} for details.


%

\begin{definition}\label{def:BFD}
  Let $(\BanachTD,\|\bullet\|_\BanachTD)$ be a Banach space.
  A family $\Psi_d = (\psi_j)_{j \in J} \subset \BanachTD \cap \BanachTD'$
  is called a \emph{Banach frame decomposition} for $\BanachTD$
  if there exist a dual family $E_d = (e_j)_{j \in J} \subset \BanachTD \cap \BanachTD'$
  and solid, rich Banach sequence spaces $(\BanachTD^\sharp, \|\bullet\|_{\BanachTD^\sharp})$
  and $(\BanachTD^\flat,\|\bullet\|_{\BanachTD^\flat})$ over $J$,
  i.e., $\BanachTD^\sharp,\BanachTD^\flat\subset\CC^J$, with the following properties:
  \setlength{\leftmargini}{0.7cm}
  \begin{itemize}
    \item The \emph{coefficient operators}
          \[
           \begin{split}
            \qquad
            &\CalC_{\Psi_d} : \quad
            \BanachTD \to \BanachTD^{\flat}, \quad
            f \mapsto \big( \langle f, \psi_j \rangle_{\BanachTD, \BanachTD'} \big)_{j \in J}
            \qquad \text{and} \\
            &\CalC_{E_d} : \quad
            \BanachTD \to \BanachTD^{\sharp}, \quad
            f \mapsto \big( \langle f, e_j \rangle_{\BanachTD, \BanachTD'} \big)_{j \in J}
            \end{split}
          \]
          are well-defined and bounded.
          \vspace*{0.1cm}

    \item The \emph{reconstruction operators}
          \[
            \CalR_{\Psi_d} : \quad
            \BanachTD^{\sharp} \to \BanachTD, \quad
            (c_j)_{j \in J} \mapsto \sum_{j \in J} c_j \, \psi_j
            \quad \text{and} \quad
            \CalR_{E_d} : \quad
            \BanachTD^{\flat} \to \BanachTD, \quad
            (c_j)_{j \in J} \mapsto \sum_{j \in J} c_j \, e_j
          \]
          are well-defined and bounded, with unconditional convergence of the defining series
          in a suitable topology.
          \vspace*{0.1cm}

    \item We have
          \(
            \CalR_{\Psi_d} \circ \CalC_{E_d}
            = \identity_{\BanachTD}
            = \CalR_{E_d} \circ \CalC_{\Psi_d} ,
          \)
          or in other words
          \[
            f
            = \sum_{j \in J}
                \langle f, e_j \rangle_{\BanachTD, \BanachTD'} \,\,
                \psi_j
            = \sum_{j \in J}
                \langle f, \psi_j \rangle_{\BanachTD, \BanachTD'} \,\,
                e_j
            \qquad \text{for all } f \in \BanachTD .
          \]
  \end{itemize}
\end{definition}

\begin{remark}
  In some recent works, atomic decompositions of Banach spaces are defined
  by a pair of systems $(\Psi_d,\widetilde{\Psi_d})$, with $\Psi_d\in B'$ providing the analysis,
  and $\widetilde{\Psi_d}\in B$ the synthesis operation, e.g., \cite[Definition 24.3.1]{ch16}.
  In that sense, Definition~\ref{def:BFD} is not dissimilar to stating that
  both $(\Psi_d,E_d)$ and $(E_d,\Psi_d)$ are atomic decompositions of $B$.
  Nonetheless, a Banach frame decomposition, which implies the existence
  of a class of test functions embedded into $B$ and $B'$, is distinct,
  since it places additional assumptions on the sequence spaces $\BanachTD^\sharp,\BanachTD^\flat$
  on which the reconstruction operators are further required to be unconditionally convergent. 
\end{remark}

\subsection{Coverings and weight functions}
\label{ssec:coverings}

For applying the discretization results of (general) coorbit theory,
we will have to construct special coverings of the phase space $\PhSpace = \R^d \times D$.
To allow for a more streamlined development later on,
the present subsection discusses the required properties of these coverings.
The most basic of these properties are admissibility and moderateness.

\begin{definition}\label{def:admissibility}
  Let $\CalO \neq \emptyset$ be a set.
  A family $\CalV = (V_j)_{j \in J}$ of non-empty subsets of $\CalO$
  is called an \emph{admissible covering} of $\CalO$,
  if we have $\CalO = \bigcup_{j \in J} V_j$ and if
  \begin{equation}
    \CalN(\CalV)
    := \sup_{j \in J} |j^\ast| < \infty
    \qquad \text{where} \qquad
    j^\ast := \{i \in J \,:\, V_i \cap V_j \neq \emptyset \}
    \quad \text{for } j \in J.
    \label{eq:IntersectionNumberDefinition}
  \end{equation}

  If $\CalO$ is a topological space, we say that a family $\CalV$
  as above is \emph{topologically admissible}
  if it is admissible and if each $V_j \subset \CalO$
  is open and relatively compact (i.e., $\overline{V_j} \subset \CalO$ is compact).  
\end{definition}

\begin{rem*}
  We remark that every topologically admissible covering is locally finite:
  Given $x \in \CalO$, we have $x \in V_{j_0}$ for some $j_0 \in J$.
  Since $V_{j_0}$ is open and since $V_j \cap V_{j_0} \neq \emptyset$
  can only hold for $i \in j_0^\ast$ with $j_0^\ast \subset J$ finite,
  we see that $\CalV$ is indeed a locally finite covering.
\end{rem*}

In the special case where $\CalO = \PhSpace$ has a product structure,
we will also use the following class of coverings.


\begin{definition}\label{def:ProductAdmissibleCovering}(\cite[Def.~2.12]{hovo20_algebra})
  Let $\PhSpace = \PhSpace_1 \times \PhSpace_2$, where each $\PhSpace_j$ is equipped
  with a measure $\mu_j$ and $\mu = \mu_1 \otimes \mu_2$.
  We say that a family $\CalU = (U_j)_{j \in J}$ is a \emph{product-admissible covering}
  of $\PhSpace$, if it satisfies the following:
  $J$ is countable, $\PhSpace = \bigcup_{j \in J} U_j$,
  each $U_j$ is non-empty and of the form $U_j = U_{1,j} \times U_{2,j}$
  with $U_{\ell,j} \subset \PhSpace_\ell$ open,
  and there is a constant $C > 0$ such that the \emph{covering weight} $w_{\CalU}$ defined by
  \begin{equation}
    (w_{\CalU})_j
    := \min \big\{
              1, \,\,
              \mu_1 (U_{1,j}), \,\,
              \mu_2 (U_{2,j}), \,\,
              \mu(U_j)
            \big\}
    \qquad \text{for } j \in J
    \label{eq:CoveringWeightDefinition}
  \end{equation}
  satisfies $(w_{\CalU})_j \leq C \cdot (w_{\CalU})_\ell$ for all $j,\ell \in J$ with
  $U_j \cap U_\ell \neq \emptyset$.

  Given a product-admissible covering $\CalU = (U_j)_{j \in J}$
  and a measurable function $u : \PhSpace \to \R^+$, we say that $u$ is \emph{$\CalU$-moderate}
  if there is a constant $C' > 0$, such that $u(\PhVar) \leq C' \cdot u(\PhVarA)$
  for all $j \in J$ and all $\PhVar, \PhVarA \in U_j$.
\end{definition}

If $\CalU = (U_j)_{j \in J}$ is a product-admissible covering of $\PhSpace$,
then with $w_{\CalU}$ as defined in \eqref{eq:CoveringWeightDefinition},
it is easy to see that there exists a measurable function $w_{\CalU}^c : \PhSpace \to \R^+$
such that
\begin{equation}
  (w_{\CalU})_j \asymp w_{\CalU}^c (\lambda)
  \quad \text{for all } j \in J \text{ and } \lambda \in U_j .
 \label{eq:ContinuousCoveringWeightCondition}
\end{equation}
Furthermore, any two such weights $w_{\CalU}^c, \widetilde{w_{\CalU}^c}$
satisfy $w_{\CalU}^c \asymp \widetilde{w_{\CalU}^c}$.
We refer to \cite[Theorem~2.13]{hovo20_algebra} for the details.

\medskip{}

In addition to such coverings, the study of specific coorbit spaces and their
properties relies on certain weighted spaces that are compatible with the given
coverings in a suitable way.
The following classes of weight functions are of particular importance.

\begin{definition}\label{def:weights}
  \begin{enumerate}
    \item Any measurable function $v : \CalO \to \RR^+$ on a measurable space $\CalO$ will be called
          a \emph{weight}, or a \emph{weight function}.
    \item A weight $m: \CalO \times \CalO \rightarrow \RR^+$ is called \emph{symmetric} if
          $m(\PhVar, \PhVarA) = m(\PhVarA, \PhVar)$ for all $\PhVar,\PhVarA \in \CalO$.
    \item Given any weight $v: \CalO \rightarrow \RR^+$, the \emph{associated weight}
          $m_v : \CalO \times \CalO \to \RR^+$ is defined by
          \begin{equation}\label{eq:vDerivedWeight}
               m_v(\PhVar,\PhVarA)
               := \max\left\{\frac{v(\PhVar)}{v(\PhVarA)},\frac{v(\PhVarA)}{v(\PhVar)}\right\},
               \quad \text{ for all } \PhVar,\PhVarA\in \CalO .
          \end{equation}
    \item A weight function $v$ on $\RR^d$ is called \emph{submultiplicative}, if
          \[
            v(\PhVar + \PhVarA)
            \leq v(\PhVar) \cdot v(\PhVarA),
            \quad \text{ for all } \PhVar, \PhVarA \in \RR^d.
          \]
          Given such a submultiplicative weight $v$, another weight function
          $\widetilde{v} : \RR^d \rightarrow \RR^+$ is called \emph{$v$-moderate} if
          \begin{equation}\label{eq:moderateWeight}
           \widetilde{v}(\PhVar + \PhVarA)
           \leq v(\PhVar) \cdot \widetilde{v}(\PhVarA),
           \quad \text{ for all } \PhVar, \PhVarA \in \RR^d.
          \end{equation}
    \item We say that a weight $v$ on $\RR^d$ is \emph{radially increasing}
          if $v(\PhVar) \leq v(\PhVarA)$ whenever $\PhVar, \PhVarA \in \RR^d$ with $|\PhVar| \leq |\PhVarA|$.
          This in particular implies that $v(\PhVar)$ only depends on $|\PhVar|$, so that we identify
          $v$ with a weight on $[0,\infty)$ and write $v(\PhVar) = v(|\PhVar|)$.
  \end{enumerate}
\end{definition}

\begin{remark}\label{rem:InverseMaxMinOfNoderateWeights}
  If $v_1,v_2$ are $v_0$-moderate weights and $v_0(\lambda) = v_0(-\lambda)$ for all $\lambda \in \R^d$,
  then a simple derivation shows that $1/v_1$, $\max \{ v_1,v_2 \}$, and $\min \{ v_1,v_2 \}$ are $v_0$-moderate as well.
\end{remark}

\subsection{Kernel spaces}%
\label{sub:KernelSpaces}

The main prerequisite of general coorbit theory is that the reproducing kernel $K_\Psi$---%
and some additional kernels derived from it---must satisfy appropriate decay conditions.
These are formulated in terms of certain \emph{Banach spaces of integral kernels}
that we review in this subsection.

Let $(\Lambda,\mu)$ be a $\sigma$-finite measure space.
Recall from Section~\ref{subsub:IntegralKernels} that a kernel is any measurable map
$K : \Lambda \times \Lambda \to \CC$.
Given such a kernel and a symmetric weight $m$ on $\Lambda \times \Lambda$, we define
${\| K \|_{\AAm(\Lambda)} := \| K \|_{\AAm}}$, where
\begin{equation}\label{eq:normA1}
  \|K\|_{\AAm}
  := \max \left\{
            \esssup_{\PhVarA \in \Lambda}
              \int_\Lambda
                \bigl| m(\PhVarA, \PhVar) \cdot K(\PhVarA,\PhVar) \bigr|
              ~d\mu(\PhVar),
            \quad
            \esssup_{\PhVar \in \Lambda}
              \int_\Lambda
                \bigl| m(\PhVarA, \PhVar) \cdot K(\PhVarA,\PhVar) \bigr|
              ~d\mu(\PhVarA)
          \right\} ,
\end{equation}
and we define
\(
  \AAm
  := \AAm(\Lambda)
  := \big\{
       K : \Lambda \times \Lambda \to \CC
       \, \colon \,
       K \text{ measurable and } \| K \|_{\AAm} < \infty
     \big\}
  .
\)
In the case where $m \equiv 1$, we use the notation $\AAi$.

For most applications, it is not enough to know that $K_\Psi \in \AAm$;
rather, it is required that the integral operator associated to $K_\Psi$ or $|K_\Psi|$
(defined in Equation~\eqref{eq:IntegralOperator}) acts boundedly on a given solid
Banach space $\BanachOne \subset \lebesgue_{\loc}^1 (\Lambda)$.
Precisely, given a kernel $K : \Lambda \times \Lambda \to \CC$,
we set $\big\|\, |K| \,\big\|_{\BanachOne \to \BanachOne} := \infty$
if the integral operator associated to $|K|$ does \emph{not} define a bounded linear map
on $\BanachOne$; otherwise, we denote by $\big\|\, |K| \,\big\|_{\BanachOne \to \BanachOne}$
the operator norm of this integral operator.
With this convention, we define
\[
  \AAmY := \big\{
             K \in \AAm
             \,\,\colon\,\,
             \big\|\, |K| \,\|_{\BanachOne \to \BanachOne} < \infty
           \big\} ,
  \quad \text{with norm} \quad
  \| K \|_{\AAmY}
  := \max \big\{
            \| K \|_{\AAm},
            \big\|\, |K| \,\big\|_{\BanachOne \to \BanachOne}
          \big\} .
\]

\begin{remark}(cf.~\cite[Lemma 2.45]{kempka2015general})
  If $K$ is measurable and if $|K|$ induces a bounded operator $\BanachOne \to \BanachOne$,
  then so does $K$ itself, since $\BanachOne$ is solid.
  A similar argument shows that $\AAmY$ is a solid space of kernels\nicki{: 
  Let $K,L$ be measurable with $K \in \AAmY$ and $|L| \leq |K|$ almost everywhere
  (with respect to the product measure).
  Then, for $\mu$-almost every $\PhVar\in\PhSpace$,
  $|L(\PhVar,\bullet)| \leq |K(\PhVar,\bullet)|$ $\mu$-almost everywhere, implying 
  \[
    ||L|(F)(\lambda)|
    \leq |L|(|F|)(\lambda)
    \leq |K|(|F|)(\lambda)
    \quad \text{$\mu$-almost everywhere.}
  \]
  Noting that $\|F\|_Y = \||F|\|_Y$ due to solidity of $Y$,
  the first inequality implies that to determine the operator norm $\||L|\|_{Y\rightarrow Y}$,
  it suffices to consider nonnegative functions $F\in Y$.
  On the other hand, for such functions, the second inequality implies
  $\||L|(F)\|_Y \leq \||K|(F)\|_Y$, by solidity of $Y$.
  Hence, we have established $\||L|\|_{Y\rightarrow Y} \leq \||K|\|_{Y\rightarrow Y}$,
  and therefore $\||L|\|_{\AAmY} \leq \||K|\|_{\AAmY}$ follows with solidity of $\AAm$,
  which is clear from the definition.
  }

  Finally, we remark that our definition of $\AAmY$ is different from the
  definition in \cite[Section 2.4]{kempka2015general} in that we take the
  norm $\big\| \, |K| \, \big\|_{\BanachOne \to \BanachOne}$
  instead of $\| K \|_{\BanachOne \to \BanachOne}$.
  Nevertheless, if a kernel $K$ satisfies $K \in \AAmY$ with our definition,
  it also satisfies $K \in \AAmY$ according to the definition in
  \cite[Section~2.4]{kempka2015general},
  so that the slightly different definition will not cause problems.
\end{remark}

For applications of coorbit theory, one has to verify $K_\Psi \in \AAmY$ for the
space $\BanachOne$ of interest and a certain weight $m$.
In many cases, it turns out to be easier to verify $K_\Psi \in \BBmo$,
where $\BBmo$ is a smaller space of kernels that satisfies $\BBmo \hookrightarrow \AAmY$,
possibly with $m_0=m$. Precisely, since we are mostly interested in the product setting of kernels on
$\PhSpace = \PhSpace_1 \times \PhSpace_2$, we will use the following spaces $\BBm$
introduced in \cite{hovo20_algebra}.

\begin{definition}\label{def:NewKernelModule}
  Let $(\PhSpace,\mu) = (\PhSpace_1 \times \PhSpace_2, \mu_1 \otimes \mu_2)$,
  where $(\PhSpace_1, \mu_1), (\PhSpace_2, \mu_2)$ are $\sigma$-finite measure spaces.
  Given a kernel $K : \PhSpace \times \PhSpace \to \CC$, we define
  \begin{equation}
    K^{(\PhVar_2,\PhVarA_2)} (\PhVar_1,\PhVarA_1) := K ( \PhVar,\PhVarA )
    \quad \text{ for } \quad
    \PhVar = (\PhVar_1,\PhVar_2), \PhVarA = (\PhVarA_1,\PhVarA_2) \in \PhSpace.
    \label{eq:PartialKernelDefinition}
  \end{equation}
  Using this notation, we define
  \[
    \| K \|_{\BBi}
    := \| K \|_{\BBi(\PhSpace)}
    := \Big\|
         (\PhVar_2,\PhVarA_2) \mapsto \big\| K^{(\PhVar_2,\PhVarA_2)} \big\|_{\AAi(\PhSpace_1)}
       \Big\|_{\AAi(\PhSpace_2)}
    \in [0,\infty] ,
  \]
  and
  \(
    \BBi
    := \BBi(\PhSpace)
    := \big\{
         K : \PhSpace \times \PhSpace \to \CC \,
         \colon
         K \text{ measurable and } \| K \|_{\BBi} < \infty
       \big\}.
  \)
  Finally, given a symmetric weight $m : \PhSpace \times \PhSpace \to \RR^+$, we define
  \(
    \BBm := \BBm(\PhSpace)
    := \big\{ K : \PhSpace \times \PhSpace \to \CC \, \colon m \cdot K \in \BBi \big\},
  \)
  with norm $\| K \|_{\BBm} := \| m \cdot K \|_{\BBi}$.
\end{definition}

As shown in \cite[Propositions 2.5 and 2.6]{hovo20_algebra}, $\BBm$ is a solid Banach space
of integral kernels that satisfies $\| K^T \|_{\BBm} = \| K \|_{\BBm}$ and furthermore
$\| K \|_{\AAm} \leq \| K \|_{\BBm}$ for every kernel $K$.
If the weight $m$ additionally satisfies $m(x,z) \leq C m(x,y)m(y,z)$,
for all $x,y,z\in\Lambda$ and some $C>0$, then it is easy to see that $\AAm,\BBm$
are algebrae with respect to the standard kernel product, defined by  
\[
   K_1\cdot K_2
   = \int_\Lambda K_1(\bullet_1,\lambda)K_2(\lambda,\bullet_2)~d\mu(\lambda).
\]
Most importantly for us, the integral operators associated to kernels in $\BB_{m_\kappa}$
act boundedly on the mixed-norm Lebesgue spaces $\lebesgue^{p,q}_\kappa (\PhSpace)$;
see the following proposition.

\begin{proposition}\label{prop:NewAlgebraGoodForMixedLebesgue}(see \cite[Proposition~2.7]{hovo20_algebra})
  Let $\PhSpace$ as in Definition~\ref{def:NewKernelModule}, let $\kappa$ be a weight
  on $\PhSpace$, and let $m_\kappa : \PhSpace \times \PhSpace \to \R^+$
  be as in Equation~\eqref{eq:vDerivedWeight}.
  Then, for each kernel $K \in \BB_{m_\kappa}(\PhSpace)$ and arbitrary $p,q \in [1,\infty]$,
  the associated integral operator $K(\bullet)$ defined in Equation~\eqref{eq:IntegralOperator}
  restricts to a bounded linear operator
  $K (\bullet) : \lebesgue^{p,q}_\kappa (\PhSpace) \to \lebesgue^{p,q}_\kappa (\PhSpace)$,
  with absolute convergence almost everywhere of the defining integral, and with
  \begin{equation}\label{eq:BBmBoundedOnLPQ}
    \| K(F) \|_{\lebesgue^{p,q}_\kappa(\PhSpace)}
    \leq \| K \|_{\mathcal{B}_{m_\kappa}}
         \cdot \| F \|_{\lebesgue^{p,q}_\kappa(\PhSpace)}
    \qquad \forall \, F \in \lebesgue^{p,q}_{\kappa}(\PhSpace).
  \end{equation}
  In particular, this implies for $\BanachOne = \lebesgue^{p,q}_\kappa(\PhSpace)$
  and any (symmetric) weight $m$ with $m \geq m_\kappa$ that $\| K \|_{\AAmY} \leq \| K \|_{\BBm}$.
\end{proposition}

\subsection{General coorbit spaces}
\label{ssec:coorbit}

In this subsection, we give a brief crash-course to general coorbit theory.
Our treatment is essentially based on \cite{kempka2015general}, but incorporates
additional simplifications (from \cite{hovo20_algebra}) that are
on the one hand due to using the kernel space $\BBm$ instead of $\AAmY$,
and on the other hand due to imposing slightly more restrictive
assumptions than in \cite{kempka2015general}.
For the warped time-frequency systems that we consider, these assumptions are automatically satisfied, 
justifying this restriction. 

To formulate our assumptions for the applicability of coorbit theory,
we need one final ingredient.

\begin{definition}\label{def:maxkern}
Let $\CalV = (V_j)_{j \in J}$ be an arbitrary open covering of $\PhSpace = \R^d \times D$. 
The \emph{maximal kernel} $\mathrm{M}_{\CalV} K$ associated to
a given kernel $K : \PhSpace \times \PhSpace \to \CC$, given by
\begin{equation}
  \mathrm{M}_{\CalV} K : \quad
  \PhSpace \times \PhSpace \to [0,\infty], \quad
  (\PhVar,\PhVarA) \mapsto \sup_{\PhVarC \in \CalV_{\PhVar}} |K(\PhVarC, \PhVarA)|
  \quad \text{where} \quad
  \CalV_{\PhVar} := \bigcup_{j \in J \text{ with } \PhVar \in V_j} V_j .
  \label{eq:def_maxkern}
\end{equation}
\end{definition}

In what follows, we shall always work in the following setting:

\begin{assumption}\label{assu:CoorbitAssumptions1}
  Let $D \subset \R^d$ be open, and let $\PhSpace = \R^d \times D$,
  equipped with the Borel $\sigma$-algebra and the Lebesgue measure $\mu$.
  We assume that
  \begin{enumerate}[itemsep=0.3em]

    \item \label{enu:CoorbitProductAdmissible}
          $\CalU = (U_j)_{j \in J}$ is a product-admissible covering of $\PhSpace$;

    \item \label{enu:CoorbitUModerateness}
          $u : \PhSpace \to \R^+$ is continuous and $\CalU$-moderate;

    \item \label{enu:CoorbitKernelWeightAssumption}
          $m_0 : \PhSpace \times \PhSpace \to \R^+$ is continuous and symmetric
          and satisfies $m_0 (\PhVar, \PhVarA) \leq C^{(0)} \cdot u(\PhVar) \, u(\PhVarA)$
          for all $\PhVar, \PhVarA \in \PhSpace$ and some $C^{(0)} > 0$;

    \item \label{enu:CoorbitParseval}
          $\Psi = (\psi_\PhVar)_{\PhVar \in \PhSpace}$ is a continuous Parseval frame for $\LtDF$,
          and the map $\PhSpace \to \lebesgue^2(\R^d), \PhVar \mapsto \psi_\PhVar$
          is continuous;
    
    \item \label{enu:CoorbitVCondition}
          $v : \PhSpace \to [1,\infty)$ is continuous and satisfies
          \(
            v(\PhVar)
            \geq c \cdot
                 \max
                 \big\{
                    \| \psi_{\PhVar} \|_{\lebesgue^2}, \,\,
                    u(\PhVar) / w_{\CalU}^c (\PhVar)
                 \big\}
          \)
          for some $c > 0$ and all $\PhVar \in \PhSpace$, with $w_{\CalU}^c$ as in
          Equation~\eqref{eq:ContinuousCoveringWeightCondition};
          
    \item \label{enu:CoorbitBanachSpaceCondition}
          $\BanachOne \subset \lebesgue_{\loc}^1(\PhSpace)$ is a rich, solid Banach space 
          such that
          $\| K(\bullet) \|_{\BanachOne \to \BanachOne} \leq \| K \|_{\BBmo}$
          for all $K \in \BBmo$;

    \item \label{enu:CoorbitKernelLocalization}
          The kernel $K_\Psi$ defined in Equation~\eqref{eq:ReproducingKernel} satisfies
          \begin{equation}
            K_\Psi \in \AAPlain_{m_v}
            \quad \text{ and } \quad
            \MaxKernel{\CalU} K_\Psi \in \BB_{m_0} .
            \label{eq:CoorbitKernelCondition}
          \end{equation}
          with $m_v$ as defined in Equation~\eqref{eq:vDerivedWeight}.
  \end{enumerate}
\end{assumption}

By Proposition~\ref{prop:NewAlgebraGoodForMixedLebesgue},
Condition~\eqref{enu:CoorbitBanachSpaceCondition} is satisfied
for $\BanachOne = \lebesgue^{p,q}_\kappa (\PhSpace)$, as long as
$\frac{\kappa(\PhVar)}{\kappa(\PhVarA)} \leq m_0(\PhVar, \PhVarA)$
for all $\PhVar,\PhVarA \in \PhSpace$.

\begin{remark}\label{rem:MaximalKernelContinuity}
  If the kernel $K$ is continuous in the second component (as is the case for the reproducing
  kernel $K_\Psi$, under the conditions in Assumption~\ref{assu:CoorbitAssumptions1} below),
  then $\MaxKernel{\UU} K$ is lower semicontinuous and hence measurable.
  To see this, let $\alpha \in \R$ and $(\lambda_0, \rho_0) \in \Lambda \times \Lambda$ with
  $\MaxKernel{\CalU} K (\lambda_0, \rho_0) > \alpha$.
  Then there are $j \in J$ with $\lambda_0 \in U_j$ and some $\nu \in U_j$
  such that $|K(\nu,\rho_0)| > \alpha$.
  By continuity of $K(\nu,\bullet)$, there is thus an open set $V \subset \Lambda$
  with $\rho_0 \in V$ and such that $|K(\nu,\rho)| > \alpha$ for all $\rho \in V$.
  Overall, we see for $(\lambda,\rho) \in U_j \times V$ that
  $\MaxKernel{\CalU} K (\lambda,\rho) \geq |K(\nu,\rho)| > \alpha$.
  Since $\CalU$ is a product-admissible covering, $U_j$ is open;
  thus, we have shown that $\MaxKernel{\CalU} K$ is indeed lower semicontinuous.
\end{remark}


The next theorem shows that the conditions in Assumption~\ref{assu:CoorbitAssumptions1}
ensure that one can extend the voice transform to a suitably defined space of distributions.

\begin{theorem}\label{thm:resanddual}
  Under Assumption~\ref{assu:CoorbitAssumptions1}, the following hold:
  The space
  \begin{equation}\label{eq:resspace}
    \mathcal H^1_{v}
    := \mathcal H^1_{v}(\Psi)
    := \big\{
         f\in\LtDF
         ~:~
         V_\Psi f\in \bd L^1_{v}
       \big\},
       \text{ with the norm }
       \|f\|_{\mathcal H^1_{v}} := \|V_\Psi f\|_{\bd L^1_{v}},
  \end{equation}
  is a Banach space satisfying $\mathcal{H}_v^1 \hookrightarrow \LtDF$, with dense image.
  Furthermore, there is some $C' > 0$ such that
  $\|\psi_{\PhVar}\|_{\mathcal{H}_v^1} \leq C' \cdot v(\PhVar) < \infty$ for all $\PhVar\in\PhSpace$.
  In fact, $\GoodVectors$ is the minimal Banach space with that property.

  Finally, for each $f \in \Reservoir$, the \emph{extended voice transform}
  \begin{equation}
    V_\Psi f : \quad
    \Lambda \to \CC, \quad
    \lambda \mapsto \langle f, \psi_\lambda \rangle_{\Reservoir, \GoodVectors} = f(\psi_\lambda)
    \label{eq:ExtendedVoiceTransform}
  \end{equation}
  satisfies $V_\Psi f \in \lebesgue_{1/v}^\infty (\Lambda)$.
  In fact, the expression $\| V_\Psi f \|_{\lebesgue_{1/v}^\infty}$ defines an equivalent
  norm on $\Reservoir$.
\end{theorem}

\begin{proof}
  Define $C' := \| K_\Psi \|_{\AAPlain_{m_v}}$.
  Then, \cite[Lemma~2.13]{kempka2015general} shows that 
  \(
    \| \psi_\PhVar \|_{\mathcal H^1_v}
    \leq C' \cdot v(\PhVar)
  \)
  holds for all $\PhVar \in \PhSpace \setminus N$, if $\PhVar \mapsto \psi_\PhVar$ is weakly measurable. 
  If 
  $\PhVar \mapsto \psi_\PhVar$ and $v$ are continuous,
  their proof is easily seen to hold pointwise for all $\PhVar \in \PhSpace$
  and hence $\Psi \subset \GoodVectors$. 
  Since $\Psi$ is a continuous frame for $\LtDF$, this in particular implies that
  $\GoodVectors \subset \LtDF$ is dense. 
  The completeness of $(\GoodVectors, \| \bullet \|_{\GoodVectors})$
  and the continuity of the embedding $\GoodVectors \hookrightarrow \LtDF$ follow from
  \cite[Lemma~8.1]{hovo20_algebra}.
  The minimality property of $\mathcal{H}_v^1$ is shown in \cite[Corollary~1]{fora05}.

  \smallskip{}

  For $\varphi \in \Reservoir$, \cite[Lemma~8.1]{hovo20_algebra} shows that $V_\Psi \varphi$
  is measurable with respect to the \emph{Lebesgue $\sigma$-algebra},
  and that $\varphi \mapsto \| V_\Psi \varphi \|_{\lebesgue_{1/v}^\infty}$
  defines an equivalent norm on $\Reservoir$.
  Thus, we only show that $V_\Psi \varphi$ is in fact measurable
  with respect to the \emph{Borel $\sigma$-algebra}.
  To see this, define $W := \{ V_\Psi f \colon f \in \GoodVectors \} \subset \lebesgue_v^1 (\PhSpace)$
  and $\gamma : W \to \CC, V_\Psi f \mapsto \overline{\varphi(f)}$, noting that this is a
  well-defined, bounded linear functional since
  \(
    |\gamma(V_\Psi f)|
    = |\varphi(f)|
    \leq C \cdot \| f \|_{\GoodVectors}
    = C \cdot \| V_\Psi f \|_{\lebesgue_v^1}
    .
  \)
  By combining the Hahn-Banach theorem with the characterization of the dual of $\lebesgue_v^1(\PhSpace)$,
  we thus see that there exists $G \in \lebesgue_{1/v}^\infty (\PhSpace)$ satisfying
  \[
    V_\Psi \varphi (\PhVar)
    = \varphi(\psi_{\PhVar})
    = \overline{\gamma(V_\Psi \psi_{\PhVar})}
    = \overline{
        \int_{\PhSpace}
          G(\PhVarA) V_\Psi \psi_\PhVar (\PhVarA)
        \, d \PhVarA
      } .
  \]
  Now, since
  \(
    (\PhVar,\PhVarA) \mapsto V_\Psi \psi_\lambda (\rho)
    = \langle \psi_{\PhVar}, \psi_{\PhVarA} \rangle
    = K_\Psi (\PhVarA,\PhVar)
  \)
  is measurable and since $G \in \lebesgue_{1/v}^\infty$ and $V_\Psi \psi_\lambda \in \lebesgue_v^1$
  (as shown above), the measurability of $V_\Psi \varphi$ is an easy consequence of the
  Fubini-Tonelli theorem (see \cite[Proposition~5.2.1]{CohnMeasureTheory}).
\end{proof}

Now that we have constructed the ``reservoir'' $\Reservoir$, we can use it to define the
coorbit space associated to the frame $\Psi$ and a solid Banach space $\BanachOne$.

\begin{theorem}\label{thm:coorbits}
  Suppose that Assumption~\ref{assu:CoorbitAssumptions1} is satisfied.
  Then the \emph{coorbit of $\BanachOne$ with respect to $\Psi$},
  \begin{equation}\label{eq:coorbits}
    \Co \BanachOne
    := \Co (\Psi,\BanachOne)
    := \bigl\{f\in \Reservoir~:~ V_\Psi f\in \BanachOne\bigr\},
  \end{equation}
  is a Banach space with natural norm $\|f\|_{\Co \BanachOne} := \|V_\Psi f\|_\BanachOne$.

  Additionally, for any $G\in \BanachOne$, the property $G=K_\Psi(G)$ is equivalent to
  $G = V_\Psi f$ for some $f\in\Co \BanachOne$.
  The map $V_\Psi:\Co \BanachOne \rightarrow \BanachOne$ is an isometry of $\Co \BanachOne$ onto the
  closed subspace $K_\Psi(\BanachOne)$ of $\BanachOne$.
  Finally, the inclusion $\Co \BanachOne \hookrightarrow \Reservoir$
  is continuous.
\end{theorem}

\begin{proof}
  This follows from \cite[Proposition~8.6]{hovo20_algebra} together with
  \cite[Sections~2.3 and 2.4]{kempka2015general}.
\end{proof}

Note that the definition of $\Co \BanachOne$ is independent of the weight $v$
in the following sense:
If $\widetilde{v}$ is another weight such that Assumption~\ref{assu:CoorbitAssumptions1} holds,
then \eqref{eq:coorbits} defines the same space,
see \cite[Lemma~2.26]{kempka2015general}.
Furthermore, according to \cite[Lemma 2.32]{kempka2015general},
we have the following special cases:
\[
  \Co \bd L^1_v
  = \mathcal{H}^1_v,\quad \Co \bd L^\infty_{1/v}
  = \Reservoir
  \quad \text{ and } \quad
  \Co \bd L^2 = \bd L^2.
\]


The coorbit spaces $(\Co \BanachOne,\|\bullet\|_{\Co \BanachOne})$ are independent of
the particular choice of the continuous frame $\Psi$, under a certain
equivalence condition on the mixed kernel associated to a pair of continuous Parseval frames.

\begin{proposition}\label{pro:mixedkern}
  If $\Psi$ and $\widetilde{\Psi}$ are continuous Parseval frames for $\LtDF$
  such that Assumption~\ref{assu:CoorbitAssumptions1} is satisfied for $\Psi$
  and also for $\widetilde{\Psi}$, and if
  ${K_{\Psi,\widetilde{\Psi}},K_{\widetilde{\Psi},\Psi}\in\AAPlain_{m_v} \cap \BB_{m_0}}$, where
  $K_{\Psi,\widetilde{\Psi}}$ is the \emph{mixed} kernel defined by
  \begin{equation}
    K_{\Psi,\widetilde{\Psi}}(\PhVar,\PhVarA)
    := \big\langle
         \widetilde{\psi_{\PhVarA}},\psi_{\PhVar}
       \big\rangle
    \label{eq:MixedKernelDefinition}
  \end{equation}
  then
  \[
    \mathcal{H}^1_v(\Psi) = \mathcal{H}^1_v(\widetilde{\Psi})
    \quad \text{ and } \quad
    \Co (\Psi,\BanachOne) = \Co (\widetilde{\Psi},\BanachOne).
  \]
\end{proposition}

\begin{proof}
  Assumption~\ref{assu:CoorbitAssumptions1} implies
  $\AAPlain_{m_v} \cap \BB_{m_0} \hookrightarrow \AAPlain_{m_v,Y}$.
  Thus, \cite[Lemma~2.29]{kempka2015general} yields the claim.
\end{proof}

\subsection{Discretization in coorbit spaces}
\label{ssec:discret}

General coorbit theory provides a machinery for constructing Banach spaces $\Co \BanachOne$
and associated (Banach) frames and atomic decompositions through
sampling of the continuous frame $\Psi$ on $\PhSpace$.
The results summarized here have been developed by Fornasier and
Rauhut~\cite{fora05} and extended
in~\cite{rauhut2011generalized,hovo20_algebra,kempka2015general,ournote,bahowi15}.

In a nutshell, the idea for discretizing the continuous frame $\Psi$ is to
consider a sufficiently fine covering $\CalV = (V_j)_{j \in J}$
such that the frame $\Psi = (\psi_{\PhVar})_{\PhVar \in \PhSpace}$ is
\emph{almost constant} (in a suitable sense) on each of the sets $V_j$.
Then, by choosing $\PhVar_j \in V_j$, it is intuitively plausible that
the discrete family $(\psi_{\PhVar_j})_{j \in J}$ behaves similarly to the
continuous frame $\Psi$.
The following definition makes this idea of $\Psi$ being almost constant on
each of the $V_j$ more precise.

\begin{definition}\label{def:genosckern}
  Let $\Gamma : \PhSpace \times \PhSpace \to S^1 \subset \CC$ be continuous.
  The \emph{$\Gamma$-oscillation} $\oscVG : \PhSpace \times \PhSpace \to [0,\infty)$
  of a continuous Parseval frame $\Psi = (\psi_\lambda)_{\lambda \in \lambda}$ with
  respect to the topologically admissible covering
  $\CalV = (V_j)_{j \in J}$ of $\PhSpace$ is defined as
  \begin{equation}
   \begin{split}
    \oscVG(\PhVar,\PhVarA)
    := {\mathrm{osc}}_{\Psi,\VV,\Gamma}(\PhVar,\PhVarA)
    & := \sup_{\PhVarC \in \CalV_{\PhVarA}}
           |
             \langle
               \psi_{\PhVar},
               \psi_{\PhVarA} - \Gamma(\PhVarA,\PhVarC)\psi_{\PhVarC}
             \rangle
           |\\
    & = \sup_{\PhVarC \in \CalV_{\PhVarA}}
         |
            K_\Psi(\PhVarA,\PhVar)
            - \overline{\Gamma(\PhVarA,\PhVarC)}
              K_\Psi(\PhVarC,\PhVar)
         | \\
    & = \sup_{\PhVarC \in \CalV_{\PhVarA}}
         |
            K_\Psi(\PhVar,\PhVarA)
            - \Gamma(\PhVarA,\PhVarC)
              K_\Psi(\PhVar,\PhVarC)
         | ,
         \label{eq:def_genosckern}
         \end{split}
  \end{equation}
  where $\CalV_{\PhVarA} := \bigcup_{j \in J \text{ with } \PhVarA \in V_j} V_j$.
\end{definition}

\begin{remark}\label{rem:OscKernelContinuity}
  The oscillation $\oscVG : \PhSpace \times \PhSpace \to [0,\infty)$ is well-defined
  and lower semicontinuous and hence measurable.
  Indeed, each set $\bf{V}_\PhVarA \subset \PhSpace$ is relatively compact as a finite union
  of relatively compact sets, where finiteness of the union is implied by
  the remark after Definition~\ref{def:admissibility}.
  Next, note that $K_\Psi$ is continuous, since the map $\lambda \mapsto \psi_\lambda$
  is (strongly) continuous by Assumption~\ref{assu:CoorbitAssumptions1}.
  Since continuous functions are bounded on relatively compact sets,
  this shows that $\oscVG$ is finite-valued. Now proceed analogous to Remark \ref{rem:MaximalKernelContinuity}.
%
\end{remark}

We further consider specific sequence spaces associated to $\BanachOne$ and a collection $\CalW$ of subsets of $\PhSpace$. 

\begin{definition}\label{def:discreteSpaces}
     For any family $\CalW = (W_j)_{j \in J}$ with a countable index set $J$
     and consisting of measurable subsets $W_j \subset \PhSpace$ with $0 < \mu(W_j) < \infty$
     and any sequence $c = (c_j)_{j \in J} \in \CC^J$, we define
          \[
            \qquad
            \| c \|_{\BanachOne^{\flat}(\CalW)}
            := \bigg\|
                 \sum_{j \in J}
                   |c_j| \Indicator_{W_j}
               \bigg\|_{\BanachOne}
            \in [0,\infty]
            \quad \text{ and } \quad
            \| c \|_{\BanachOne^{\sharp}(\CalW)}
            := \bigg\|
                 \sum_{j \in J}
                   \frac{|c_j|}{\mu(W_j)}
                   \Indicator_{W_j}
               \bigg\|_{\BanachOne}
            \in [0,\infty] ,
          \]
          and finally
          \begin{equation}
            \begin{split}
              & \BanachOne^{\flat}(\CalW)
                := \{
                     c \in \CC^{J}
                     \colon
                     \| c \|_{\BanachOne^{\flat}(\CalW)} < \infty
                   \} \\
              \text{and} \quad
              & \BanachOne^{\sharp}(\CalW)
                := \{
                     c \in \CC^{J}
                     \colon
                     \| c \|_{\BanachOne^{\sharp}(\CalW)} < \infty
                   \}
                .
            \end{split}
            \label{eq:sequenceSpacesGeneral}
          \end{equation}            
\end{definition}

The following set of assumptions summarizes the conditions that ensure applicability
of the discretization results from coorbit theory.

\begin{assumption}\label{assu:CoorbitAssumptions2}
  In addition to Assumption~\ref{assu:CoorbitAssumptions1},
  assume the following conditions:
  \begin{enumerate}[itemsep=0.3em]
    \item $\CalV = (V_j)_{j \in J}$ is a topologically admissible covering of $\PhSpace$;

    \item $\Gamma : \PhSpace \times \PhSpace \to S^1$ is continuous;

    \item \label{enu:smallOscillation}With $m := \max\{m_0,m_v\}$, we have
          \[
            \| \oscVG \|_{\BBm}
            \cdot (2 \| K_\Psi \|_{\BBm} + \| \oscVG \|_{\BBm})
            < 1 ;
          \] 
  \end{enumerate}
\end{assumption}
\begin{remark}\label{rem:KernOscEstimate}
  If $\CalW$ is identical to the topologically admissible covering $\CalV = (V_j)_{j \in J}$,
  we often write $\BanachOne^{\flat}$ and $\BanachOne^{\sharp}$ for $\BanachOne^{\flat}(\CalV)$
  or $\BanachOne^{\sharp}(\CalV)$.
  In fact, it is often possible to choose the product-admissible covering $\CalU$
  from Assumption~\ref{assu:CoorbitAssumptions1} identical
  to the topologically admissible covering $\CalV$,
  and we will indeed do so, but this is not required.
  However, the oscillation of $\Psi$ provides a useful,
  straightforward estimate for the maximal kernel associated to $K_\Psi$:
  \begin{equation}\label{eq:MaxKernOscEstimate}
    \mathrm{M}_{\CalV} K_\Psi(\PhVar,\PhVarA)
    \leq |K_\Psi(\PhVar,\PhVarA)| + \mathrm{osc}_{\Psi,\VV,\Gamma}^\ast(\PhVar,\PhVarA), \text{ a.e.},
  \end{equation}
  for any choice of $\Gamma$.
  Hence, Assumption~\ref{assu:CoorbitAssumptions2}\eqref{enu:smallOscillation} implies
  the second part of Assumption~\ref{assu:CoorbitAssumptions1}\eqref{enu:CoorbitKernelLocalization}
  if $\CalU = \CalV$. 
\end{remark}

\begin{remark}\label{rem:ChoiceOfGamma}
  Note that an appropriate choice of the map $\Gamma : \PhSpace \times \PhSpace \to S^1$
  is crucial to achieve small $\BBm$-norm of the oscillation and, consequently,
  for satisfying Item~\ref{enu:smallOscillation} above.
  In  this work, we will only consider a single, straightforward choice for $\PhSpace$
  and the map $\Gamma : \PhSpace \times \PhSpace \to S^1$, namely $\PhSpace = \RR^d\times D$,
  with $D\subset\RR^d$ open, and $\Gamma : \bigl((y,\omega),(z,\eta)\bigr) \mapsto e^{-2\pi i\langle y-z,\omega\rangle}$,
  cf.\ Theorem~\ref{thm:main2_discreteframes}.
  However, other continuous frames $\Psi$ may require a different choice of $\Gamma$. 
\end{remark}

The following theorem shows that the preceding conditions indeed imply that suitably sampling
the continuous frame $\Psi$ produces a Banach frame decomposition of $\Co(\BanachOne)$.

\begin{theorem}\label{thm:CoorbitDiscretization}
  If Assumption~\ref{assu:CoorbitAssumptions2} holds and if for each $j \in J$
  some $\lambda_j \in V_j$ is chosen, then the discrete frame $\Psi_d = (\psi_{\lambda_j})_{j \in J}$
  forms a Banach frame decomposition for $\Co (\BanachOne) = \Co(\Psi,\BanachOne)$,
  with the sequence space $\BanachOne^{\flat}$ and $\BanachOne^{\sharp}$
  taking the place of $\BanachTD^{\flat}$ and $\BanachTD^{\sharp}$.
\end{theorem}

\begin{proof}
  This follows from \cite[Proposition~8.7]{hovo20_algebra},
  by choosing $L := \oscVG$ and $\widetilde{\CalU} = \CalV$ and by noting that the
  \emph{topologically admissible covering} $\CalV$ is admissible in the terminology
  of \cite{hovo20_algebra}.
\end{proof}

One strategy to satisfy the conditions of Theorem~\ref{thm:CoorbitDiscretization}
is the construction of a parametrized family of topologically admissible coverings
$\CalV^\delta$ such that
\begin{equation}\label{eq:convergentcover}
  \|\text{osc}_{\mathcal{V}^\delta,\Gamma}\|_{\BBm}
  \overset{\delta \rightarrow 0}{\longrightarrow} 0.
\end{equation}
Then, $\delta_0 > 0$ can be found such that Theorem~\ref{thm:CoorbitDiscretization}
holds for the fixed frame $\Psi$ and all $\mathcal{V}^\delta$ with $\delta \leq \delta_0$.

In \cite{rauhut2011generalized}---later generalized in \cite[Theorem~2.50]{kempka2015general}%
---a complementary discretization result is introduced,
which allows to derive Banach frame decompositions for all appropriate $\Co \BanachOne$
directly from (discrete) frames on the Hilbert space $\Hil$, obtained by sampling a continuous frame.
This is an intriguing and important result, given that the explicit construction of frames
for $\Hil$ by sampling a continuous frame is often straightforward, see, e.g.,~\cite{howi14}.
Although we do not consider this result in detail here,
we would like to note that its adjustment to our setting is straightforward.

\subsection{Sequence spaces associated to mixed-norm Lebesgue spaces}%
\label{sub:MixedLebesgueSequenceSpaces}

In this subsection, we show for $\BanachOne = \lebesgue^{p,q}_\kappa(\PhSpace)$
and under suitable conditions on the covering $\CalW$,
that the coefficient spaces $\BanachOne^\flat(\CalW)$
and $\BanachOne^\sharp(\CalW)$ coincide with certain mixed-norm sequence spaces
$\ell^{p,q}_{\widetilde{\kappa}}(J)$.
Here, given a (countable) index set $J$ of the form $J = J_1 \times J_2$,
and any fixed discrete weight $\widetilde{\kappa} \colon J \rightarrow \RR^+$,
the space $\ell^{p,q}_{\widetilde{\kappa}}(J)$ consists of all sequences
$c = (c_{\ell,k})_{(\ell,k) \in J} \in \CC^J$ for which
\begin{equation}\label{eq:mixedsequencespace}
  \| c \|_{\ell^{p,q}_{\widetilde{\kappa}}(J)}
  := \left\|
       k \mapsto \|\widetilde{\kappa}(\bullet,k) \, c_{\bullet,k} \|_{\ell^{p}(J_1)}
     \right\|_{\ell^{q}(J_2)}
  < \infty.
\end{equation}

Precisely, our result is as follows:

\begin{lemma}\label{lem:SequenceSpaceIdentification}
  Let $J = J_1 \times J_2$ be a countable index set
  and $\CalQ =  ( Q_{k} )_{k\in J_2}$ an admissible covering of $\PhSpace_2$.
  For each $k \in J_2$, let $\mathcal{P}_k = ( P_{\ell,k} )_{\ell \in J_1}$ be an admissible
  covering of $\PhSpace_1$ such that $\sup_{k \in J_2} \CalN(\CalP_k) < \infty$.
  Define $\CalU = ( U_{\ell,k} )_{(\ell,k)\in J}$ by
  \begin{equation}\label{eq:prodcover}
     U_{\ell,k} := P_{\ell,k} \times Q_k, \text{ for all } (\ell,k)\in J.
  \end{equation}

  If the weight function $\kappa \colon \PhSpace \rightarrow \RR^+$ satisfies
  \begin{equation}\label{eq:wAlmostConstOnUi}
    \kappa(\PhVar_0)/\kappa(\PhVar_1) \leq C,
    \text{ for some } C > 0,
    \text{ all } \PhVar_0,\PhVar_1\in U_{\ell,k}
    \text{ and all } (\ell,k) \in J = J_1 \times J_2,
  \end{equation}
  then, for all $1\leq p,q\leq \infty$,
  \begin{equation}\label{eq:sequencespacesLpq}
    \bigl(\lebesgue^{p,q}_\kappa(\PhSpace)\bigr)^{\flat}(\CalU)
    = \ell^{p,q}_{\kappa^\flat_{\CalU}}(J)
    \quad \text{and} \quad
    \bigl(\lebesgue^{p,q}_\kappa(\PhSpace)\bigr)^{\sharp}(\CalU)
    = \ell^{p,q}_{\kappa^\sharp_{\CalU}}(J),
    \quad \text{with equivalent norms.}
  \end{equation}
  Here, $\bigl(\lebesgue^{p,q}_\kappa(\PhSpace)\bigr)^{\flat}(\CalU)$
  and $\bigl(\lebesgue^{p,q}_\kappa(\PhSpace)\bigr)^{\sharp}(\CalU)$ are the spaces
  defined in \eqref{eq:sequenceSpacesGeneral} and the weights $\kappa^\flat_{\CalU}$
  and $\kappa^\sharp_{\CalU}$ are given by
  \[
    \kappa^{\flat}_{\CalU} (\ell,k)
    = [\mu_1(P_{\ell,k})]^{1/p}
      \cdot [\mu_2(Q_k)]^{1/q}
      \cdot \kappa_{\ell,k}
    \quad \text{and} \quad
    \kappa^{\sharp}_{\CalU} (\ell,k)
    = [\mu_1(P_{\ell,k})]^{1/p - 1}
      \cdot [\mu_2(Q_k)]^{1/q - 1}
      \cdot \kappa_{\ell,k} ,
  \]
  where $\kappa_{\ell,k} := \sup_{\PhVar \in U_{\ell,k}} \kappa(\PhVar)$ for all $(\ell,k) \in J$.
\end{lemma}

\begin{proof}
  We prove the assertion for $p,q < \infty$;
  the proof for the cases $p = \infty$ or $q = \infty$ is similar and hence omitted.

  Note that if $\CalV = (V_j)_{j \in J}$ is an admissible covering of a set $\CalO$
  and if $(a_j)_{j \in J} \in [0,\infty)^J$, then at most $\CalN(\CalV)$ summands
  of the sum $\sum_{j \in J} a_j \Indicator_{V_j}(x)$ are non-zero for each fixed $x \in \CalO$.
  Therefore, given any $r \in \RR^+$, we have
  \(
    \big(
      \sum_{j \in J}
        a_j \Indicator_{V_j}(x)
    \big)^{r}
    \asymp \sum_{j \in J}
             a_j^r \Indicator_{V_j}(x)
    ,
  \)
  where the implied constant only depends on $r$ and on $\CalN(\CalV)$.

  Let $(c_{\ell,k})_{(\ell,k)\in J} \in \CC^J$ and set
  $f_c (\PhVar) := \sum_{\ell,k \in J} |c_{\ell,k}| \, \Indicator_{U_{\ell,k}}(\PhVar)$.
  The estimate from the preceding paragraph,
  first applied to $\CalV = \CalQ$, and then applied to $\CalV = \CalP_k$ for fixed $k \in J_2$,
  shows
  \begin{equation}
    \begin{split}
      \big( f_c (\PhVar) \big)^p
      & = \Big(
            \sum_{k \in J_2}
              \Indicator_{Q_k} (\PhVar_2)
              \sum_{\ell \in J_1}
                |c_{\ell,k}| \,
                \Indicator_{P_{\ell,k}} (\PhVar_1)
          \Big)^p
        \asymp \sum_{k \in J_2}
               \bigg[
                 \Indicator_{Q_k}(\PhVar_2)
                 \Big(
                   \sum_{\ell \in J_1}
                     |c_{\ell,k}| \,
                     \Indicator_{P_{\ell,k}} (\PhVar_1)
                 \Big)^p
               \bigg] \\
      & \asymp \sum_{k \in J_2}
                 \Indicator_{Q_k} (\PhVar_2)
                 \sum_{\ell \in J_1}
                   |c_{\ell,k}|^p
                   \Indicator_{P_{\ell,k}}(\PhVar_1) .
    \end{split}
    \label{eq:SequenceSpaceIdentificationPointwise}
  \end{equation}
  Furthermore, note that Equation~\eqref{eq:wAlmostConstOnUi} implies
  $\kappa(\PhVar) \asymp \kappa_{\ell,k}$ for $\PhVar \in U_{\ell,k} = P_{\ell,k} \times Q_k$.
  Therefore, integrating the estimate \eqref{eq:SequenceSpaceIdentificationPointwise}
  over $\PhVar_1 \in \PhSpace_1$, we see
  \begin{align*}
    g_c (\PhVar_2)
    & := \!\!
       \int_{\PhSpace_1} \!\!
         \bigl(f_c (\PhVar_1,\PhVar_2) \cdot \kappa(\PhVar_1,\PhVar_2) \bigr)^p
       \, d \mu_1(\PhVar_1)\\
    & \asymp \sum_{k \in J_1} \!
               \Indicator_{Q_k} (\PhVar_2)
               \sum_{\ell \in J_1}
                 |c_{\ell,k}|^p \!
                 \int_{\PhSpace_1} \!\!
                   \bigl(\kappa(\PhVar_1,\PhVar_2)\bigr)^p
                   \cdot \Indicator_{P_{\ell,k}}(\PhVar_1)
                 \, d \mu_1 (\PhVar_1) \\
    & \asymp \sum_{k \in J_1}
             \Big[
               \Indicator_{Q_k} (\PhVar_2)
               \sum_{\ell \in J_1}
                 |c_{\ell,k} \cdot \kappa_{\ell,k}|^p
                 \cdot \mu_1(P_{\ell,k})
             \Big]
    .
  \end{align*}
  Now, we again use the estimate from the beginning of the proof (for $\CalV = \CalQ$) to obtain
  \[
   \begin{split}
    [g_c(\PhVar_2)]^{q/p}
    & \asymp \bigg(
             \sum_{k \in J_1}
             \Big[
               \Indicator_{Q_k} (\PhVar_2)
               \sum_{\ell \in J_1}
                 |c_{\ell,k} \cdot \kappa_{\ell,k}|^p
                 \cdot \mu_1(P_{\ell,k})
             \Big]
           \bigg)^{\frac{q}{p}}\\
           & \asymp \sum_{k \in J_1} \!
           \bigg[
             \Indicator_{Q_k}(\PhVar_2)
             \Big(
               \sum_{\ell \in J_1}
                 |c_{\ell,k} \cdot \kappa_{\ell,k}|^p
                 \cdot \mu_1(P_{\ell,k}) \!
             \Big)^{\frac{q}{p}}
           \bigg].
    \end{split}
  \]
  Integrating this over $\PhVar_2 \in \PhSpace_2$, we finally see
  \begin{align*}
    \| c \|_{(\lebesgue^{p,q}_\kappa(\PhSpace))^{\flat}(\CalU)}^q
    & = \| f_c \|_{\lebesgue^{p,q}_\kappa(\PhSpace)}^q
      = \int_{\PhSpace_2}
          [g_c(\PhVar_2)]^{q/p}
        \, d \mu_2(\PhVar_2) \\
    & \asymp \sum_{k \in J_1}
             \bigg[
               \mu_2(Q_k)
               \Big(
                 \sum_{\ell \in J_1}
                   |c_{\ell,k} \cdot \kappa_{\ell,k}|^p
                   \cdot \mu_1(P_{\ell,k}) \!
               \Big)^{\frac{q}{p}}
             \bigg]
      =      \| c \|_{\ell^{p,q}_{\kappa^{\flat}_{\CalU}}}^q ,
  \end{align*}
  which completes the proof for the identification of the space
  $\bigl(\lebesgue^{p,q}_\kappa(\PhSpace)\bigr)^{\flat}(\CalU)$.

  The identification of $\bigl(\lebesgue^{p,q}_\kappa(\PhSpace)\bigr)^{\sharp}(\CalU)$
  follows by substituting $c_{\ell,k} \mu(U_{\ell,k})^{-1}$ for $c_{\ell,k}$
  everywhere in the derivations above.
\end{proof}

Our proof of the above result relies heavily on the product structure of the covering $\CalU$
in \eqref{eq:prodcover}.
Although minor generalizations of the conditions placed on $\CalU$
are possible without significant complications, one cannot expect to recover a similar result
without restrictions on $\CalU$.
However, in our setting of warped time-frequency systems,
product coverings as in \eqref{eq:prodcover} arise quite naturally and the result above is entirely sufficient.


 \section{Frequency-adapted tight continuous frames through warping}
\label{sec:warpedsystems}

In this section, we define the class of warped time-frequency systems
as tools for the analysis and synthesis of functions.
The framework presented here generalizes the systems
introduced in \cite{bahowi15} to arbitrary dimensions.
The basic properties presented in this section are proven analogous
to the one-dimensional case, such that we only provide references.

As explained in the introduction, a warped time-frequency system generates
a joint time-frequency representation
in which the trade-off between time- and frequency-resolution at any given
frequency position is governed by the associated frequency scale.
That frequency scale is generated by the warping function.

\begin{definition}\label{def:warpfun}
  Let $D \subset \RR^d$ be open.
  A $C^1$ diffeomorphism $\Phi: D \rightarrow \RR^d$ is called
  a \emph{warping function}, if $\det(\mathrm{D}\Phi^{-1}(\tau))>0$
  for all $\tau \in \RR^d$ and if the associated weight function
  \begin{equation}
    w : \quad
    \RR^d \to \RR^+, \quad
    w(\tau) = \det(\mathrm{D}\Phi^{-1}(\tau)),
    \label{eq:wDefinition}
  \end{equation}
  is $w_0$-moderate for some submultiplicative weight
  $w_0 : \RR^d \to \R^+$.
\end{definition}

\begin{rem*}
  We note that $w_0$ is automatically locally bounded,
  as shown in \cite[Theorem~2.1.4]{HeilPhDThesisWienerAmalgam}
  and \cite[Theorem~2.2.22]{VoigtlaenderPhDThesis}.
\end{rem*}

Let us collect some basic results that are direct consequences
of $w$ being $w_0$-moderate.
For the sake of brevity, set
\begin{equation}\label{eq:DefofA}
  A(\tau) := \mathrm{D}\Phi^{-1}(\tau) \qquad \forall \, \tau \in \mathbb{R}^d
\end{equation}
for the remainder of this article.
First, note that the chain rule---applied to the identity
$\tau = \Phi(\Phi^{-1}(\tau))$ for $\tau \in \mathbb{R}^d$---yields
\begin{equation}\label{eq:DerivativeOfInverse}
  \mathrm{id}
  = \mathrm{D}\Phi(\Phi^{-1}(\tau)) \cdot A(\tau),
  \quad \text{ i.e.,} \quad
  A(\tau) = [\mathrm{D}\Phi(\Phi^{-1}(\tau))]^{-1}.
\end{equation}
In particular, we get (for arbitrary $\tau = \Phi(\xi)$)
that $w(\Phi(\xi)) = \frac{1}{\det(\mathrm{D}\Phi(\xi))}$.  
Thus, given any measurable nonnegative $f : \RR^d \to [0,\infty)$,
a change of variables leads to the frequently useful formulae
\begin{equation}
  \int_D f(\Phi(\xi)) \, d\xi
  = \int_{\mathbb{R}^d}
      w(\tau) \cdot f(\tau)
    \, d\tau \, ,
  \quad \text{ and consequently } \quad
  \| f \circ \Phi\|_{\bd L^p(D)} = \|f\|_{\bd L_{w^{1/p}}^p}.
  \label{eq:StandardChangeOfVariables}
\end{equation}

Finally, we note that submultiplicativity of $w_0$
and $w_0$-moderateness of $w$ yields translation invariance of
$\bd L_{w^{1/p}}^p$ and $\bd L_{w_0^{1/p}}^p$.
Indeed, if $w$ is any $w_0$-moderate weight
(not necessarily given by \eqref{eq:wDefinition}), then
$w(\tau+\upsilon) \leq \min\{w(\tau) w_0 (\upsilon),w(\upsilon) w_0 (\tau)\}$,
so that \eqref{eq:StandardChangeOfVariables} yields
\begin{equation}
    \| \bd T_\upsilon f \|^p_{\bd L_{w^{1/p}}^p}
    \leq  w_0(\upsilon) \cdot \|f\|^p_{\bd L_{w^{1/p}}^p}
    \qquad \text{and} \qquad
    \| \bd T_\upsilon f\|_{\bd L_{w^{1/p}}^p}^p
    \leq  w(\upsilon) \cdot \|f\|_{\bd L_{w_0^{1/p}}^p}^p
\label{eq:StandardTrafoAndTranslationEstimate}
\end{equation}
for all measurable $f : \RR^d \to \CC$ and all $w_0$-moderate weights $w$.
In particular, one can choose $w=w_0$,
since $w_0$ is submultiplicative and hence $w_0$-moderate.

Moderateness (and positivity) of the weight function $w$
associated to the warping function $\Phi$ ensure that
warped time-frequency systems and the associated representations
are well-defined and possess some essential properties,
as we will see shortly. 
But first, let us formally introduce warped time-frequency systems.

\begin{definition}\label{def:warpedsystem}
  Let $\Phi$ be a warping function and $\theta \in \Ltw(\RR^d)$.
  The \emph{(continuous) warped time-frequency system} generated by $\theta$
  and $\Phi$ is the collection of functions
  $\mathcal G(\theta,\Phi):=(g_{y,\omega})_{(y,\omega)\in\Lambda}$, where
  \begin{equation}
    g_{y,\omega} := \bd T_{y} \widecheck{g_\omega},
    \quad \text{ with } \quad
    g_\omega := w(\Phi(\omega))^{-1/2} \cdot (\bd T_{\Phi(\omega)} \theta)\circ \Phi
    \text{ for all } y\in \RR^d,\ \omega\in D.
    \label{eq:WarpedSystemDefinition}
  \end{equation}
  Here, the function $g_\omega : D \to \CC$ is extended by zero to a
  function on all of $\RR^d$, so that $\widecheck{g_\omega}$ is well-defined.
  The \emph{phase space} associated with this family is $\Lambda = \RR^d\times D$.
\end{definition}

Since $w$ is moderate with respect to $w_0$, we obtain $g_{y,\omega}\in \LtDF$.
In fact, \eqref{eq:StandardChangeOfVariables} and \eqref{eq:StandardTrafoAndTranslationEstimate}
show
\begin{equation}\label{eq:FrameElementsBounded}
  \|\widehat{g_{y,\omega}}\|_{\lebesgue^2(D)}^2
  \leq \frac{w_0(\Phi(\omega))}{w(\Phi(\omega))}\|\theta\|_{\Ltw(\RR^d)}^2
  < \infty
  \quad\text{and}\quad
  \|\widehat{g_{y,\omega}}\|_{\lebesgue^2(D)}^2
  \leq \|\theta\|_{\Ltv}^2 \in [0,\infty].
\end{equation}
Thus, $\mathcal G(\theta,\Phi) \subset \LtDF$ and the associated
analysis operation, i.e., taking inner products with the functions $g_{y,\omega}$,
defines a transform on $\LtDF$.

\begin{definition}\label{def:WarpingFunction}
  Let $\Phi$ be a warping function and $\theta \in \Ltw(\RR^d)$.
  The \emph{$\Phi$-warped time-frequency transform} of $f\in \LtDF$ with respect
  to the \emph{prototype} $\theta$ is defined as
  \begin{equation}\label{eq:warpedtransform}
    V_{\theta, \Phi} f: \quad
    \RR^d  \times D \to \CC, \quad
    (y,\omega) \mapsto \inner{f}{g_{y,\omega}}_{\lebesgue^2(\RR^d)}.
  \end{equation}
\end{definition}


For $\lambda = (y,\omega)\in\Lambda = \RR^d\times D$, we will alternatively use
the notations $V_{\theta,\Phi}f(y,\omega) = V_{\theta,\Phi}f(\lambda)$ and
$g_{y,\omega} = g_{\lambda}$, whenever one or the other is more convenient.

By definition and \eqref{eq:FrameElementsBounded}, we have $V_{\theta, \Phi}f
\in \bd L^\infty(\Lambda)$, whenever $\theta\in \Ltv(\RR^d)$.
Furthermore, using that $\Phi \in \mathcal C^1$ and the translation-invariance of $\Ltw(\RR^d)$,
one can also deduce that $V_{\theta,\Phi}f\in \mathcal C(\Lambda)$,
even under the weaker assumption $\theta \in \Ltw (\RR^d)$.

\begin{proposition}\label{prop:WarpedSystemWeaklyContinuous}
  Let $\Phi$ be a warping function and $\theta\in \Ltw(\RR^d)$.
  Then
  \begin{equation}
    V_{\theta,\Phi}f\in \mathcal C(\Lambda),\ \text{ for all } f\in \LtDF . 
  \end{equation}
  In fact, the mapping
  $\RR^d \times D \to \LtDF, (y,\omega) \mapsto g_{y,\omega}$ is continuous.
\end{proposition}

\begin{proof}
Analogous to the proof of \cite[Proposition~4.5]{bahowi15}.
\end{proof}

The next result provides the crucial property that makes warped time-frequency systems so attractive.
Namely, $V_{\bullet, \Phi}$ possesses a norm-preserving property similar to the orthogonality relations
(Moyal's formula~\cite{mo49,gr01}) for the short-time Fourier transform.

\begin{theorem}\label{thm:orthrel}
  Let $\Phi$ be a warping function and $\theta_1, \theta_2 \in \Ltw\cap\LtRd$.
  Then the following holds for all $f_1,f_2 \in \LtDF$:
  \begin{equation}
    \int_\Lambda V_{\theta_1,\Phi}f_1(\lambda) \overline{V_{\theta_2,\Phi}f_2(\lambda)}\; d\lambda
    = \inner{f_1}{f_2}_{\lebesgue^2(\RR^d)}\langle \theta_2,\theta_1 \rangle_{\lebesgue^2(\RR^d)}.
    \label{eq:orthrel}
  \end{equation}
  In particular, for any $\theta\in \Ltw\cap\LtRd$, $\mathcal G(\theta,\Phi)$ is a continuous tight frame
  with frame bound $\|\theta\|_{\lebesgue^2}^2$.
\end{theorem}

\begin{proof}
Analogous to \cite[Theorem~4.6]{bahowi15}.
Note that $\theta_1, \theta_2 \in \LtRd$ implies the admissibility condition required there,
and moreover serves to justify the application of Plancherel's theorem in the proof.
\end{proof}

As already remarked in \cite{bahowi15}, $\theta_1,\theta_2\in\LtRd$
is a sort of admissibility condition and, in fact,
yields the classical wavelet admissibility, if $d=1$ and $\Phi = \log$.
Besides the tight frame property, Theorem~\ref{thm:orthrel} shows
that the warped time-frequency representations with respect to orthogonal windows,
but the same warping function, span orthogonal subspaces of $\bd L^2(\Lambda)$.
Similarly, orthogonal functions $f_1,f_2$ have orthogonal representations,
independent of the prototypes $\theta_1,\theta_2$.
These additional properties are useful, e.g., for constructing superframes
for multiplexing~\cite{grochenig2009gabor,balan2000multiplexing} or multitapered representations
\cite{thomson1982spectrum,xiao2007multitaper,daubechies2016conceft}.

The tight frame property itself is a basic requirement for general coorbit theory,
and provides a convenient inversion formula:

\begin{corollary}\label{cor:WarpingInversion}
  Given a warping function $\Phi$ and some nonzero $\theta\in \Ltw\cap \bd L^2 (\R^d)$.
  Then any $f \in \mathcal F^{-1}(\LtD)$ can be reconstructed from $V_{\theta, \Phi}f$ by
  \begin{equation}
    f
    = \frac{1}{\|\theta\|^2_{\bd L^2}}
      \int_\Lambda
        V_{\theta,\Phi}f(\lambda) \, g_{\lambda}
      \; d\lambda.
    \label{eq:WeakSenseReconstruction}
  \end{equation}
  The equation holds in the weak sense.
\end{corollary}

\begin{proof}
  The assertion is a direct consequence of $\mathcal G(\theta,\Phi)$ being a tight continuous frame
  with bound $\|\theta\|^2_{\bd L^2}$.
\end{proof}

Now that the essential properties of warped time-frequency systems are established,
and before proceeding to construct and examine coorbit spaces
associated to warped time-frequency systems,
we provide some instructive examples of warping functions
and the resulting warped time-frequency systems.

\subsection{Examples}
\label{ssec:examples}

We present several examples of warping functions. We begin by constructing a $d$-dimensional function
as a separable (coordinate-wise) combination of $1$-dimensional warping functions.
Examples of such $1$-dimensional warping functions can be found in~\cite{bahowi15}.\\

\noindent
\textbf{Separable warping.}
Fix $\mathcal C^1$-diffeomorphisms $\Phi_i: D_i \rightarrow \RR$, $i\in\underline{d}$,
such that $\mathrm{D}\Phi_i^{-1}(\tau) = \frac{\partial \Phi_i^{-1}}{\partial \tau}(\tau) > 0$,
for all $\tau\in\RR$, $i\in\underline{d}$.
If each $\mathrm{D}\Phi_i^{-1}$, $i\in\underline{d}$,
is $w_{0,i}$-moderate and we take $\Phi$ to be defined as
\[
  \Phi(\xi)
  = \bigl(\Phi_1(\xi_1),\ldots,\Phi_d(\xi_d)\bigr)^T,
  \quad \text{for all} \quad \xi \in D := D_1 \times \cdots \times D_d,
\]
then clearly $\Phi : D \to \R^d$ is a diffeomorphism and $\mathrm{D}\Phi^{-1}$ is diagonal,
and hence
\[
  w(\tau)
  = \det(\mathrm{D}\Phi^{-1}(\tau))
  = \prod_{i\in\underline{d}}
      \mathrm{D}\Phi_i^{-1}(\tau_i)
  > 0
  \qquad \forall \, \tau \in \R^d ,
\]
and $w$ is $w_0$-moderate for $w_0 (\tau) := \prod_{i\in\underline{d}} w_{0,i}(\tau_i)$.

A family of anisotropic wavelets can be constructed by selecting $\Phi = \textbf{log}$,
where ${\textbf{log}:(\RR^+)^d\rightarrow \RR^d}$ denotes the map
$\xi \mapsto (\log(\xi_1),\ldots,\log(\xi_d))^T$.
It follows that $\Phi^{-1}$ is the componentwise exponential function and satisfies
\[
  \mathrm{D}\Phi_i^{-1}(\tau) = \mathrm{diag}(e^{\tau_1},\dots,e^{\tau_d})
  \quad \text{and} \quad
  w(\tau) = \exp(\tau_1+ \cdots +\tau_d),
\]
for all $\tau\in\RR^d$.
Hence, $w$ is submultiplicative and moderate with respect to itself.
Furthermore, writing $\bd{d}(\omega) := \mathrm{diag}(\omega_1,\dots,\omega_d)\in \RR^{d\times d}$
for $\omega \in (\R^+)^d$, we see that the elements of $\mathcal G(\theta,\Phi)$ are given by
\[
  \begin{split}
    g_{y,\omega}
    & = w(\Phi(\omega))^{-1/2}
        \cdot \bd T_{y}\mathcal{F}^{-1} \left(
                                          (\bd T_{\textbf{log}(\omega)}\theta)
                                          \circ \textbf{log}
                                        \right) \\
    & = \det(\bd d(\omega))^{-1/2}
        \cdot \bd T_{y} \mathcal{F}^{-1} \left(
                                           \theta
                                           \circ \textbf{log}
                                                 (
                                                  [\bd d(\omega)]^{-1}
                                                  \langle\cdot\rangle
                                                 )
                                         \right) \\
    & = \det(\bd d(\omega))^{1/2}
        \cdot \big[
                \mathcal{F}^{-1} \left(
                                   \theta\circ \textbf{log}
                                 \right)
              \big]
              (\bd d(\omega) \langle \cdot-y\rangle) \\
    & = \det(\bd d(\omega))^{1/2} \cdot \widetilde{g}
                                        (
                                         \bd d(\omega)\langle \cdot-y\rangle
                                        ),
     \text{ with }
     \widetilde{g} := \mathcal{F}^{-1}\bigl( \theta\circ \textbf{log}\bigr).
  \end{split}
\]
Thus, $\mathcal G(\theta,\Phi)$ is a wavelet system in the sense
of~\cite{bernier1996wavelets,fuhr1996wavelet}, with the dilation group
given by the diagonal $d\times d$-matrices with entries in $\RR^+$. 
The derivations above do not seem to generalize, however, to a setting that recovers 
wavelets with respect to general dilation groups. 
Finally, the expression of $g_{y,\omega}$ through linear operators applied to
a single \emph{mother wavelet} $\widetilde{g}$ defined in the time-domain relies on 
properties of the coordinate-wise logarithm $\textbf{log}$ and does not generalize to arbitrary warping functions $\Phi$.\\

\noindent
\textbf{Radial warping.}
By choosing the warping function $\Phi$ to be radial,
we can construct time-frequency systems with frequency resolution depending
on the modulus $|\xi|$ of $\xi\in\RR^d$. The deformation is then fixed on any 
$(d-1)$-sphere of fixed radius, similar to isotropic wavelets
(see \cite[Section~2.6]{da92} and \cite[Example~2.30]{fuhr2005abstract}).
Generally, radial warping functions are of the form
\[
  \Phi_\varrho: \quad
  \RR^d \to \RR^d, \quad
  \xi \mapsto \varrho(|\xi|) \cdot \xi/|\xi|,
\]
for a strictly increasing diffeomorphism $\varrho : \RR \to \RR$.
Under suitable additional assumptions on $\varrho$,
it can then be shown that $(\Phi_\varrho)^{-1} = \Phi_{\varrho^{-1}}$
and that $\Phi_\varrho$ is a warping function as per Definition \ref{def:warpfun}.
\nicki{It will be shown in future work that radial warping does \emph{not}
recover isotropic wavelets exactly in dimensions $d > 1$,
for any choice of $\varrho$, but that warped time-frequency systems can be close to isotropic wavelets
in a sense that will be made formal in the mentioned follow-up work.}
An in depth study of radial warping with some specific examples is provided in Section~\ref{sec:RadialWarping}.
\\


\noindent
\textbf{An explicit, exotic example for $d=2$.}
To demonstrate that there is potential for warping functions
beyond the separable and radial cases, consider the
continuous $\mathcal C^1$-diffeomorphism
\[
  \Phi : \quad
  \R^2 \to \R^2, \quad
  \xi \mapsto \big( e^{\xi_2} \, \xi_1 , \,\, \xi_2 \big)^T .
\]
It is straightforward to see that $\Phi$ is a diffeomorphism with inverse
$\Phi^{-1} (\tau) = (e^{-\tau_2} \, \tau_1, \,\, \tau_2)^T$, which satisfies
\[
  \mathrm{D}\Phi^{-1}(\tau)
  = \begin{pmatrix}
      e^{-\tau_2} & -e^{-\tau_2} \, \tau_1 \\
      0           & 1 \\
    \end{pmatrix}
\]
and hence $w(\tau) = \det (\mathrm{D}\Phi^{-1}(\tau)) = e^{-\tau_2} > 0$.
Moreover, it is easy to see that $w$ is multiplicative (and in particular submultiplicative)
and hence self-moderate.
Thus, $\Phi$ is a valid warping function that is neither separable nor radial.


\section{Membership of the reproducing kernel in \texorpdfstring{$\BBm$}{𝓑ₘ}}
\label{sec:coorbits}

As we saw in Section~\ref{ssec:coorbit} (see in particular Assumption~\ref{assu:CoorbitAssumptions1}),
the main challenge in verifying the applicability of coorbit theory
for a continuous Parseval frame $\Psi$ lies in showing that (the maximal function of)
the reproducing kernel $K_\Psi$ is contained $\AAPlain_{m_v}$ or $\BB_{m_0}$,
for suitable weights $m_v, m_0 : \PhSpace \times \PhSpace \to \R^+$.
We will do so in two steps:
(1) In the present section, we will derive verifiable conditions on the warping function $\Phi$
and the prototype function $\theta$ which ensure that the warped time-frequency system
$\Psi = \CalG(\theta,\Phi)$ satisfies $K_{\Psi} \in \BBm$,
for a weight $m$ satisfying suitable assumptions.
(2) In Section \ref{sec:discretewarped}, we do the same for the $\Gamma$-oscillation
of $\Psi$ and additionally demonstrate that $\|\oscVG\|_{\BBm}$ can be made arbitrarily small
by choosing an appropriate covering $\CalV$.
Then, the desired properties of the maximal kernel $\textrm{M}_\CalV K_\Psi$
are a consequence of Remark \ref{rem:KernOscEstimate}.

To prepare for the treatment of the $\Gamma$-oscillation,
we already consider mixed kernels in the present section.
This setting only requires little additional effort.
We begin by introducing some notation and conditions that will be used throughout this section.

\begin{notanddef}\label{assu:WeightFunctionAssumptions}
  By $\Phi$, we denote a warping function $\Phi : D \to \R^d$,
  with associated weights $w,w_0$ as in \Cref{def:warpfun},
  $A = \mathrm{D}\Phi^{-1}$ and $\PhSpace := \R^d \times D$.
  In all instances, $\theta,\theta_1,\theta_2 \in \lebesgue_{\sqrt{w_0}}^2 (\R^d)$
  and we denote the mixed kernel associated with $\CalG(\theta_1,\Phi)$ and $\CalG(\theta_2,\Phi)$
  by $K_{\theta_1,\theta_2}:= K_{\CalG(\theta_1,\Phi),\CalG(\theta_2,\Phi)}$.
  Finally, for $\ell\in\{1,2\}$, we write
  \(
    \CalG(\theta_\ell,\Phi)
    = \bigl(g_{y,\omega}^{[\ell]}\bigr)_{y \in \R^d, \omega \in D}
    = \bigl(\translation_y \, \widecheck{g_{\omega}^{[\ell]}}\bigr)_{y \in \R^d, \omega \in D}.
  \)
%
   \begin{enumerate}
   \item \label{enu:M0domination}If there is a continuous function $m^{\Phi} : \R^d \times \R^d \to \R^+$
   satisfying
          \begin{equation}
            m \big( (x,\Phi^{-1}(\sigma)), (y,\Phi^{-1}(\tau)) \big)
            \leq m^{\Phi} (x-y, \sigma-\tau)
            \qquad \forall \, x,y,\sigma,\tau \in \R^d,
            \label{eq:M0Domination}
          \end{equation}
   then we say that $m$ is \emph{$\Phi$-convolution-dominated} (by $m^{\Phi}$).
   If that is the case, we denote by $M : \R^d \times \R^d \to \R^+$ the weight
   \begin{equation}
            M(x,\tau)
            := \sup_{y \in \R^d, |y| \leq R_{\Phi} |x|}
               \big[
                 \sqrt{w_0(\tau)} \, m^{\Phi} (y, \tau)
               \big]
            \quad \text{where} \quad
            R_\Phi := \sup_{\xi \in D} \| \mathrm{D}\Phi (\xi) \| \in \RR^+\cup\{\infty\}.
            \label{eq:BigWeightDefinition}
   \end{equation}
   \item If there exists an $m^{\Phi}$ as in \eqref{enu:M0domination},
         such that $m$ is $\Phi$-convolution-dominated by $m^{\Phi}$ and
        \begin{equation}
            R_\Phi < \infty
            \qquad \text{or} \qquad
            m^{\Phi}(x,\sigma) \lesssim m^{\Phi}(0,\sigma) \text{ for all } x,\sigma \in \R^d,
            \label{eq:M0XDependence}
          \end{equation}
    then we say that $m$ is \emph{$\Phi$-compatible} (with dominating weight $m^{\Phi}$).
  \end{enumerate}
\end{notanddef}

Furthermore, we require a slightly stricter and more structured notion
of regularity for warping functions.

\begin{definition}\label{assume:DiffeomorphismAssumptions}
  Let $\emptyset \neq D \subset \RR^d$ be an open set and fix an integer $k\in\NN_0$.
  A map $\Phi:D\to\mathbb{R}^{d}$ is a \emph{$k$-admissible warping function}
  with \emph{control weight} $v_0:\RR^d\rightarrow \RR^+$,
  if $v_0$ is continuous, submultiplicative and radially increasing
  and $\Phi$ satisfies the following assumptions:
  \begin{itemize}
   \item $\Phi$ is a $C^{k+1}$-diffeomorphism.
   \item $A = \mathrm{D}\Phi^{-1}$ has positive determinant.
   \item With
          \begin{equation}
              \phi_{\tau}\left(\upsilon\right)
              :=\left(A^{-1}(\tau) A(\upsilon+\tau)\right)^T
              = A^T (\upsilon+\tau) \cdot A^{-T}(\tau),
              \label{eq:PhiDefinition}
          \end{equation}
          we have
          \begin{equation}\label{eq:PhiHigherDerivativeEstimate}
            \left\Vert
              \partial^{\alpha}\phi_{\tau}\left(\upsilon\right)
            \right\Vert
            \leq v_0 (\upsilon) 
            \qquad \text{ for all } \tau, \upsilon \in \mathbb{R}^{d}
                   \text{ and all multiindices }
                   \alpha \in \mathbb{N}_{0}^{d}, ~ \left|\alpha\right|\leq k.  
          \end{equation}
  \end{itemize}
\end{definition}

\begin{remark}
  1) The function $\phi_\tau$ describes the regularity of $A$ around $\tau$;
      its relevance will become clear before long,
      see Equation~\eqref{eq:PsiDerivativeIsPhi} below.

  \smallskip{}

  \noindent 2) On the right-hand side of \eqref{eq:PhiHigherDerivativeEstimate},
     one could allow constants $C_{\alpha}$ and different weights $\tilde{v}_{\alpha}$
     not necessarily being radially increasing, therefore obtaining tighter bounds on
     $\left\Vert \partial^{\alpha}\phi_{\tau}\left(\upsilon\right)\right\Vert$.
     However, whenever such $C_{\alpha}$, $\tilde{v}_{\alpha}$ exist, there also exists
     a weight $v_0$ satisfying all the requirements
     of Definition~\ref{assume:DiffeomorphismAssumptions}.

  \smallskip{}

 \noindent  3) We remark that \eqref{eq:PhiHigherDerivativeEstimate} generalizes the conditions
     mentioned in~\cite{bahowi15}, even for the case $d = 1$ considered there.
\end{remark}

\Cref{thm:MR1_kernel_is_in_AAm} below shows that smoothness of the prototypes $\theta_1,\theta_2$
and decay (or localization) of their partial derivatives implies $K_{\theta_1,\theta_2} \!\in\! \BBm$,
provided that $m$ is $\Phi$-compatible.
In particular, all conditions are surely satisfied
for arbitrary $\theta_1, \theta_2 \in \mathcal{C}^{\infty}_c(\RR^d)$.
The proof of \Cref{thm:MR1_kernel_is_in_AAm} is deferred to the end of the section.

%

\begin{theorem}\label{thm:MR1_kernel_is_in_AAm}
  Let $\Phi$ be a $(d+p+1)$-admissible warping function
  with control weight $v_0$, where $p=0$ if $R_\Phi = \infty$, \nicki{defined as in \eqref{eq:BigWeightDefinition}, }and $p\in\NN_0$ otherwise.
  Let furthermore $m : \Lambda \times \Lambda \to \R^+$ be a symmetric weight satisfying
  \begin{equation}\label{eq:m_weight_estimate}
    m\bigl((y, \xi), (z, \eta)\bigr)
    \leq (1 + |y-z|)^p \cdot v_1 \bigl(\Phi(\xi) - \Phi(\eta)\bigr),
    \text{ for all } y,z\in\RR^d \text{ and } \xi,\eta\in D,
  \end{equation}
  for some continuous and submultiplicative weight $v_1 : \RR^\dimension \to \R^+$
  satisfying $v_1 (\upsilon) = v_1(-\upsilon)$ for all $\upsilon \in \R^d$.

  Finally, with
  \[
     w_2 : \quad
     \RR^d \to \R^+, \quad
     \upsilon \mapsto (1+|\upsilon|)^{d+1} \cdot v_1(\upsilon) \cdot [v_0 (\upsilon)]^{9d/2+3p+3},
  \]
  assume that $\theta_1,\theta_2\in \mathcal C^{d+p+1}(\RR^d)$ and
  \[
    \frac{\partial^n}{\partial \upsilon_j^n} \theta_\ell\in \lebesgue^2_{w_2}(\RR^d),
    \qquad \text{ for all } j \in \underline{d},\  \ell \in \{1,2\},\ 0\leq n\leq d+p+1 \, ,
  \]
  and let
  \begin{equation}\label{eq:DefOfCmax}
    C_{\max}
    := \nicki{C_{\max}(d+p+1,\theta_1,\theta_2) :=}
    \prod_{\ell \in \{1,2\}} 
          \bigg(
             \max_{j \in \underline{d}} \,\,
               \max_{0 \leq n \leq d+p+1}
                 \Big\|
                   \frac{\partial^n}{\partial \upsilon_j^n} \theta_\ell
                 \Big\|_{\lebesgue^2_{w_2}(\RR^d)}
          \bigg).
  \end{equation}

  Then, $m$ is $\Phi$-compatible with dominating weight $m^{\Phi}(x,\tau) = (1 + |x|)^p \cdot v_1(\tau)$
  and there is a constant $C>0$, independent of $\theta_1,\theta_2$ and $m$, 
  satisfying
  \[
    \big\| K_{\theta_1, \theta_2} \big\|_{\BBm}
    \leq C \cdot C_{\max}
    < \infty.
  \]
\end{theorem}

\subsection{Bounding \texorpdfstring{$\|K_{\theta_1,\theta_2}\|_{\BBm}$}{the 𝓑ₘ norm of the kernel} via Fourier integral operators}

Towards an explicit estimate for $\|K_{\theta_1,\theta_2}\|_{\BBm}$, the next result
provides an estimate
in terms of families of Fourier integral operators~\cite{hormander1971fourier,duistermaat1972fourier,duistermaat2011fourier,st93}
dependent on $\theta_1,\theta_2$.
\nicki{Here and in the following, we use $e_{\tau}$, $\tau\in\RR^d$, as short-hand for the ($\Phi$-dependent) map
 \begin{equation}\label{eq:DefOfEtau}
     e_{\tau} \colon \quad \RR^d\times\RR^d \rightarrow \CC, \quad
     (x,\upsilon) \mapsto e^{-2\pi i \left\langle A^{-T}(\tau) \langle x \rangle, \Phi^{-1}(\upsilon+\tau)\right\rangle}.
  \end{equation}
  }

\begin{theorem}\label{thm:BBmFourierIntegralEstimate}
Define
  \nicki{\begin{equation}\label{eq:defOfL}
      L_{\tau_0}^{(\ell)} (x,\tau)
      := L_{\tau_0}[\theta_\ell,\theta_{3-\ell}] (x,\tau)
      := 
             \int_{\RR^d} \!\!\!
                 \frac{w(\upsilon+\tau_0)}{w(\tau_0)}
                 \cdot (\theta_{3-\ell}\cdot \overline{\bd T_\tau\theta_\ell})(\upsilon)
                 \cdot e_{\tau_0}(x,\upsilon)~
             d\upsilon\, ,
  \end{equation}
  for $\ell \in \{ 1,2 \}$ and $x,\tau,\tau_0 \in \R^d$.}
  If $m$ is $\Phi$-compatible with dominating weight $m^\Phi$, then we have
  \begin{equation}\label{eq:normcond_warped}
    \|K_{\theta_1,\theta_2}\|_{\BBm}
    \leq \max_{\ell\in\{1,2\}}
         \left[
           \esssup_{\eta\in D} \big\| L_{\Phi(\eta)}^{(\ell)} \big\|_{\lebesgue_M^1 (\R^d \times \R^d)}
         \right],
  \end{equation}
  with $M$ as in \eqref{eq:BigWeightDefinition}.
  In particular, if
  \(
    \esssup_{\tau_0 \in \R^d}
      \big\| L_{\tau_0}^{(\ell)} \big\|_{\lebesgue_M^1 (\R^d \times \R^d)}
    < \infty
  \),
  for $\ell \in \{ 1,2 \}$, then $\| K_{\theta_1,\theta_2} \|_{\BBm}$ is finite.
\end{theorem}

We prove \Cref{thm:BBmFourierIntegralEstimate} by means of two intermediate results.
First, an  (elementary) lemma concerned with the $\BBm$-norm of $K_{\theta_1,\theta_2}$.

\begin{lemma}\label{lem:CrossGramianBabyStep}
  If $m$ is $\Phi$-convolution-dominated by $m^\Phi$, we have
  \begin{equation}\label{eq:BBm_norm_warped_mixed}
  \begin{split}
  & \| K_{\theta_1, \theta_2} \|_{\BBm}\\
  & \leq \max_{\ell \in\{1,2\}}
          \left[
              \esssup_{\eta \in D}
                  \int_{D}
                      \esssup_{z \in \RR^d}
                          \int_{\RR^d}
                              m^{\Phi}(y \!-\! z,\Phi(\omega) \!-\! \Phi(\eta)) \cdot
                              |K_{\theta_\ell,\theta_{3-\ell}}((y,\omega),(z,\eta))|
                          ~dy
                  ~d\omega
          \right]
    \! .
  \end{split}
  \end{equation}
\end{lemma}
\begin{proof}
  If we define $\widetilde{m}^{\Phi} (x,\tau) := \min \{ m^{\Phi}(x,\tau), m^{\Phi}(-x,-\tau) \}$,
  the symmetry of $m$ easily shows that \eqref{eq:M0Domination} also holds for
  $\widetilde{m}^{\Phi}$ instead of $m^{\Phi}$.
  Hence, we can assume in what follows that $m^{\Phi}$ satisfies
  $m^{\Phi}(-x,-\tau) = m^{\Phi}(x,\tau)$ for all $x,\tau \in \R^d$.

  For $\ell \in \{ 1,2 \}$ and $\omega,\eta \in D$, define
  \[
    B_\ell(\omega,\eta)
    := \esssup_{z \in \R^d}
         \int_{\R^d}
           m^{\Phi}\bigl(y-z, \Phi(\omega) - \Phi(\eta)\bigr)
           \cdot \big| K_{\theta_{\ell},\theta_{3 - \ell}} \big( (y,\omega), (z,\eta) \big) \big|
         \, d  y ,
  \]
  and let $C := \max_{\ell \in \{ 1,2 \}} \esssup_{\eta \in D} \int_D B_\ell(\omega,\eta) \, d \omega$,
  which is precisely the right-hand side of the target inequality.
  Equation~\eqref{eq:M0Domination} yields
  \begin{equation}
    \begin{split}
      \esssup_{z \in \R^d}
        \int_{\R^d}
          \bigl|(m \cdot K_{\theta_1,\theta_2}) \big( (y,\omega),(z,\eta) \big)\bigr|
        \, d y
      \leq B_1(\omega,\eta) .
    \end{split}
    \label{eq:CrossGramianBabyStepsStep1}
  \end{equation}

  Next, note that $\langle \bd T_yf_1,\bd T_z f_2\rangle = \langle \bd T_{-z}f_1,\bd T_{-y} f_2\rangle$
  and $\langle f_1, f_2 \rangle = \overline{\langle f_2,f_1 \rangle}$ for all $f_1,f_2\in\LtRd$.
  Based on these identities and the translation-invariant structure of warped time-frequency
  systems, we see
  \begin{equation}
    \begin{split}
      K_{\theta_1,\theta_2}\bigl((y,\omega),(z,\eta)\bigr)
      & = K_{\theta_1,\theta_2}\bigl((-z,\omega),(-y,\eta)\bigr) \\
      & = \overline{K_{\theta_2,\theta_1}\bigl((-y,\eta),(-z,\omega)\bigr)}
        = \overline{K_{\theta_2,\theta_1}\bigl((z,\eta),(y,\omega)\bigr)}.
    \end{split}
   \label{eq:KernelEasyIdentities}
  \end{equation}
  Using these identities and renaming $\widetilde{z} = -y$ and $\widetilde{y} = -z$, we see
  \begin{equation}
    \begin{split}
      & \esssup_{y \in \R^d}
          \int_{\R^d}
            \bigl|(m \cdot K_{\theta_1,\theta_2}) \big( (y,\omega),(z,\eta) \big)\bigr|
          \, d z \\
      & \leq \esssup_{y \in \R^d}
               \int_{\R^d}
                 m^{\Phi}\bigl( (-z) - (-y), \Phi(\omega) - \Phi(\eta)\bigr)
                 \cdot \bigl|K_{\theta_1,\theta_2} \big( (-z,\omega),(-y,\eta) \big)\bigr|
               \, d z \\
      & = \esssup_{\widetilde{z} \in \R^d}
            \int_{\R^d}
              m^{\Phi}\bigl(\widetilde{y} - \widetilde{z}, \Phi(\omega) - \Phi(\eta)\bigr)
              \cdot \bigl|K_{\theta_1,\theta_2} \big( (\widetilde{y}, \omega), (\widetilde{z},\eta) \big)\bigr|
            \, d \widetilde{y}
        = B_1 (\omega,\eta) .
    \end{split}
    \label{eq:CrossGramianBabyStepsStep2}
  \end{equation}
  Combining \eqref{eq:CrossGramianBabyStepsStep1} and \eqref{eq:CrossGramianBabyStepsStep2},
  we see with notation as in \eqref{eq:PartialKernelDefinition} that
  \[
    \big\| (m \cdot K_{\theta_1,\theta_2})^{(\omega,\eta)} \big\|_{\AAi}
    \leq B_1 (\omega, \eta)
    \qquad \forall \, \omega,\eta \in D.
  \]

  A simple calculation using \eqref{eq:KernelEasyIdentities}
  and the symmetry $m^{\Phi}(-x,-\tau) = m^{\Phi}(x,\tau)$ proves the identity
  $B_1(\omega,\eta) = B_2(\eta,\omega)$.
  Overall, we thus see
  \begin{align*}
    \| m \cdot K_{\theta_1,\theta_2} \|_{\BBm}
    & = \max
        \Big\{
          \esssup_{\eta \in D}
            \int_{D}
              \| (m \cdot K_{\theta_1,\theta_2})^{(\omega,\eta)} \|_{\AAi}
            \, d \omega
          , \,\,
          \esssup_{\omega \in D}
            \int_{D}
              \| (m \cdot K_{\theta_1,\theta_2})^{(\omega,\eta)} \|_{\AAi}
            \, d \eta
        \Big\} \\
    & \leq \max
           \Big\{
             \esssup_{\eta \in D}
               \int_D
                 B_1(\omega,\eta)
               \, d \omega
             , \,\,
             \esssup_{\eta \in D}
               \int_D
                 B_2 (\eta,\omega)
               \, d \omega
           \Big\},
  \end{align*}
  which completes the proof.
\end{proof}

The second intermediate result expresses the integral over $D$ in \eqref{eq:BBm_norm_warped_mixed}
through the Fourier integral operators $L^{(\ell)}_{\tau_0}$.

\begin{lemma}\label{pro:kern_in_theta}
  Let $L_{\tau_0}^{(\ell)}$, $\ell \in \{ 1,2 \}$, be as in Theorem \ref{thm:BBmFourierIntegralEstimate}.
  For all $(y,\omega),(z,\eta)\in\Lambda$ and $\ell \in\{1,2\}$, we have
  \begin{equation}\label{eq:kern_in_terms_of_theta}
    \bigl|K_{\theta_\ell,\theta_{3-\ell}}\bigl((y,\omega),(z,\eta)\bigr)\bigr|
    = \sqrt{\!\frac{w(\Phi(\eta))}{w(\Phi(\omega))}}
      \cdot \left| L^{(\ell)}_{\Phi(\eta)}\bigl(A^T (\Phi(\eta))\langle z-y \rangle,\Phi(\omega) - \Phi(\eta)\bigr)\right|.
  \end{equation}

  If $m$ is $\Phi$-compatible with dominating weight $m^\Phi$, then we have,
  for given arbitrary $\ell \in \{ 1,2 \}$ and $\eta \in D$,
  \begin{equation}\label{eq:kerncond_warped}
    \int_{D}
       \esssup_{z \in \RR^d}
         \int_{\RR^d}
             m^{\Phi}\bigl(y-z,\Phi(\omega) - \Phi(\eta)\bigr)
             \cdot |K_{\theta_\ell,\theta_{3-\ell}}\bigl((y,\omega),(z,\eta)\bigr)|
         ~dy
    ~d\omega
    \leq \big\| L_{\Phi(\eta)}^{(\ell)} \big\|_{\lebesgue_M^1 (\R^d \times \R^d)},
  \end{equation}
  with $M$ as in \eqref{eq:BigWeightDefinition}.
\end{lemma}

\begin{proof}
  We provide the proof for $\ell =1$; the proof for $\ell =2$ follows the same steps.
  First, recall from after Equation~\eqref{eq:DerivativeOfInverse} the identity
  $0 < w(\Phi(\xi)) = [\det \mathrm{D}\Phi (\xi)]^{-1}$ for all $\xi \in D$.
  This identity will be applied repeatedly.
  To show \eqref{eq:kern_in_terms_of_theta}, apply Plancherel's theorem
  and perform the change of variable $\upsilon = \Phi(\xi)-\Phi(\eta)$ to derive
  \begin{equation}\label{eq:firstKest}
    \begin{split}
       & \left| K_{\theta_1, \theta_2} \bigl((y,\omega), (z,\eta)\bigr) \right|
        = \bigl| \big\langle g_{z,\eta}^{[2]}, g_{y,\omega}^{[1]} \big\rangle \bigr|
        = \big|
            \big\langle
              \widehat{\strut \smash{g_{z,\eta}^{[2]}}},
              \widehat{\strut \smash{g_{y,\omega}^{[1]}}}
            \big\rangle
          \big| \\
       &= \left|
             \int_{D}
                 \frac{\theta_2 (\Phi(\xi) - \Phi(\eta))
                       \cdot \overline{\theta_1 (\Phi(\xi) - \Phi(\omega))}}
                      {\sqrt{w(\Phi(\eta)) \cdot w(\Phi(\omega))}}
                 \cdot e^{-2\pi i \left\langle z-y, \xi\right\rangle}
             ~d\xi
          \right| \\
       &= \left|
             \int_{\RR^d}
                 \theta_2 (\upsilon)
                 \cdot \overline{\theta_1 (\upsilon + \Phi(\eta) - \Phi(\omega))}
                 \cdot \frac{w(\upsilon + \Phi(\eta))}{\sqrt{w(\Phi(\eta)) w(\Phi(\omega))}}
                 \cdot e^{-2\pi i \left\langle z-y, \Phi^{-1}(\upsilon + \Phi(\eta))\right\rangle}
             ~d\upsilon
          \right|.
    \end{split}
  \end{equation}
  This easily implies \eqref{eq:kern_in_terms_of_theta}.

  To prove \eqref{eq:kerncond_warped}, set $\tau_0 := \Phi(\eta)$ and note that
  \eqref{eq:kern_in_terms_of_theta} implies that the left-hand side of
  \eqref{eq:kerncond_warped} satisfies
  \begin{equation*}
   \begin{split}
    \circledast
    & := \mathrm{LHS} \eqref{eq:kerncond_warped}\\
    & = \int_{D}
         \esssup_{z \in \RR^d}
             \int_{\RR^d}
                 m^{\Phi}(y-z, \Phi(\omega) - \tau_0)
                 \cdot \sqrt{\frac{w(\tau_0)}{w(\Phi(\omega))}}
                 \cdot \left| L_{\tau_0}^{(1)} (A^T(\tau_0)\langle z-y\rangle, \Phi(\omega)-\tau_0)\right|
             ~dy
       ~d\omega .
       \end{split}
  \end{equation*}
  Next, perform the change of variable $\tau = \Phi(\omega) - \tau_0$ to obtain
  \[
    \circledast
    = \int_{\RR^d}
         \esssup_{z \in \RR^d}
            \int_{\RR^d}
               m^{\Phi}(y-z,\tau)
               \sqrt{w(\tau_0)w(\tau+\tau_0)}
               \cdot \left|L_{\tau_0}^{(1)} (A^T(\tau_0)\langle z-y\rangle,\tau)\right|
            ~dy
      ~d\tau
    =: \odagger .
  \]
  Next, perform the change of variables $x = A^T(\tau_0)\langle z-y\rangle$ in the inner integral
  and apply the estimate $\sqrt{\frac{w(\tau + \tau_0)}{w(\tau_0)}} \leq \sqrt{w_0 (\tau)}$ to derive
  \[
    \odagger
    \leq \int_{\RR^d}
            \int_{\RR^d}
                m^{\Phi}(-A^{-T}(\tau_0)\langle x\rangle, \tau)
                \cdot \sqrt{w_0(\tau)}
                \cdot \left|L_{\tau_0}^{(1)} (x,\tau)\right|
            ~dx
        ~d\tau .
  \]

  Now, in the case where $\mathrm{D}\Phi$ is unbounded, we are done, since in this case
  Equations~\eqref{eq:M0XDependence} and \eqref{eq:BigWeightDefinition} show
  \[
    m^{\Phi}(-A^{-T}(\tau_0)\langle x\rangle, \tau) \cdot \sqrt{w_0 (\tau)}
    = m^{\Phi}(0,\tau) \cdot \sqrt{w_0 (\tau)}
    \leq M(x,\tau) .
  \]
  For the case that $\mathrm{D}\Phi$ is bounded,
  recall from \eqref{eq:DerivativeOfInverse}
  that $A^{-T}(\tau_0) = A^{-T}(\Phi(\eta)) = [\mathrm{D}\Phi]^T(\eta)$ and thus
  $|A^{-T}(\tau_0)\langle x \rangle| = |\mathrm{D}\Phi^T (\eta) \langle x \rangle| \leq R |x|$
  by choice of $R$ in \eqref{eq:BigWeightDefinition}.
  Therefore, by choice of $M$, we see
  \[
           m^{\Phi}(-A^{-T}(\tau_0)\langle x\rangle,\tau) \cdot \sqrt{w_0(\tau)}
      \leq M(x,\tau).
      \qedhere
  \]
\end{proof}

We now obtain Eq.~\eqref{eq:normcond_warped} in Theorem~\ref{thm:BBmFourierIntegralEstimate}
simply by inserting Eq.~\eqref{eq:kerncond_warped} into Eq.~\eqref{eq:BBm_norm_warped_mixed}.

\subsection{Uniform integrability of the integral kernels \texorpdfstring{$L^{(\ell)}_{\tau_0}$}{from Equation (\ref{eq:defOfL})}}

To control $\esssup_{\eta\in D} \|L^{(\ell)}_{\Phi(\eta)}\|_{\bd L^1_M}$,
we find that $k$-admissibility of the warping function $\Phi$ is crucial.
The remainder of this subsection is dedicated to proving
Theorem~\ref{lem:NiceCrossGramianEstimate} below,
which will in turn be central to proving Theorem~\ref{thm:MR1_kernel_is_in_AAm}.

\begin{theorem}\label{lem:NiceCrossGramianEstimate}
  Let $\Phi$ be a $k$-admissible warping function with control weight $v_0$.
  Furthermore, let $w_1 : \RR^d \to \R^+$ be continuous
  and submultiplicative and such that $w_1(-\upsilon) = w_1(\upsilon)$ for all $\upsilon \in \R^d$.
  Define
  \[
    w_2 : \quad
    \RR^d \to \R^+, \quad
    \upsilon \mapsto w_1 (\upsilon) \cdot [v_0 (\upsilon)]^{d+3k} ,
  \]
  assume that $\theta_1,\theta_2\in \mathcal C^{k}(\RR^d)$ are such that
  \begin{equation}
    \frac{\partial^n}{\partial \upsilon_j^n} \theta_\ell \in \lebesgue^2_{w_2}(\RR^d),
    \qquad \text{ for all } j \in \underline{d},\  \ell \in \{1,2\},\ 0\leq n\leq k,
    \label{eq:NiceCrossGramianAssumption}
  \end{equation}
  \nicki{and recall from Equation \eqref{eq:DefOfCmax}, that }
  \[
      C_{\max}
     = \prod_{\ell \in \{1,2\}}
           \Bigg(
              \max_{j \in \underline{d}} \,\,
                  \max_{0 \leq n \leq k}
                      \bigg\|
                          \frac{\partial^n}{\partial \upsilon_j^n} \theta_\ell
                      \bigg\|_{\lebesgue^2_{w_2}(\RR^d)}
           \Bigg) .
  \]
  
  Then, with $e_{\tau}$ as defined in Equation~\eqref{eq:DefOfEtau}
  and $L_{\tau_0}^{(\ell)}$ as in \Cref{thm:BBmFourierIntegralEstimate}
  there exists a constant 
 $C = C(d, k, v_0) > 0$ satisfying
 for all $x,\tau,\tau_0 \in \R^d$ and $\ell \in \{ 1,2 \}$ the estimate
 \begin{equation}\label{eq:NiceCrossGramianEstimate}
    \left|L_{\tau_0}^{(\ell)} (x,\tau)\right|
    = \left|
        \int_{\RR^d}
         \frac{w(\upsilon+\tau_0)}{w(\tau_0)}
         \left(
           \theta_{3-\ell} \cdot \overline{\bd T_{\tau}\theta_\ell}
         \right) \! (\upsilon)
         \, e_{\tau_0}(x,\upsilon)
       ~d\upsilon
      \right|
    \leq C \cdot  C_{\max} \cdot (1 + |x|)^{-k} \cdot [w_1 (\tau)]^{-1} .
  \end{equation}
\end{theorem}

\begin{rem*}
  In Section \ref{sec:discretewarped}, we will apply \Cref{lem:NiceCrossGramianEstimate} in a setting
  in which $\theta_1, \theta_2$ depend on $x,\tau,\tau_0$.
  We suggest that the reader keeps this potential dependency in mind.
\end{rem*}

%

In a first step, we derive a number of important consequences of Definition~\ref{assume:DiffeomorphismAssumptions}
that will be used repeatedly.

\begin{lemma}\label{lem:assume_conclude}
  If $\Phi$ is a $0$-admissible warping function with control weight $v_0$,
  then $\Phi$ is a warping function in the sense of Definition~\ref{def:warpfun}.
  In particular, $w = \det A$ is $w_0$-moderate with $w_0 = v_0^d$, i.e.
  \begin{align}
    & w(\upsilon+\tau) \leq w(\upsilon) \cdot [v_0(\tau)]^d
      \qquad \forall \, \tau, \upsilon \in \R^d
    \label{eq:wModerateness}
    \\
    \text{and} \quad
    &
    \|A(\upsilon+\tau)\|
    \leq \|A(\upsilon)\| \cdot v_0(\tau)
    \qquad \forall \, \tau,\upsilon\in\RR^d.
    \label{eq:ModeratenessAssumption}
  \end{align}
  Additionally, for arbitrary $\gamma \in S^{d-1}$ and $\tau, \upsilon \in \R^d$, we have
  \begin{equation}
    [v_0 (\upsilon-\tau)]^{-1}
    \leq \|A^{-1}(\tau) A(\upsilon)\|^{-1}
    \leq |\phi_{\upsilon}(\tau - \upsilon) \langle \gamma \rangle|
    \leq \|A^{-1}(\upsilon) \cdot A (\tau)\|
    \leq v_0 (\tau-\upsilon)
    \label{eq:PhiUpperLowerEstimate}
  \end{equation}
  and
  \begin{equation}
    \phi_{\tau_0}(\upsilon)
    = \phi_{\tau_0 + \tau} (\upsilon-\tau) \cdot \phi_{\tau_0}(\tau)
    .
    \label{eq:PhiAlmostGroup}
  \end{equation}
  Finally, we have
  \begin{equation}
    [v_0(\tau)]^{-1} \cdot |\gamma|
    \leq |\phi_{\upsilon}(\tau) \langle \gamma \rangle|
    \leq v_0(\tau) \cdot |\gamma|
    \qquad \forall \, \gamma \in \R^d \text{ and } \tau,\upsilon \in \RR^d.
    \label{eq:PhiTauUpperLowerBounds}
  \end{equation}
\end{lemma}

\begin{proof}
  To show that $\Phi$ is a warping function, we need only verify moderateness of $w = \det A$.
  To prove this moderateness, apply \emph{Hadamard's inequality}
  $|\det M| \leq \|M\|^d = \|M^T\|^d$ (see \cite[Chapter~75]{RieszNagyFunctionalAnalysis})
  for $M \in \RR^{d \times d}$, combined with \eqref{eq:PhiHigherDerivativeEstimate}
  (for $\alpha = 0$) to see that
  \[
    \frac{w(\upsilon+\tau)}{w(\upsilon)}
    = \det \left([A(\upsilon)]^{-1} A(\upsilon+\tau) \right)
    \leq \big\| [A(\upsilon)]^{-1} A(\upsilon+\tau) \big\|^d
    =    \big\| [\phi_{\upsilon} (\tau)]^T \big\|^d
    \leq [v_0(\tau)]^d.
  \]
  Hence, we obtain \eqref{eq:wModerateness}.
  Moreover,
  \[
    \|A(\upsilon+\tau)\|
    = \| A(\upsilon) A^{-1}(\upsilon) A(\upsilon+\tau) \|
    \leq \|A(\upsilon)\| \cdot \| [\phi_{\upsilon}(\tau)]^T \|
    \smash{\overset{\eqref{eq:PhiHigherDerivativeEstimate}}{\leq}} \|A(\upsilon)\| \cdot v_0(\tau),
  \]
  proving \eqref{eq:ModeratenessAssumption}.
  To show \eqref{eq:PhiUpperLowerEstimate}, first note for $\gamma \in S^{d-1}$
  and any $M \in \mathrm{GL}(\RR^d)$ that $|M\gamma| \geq \|M^{-1}\|^{-1}$,
  and then apply \eqref{eq:PhiHigherDerivativeEstimate} twice:
  \begin{align*}
    \frac{1}{v_0 (\upsilon-\tau)}
    & \leq  \big\| [\phi_\tau (\upsilon - \tau)]^T \big\|^{-1}
      =     \|A^{-1}(\tau) A(\upsilon)\|^{-1}
      =     \big\| [A^T(\tau) \cdot A^{-T}(\upsilon)]^{-1} \big\|^{-1}
      =     \big\| [\phi_{\upsilon}(\tau - \upsilon)]^{-1} \big\|^{-1} \\
    & \leq  |\phi_{\upsilon}(\tau - \upsilon)\gamma|
      \leq  \|\phi_{\upsilon}(\tau - \upsilon)\|
      =     \|A^{-1}(\upsilon) \cdot A (\tau)\| \\
    & \leq  v_0 (\tau-\upsilon),
    \quad \text{ for all }\quad  \upsilon,\tau \in \RR^d,\quad  \gamma \in S^{d-1}.
  \end{align*}

  Finally, assertion \eqref{eq:PhiAlmostGroup} is easily verified using direct computation,
  and $[v_0(\tau)]^{-1}\leq |\phi_{\upsilon}(\tau)\cdot \gamma| \leq v_0(\tau)$
  for $|\gamma| = 1$ is obtained from \eqref{eq:PhiUpperLowerEstimate} through the bijective map
  $\tau \mapsto \tau-\upsilon$ and using that $v_0$ is radial.
  This proves \eqref{eq:PhiTauUpperLowerBounds}.
\end{proof}

Lemma~\ref{lem:assume_conclude} shows that $w$ is $v_0^d$-moderate.
The next result provides $v_0^d$-moderateness (up to a constant) for the partial derivatives of $w$.

\begin{lemma}\label{cor:deriv_of_w_estimate}
  Let $\Phi$ be a $k$-admissible warping function with control weight $v_0$.
  For every $j \in \underline{d}$ and $n \in \N_0$ with $n \leq k$, we have
  \begin{equation}\label{eq:deriv_of_w_estimate}
    \left|
        \frac{\partial^n}{\partial \upsilon^n_j} w(\upsilon+\tau)
    \right|
    \leq D_{n} \cdot [v_0(\upsilon)]^d \cdot w(\tau),
    \text{ for all } \upsilon,\tau\in\RR^d,
  \end{equation}
   with $D_{n} := D_{n}(d) := d!\cdot d^n$.
\end{lemma}

\begin{proof}
  We begin by rewriting $\frac{\partial^n}{\partial \upsilon^n_j}w(\upsilon+\tau)$
  using some simple properties of determinants:
  \begin{equation*}
  \begin{split}
  \frac{\partial^n}{\partial \upsilon^n_j}w(\upsilon+\tau)
  & = \frac{\partial^n}{\partial \upsilon^n_j}\det(A(\upsilon+\tau))
    = \frac{\partial^n}{\partial \upsilon^n_j}\det(A^T(\upsilon+\tau)) \\
  & = \det(A^T(\tau))\frac{\partial^n}{\partial \upsilon^n_j} \det(A^T(\upsilon+\tau)A^{-T}(\tau))
  \overset{\eqref{eq:PhiDefinition}}{=}
      w(\tau)\frac{\partial^n}{\partial \upsilon^n_j}\det(\phi_\tau(\upsilon)) .
  \end{split}
  \end{equation*}
  Let $S_d$ be the set of permutations on $\underline{d}$.
  Then, the definition of the determinant yields
  \[
    \frac{\partial^n}{\partial \upsilon^n_j}\det(\phi_\tau(\upsilon))
    = \frac{\partial^n}{\partial \upsilon^n_j}
        \left[
            \sum_{\sigma\in S_d}\sgn(\sigma)
                \prod_{i=1}^d
                    [\phi_\tau(\upsilon)]_{i,\sigma(i)}
        \right]
    = \sum_{\sigma\in S_d}
         \sgn(\sigma)
         \frac{\partial^n}{\partial \upsilon^n_j}
         \prod_{i=1}^d [\phi_\tau(\upsilon)]_{i,\sigma(i)}.
  \]
  The general Leibniz rule for products with $d$ terms shows
  \[
    \frac{\partial^n}{\partial \upsilon^n_j} \prod_{i=1}^d [\phi_\tau(\upsilon)]_{i,\sigma(i)}
    = \sum_{\substack{m_1,\dots,m_d \in \N_0, \\ m_1+\ldots+m_d = n}}
        \binom{n}{m_1,\ldots,m_d}
        \prod_{i=1}^d
          \frac{\partial^{m_i}}{\partial \upsilon_j^{m_i}}
            [\phi_\tau(\upsilon)]_{i,\sigma(i)},
  \]
  where $\binom{n}{m_1,\ldots,m_d} := \frac{n!}{m_1!\cdots m_d!}$
  is the usual multinomial coefficient.
  Moreover, the estimate \eqref{eq:PhiHigherDerivativeEstimate} yields
  \[
   \left|
     \frac{\partial^{m_i}}{\partial \upsilon_j^{m_i}}[\phi_\tau(\upsilon)]_{i,\sigma(i)}
   \right|
   \leq \left\Vert
                  \partial^{m_ie_j}\phi_{\tau}\left(\upsilon\right)
              \right\Vert \leq v_0(\upsilon).
  \]
  Altogether, we obtain
  \[
    \begin{split}
    \left|\frac{\partial^n}{\partial \upsilon^n_j} w(\upsilon+\tau)\right|
    & = \left|
          w(\tau)
          \cdot \frac{\partial^n}{\partial \upsilon^n_j}\det(\phi_\tau(\upsilon))
        \right| \\
    & \leq w(\tau)
           \cdot [v_0(\upsilon)]^d
           \cdot \sum_{\sigma\in S_d} \,\,\,
                   \sum_{m_1+\ldots+m_d = n}
                     \binom{n}{m_1,\ldots,m_d} \\
    & = d!
        \cdot d^n
        \cdot [v_0(\upsilon)]^d
        \cdot w(\tau)
      = D_{n} \cdot [v_0(\upsilon)]^d \cdot w(\tau) ,
    \end{split}
  \]
  where we used $|S_d| = d!$ and the multinomial theorem
  (see e.g.\@ \cite[Exercise 2(a)]{FollandRA}), i.e.
  \[
    \sum_{m_1+\ldots+m_d = n}
      \binom{n}{m_1,\ldots,m_d}
      \prod_{i=1}^d a_i^{m_i}
    = (a_1+\ldots+a_d)^n,
    \text{ for all } n \in \NN,\ (a_i)_{i\in\underline{d}} \in \RR^d,
  \]
  for $a_1,\ldots,a_d = 1$.
  Thus, the proof is complete.
\end{proof}

We now turn our attention towards the Fourier integral operators $L^{(\ell)}_{\tau_0}$ defined in \eqref{eq:defOfL}.
We will obtain the desired integrability with respect to $x\in\RR^d$
by means of an integration by parts argument of the kind well-known
for establishing the smoothness-decay duality of a function and its Fourier transform,
as well as the asymptotic behavior of oscillatory integrals, cf.~\cite[Chapter~VIII]{st93}.
An additional complication in our setting is that we require a uniform estimate
over all $L^{(\ell)}_{\tau_0}$, $\ell\in\{1,2\}$, $\tau_0\in\RR^d$.

For now, we replace
$\frac{w(\upsilon+\tau_0)}{w(\tau_0)} \left(\theta_2 \cdot \overline{\bd T_{\tau}\theta_1}\right)(\upsilon)$
in \eqref{eq:defOfL} by an unspecific, compactly supported function $g\in\mathcal C^k_c(\R^d)$, i.e., we consider
\begin{equation}\label{eq:exp_abbrv}
 \int_{\RR^d} g(\upsilon)\cdot e_{\tau_0}(x,\upsilon)~d\upsilon, \text{ recalling } e_{\tau}(x,\upsilon)
  = e^{-2\pi i \langle A^{-T}(\tau)\langle x\rangle, \Phi^{-1}(\upsilon+\tau)\rangle},
  \quad \text{ for all } x,\upsilon,\tau\in\RR^d.
\end{equation}
Note that, with $f =  e_{\tau}(x,\bullet)$, we have
\[
 \frac{\partial}{\partial \upsilon_j} f(\upsilon)
 = -2\pi i
   \cdot \left\langle
                A^{-T}(\tau) \langle x\rangle ,
                \frac{\partial}{\partial \upsilon_j}\Phi^{-1}(\upsilon+\tau)
         \right\rangle
   \cdot e_{\tau}(x,\upsilon)
 = -2\pi i
   \cdot (\phi_{\tau}(\upsilon)\cdot x)_j
   \cdot e_{\tau}(x,\upsilon) .
\]
The final equality can be verified by observing
\begin{equation}
  \begin{split}
       \big\langle
            A^{-T}(\tau) \langle \eta\rangle,
            \frac{\partial}{\partial \upsilon_j}\Phi^{-1}(\upsilon+\tau)
          \big\rangle
      & = \inner{A^{-T}(\tau) \langle\eta\rangle}{\mathrm{D}\Phi^{-1}(\upsilon+\tau)\langle e_j\rangle}\\
      & =  \inner{A^{-T}(\tau) \langle \eta\rangle}
                 {A(\upsilon+\tau) \langle e_j\rangle}
        = (\phi_{\tau}(\upsilon) \cdot \eta)_j,
  \end{split}
  \label{eq:PsiDerivativeIsPhi}
\end{equation}
which motivates the definition of $\phi_{\tau}$.
Provided $(\phi_{\tau}(\upsilon)\cdot x)_j\neq 0$ on the support of $g$, we obtain, with
\(
  \tilde{g}(\upsilon)
  = \left(-2\pi i \cdot (\phi_{\tau}(\upsilon) \langle x \rangle )_j\right)^{-1}\cdot g(\upsilon)
  ,
\)
\[
  \int g(\upsilon) f(\upsilon) d \upsilon
  = \int \tilde{g}(\upsilon)\frac{\partial}{\partial \upsilon_j} f(\upsilon) d\upsilon
  = -\int\frac{\partial}{\partial \upsilon_j} \tilde{g}(\upsilon)f(\upsilon) d\upsilon,
\]
where the last equality is obtained through integration by parts.

For fixed $x,\tau\in\RR^d$ and $j \in \underline{d}$ and all $g\in\mathcal C^k_c (\R^d)$ such that
$(\phi_{\tau} (\cdot) \langle x \rangle)_j \neq 0$ on the support of $g$,
we define the differential operator $\Square_{j,\tau,x}$ by
\begin{equation}\label{eq:differentialOperator}
  \left(\Square_{j,\tau,x} \, g\right)(\upsilon)
    := (2\pi i)^{-1}
       \frac{\partial}{\partial \upsilon_j}
          \left[
            \frac{g(\upsilon)}{(\phi_{\tau}(\upsilon) \langle x \rangle )_j}
          \right]
    = (2\pi i|x|)^{-1}
       \frac{\partial}{\partial \upsilon_j}
           \left[
               \frac{g(\upsilon)}{(\phi_{\tau}(\upsilon) \langle \rho_x \rangle)_j}
           \right],
\end{equation}
where $\rho_x\in S^{d-1}$ with  $x = |x|\rho_x$.
We can rewrite the integral in \eqref{eq:exp_abbrv} as
\begin{equation}\label{eq:partialint}
 \int_{\RR^d} g(\upsilon) \, e_{\tau}(x,\upsilon) d\upsilon
 = \int_{\RR^d} \!
                            \left(\Square_{j,\tau,x} \, g\right)(\upsilon)
                            e_{\tau}(x,\upsilon)
                        d\upsilon
   = \int_{\RR^d} \!
          \left(\Square_{j,\tau,x}^n \, g\right) \!\! (\upsilon)
          \, e_{\tau}(x,\upsilon)
        d\upsilon,
  \text{ for } n \leq k.
\end{equation}
where $\Square_{j,\tau,x}^n$ denotes $n$-fold application of $\Square_{j,\tau,x}$.

By \eqref{eq:differentialOperator}, each application of $\Square_{j,\tau,x}$ provides additional,
linear decay with respect to $|x|$, $x\in\RR^d$.
For a given pair $(\Phi,\theta)$ of warping function and prototype, however,
we cannot expect the support restriction required for the application
of the differential operator $\Square_{j,\tau,x}$, i.e.,
$(\phi_{\tau}(\cdot) \langle x \rangle )_j \neq 0$ on the support of $g$, to hold.
To account for this, we decompose
\(
  g_{\tau,\tau_0}(\upsilon)
  := \frac{w(\upsilon+\tau_0)}{w(\tau_0)}
     \left(
       \theta_2 \cdot \overline{\bd T_{\tau}\theta_1}
     \right)(\upsilon)
\)
into compactly supported functions, such that each of them allows
the application of $\Square_{j,\tau,x}$, for some $j \in \underline{d}$.
Therefore, our next steps are:
\begin{itemize}
    \item \textbf{Step 1: } Find a suitable splitting
          \(
            g_{\tau,\tau_0}
            = \sum_{i \in I} g_{i,\tau,\tau_0}
            = \sum_{i \in I} \varphi_i g_{\tau, \tau_0}
          \)
          (with $(\varphi_i)_i$ only depending on $\tau_0$) into compactly supported elements
          $g_{i,\tau,\tau_0} = \varphi_i g_{\tau, \tau_0}$,
          such that, for any fixed $\rho_x\in S^{d-1}$, $\tau_0\in\RR^d$ and $i \in I$,
          there is an index $j = j(\rho_x,\tau_0,i)\in\underline{d}$
          and a positive function $\tilde{v}$ (\emph{independent} of $i,\rho_x,\tau_0, \tau$),
          such that
          $| (\phi_{\tau_0}(\upsilon) \langle \rho_x \rangle )_j | \geq \tilde{v}(\upsilon) > 0$
          for $\upsilon \in \supp \varphi_i \supset \supp(g_{i,\tau,\tau_0})$.
          Besides being able to apply $\Square_{j,\tau_0,x}$, this property lets us control
          the growth of $\frac{1}{(\phi_{\tau_0}(\upsilon) \langle \rho_x \rangle)_j}$
          independently of the orientation $\rho_x\in S^{d-1}$ and of $\tau, \tau_0$.

    \item \textbf{Step 2: } Estimate $\left(\Square_{j,\tau_0,x}^n \,\, g_{i,\tau,\tau_0}\right)(\upsilon)$,
          for $x=|x| \cdot \rho_x \neq 0$, independently of $i,\rho_x,\tau_0$.
          In fact, this estimate will exhibit rapid decay with respect to $|x|$
          and depend boundedly on the derivative of $g_{i,\tau,\tau_0}$,
          which can be used to obtain decay with respect to $|\tau|$.
\end{itemize}

Towards Step 1, we introduce a specific family of coverings
in the following lemma.
The smooth splitting of $g_{\tau,\tau_0}$ into the building blocks
$g_{i,\tau,\tau_0} = \varphi_i g_{\tau, \tau_0}$, see Lemma~\ref{lem:IntegralSegmentation}, is provided
by a $\mathcal C^\infty_c$ partition of unity $(\varphi_i)_i$ with respect to these coverings,
introduced in Lemma~\ref{lem:SmoothPartitionOfUnity}.
Lemmas~\ref{lem:partial_derivs_of_localized_gte} and \ref{lem:estimate_diffop_abs}
take care of Step 2.

\begin{lemma}\label{lem:phiIsReasonable}
  Let $\Phi$ be a $1$-admissible warping function with control weight $v_0$.
  For any $\upsilon_0,\tau_0\in\RR^d$, the following are true:
  \begin{enumerate}
   \item The family $\left(U_j^{(\upsilon_0,\tau_0)}\right)_{j \in\underline{d}}$ defined by
         \[
              U_j^{(\upsilon_0,\tau_0)}
              := \left\{
                    \gamma\in S^{d-1}
                    ~:~
                    \left|
                      \left(\phi_{\tau_0}(\upsilon_0) \langle \gamma \rangle \right)_j
                    \right|
                    > \frac{1}{2d}
                      \left|
                          \phi_{\tau_0}(\upsilon_0) \langle \gamma \rangle
                      \right|
                 \right\}
         \]
         is a covering of $S^{d-1}$.

   \item For any $\delta > 0$ satisfying $\delta \cdot v_0(\delta / (4d) \cdot e_1) \leq 1/\sqrt{d}$
         and arbitrary $\upsilon\in B_{\delta/(4d)}(\upsilon_0)$
         and $\gamma\in U_j^{(\upsilon_0, \tau_0)}$, we have
         \[
           \big| \left(\phi_{\tau_0}(\upsilon) \langle \gamma \rangle \right)_j \big|
           \geq C_{\delta} \cdot [v_0(\upsilon)]^{-1},
         \]
         with $C_{\delta}\nicki{: = C_{\delta}(d,v_0)} := \bigl[4d \cdot v_0(\delta/(4d)\cdot e_1)\bigr]^{-1}$.
  \end{enumerate}
\end{lemma}

\begin{rem*}
  If $\delta \leq \min \bigl\{ 1, 1/ (\sqrt{d} \cdot v_0(e_1/(4d))) \bigr\}$, then
  $\delta \cdot v_0(\delta/(4d) \cdot e_1) \leq \delta \cdot v_0(e_1 / (4d)) \leq 1/\sqrt{d}$.
  Hence, the condition of Part~(2) of the lemma is satisfied for all sufficiently
  small $\delta > 0$.
\end{rem*}

\begin{proof}
  Part (1) does not use any of the properties of $\Phi$,
  except that $\phi_{\tau_0} (\upsilon) \in \GL (\RR^d)$:
  We simply note that any $z\in\mathbb{R}^{d}\setminus\left\{ 0\right\} $ satisfies
  \[
      \left|z\right|
      \leq \sum_{j=1}^{d}
             \left|z_{j}\right|
      \leq d  \cdot \max\bigl\{ \left|z_{j}\right| \colon j \in \underline{d} \bigr\}
      <    2d \cdot \max\bigl\{ \left|z_{j}\right| \colon j \in \underline{d} \bigr\} .
  \]
  Hence, there is some $j \in \underline{d}$ with $| z_{j} | > \frac{1}{2d} \cdot |z|$.
  Now apply this to $z=\phi_{\tau_0}(\upsilon_0) \langle \gamma \rangle$, noting
  that $z\neq0$ since $\phi_{\tau_0}(\upsilon_0)\in\mathrm{GL}(\mathbb{R}^{d})$
  and $\gamma\in S^{d-1}$.

  For part (2), let $\upsilon\in B_{\delta / (4d)} (\upsilon_0)$
  and $\gamma\in U_j^{(\upsilon_0,\tau_0)} \subset S^{d-1}$ be arbitrary.
  The triangle inequality provides
  \begin{align*}
           |(\phi_{\tau_0}(\upsilon) \langle \gamma \rangle )_j|
     &\geq |(\phi_{\tau_0}(\upsilon_0) \langle \gamma \rangle )_j|
           - |
                (\phi_{\tau_0}(\upsilon) \langle \gamma \rangle
              - \phi_{\tau_0}(\upsilon_0) \langle \gamma \rangle )_j
             | \\
     ({\scriptstyle{\text{since }\gamma \in U_j^{(\upsilon_0,\tau_0)}}})
     & \geq   \frac{|\phi_{\tau_0}(\upsilon_0)\langle \gamma \rangle |}{2d}
            - |
               \left(
                 \phi_{\tau_0}(\upsilon) - \phi_{\tau_0}(\upsilon_0)
               \right)
               \langle \gamma \rangle
              |.
  \end{align*}
  Note that
  \begin{equation}
      \begin{split}
            \phi_{\tau_0}(\upsilon) - \phi_{\tau_0}(\upsilon_0)
         &= \left(\phi_{\tau_0 + \upsilon_0} (\upsilon - \upsilon_0) - \mathrm{id}\right)
            \phi_{\tau_0}(\upsilon_0),
      \end{split}
      \label{eq:differenceOfphi}
  \end{equation}
  where we used the identity \eqref{eq:PhiAlmostGroup} of Lemma~\ref{lem:assume_conclude},
  with $\tau = \upsilon_0$.

  To estimate the first factor on the right-hand side of \eqref{eq:differenceOfphi},
  recall that $\phi_{\tau_0 +\upsilon_0}(0) = \mathrm{id}$.
  Therefore,
  \begin{align*}
     \big\| \mathrm{id} - \phi_{\tau_0 + \upsilon_0}(\upsilon - \upsilon_0) \big\|
     &= \big\|
          \phi_{\tau_0 + \upsilon_0}(0)
          - \phi_{\tau_0 + \upsilon_0} (\upsilon - \upsilon_0)
        \big\| \\
     &=\left\|
         \int_0^1
           \frac{d}{dt}\bigg|_{t=s}
             \big[
               \phi_{\tau_0 + \upsilon_0} (t(\upsilon - \upsilon_0))
             \big]
         \, ds
       \right\|\\
     & \leq \int_0^1
               \sum_{\ell=1}^d
                 \big\|
                   ( \partial_\ell \,\, \phi_{\tau_0 +\upsilon_0}) (s(\upsilon - \upsilon_0))
                 \big\|
                 \cdot \bigl|(\upsilon - \upsilon_0)_\ell\bigr|
            \, ds
     =: (\ast) .
  \end{align*}
  We now rewrite this expression further, recalling that $v_0$ is radially increasing and applying
  the Cauchy-Schwarz inequality, and inequality \eqref{eq:PhiHigherDerivativeEstimate}:
  \begin{align*}
      (\ast)
      & \overset{\text{CS}}{\leq}
         |\upsilon - \upsilon_0| \cdot
         \sup_{t \in [0,1]}
           \left|
            \left(
               \begin{matrix}
                  \left\|
                    (\partial_1 \, \phi_{\tau_0 + \upsilon_0}) (t(\upsilon-\upsilon_0))
                  \right\| \\
                  \vdots \\
                  \left\|
                    (\partial_d \, \phi_{\tau_0 + \upsilon_0}) (t(\upsilon-\upsilon_0))
                  \right\|
               \end{matrix}
            \right)
           \right|\\
     & \overset{\eqref{eq:PhiHigherDerivativeEstimate}}{\leq}
         |\upsilon - \upsilon_0| \cdot \sqrt{d}
         \cdot \sup_{t \in [0,1]}v_0 (t (\upsilon - \upsilon_0)) \\
     ({\scriptstyle{\text{since } |\upsilon - \upsilon_0| < \delta / (4d)}})
     & \leq
     \frac{\sqrt{d} \cdot \delta \cdot v_0 (\delta/(4d)\cdot e_1)}{4d}
     \leq \frac{1}{4d}.
  \end{align*}
  Hence,
  \[
    |(\phi_{\tau_0}(\upsilon)\langle  \gamma \rangle )_j|
    \geq \frac{|\phi_{\tau_0}(\upsilon_0)\langle \gamma \rangle |}{2d}
         - \|\phi_{\tau_0 + \upsilon_0}(\upsilon - \upsilon_0) - \mathrm{id}\|
           \cdot |\phi_{\tau_0}(\upsilon_0) \langle \gamma \rangle | \\
    \geq \frac{|\phi_{\tau_0}(\upsilon_0) \langle \gamma \rangle |}{4d}.
  \]

  To finish the proof, it remains to show
  $|\phi_{\tau_0}(\upsilon_0)\langle \gamma \rangle | \geq 4d C_\delta \cdot [v_0(\upsilon)]^{-1}$.
  To see this, note
  \begin{equation*}
    |\phi_{\tau_0}(\upsilon_0) \langle \gamma \rangle |
    \overset{\eqref{eq:PhiAlmostGroup}}{=}
      |
        \phi_{\tau_0 + \upsilon}(\upsilon_0 - \upsilon)
        \cdot \phi_{\tau_0}(\upsilon)\langle \gamma \rangle
      |
    \overset{\eqref{eq:PhiTauUpperLowerBounds}}{\geq}
      \frac{1}{v_0 (\upsilon - \upsilon_0) v_0(\upsilon)}
    \geq 4d C_\delta \cdot [v_0(\upsilon)]^{-1} ,
  \end{equation*}
  where we inserted $C_{\delta} = (4d \cdot v_0(\delta/(4d)\cdot e_1))^{-1}$,
  using $|\upsilon - \upsilon_0| < \delta / (4 d)$.
\end{proof}

\begin{lemma}\label{lem:SmoothPartitionOfUnity}
  Let $\delta' > 0$ be arbitrary.
  The sequence $\bigl(B_{\delta'}(\upsilon_i)\bigr)_{i\in \ZZ^d}$,
  with $(\upsilon_i)_{i \in \ZZ^d} = \bigl(\frac{\delta'}{\sqrt{d}} i\bigr)_{i \in \ZZ^d}$,
  is an open cover of $\RR^d$.
  Moreover, there is a collection of smooth functions $(\varphi_i)_{i \in \ZZ^d}$, such that
  \begin{enumerate}
   \item $\varphi_i \geq 0$ and $\varphi_i\in\mathcal C^\infty(\RR^d)$,

   \item $\operatorname{supp}(\varphi_i) \subset B_{\delta'}(\upsilon_i)$,

   \item $\sum_i \varphi_i \equiv 1$ on $\RR^d$, and

   \item for every multi-index $\alpha\in\NN^d_0$, there exists a constant $D_\alpha^ {(\delta')} > 0$
         such that $|\partial^\alpha \varphi_i(\upsilon)|\leq D_\alpha^{(\delta')}$
         uniformly over $i \in \ZZ^d$ and $\upsilon \in \RR^d$.
  \end{enumerate}
\end{lemma}

\begin{proof}
  The result is a direct consequence of standard constructions of smooth partitions of unity;
  see e.g.\ \cite[Theorem~1.4.6]{Hoermander2015Analysis}.
\end{proof}

\begin{lemma}\label{lem:IntegralSegmentation}
  Let $\Phi$ be a $k$-admissible warping function with control weight $v_0$ and $\delta>0$ be
  such that $\delta\cdot v_0(\delta/(4d)\cdot e_1) \leq 1/\sqrt{d}$.
  Set $\delta' = \delta/(4d)$ and let $(\upsilon_i)_{i\in\ZZ^d}$,
  $\left( B_{\delta / (4d)} (\upsilon_i) \right)_{i \in \ZZ^d}$ and $(\varphi_i)_{i \in \ZZ^d}$
  be as in \Cref{lem:SmoothPartitionOfUnity}.
  Then
  \begin{equation}
  \begin{split}
    \sharp\{i\in\ZZ^d~:~ B_{\delta'}(\upsilon_\ell) \cap B_{\delta'}(\upsilon_i) \neq \emptyset\}
    & \leq (1 \!+\! 4d)^d
    \quad \text{ and } \quad
    \sharp\{i\in\ZZ^d~:~ \upsilon \in B_{\delta'}(\upsilon_i)\} \leq (1 \!+\! 4d)^d,
    \label{eq:uniform_covering_admissible}
  \end{split}
  \end{equation}
  for all $\ell \in \ZZ^d$ and $\upsilon \in \RR^d$.
  For $\theta_1,\theta_2\in\lebesgue^2_{\sqrt{w_0}} (\RR^d)\cap \mathcal C^k(\RR^d)$,
  $\upsilon, \tau, \tau_0 \in \R^d$ and $i \in \Z^d$, define
  \begin{equation}
    g_{i,\tau,\tau_0}(\upsilon)
    := \varphi_i(\upsilon) g_{\tau,\tau_0}(\upsilon)
    \quad\text{ with }\quad
    g_{\tau,\tau_0}(\upsilon)
    := \frac{w(\upsilon + \tau_0)}{w(\tau_0)}
       \left(\theta_2 \cdot \overline{\bd T_{\tau}\theta_1}\right)(\upsilon).
  \label{eq:local_g_definition}
  \end{equation}
  Then $g_{i,\tau,\tau_0}\in \mathcal C^k_c(B_{\delta'}(\upsilon_i))$ and,
  for any fixed $\tau_0\in\RR^d$ and $x\in\RR^d\setminus\{0\}$,
  there exists a sequence $(j_i)_{i\in\ZZ^d}$ with $j_i \in \underline{d}$, such that
  $\rho_x := x/|x| \in U_{j_i}^{(\upsilon_i, \tau_0)}$
  (where this set is defined is in \Cref{lem:phiIsReasonable}) for all $i \in \ZZ^d$ and such that
  \begin{equation}
    \begin{split}
      \int_{\RR^d} g_{\tau,\tau_0}(\upsilon)
         e_{\tau_0}(x,\upsilon)
         ~d\upsilon
     &= \sum_{i \in \ZZ^d}
          \int_{\R^d}
            \left(\Square_{j_i,\tau_0,x}^n \, g_{i,\tau,\tau_0}\right)\!(\upsilon)
            \cdot e_{\tau_0}(x,\upsilon)
          ~d\upsilon,
      \quad \text{ for all } n \leq k.
    \end{split}
    \label{eq:partialint2}
  \end{equation}
\end{lemma}

\begin{proof}
The first assertion, \eqref{eq:uniform_covering_admissible}, is verified
by a straightforward calculation and $g_{i,\tau,\tau_0}\in \mathcal C^k_c(B_{\delta'}(\upsilon_i))$
is a consequence of \Cref{lem:SmoothPartitionOfUnity}, with $k$-admissibility of $\Phi$
and $\theta_1,\theta_2\in\mathcal C^k(\RR^d)$.
\Cref{lem:phiIsReasonable}(1) provides the existence of
$j_i = j_i (i, \tau_0,\rho_x) \in \underline{d}$ satisfying
$\rho_x \in U_{j_i}^{(\upsilon_i, \tau_0)}$, for arbitrary,
fixed $\tau_0,\rho_x$ and each $i\in\ZZ^d$.
The elements of the covering $\left( B_{\delta / (4d)} (\upsilon_i) \right)_{i \in \ZZ^d}$
are specific instances of the set in \Cref{lem:phiIsReasonable}(2),
such that the application of $\Square_{j_i,\tau_0,x}^n$, $n\in\NN$,
to $g_{i,\tau,\tau_0}$ is well-defined.
Thus, to prove \eqref{eq:partialint2} it only remains to justify the interchange
of integral and summation
\begin{equation}\label{eq:splitGSumIntegralSwitch}
  \int_{\RR^d}
        \sum_{i \in \ZZ^d}
          g_{i,\tau,\tau_0}(\upsilon)
        e_{\tau_0}(x,\upsilon)
      ~d\upsilon
   =
       \sum_{i\in\ZZ^d}
        \int_{\R^d}
          g_{i,\tau,\tau_0}(\upsilon)
          e_{\tau_0}(x,\upsilon)
        ~d\upsilon.
\end{equation}
Since
\begin{align*}
  \int_{\RR^d}
    \sum_{i \in \ZZ^d}
      |g_{i,\tau,\tau_0} (\upsilon) e_{\tau_0}(x,\upsilon) |
  \, d\upsilon
  &\leq \int_{\RR^d}
          \sum_{i \in \ZZ^d}
            \Indicator_{B_{\delta'}(\upsilon_i)}(\upsilon)
            \cdot |g_{\tau, \tau_0} (\upsilon)|
        \, d\upsilon \\
  ({\scriptstyle{\text{Eq. } \eqref{eq:uniform_covering_admissible}}})
  & \leq (1 + 4d)^d \cdot \| g_{\tau, \tau_0} \|_{\lebesgue^1} \\
  ({\scriptstyle{\text{since } \frac{w(\upsilon + \tau_0)}{w(\tau_0)} \leq w_0 (\upsilon)}})
  & \leq (1 + 4d)^d
         \cdot \|
                 w_0
                 \cdot \theta_2
                 \cdot \translation_{\tau} \theta_1
               \|_{\lebesgue^1} \\
   ({\scriptstyle{w_0 \text{ is submultiplicative}}})&\leq \sqrt{w_0(\tau)}\cdot (1 + 4d)^d
        \cdot \| \theta_2 \|_{\lebesgue^2_{\sqrt{w_0}}}
        \cdot \| \translation_{\tau} \theta_1 \|_{\lebesgue^2_{\sqrt{w_0}}}
  < \infty,
\end{align*}
the dominated convergence theorem justifies \eqref{eq:splitGSumIntegralSwitch}.
\end{proof}

To prepare for an estimate of $\Square_{j_i,\tau_0,x}^n \, g_{i,\tau,\tau_0}$ itself,
we consider the partial derivatives of $g_{i,\tau,\tau_0}$.

\begin{lemma}\label{lem:partial_derivs_of_localized_gte}
  Let $\Phi$ be a $k$-admissible warping function with control weight $v_0$,
  let $\theta_1,\theta_2 \in \Ltv \cap \mathcal{C}^k$,
  and let $(\varphi_i)_{i \in \ZZ^d}$ be a bounded partition of unity
  as in Lemma~\ref{lem:SmoothPartitionOfUnity}, for some given $\delta' > 0$.
  For any fixed $j \in \underline{d}$, $i \in \ZZ^d$ and $\tau, \tau_0 \in \RR^d$, we have 
  \nicki{\begin{equation}\label{eq:estimate_of_partial_deriv_abs}
    \left| \frac{\partial^n}{\partial \upsilon_j^n} \, g_{i,\tau,\tau_0}\right|
    \leq C_n \cdot v_0^{d} \cdot
         \sum_{\substack{m_1,m_2\in\NN_0\\ m_1+m_2\leq n}}
             \left|
               \frac{\partial^{m_1}}{\partial \upsilon_j^{m_1}}
               \theta_2
               \cdot \frac{\partial^{m_2}}{\partial \upsilon_j^{m_2}}
               \overline{\bd T_\tau \theta_1}
             \right|
    ,\text{ for all } n\leq k,
  \end{equation}
  for some constant $C_n = C_n (\delta', d) > 0$. Here 
  \(
    g_{i,\tau,\tau_0}(\upsilon)
    = \frac{w(\upsilon+\tau_0)}{w(\tau_0)}
      \varphi_i(\upsilon)
      \left( \theta_2 \cdot \overline{\bd T_\tau \theta_1} \right)(\upsilon)
  \)
  is as in \eqref{eq:local_g_definition}.}
\end{lemma}

\begin{proof}
  \nicki{We begin by applying the general Leibniz rule, with $4$ terms in this case,
    to rewrite the partial derivatives of  $g_{i,\tau,\tau_0}$:
  \begin{equation}\label{eq:form_of_partial_deriv}
   \begin{split}
   & \frac{\partial^n}{\partial \upsilon_j^n} \, g_{i,\tau,\tau_0}
     = \frac{\partial^n}{\partial \upsilon_j^n}
          \left(
            \frac{\bd T_{-\tau_0}w}{w(\tau_0)}
            \cdot \varphi_i
            \cdot \theta_2
            \cdot \overline{\bd T_\tau \theta_1}
          \right) \\
            & = \frac{1}{w(\tau_0)}
       \sum_{\substack{n_1,\ldots,n_4\in\NN_0\\ n_1+\cdots+n_4=n}}
         \binom{n}{n_1,\ldots,n_4} \cdot 
                    \frac{\partial^{n_1}}{\partial \upsilon_j^{n_1}}\bd T_{-\tau_0}w
                    \cdot \frac{\partial^{n_2}}{\partial \upsilon_j^{n_2}}
                            \varphi_i
                    \cdot \frac{\partial^{n_3}}{\partial \upsilon_j^{n_3}}
                            \theta_2
                    \cdot \frac{\partial^{n_4}}{\partial \upsilon_j^{n_4}}
                            \overline{\bd T_\tau \theta_1},
   \end{split}
  \end{equation}
  where $\binom{n}{n_1,\ldots,n_4} := \frac{n!}{n_1!n_2!n_3!n_4!}$ is, once more, the usual multinomial coefficient. 
 
  We now consider each term
  appearing in \eqref{eq:form_of_partial_deriv} individually.
  Since all the involved sums are finite, there is a finite constant
  $\widetilde{C}_n > 0$, depending only on $\delta'>0$,
  the chosen partition of unity $(\varphi_i)_{i \in \ZZ^d}$, and (implicitly) $d\in\NN$, such that
  \[
    \max_{\substack{n_1,\ldots,n_4\in\NN_0\\ n_1+\cdots+n_4=n}}
      \left|
        \binom{n}{n_1,\ldots,n_4} \cdot 
        \frac{\partial^{n_2}}
             {\partial \upsilon_j^{n_2}}
          \varphi_i
      \right|
    \leq \max_{\substack{n_1,\ldots,n_4\in\NN_0\\ n_1+\cdots+n_4=n}}
           \left( \binom{n}{n_1,\ldots,n_4} \cdot D^{(\delta')}_{n_2e_j}\right)
    \leq \widetilde{C}_n \, ,
  \]
  where property (4) of $(\varphi_i)_{i \in \ZZ^d}$ in Lemma~\ref{lem:SmoothPartitionOfUnity}
  was used, and $n_2e_j$ is interpreted as a multi-index.

  For the term
  $[w(\tau_0)]^{-1} \cdot \frac{\partial^{n_1}}{\partial \upsilon_j^{n_1}}w(\upsilon+\tau_0)$
  on the other hand, we apply the estimate given in Lemma~\ref{cor:deriv_of_w_estimate}, i.e.
  \[
    [w(\tau_0)]^{-1}
    \cdot \left|\frac{\partial^{n_1}}{\partial \upsilon_j^{n_1}}w(\upsilon+\tau_0)\right|
    \leq D_{n_1} \cdot  [v_0(\upsilon)]^{d}
    \leq \left(\max_{0 \leq m \leq n} D_{m} \right) \cdot [v_0 (\upsilon)]^d
    = D_n \cdot [v_0 (\upsilon)]^d ,
  \]
  where $D_n = d!d^n$ as in Lemma~\ref{cor:deriv_of_w_estimate}. With $C_n := D_n\widetilde{C_n}$, we see that 
  \[
   \begin{split}
     \left|\frac{\partial^n}{\partial \upsilon_j^n} \, g_{i,\tau,\tau_0} \right|
     & \leq D_{n} \widetilde{C}_n
            \cdot \sum_{\substack{n_1,\ldots,n_4\in\NN_0\\ n_1+\cdots+n_4=n}}
                          \left|
                            v_0^{d}
                            \cdot \frac{\partial^{n_3}}{\partial \upsilon_j^{n_3}}
                                    \theta_2
                            \cdot \frac{\partial^{n_4}}{\partial \upsilon_j^{n_4}}
                                    \overline{\bd T_\tau \theta_1}
                          \right| \\
     & \leq C_n
            \cdot v_0^{d}
            \cdot \sum_{\substack{n_3,n_4\in\NN_0\\ n_3+n_4\leq n}}
                        \left|
                            \frac{\partial^{n_3}}{\partial \upsilon_j^{n_3}}
                                    \theta_2
                            \cdot \frac{\partial^{n_4}}{\partial \upsilon_j^{n_4}}
                                    \overline{\bd T_\tau \theta_1}
                        \right|.
    \qedhere
   \end{split}
  \]}
\end{proof}

The next lemma provides an estimate of $|\Square_{j_i,\tau_0,x}^n \, g|$
in terms of the partial derivatives of $g$ and the weight function $v_0$
from Definition~\ref{assume:DiffeomorphismAssumptions}.

\begin{lemma}\label{lem:estimate_diffop_abs}
  Let $\Phi$ be a $k$-admissible warping function with control weight $v_0$
  and choose $\delta > 0$ such that $\delta \cdot v_0(\delta / (4d) \cdot e_1) \leq 1/\sqrt{d}$.
  Fix $j \in \underline{d}$ and $\upsilon_0,\tau_0 \in\RR^d$, and let
  $U_j^{(\upsilon_0,\tau_0)}$ be as in Lemma~\ref{lem:phiIsReasonable}(1).
  If $g\in\mathcal C^k_c(B_{\delta/(4d)}(\upsilon_0))$ and if $x \in \RR^d \setminus \{0\}$
  satisfies $x/|x|\in U_j^{(\upsilon_0,\tau_0)}$ then, with
  \[
    \left(\Square_{j,\tau_0,x} \, g\right)
    = (2\pi i|x|)^{-1}
      \frac{\partial}{\partial \upsilon_j}
      \left[\frac{g(\bullet)}{\bigl(\phi_{\tau_0}(\bullet) \langle x/|x| \rangle \bigr)_j}\right]
  \]
  as in \eqref{eq:differentialOperator},
  there exists \nicki{$D_{n,\delta} := D_{n,\delta}(v_0) >0$, independent of $j,x,\tau_0$,
  as well as $\upsilon_0$ and the function $g\in\mathcal C^k_c(B_{\delta/(4d)}(\upsilon_0))$, }
  such that
  \[
    \left|\Square^n_{j,\tau_0,x} \, g\right|
    \leq D_{n,\delta}
         \cdot (2\pi |x|)^{-n}
         \cdot v_0^{3n}
         \cdot \sum_{m=0}^n
                  \left|\frac{\partial^m}{\partial \upsilon^m_j} \, g\right|
  \]
  holds for all $0 \leq n \leq k$.
\end{lemma}

\begin{proof}
  \textbf{Step~1 (Preparation):}
  Given $j \in \underline{d}$ and a strictly positive (or strictly negative)
  function $h \in \mathcal{C}^1(U)$ defined on an open set $\emptyset \neq U \subset \R^d$,
  we define the differential operator $\blacksquare_{j,h}$ by
  $\blacksquare_{j,h} \,\, g := \frac{\partial}{\partial \upsilon_j} \left(\frac{g}{h}\right)$.
  Then the following identity can be derived from the quotient rule by a tedious,
  but straightforward induction:
  \begin{equation}\label{eq:genquotrulelike}
    \blacksquare_{j,h}^n \,\, g
    = h^{-2n}
      \cdot \sum_{m=0}^n
              \Bigg(
                  \frac{\partial^{m} g}{\partial \upsilon_j^{m}}
                  \cdot
                  \sum_{\substack{\alpha \in \NN_0^{n} \\ |\alpha| = n - m}}
                      \bigg(
                          C^{(m,\alpha)}
                          \cdot \prod_{\ell=1}^{n}
                                  \frac{\partial^{\alpha_\ell} h}{\partial \upsilon_j^{\alpha_\ell}}
                      \bigg)
              \Bigg),
              \text{ for all } g\in {\mathcal C}^k(U) \text{ and } n \in \underline{k},
  \end{equation}
  for suitable constants $C^{(m,\alpha)} \in \ZZ$ that depend only on $\alpha \in \NN_0^{n}$
  and on $m \in \{0,\ldots,n\}$. Furthermore, we have the equality
  \[
    \Square^n_{i,\tau_0,x} \, g(\upsilon)
    = (2\pi i|x|)^{-n}
      \cdot \blacksquare_{j, \left(\phi_{\tau_0} (\cdot) \langle x/|x| \rangle \right)_j}^n \,
      g(\upsilon).
  \]

  \medskip{}

  \textbf{Step~2 (Completing the proof):}
  For $n=0$, there is nothing to prove.
  Hence, we can assume $n \in \underline{k}$.
  With $U_j^{(\upsilon_0,\tau_0)}\subset S^{d-1}$ as in \Cref{lem:phiIsReasonable}(1),
  there is a $j\in\underline{d}$, such that $x/|x| = \rho_x \in U_j^{(\upsilon_0,\tau_0)}$
  and therefore, $\Square_{j,\tau_0,x}g$ is well-defined
  for arbitrary $g \in \mathcal C_c^k(B_{\delta/(4d)}(\upsilon_0))$ by \Cref{lem:phiIsReasonable}(2).
  Now, \eqref{eq:genquotrulelike} provides
  \begin{equation}\label{eq:expression_for_tilde_square}
   \begin{split}
   \blacksquare_{j,\left(\phi_{\tau_0} (\cdot) \langle \rho_x \rangle \right)_j}^n \, g
   & = \bigl( \phi_{\tau_0} (\cdot) \langle \rho_x \rangle \bigr)_j^{-2n}
       \cdot \sum_{m=0}^n
               \Bigg(
                 \frac{\partial^{m} g}{\partial \upsilon_j^{m}}
                 \cdot
                 \sum_{\substack{\alpha\in\NN_0^{n}\\|\alpha|=n-m}}
                     \bigg(
                       C^{(m,\alpha)}
                       \cdot
                       \prod_{\ell=1}^{n}
                         \frac{\partial^{\alpha_\ell}}{\partial \upsilon_j^{\alpha_\ell}}
                         \big(
                           \phi_{\tau_0} (\cdot) \langle \rho_x \rangle
                         \big)_j
                     \bigg)
               \Bigg).
   \end{split}
  \end{equation}
  We now estimate the modulus of the innermost product
  by using \eqref{eq:PhiHigherDerivativeEstimate}:
  \[
     \left|
       \prod_{\ell=1}^{n}
         \frac{\partial^{\alpha_\ell}}{\partial \upsilon_j^{\alpha_\ell}}
         \big(
           \phi_{\tau_0}(\upsilon) \langle \rho_x \rangle
         \big)_j
     \right|
     \overset{|\rho_x|=1}{\leq}
       \prod_{\ell=1}^{n}
        \left\|
            \frac{\partial^{\alpha_\ell}}{\partial \upsilon_j^{\alpha_\ell}}
            \phi_{\tau_0}(\upsilon)
        \right\|
      \leq
        v_0^n (\upsilon).
  \]
  Insert this estimate into \eqref{eq:expression_for_tilde_square} to obtain
  \[
   \begin{split}
   \left|
     \blacksquare_{j,\left(\phi_{\tau_0} (\cdot) \langle \rho_x \rangle \right)_j}^n \, g (\upsilon)
   \right|
   & \leq [v_0 (\upsilon)]^n
          \cdot \left|
                  \left(\phi_{\tau_0}(\upsilon) \langle \rho_x \rangle \right)_j
                \right|^{-2n}
          \sum_{m=0}^n
            \Bigg(
                \left|
                    \frac{\partial^{m}}{\partial \upsilon_j^{m}} g (\upsilon)
                \right|
                \cdot
                \sum_{\substack{\alpha\in\NN_0^{n}\\|\alpha|=n-m}}
                    \bigl|C^{(m,\alpha)}\bigr|
            \Bigg) \\
    ({\scriptstyle{\text{Lemma } \ref{lem:phiIsReasonable}}})
    & \leq [v_0(\upsilon)]^{n}
           \cdot C_\delta(d,v_0)^{-2n}
           \cdot [v_0(\upsilon)]^{2n}
           \cdot
           \sum_{m=0}^n
           \Bigg(
             \left|
                 \frac{\partial^m}{\partial \upsilon^m_j} g(\upsilon)
             \right|
             \cdot \sum_{\substack{\alpha\in\NN_0^{n}\\|\alpha|=n-m}}
                     \left|C^{(m,\alpha)}\right|
           \Bigg)\\
    & \leq D_{n,\delta}
           \cdot [v_0(\upsilon)]^{3n}
           \cdot \sum_{m=0}^n
                    \left|
                      \frac{\partial^m}{\partial \upsilon^m_j} g(\upsilon)
                    \right|,
   \end{split}
  \]
  where
  \(
    D_{n,\delta}(v_0)
    := C_\delta(d,v_0)^{-2n}
       \cdot \max_{m=0,\ldots,n}
             \left(
               \sum_{|\alpha| = n-m}
                 |C^{(m,\alpha)}|
             \right)
  \)
  only depends on $n\leq k$, $\delta>0$, and \nicki{on the control weight $v_0$.}
\end{proof}

We are ready to prove \Cref{lem:NiceCrossGramianEstimate}, in particular
we can now estimate the integral appearing on the right-hand side of \eqref{eq:defOfL}.

\begin{proof}[Proof of Theorem \ref{lem:NiceCrossGramianEstimate}]
  Recall from Lemma~\ref{lem:assume_conclude} that $w_0 = v_0^d$.
  Furthermore, note by submultiplicativity of $w_1$
  that $w_1 (0) = w_1 (0+0) \leq [w_1 (0)]^2$, and hence $w_1(0) \geq 1$.
  This implies $w_1 \geq 1$:
  Another application of submultiplicativity yields
  $1 \leq w_1(0) = w_1 (\upsilon + (-\upsilon)) \leq w_1(\upsilon) \cdot w_1(-\upsilon) = [w_1(\upsilon)]^2$,
  since $w_1(-\upsilon) = w_1(\upsilon)$.
  By the same arguments, we see $v_0 \geq 1$.
  Therefore, we conclude that \eqref{eq:NiceCrossGramianAssumption} implies
  $\theta_{\ell} \in \lebesgue_{v_0^{d/2}}^2 (\RR^d) = \lebesgue_{\sqrt{w_0}}^2 (\RR^d)$,
  i.e., $\theta_1,\theta_2$ satisfy the conditions of \Cref{lem:IntegralSegmentation}.

  In the following, we only consider the case $\ell = 1$; the corresponding estimates for $\ell = 2$
  can be obtained simply by swapping $\theta_1, \theta_2$;
  our assumptions, and the definition of $C_{\max}$, are invariant under this operation.

  A first estimate \nicki{for the modulus of $L_{\tau_0}^{(1)}$
  (as defined in \eqref{eq:defOfL})}---which is effective for $|x| \leq 1$
  and which can be obtained using the $v_0^d$-moderateness of $w$
  (see Lemma~\ref{lem:assume_conclude})
  and the submultiplicativity of $w_1,v_0$---reads as follows:
  \begin{equation}\label{eq:estimate_no_deriv}
   \begin{split}
       & \left|
            \int_{\RR^d}
               \frac{w(\upsilon+\tau_0)}{w(\tau_0)}
               \left(\theta_2 \cdot \overline{\bd T_{\tau}\theta_1}\right)\!(\upsilon)
               \, e_{\tau_0}(x,\upsilon)
            ~d\upsilon
         \right| \\
       &\leq \int_{\RR^d}
                v_0^d(\upsilon)
                \cdot \left|\theta_2(\upsilon)\right|
                \cdot \left|\theta_1(\upsilon-\tau)\right|
             ~d\upsilon \\
       &=  w_1(\tau)^{-1}
           \cdot w_1(\tau)
           \cdot \int_{\RR^d}
                    v_0^d(\upsilon)
                    \cdot \left|\theta_2(\upsilon)\right|
                    \cdot \left|\theta_1(\upsilon-\tau)\right|
                 ~d\upsilon\\
       & \leq w_1(\tau)^{-1}
              \int_{\RR^d}
                \left|
                    v_0^d(\upsilon)w_1(\upsilon)
                    \theta_2(\upsilon)
                \right|
                \left|
                    w_1(\tau-\upsilon)\theta_1(\upsilon-\tau)
                \right|
              ~d\upsilon\\
       \text{\scriptsize{($w_1$  is radial) }}
       & \leq w_1(\tau)^{-1}
              \cdot \|\theta_1\|_{\bd L^2_{w_1}}
              \cdot \|\theta_2\|_{\bd L^2_{v_0^{d}w_1}}
         \leq C_{\max} \cdot [w_1(\tau)]^{-1}.
  \end{split}
  \end{equation}
  The last step used 
  $v_0\geq 1$, such that 
  $\| \theta_1 \|_{\lebesgue^2_{w_1}} \leq 
  \| \theta_1 \|_{\lebesgue^2_{w_2}}$
  and likewise $\| \theta_2 \|_{\lebesgue^2_{v_0^d w_1}} \leq 
  \| \theta_2 \|_{\lebesgue^2_{w_2}}$.
%

  To obtain an estimate which is effective for large $|x|$, we have to work harder:
  We fix some $\delta = \delta (d, \Phi, v_0) > 0$,
  such that $\delta v_0(\delta/(4d)\cdot e_1) < 1/\sqrt{d}$.
  Hence, we can apply \Cref{lem:IntegralSegmentation} to obtain
  a sequence $(j_i)_{i \in \ZZ^d}$, with $j_i \in \underline{d}$, such that
  $\rho_x = x/|x| \in U^{(i)}_{j_i}$ for all $i \in \ZZ^d$,
  and
  \begin{equation}\label{eq:partialint_again}
    \begin{split}
      \left|
        \int_{\RR^d}
            \frac{w(\upsilon+\tau_0)}{w(\tau_0)}
           \left(
                \theta_2 \cdot \overline{\bd T_{\tau}\theta_1}
            \right)(\upsilon)
           \cdot e_{\tau_0}(x,\upsilon)
        ~d\upsilon
      \right|
      & = \left|
            \sum_{i \in \ZZ^d}
                \int_{\RR^d}
                    \left(\Square_{j_i,\tau_0,x}^{k} \,\, g_{i,\tau,\tau_0}\right)(\upsilon)
                    \cdot e_{\tau_0}(x,\upsilon)
                ~d\upsilon
          \right|\\
       & \leq
            \int_{\RR^d}
                \sum_{i\in \ZZ^d}
                    \left|
                        \left(\Square_{j_i,\tau_0,x}^{k} \,\, g_{i,\tau,\tau_0}\right)(\upsilon)
                    \right|
            ~d\upsilon\\
       & = \sum_{j \in \underline{d}}
             \int_{\RR^d}
               \sum_{\substack{i\in \ZZ^d\\\text{s.t. }j_i = j}}
                 \left|
                   \left( \Square_{j,\tau_0,x}^{k} \,\, g_{i,\tau,\tau_0}\right) (\upsilon)
                 \right|
             ~d\upsilon,
    \end{split}
  \end{equation}
  for any $x\in\RR^d \setminus \{0\}$, $\tau,\tau_0\in\RR^d$.

  For $j_i = j$ (which implies $\rho_x \in U_{j_i}^{(i)} = U_{j}^{(i)}$) we further see that
  \nicki{
  \[
    \begin{split}
      & \left|
          \left( \Square_{j,\tau_0,x}^{k} \,\, g_{i,\tau,\tau_0}\right) (\upsilon)
        \right| \\
      & \overset{\text{Lem. \ref{lem:estimate_diffop_abs}}}{\leq}
         D_{k,\delta}
         \cdot (2\pi |x|)^{-k}
         \cdot v_0^{3k}(\upsilon)
         \cdot \sum_{n=0}^{k}
                 \left|
                    \frac{\partial^n}{\partial \upsilon_{j}^n}
                        g_{i,\tau,\tau_0}(\upsilon)
                 \right|\\
      & \overset{\text{Lem. \ref{lem:partial_derivs_of_localized_gte}}}{\leq}
        D_{k,\delta}
        \cdot \Indicator_{B_{\delta / (4d)}(\upsilon_i)}(\upsilon)
        \cdot (2\pi |x|)^{-k}
        \cdot v_0^{d+3k}(\upsilon)
        \cdot \sum_{n=0}^{k}
                C_n \cdot
                 \sum_{\substack{m_1,m_2\in\NN_0\\ m_1+m_2 \leq n}}
                    \left|
                        \frac{\partial^{m_1}}
                             {\partial \upsilon_{j}^{m_1}}
                        \theta_2(\upsilon)
                        \cdot \frac{\partial^{m_2}}
                                   {\partial \upsilon_{j}^{m_2}}
                              \overline{\theta_1(\upsilon-\tau)}
                    \right|.
    \end{split}
  \]
  Note that constants above are independent of $i \in \ZZ^d$.
  Next, using the finite overlap property, \eqref{eq:uniform_covering_admissible}, we get
  \[
    \sum_{\substack{i \in \ZZ^d\\\text{s.t. } j_i = j}}
         \left|
             \left( \Square_{j,\tau_0,x}^{k} \,\, g_{i,\tau,\tau_0}\right)(\upsilon)
         \right|
    \leq \widetilde{C} \cdot
         (2\pi |x|)^{-k}
         v_0^{d+3k}(\upsilon) \cdot
         \sum_{\substack{m_1,m_2\in\NN_0\\ m_1+m_2\leq k}}
                 \left|
                     \frac{\partial^{m_1}}
                          {\partial \upsilon_{j}^{m_1}}
                     \theta_2(\upsilon)
                     \cdot \frac{\partial^{m_2}}
                                {\partial \upsilon_{j}^{m_2}}
                           \overline{\theta_1(\upsilon-\tau)}
                 \right|,
  \]
  \nicki{where $\widetilde{C} := (k+1) \cdot (1+4d)^d \cdot D_{k,\delta} \cdot \max_{n=0,\dots,k} C_n$.}
  Insert this estimate into the final line of \eqref{eq:partialint_again},
  apply the Cauchy-Schwarz inequality, and
  recall that $w_1$ is submultiplicative and satisfies $w_1(-\upsilon) = w_1(\upsilon)$, whence
  \(
    1
    = [w_1(\tau)]^{-1} \cdot w_1(\upsilon + \tau - \upsilon)
    \leq [w_1(\tau)]^{-1} \cdot w_1(\upsilon) \cdot w_1(\upsilon - \tau)
    ,
  \)
  to obtain
  \[
   \begin{split}
   \lefteqn{
       \sum_{j \in \underline{d}}
         \int_{\RR^d}
             \sum_{\substack{i \in \ZZ^d\\\text{s.t. } j_i = j}}
                 \left|
                     \left(\Square_{j,\tau_0,x}^{k} \,\, g_{i,\tau,\tau_0}\right) (\upsilon)
                 \right|
         ~d\upsilon}\\
   & \leq \widetilde{C}
          \cdot (2\pi |x|)^{-k}
          \sum_{j \in \underline{d}}
             \int_{\RR^d}
                 v_0^{d+3k}(\upsilon) \cdot
               \sum_{\substack{m_1,m_2\in\NN_0\\ m_1+m_2\leq k}}
                         \left|
                             \frac{\partial^{m_1}}
                                  {\partial \upsilon_{j}^{m_1}}
                             \theta_2(\upsilon)
                             \cdot \frac{\partial^{m_2}}
                                        {\partial \upsilon_{j}^{m_2}}
                                   \overline{\theta_1(\upsilon-\tau)}
                         \right|
             ~d\upsilon\\
   & \leq \widetilde{C}
          \cdot (2\pi |x|)^{-k}
          [w_1(\tau)]^{-1}
          \cdot \\
   & \hspace{40pt}
          \sum_{j \in \underline{d}} \,\,
              \sum_{\substack{m_1,m_2\in\NN_0\\ m_1+m_2\leq k}}
                    \int_{\RR^d}
                        \left|
                            v_0^{d+3k}(\upsilon)
                            \cdot w_1(\upsilon)
                            \frac{\partial^{m_1}}
                                 {\partial \upsilon_{j}^{m_1}}
                            \theta_2(\upsilon)
                            \cdot w_1(\upsilon - \tau)
                            \frac{\partial^{m_2}}
                                 {\partial \upsilon_{j}^{m_2}}
                            \overline{\theta_1(\upsilon-\tau)}
                        \right|
                    ~d\upsilon\\
   & \leq \widetilde{C}
          \cdot (2\pi |x|)^{-k}
          w_1(\tau)^{-1}
          \sum_{j \in \underline{d}} \,\,
             \sum_{\substack{m_1,m_2\in\NN_0\\ m_1+m_2\leq k}}
                      \left\|
                          \frac{\partial^{m_1}}
                               {\partial \upsilon_{j}^{m_1}}
                          \theta_2
                      \right\|_{\bd L^2_{w_2}}
                      \cdot
                      \left\|
                          \frac{\partial^{m_2}}
                               {\partial \upsilon_{j}^{m_2}}
                          \overline{\theta_1}
                      \right\|_{\bd L^2_{w_1}}.
   \end{split}
  \]}
  Since all the involved sums are finite, so is the total number of summands.
  Moreover, the highest order partial derivatives that appear
  are $\frac{\partial^{k}}{\partial \upsilon_{j}^{k}}\theta_1$ and
  $\frac{\partial^{k}}{\partial \upsilon_{j}^{k}}\theta_2$,
  for arbitrary $j \in \underline{d}$.
  Hence, a joint maximization over $j \in \underline{d}$
  and the partial derivatives of $\theta_1, \theta_2$ yields
  \begin{equation}\label{eq:estimate_with_deriv}
    \begin{split}
      \lefteqn{
        \left|
          \int_{\RR^d}
            \frac{w(\upsilon+\tau_0)}{w(\tau_0)}
            \left(\theta_2 \cdot \overline{\bd T_{\tau}\theta_1}\right) \! (\upsilon)
            \, e_{\tau_0}(x,\upsilon)
          ~d\upsilon
        \right|
      } \\
     & \leq C'
            \cdot (2\pi |x|)^{-k}
            w_1(\tau)^{-1}
            \max_{j \in \underline{d}}
            \left\{
              \left(
                  \max_{n=0,\dots,k}
                    \left\|
                      \frac{\partial^{n}}{\partial \upsilon_{j}^{n}}\theta_1
                    \right\|_{\bd L^2_{w_1}}
              \right)
              \cdot \left(
                      \max_{n=0,\dots,k}
                          \left\|
                              \frac{\partial^{n}}{\partial \upsilon_{j}^{n}}\theta_2
                          \right\|_{\bd L^2_{w_2}}
                    \right)
          \right\}  \\
          & \leq C' \cdot (2\pi |x|)^{-k} \cdot [w_1 (\tau)]^{-1} \cdot C_{\max},
   \end{split}
  \end{equation}
  for a suitable (large) constant $C' >0$ 
  Here, the last step used again that $w_1 \leq w_2$.

  Now, define
  \[
    F(x) := \begin{cases}
                 C_{\max} , & \text{ if } |x|< 1\\
                 C' \cdot C_{\max} \cdot (2\pi |x|)^{-k} , &  \text{ else.}
             \end{cases}
  \]
  It is not hard to see $|F (x)| \leq C'' \cdot C_{\max} \cdot (1+|x|)^{-k}$ for some constant
  $C'' >0$. 
  Combining the inequalities \eqref{eq:estimate_no_deriv} and \eqref{eq:estimate_with_deriv},
  we obtain \nicki{for all $x,\ \tau,\ \tau_0\in\RR^d$ that 
  \begin{align*}
   |L_{\tau_0}^{(1)}(x,\tau)| = \left|
     \int_{\RR^d}
       \frac{w(\upsilon+\tau_0)}{w(\tau_0)}
       \left(\theta_2 \cdot \overline{\bd T_{\tau}\theta_1}\right) \! (\upsilon)
       \, e_{\tau_0}(x,\upsilon)
     ~d\upsilon
   \right|
   & \leq [w_1(\tau)]^{-1} \cdot F(x) \\
   & \leq C'' \cdot C_{\max} \cdot (1+|x|)^{-k} \cdot [w_1 (\tau)]^{-1}.
  \end{align*}
  If we collect all the hidden dependencies, then we note that the final constant $C''$
  depends on $D_{k,\delta} = D_{k,\delta}(v_0)$ and $C_n = C_n(\delta',d)$,
  and also directly on $d,\ k$.
  However, the support radius $\delta'$ of the assumed partition of unity is derived directly
  from $\delta, d$, where the largest valid choice of $\delta$ itself depends only on $v_0, d$,
  see Lemma~\ref{lem:SmoothPartitionOfUnity} for both dependencies.
  Further, noting that $D_{k,\delta}(v_0)$ is increasing in $\delta$
  (see proof of Lemma~\ref{lem:estimate_diffop_abs}), we can choose, without loss of generality,
  the largest possible value of $\delta$.
  Overall, $C''$ is a function of $d,\ k$ and $v_0$, as desired. 
  }  
\end{proof}

\subsection{Proof of Theorem \ref{thm:MR1_kernel_is_in_AAm}}

  Recall that $w = \det A$ is $w_0 := v_0^d$-moderate (Lemma~\ref{lem:assume_conclude}) and
  $v_0,v_1 \geq 1$ (see proof of Theorem~\ref{lem:NiceCrossGramianEstimate}), such that
  $w_2 \geq v_0^{d/2} = \sqrt{w_0}$ and $\theta_1,\theta_2 \in \lebesgue_{\sqrt{w_0}}^2 (\R^d)$
  follows. That $m$ is $\Phi$-compatible with dominating weight $m^{\Phi}$ is an immediate consequence
  of the inequality \eqref{eq:m_weight_estimate}, i.e.,
  \[
   m\bigl((y, \xi), (z, \eta)\bigr)
    \leq (1 + |y-z|)^p \cdot v_1 \bigl(\Phi(\xi) - \Phi(\eta)\bigr),
    \text{ for all } y,z\in\RR^d \text{ and } \xi,\eta\in D,
  \]
 and the choice of $p\in\NN_0$ (in particular, $p=0$ if $R_\Phi = \infty$).

  Thus, Lemma~\ref{lem:CrossGramianBabyStep} and Lemma~\ref{pro:kern_in_theta}
  can be applied, showing that
  \[
   \| K_{\theta_1,\theta_2} \|_{\BBm}
   \leq \max_{\ell \in\{1,2\}}
          \esssup_{\tau_0 \in \RR^d}
            \int_{\RR^d}
              \int_{\RR^d}
                M(x,\tau)\cdot
                |L^{(\ell)}_{\tau_0} (x,\tau)|
              ~dx
            ~d\tau,
  \]
  where
  \[
    M(x,\tau)
    = \sup_{y \in \R^d, |y| \leq R |x|}
        (1 + |y|)^p \cdot v_0^{d/2}(\tau) \cdot v_1(\tau)
    .
  \]
  Note that $M(x,\tau) \!\leq\! C_\Phi \cdot (1 + |x|)^p \cdot v_0^{d/2}(\tau) \cdot v_1(\tau)$,
  where $C_\Phi := \max \big\{ 1, \sup_{\xi \in D} \| \mathrm{D}\Phi (\xi) \|^p \big\}$ if $p > 0$
  and $C_\Phi := 1$ otherwise.

  Define
  \(
    w_1 :
    \R^d \to \R^+,
    \upsilon \mapsto (1 + |\upsilon|)^{d+1} \cdot v_1(\upsilon) \cdot [v_0(\upsilon)]^{d/2}.
  \)
  Since $v_0,v_1$ are submultiplicative and satisfy $v_\ell (-\upsilon) = v_\ell(\upsilon)$
  for $\ell \in \{ 0,1 \}$ and $\upsilon \in \R^d$, it is easy to see that $w_1$
  satisfies the same two properties.
  Furthermore, $w_2(\upsilon) = w_1(\upsilon) \cdot [v_0(\upsilon)]^{d + 3 (d+p+1)}$,
  so that Theorem~\ref{lem:NiceCrossGramianEstimate},
  with $k = d + p + 1$, yields a \nicki{constant $C = C(d, d+p+1, v_0) > 0$} satisfying
  \[
    \begin{split}
      \|K_{\theta_1,\theta_2}\|_{\BBm}
      & \leq \max_{\ell \in\{1,2\}}
               \esssup_{\tau_0 \in \RR^d}
                 \int_{\RR^d}
                   \int_{\RR^d}
                      M(x,\tau)\cdot 
                     \nicki{ |L^{(\ell)}_{\tau_0} (x,\tau)|}
                   ~dx
                  ~d\tau\\
      & \leq C_\Phi\cdot
             \esssup_{\tau_0 \in \RR^d}
                 \int_{\RR^d}
                     \int_{\RR^d}
                        (1+|x|)^p \cdot
                        v_0^{d/2}(\tau) \cdot
                        v_1(\tau)
                        \cdot \max_{\ell \in\{1,2\}}
                                \nicki{|L^{(\ell)}_{\tau_0} (x,\tau)|}
                     ~dx
                 ~d\tau\\
      \text{\scriptsize{(Thm. \ref{lem:NiceCrossGramianEstimate})}}
      & \leq C C_\Phi C_{\max}\cdot
             \int_{\RR^d}
                \int_{\RR^d}
                    v_0^{d/2}(\tau)
                    \cdot v_1(\tau)
                    \cdot [w_1(\tau)]^{-1}
                    \cdot (1+|x|)^{-(d+1)}
                ~d\tau
             ~dx\\
      & \leq C C_\Phi C_{\max}\cdot
             \int_{\RR^d}
               \int_{\RR^d}
                 (1+|\tau|)^{-(d+1)} (1+|x|)^{-(d+1)}
               ~d\tau
             ~dx\\
      & =: \widetilde{C} \cdot C_{\max}
        <  \infty.
    \end{split}
  \]
Here, the final constant $\widetilde{C} = \widetilde{C}(d, p, \Phi, v_0) > 0$ is finite,
simply because $(1+|\cdot|)^{-(d+1)}\in \bd L^1(\RR^d)$.
\nicki{Arguably, the dependence of $\widetilde{C}$ on $v_0$ could be expressed
as a consequence of the dependence on $\Phi$, but there may be cases where
different choices of $v_0$ could be of interest, such that we prefer to keep it explicit.}
This concludes the proof.
\hfill\qed


\section{The phase-space coverings induced by the warping function \texorpdfstring{$\Phi$}{Φ}}
\label{sec:coverings}

To prepare for the estimation of $\| \oscVGd \|_{\BBm}$ we construct families of coverings 
$\CalV_{\Phi}^{\delta} = (V_i^{\delta})_{i \in I}$ of the phase space $\Lambda$, 
induced by a given warping function $\Phi$ and study their properties. In the next section, 
we will show that $\| \oscVGd \|_{\BBm} \to 0$ as $\delta \to 0$,
with $\oscVGd$ as introduced in Definition~\ref{def:genosckern}.

\begin{definition}\label{def:inducedcover}
  Let $\Phi\colon D\rightarrow \RR^d$ be a warping function.
  Define
  \begin{equation}
    Q_{\Phi,\tau}^{(\delta,r)}
    := \Phi^{-1}(\delta \cdot B_{r}(\tau)),
    \quad \text{for all } r, \delta > 0 \text{ and } \tau \in \RR^d.
  \label{eq:freq_cover0}
  \end{equation}
  We call $\mathcal{V}^\delta_\Phi = (V^\delta_{\ell,k})_{\ell,k\in\ZZ^d}$,
  defined by
  \begin{equation}\label{eq:deltacover}
    V_{\ell,k}^{\delta}
    := A^{-T}( \delta k / \sqrt{d} )
         \left\langle \delta \cdot B_1 (\ell / \sqrt{d}) \right\rangle
       \times Q_{k}^{\delta},
    \quad \text{with} \quad
    Q_{k}^{\delta}
    := Q_{\Phi,k/\sqrt{d}}^{(\delta,1)}
     = \Phi^{-1}\bigl(\delta \cdot B_1(k/\sqrt{d})\bigr),
  \end{equation}
  the \emph{$\Phi$-induced $\delta$-fine (phase-space) covering}.
\end{definition}

By allowing $r\neq 1$ in \eqref{eq:freq_cover0}, it is possible to control the
overlap of the covering elements.
In particular, any radius strictly larger than $1/2$ provides a covering.
For proving the feasibility of discretization in coorbit spaces, however, the above
choice of $r = 1$ in \eqref{eq:deltacover} is completely sufficient.

\begin{proposition}\label{pro:inducedcoverProps}
  Let $\Phi$ be a $0$-admissible warping function with control weight $v_0$
  (see Definition~\ref{assume:DiffeomorphismAssumptions}).
  Then the $\Phi$-induced $\delta$-fine phase-space covering
  $\mathcal{V}^\delta_\Phi = (V^\delta_{\ell,k})_{\ell,k\in\ZZ^d}$
  is a topologically admissible cover of $\Lambda = \RR^d\times D$
  which is also product-admissible as per Definition~\ref{def:ProductAdmissibleCovering}.
  More precisely, we have the following properties:
  \begin{enumerate}[label=(\arabic*),leftmargin=0.7cm]
   \item \label{enu:admissibility}
         If $k,\ell, k_0, \ell_0 \in \ZZ^d$ satisfy
         $|k - k_0| > 2\sqrt{d}$, then
         $V^\delta_{\ell,k}\cap V^\delta_{\ell_0,k_0} = \emptyset$.
         Furthermore,
         \[
           \qquad
           \sup_{(\ell,k)\in \ZZ^{2d}}
              \# \big\{
                   (\ell_0,k_0)\in \ZZ^{d} \times \Z^d
                   ~:~
                   V^\delta_{\ell,k} \cap V^\delta_{\ell_0,k_0} \neq \emptyset
                 \big\}
           \leq (1 + 4\dimension)^{\dimension}
                \big(
                  1 + 2 \sqrt{\dimension} \cdot \left(1 + v_{0} (2\delta) \right)
                \big)^{\dimension}.
         \]

   \item \label{enu:measure_almost_constant}
         We have
         \(
           [v_0 (\delta)]^{-d}
           \leq \frac{\mu(V^\delta_{\ell,k})}{[\mu(B_1(0))]^2 \cdot \delta^{2d}}
           \leq [v_0 (\delta)]^{d}
         \)
         for all $k,\ell \in \Z^d$.

   \item \label{enu:moderateness}
         We have $\mu(V_{\ell, k}^\delta) / \mu(V_{\ell_0, k_0}^\delta) \leq [v_0 (\delta)]^{2d}$
         for arbitrary $\ell, k, \ell_0, k_0 \in \ZZ^d$.

    \item \label{enu:prodAdmissibility}
          For each fixed $\delta > 0$, the weight $w_{\CalV_\Phi^\delta}$
          as given in Equation~\eqref{eq:CoveringWeightDefinition} satisfies 
          \begin{equation}\label{eq:CoveringWeightEst}
            \begin{split}
              \quad
              \bigl(w_{\CalV_\Phi^\delta}\bigr)_{\ell,k}
              & \asymp \min
                       \big\{
                         w(\delta \cdot k/\sqrt{d}),
                         [w(\delta \cdot k/\sqrt{d})]^{-1}
                       \big\} \\
              & \asymp \min
                       \big\{
                         w(\Phi(\xi)),
                         [w(\Phi(\xi))]^{-1}
                       \big\}
                \gtrsim [v_0(\Phi(\xi))]^{-d}
              ,\ \text{for all }
                     \ell,k\in\ZZ^d,\ \xi \in Q^{\delta}_{k}.
            \end{split}
          \end{equation}  
          In particular, there exists a constant \nicki{$C = C(d,\delta,v_0) > 0$} such that
          \(
            (w_{\CalV_\Phi^\delta})_{\ell,k} \big/ (w_{\CalV_\Phi^\delta})_{\ell_0,k_0}
            \leq C
          \)
          for all $\ell, k, \ell_0, k_0 \in \ZZ^d$
          with $V^\delta_{\ell,k} \cap V^\delta_{\ell_0,k_0} \neq \emptyset$.
          Moreover, \eqref{eq:ContinuousCoveringWeightCondition} holds with
          \[
            w_{\CalV_\Phi^\delta}^c : \quad
            \Lambda \rightarrow \R^+, \quad
            (x,\xi) \mapsto \min
                            \big\{
                              w(\Phi(\xi)), \,\,
                              [w(\Phi(\xi))]^{-1}
                            \big\}
            .
          \]
  \end{enumerate}
\end{proposition}

\begin{proof}
  Note that the family $\delta \cdot B_1 (\ell/\sqrt{d})$, $\ell \in\ZZ^d$ forms a covering
  of $\RR^\dimension$, since $\frac{1}{\sqrt{d}} \big(\ell + [0,1)^d\big) \!\subset\! B_1 (\ell/\sqrt{d})$.
  Considering that $\Phi: D \to \RR^\dimension$ is a diffeomorphism and
  $A^{-T}(\delta k/\sqrt{d})$, for any $k\in\ZZ^d$, is an invertible matrix,
  it follows that $\mathcal{V}^\delta_\Phi$ indeed covers all of $\Lambda$.

  We first prove part \ref{enu:admissibility}.
  For $k,\ell \in \Z^d$, let
  \[
    J_{\ell,k}
    := \left\{
         \left(\ell_{0},k_{0}\right) \in \ZZ^d \times \Z^d
         \,:\,
         V_{\ell,k}^{\delta} \cap V_{\ell_{0},k_{0}}^{\delta} \neq \emptyset
       \right\}.
  \]
  If $V_{\ell,k}^{\delta} \cap V_{\ell_{0},k_{0}}^{\delta} \neq \emptyset$,
  then in particular $Q_{k}^{\delta} \cap Q_{k_{0}}^{\delta} \neq \emptyset$.
  Straightforward calculations show that the latter implies
  $|k_0-k| \leq 2 \sqrt{d}$, and then
  $k_{0}\in k+\left\{ -2\dimension,\dots,2\dimension\right\} ^{\dimension}$.
  Moreover, if $(\ell_0,k_0)\in J_{\ell,k}$, then an easy calculation shows that
  there exist $x_{1},x_{2}\in B_{1}(0)$ such that
  \[
    \ell_{0}
    = A^{T}\left(\delta k_{0}/\smash{\sqrt{\dimension}}\right)
      \cdot A^{-T}\left(\delta k/\smash{\sqrt{\dimension}}\right)
      \left\langle
          \ell+\sqrt{\dimension}\cdot x_{1}
      \right\rangle
      -\sqrt{\dimension}\cdot x_{2}.
  \]
  Property \eqref{eq:PhiHigherDerivativeEstimate} shows that
  \(
    A_{k,k_{0}}
    := A^{T}\left(\delta k_{0}/\smash{\sqrt{\dimension}}\right)
       \cdot A^{-T}\left(\delta k/\smash{\sqrt{\dimension}}\right)
     = \phi_{\delta k / \sqrt{d}} (\delta \cdot (k_0 - k) / \sqrt{d})
  \)
  satisfies
  \[
    \left\Vert A_{k,k_{0}} \right\Vert
    \leq v_{0} \left( \tfrac{\delta}{\sqrt{\dimension}} (k_{0} - k) \right)
    \leq v_{0} (2\delta).
  \]
  Here, we used $|k_{0} - k| \leq 2 \sqrt{\dimension}$ and that $v_0$ is radially increasing.
  Since $x_{1},x_{2}\in B_{1}(0)$, we thus have $\ell_{0} \in A_{k,k_{0}} \ell + [-C_1,C_1]^d$,
  where
  \[
    C_{1}
    := \sqrt{\dimension}
       \cdot \bigl(1 + v_{0} (2\delta)\bigr)
    \geq \left|
           \sqrt{\dimension} \cdot A_{k,k_{0}} x_{1}
           - \sqrt{\dimension} \cdot x_{2}
         \right|.    
  \]
  Altogether, we have shown
  \[
    J_{\ell,k}
    \subset \bigcup_{k_{0} \in k + \left\{ -2\dimension, \dots ,2\dimension \right\}^{\dimension}}
              \left(
                \left[
                  \ZZ^{\dimension}
                  \cap \left(
                          A_{k,k_{0}} \ell
                         + \left[-C_{1},C_{1}\right]^{\dimension}
                       \right)
                \right]
                \times \left\{ k_{0}\right\}
              \right).
  \]
  But we have
  \(
    \#\big[
      \ZZ^{\dimension}
      \cap \big(
              A_{k,k_{0}} \ell + \left[-C_{1}, C_{1}\right]^{\dimension}
           \big)
     \big]
   \leq \left(1 + 2 C_{1} \right)^{\dimension}
  \)
  and hence
  \[
    |J_{\ell,k}|
    \leq \sum_{k_{0} \in k + \{-2\dimension, \dots, 2\dimension\}^{\dimension}}
           (1 + 2 C_{1})^{d}
    =    (1+4\dimension)^{\dimension} (1 + 2 C_{1})^{\dimension},
  \]
  completing the proof of part (1).
  This also shows that $\CalV_{\Phi}^\delta$ is an admissible covering.
  Since each $V_{\ell,k}^\delta$ is open and relatively compact in $\PhSpace$,
  we see that $\CalV_{\Phi}^\delta$ is topologically admissible.

  \medskip

  We proceed to prove Item \ref{enu:measure_almost_constant}.
  By the change of variables formula, cf.~Equation~\eqref{eq:StandardChangeOfVariables}, we get
  \begin{equation*}
    \mu (Q_{ k}^\delta)
    = \int_{D}
         \Indicator_{\delta \cdot B_1(k/\sqrt{d})} (\Phi(\xi))
      ~ d \xi
    = \int_{\delta \cdot B_1(k/\sqrt{d})}
          w(\tau)
      ~ d \tau.
  \end{equation*}
  Recall that $w$ is $v_0^d$-moderate by Lemma~\ref{lem:assume_conclude},
  where $v_0$ is submultiplicative and radially increasing.
  Therefore,
  \[
    [v_0 (\delta)]^{-d}
    \leq \frac{w(\tau)}{w(\delta \cdot k / \sqrt{d})}
    \leq [v_0 (\delta)]^d,
    \text{ for all } \tau\in \delta \cdot \overline{B_1}(k/\sqrt{d})
  \]
  In combination, the two preceding displayed equations show that
  \begin{equation}\label{eq:U2est}
    \mu (Q_{ k}^\delta)
    \in \mu(B_1(0))
        \cdot \delta^d
        \cdot w(\delta \cdot k/\sqrt{d})
        \cdot \left[[v_0 (\delta)]^{-d}, [v_0 (\delta)]^{d}\right].
  \end{equation}
  Moreover,
  \begin{equation}\label{eq:U1est}
     \mu
     \left(
       A^{-T}(\delta k/\sqrt{d})\langle\delta\cdot B_1(\ell / \sqrt{d})\rangle
     \right)
   = \left|\det \left( A^{-T} ( \delta \cdot k / \sqrt{d} ) \right) \right|
     \cdot \mu\left(\delta \cdot B_1(\ell / \sqrt{d})\right)
   = \frac{\mu(B_1(0)) \cdot \delta^{d}}{w(\delta \cdot k / \sqrt{d})}.
  \end{equation}
  Since
  \(
    \mu(V_{\ell,k}^\delta)
    = \mu
      \big(
        A^{-T}(\delta k/\sqrt{d})
        \langle \delta \cdot B_1(\ell / \sqrt{d}) \rangle
      \big)
      \cdot \mu (Q_{ k}^\delta)
  \),
  this proves part~\ref{enu:measure_almost_constant}.
  Finally, part \ref{enu:moderateness} is a direct consequence of part \ref{enu:measure_almost_constant}.

  \medskip{}

  It remains to prove part~\ref{enu:prodAdmissibility}.
  Since $\CalV^\delta_\Phi$ is a covering of $\PhSpace = \R^d \times D$ with countable index set
  and with each set $V_{\ell,k}^\delta$
  being a Cartesian product of open sets, this will then imply
  that $\CalV^\delta_\Phi$ is product-admissible.
  First note that $\min\bigl\{1, \mu(V^\delta_{\ell,k})\bigr\} \asymp 1$
  as a function in $\ell,k \in \ZZ^d$ and that
  $V_{\ell,k}^\delta = V_{1,(\ell,k)}^\delta \times V_{2,(\ell,k)}^\delta$ with
  $V^\delta_{1,(\ell,k)} = A^{-T}(\delta k/\sqrt{d})\langle\delta\cdot B_1(\ell / \sqrt{d})\rangle$
  and $V^\delta_{2,(\ell,k)} = Q^{\delta}_{k}$.
  Hence, by \eqref{eq:U1est} and \eqref{eq:U2est}, we have 
  \[
    \min
    \big\{
      \mu(V^\delta_{1,(\ell,k)}), \,\,
      \mu(V^\delta_{2,(\ell,k)})
    \big\}
    \asymp \min
           \big\{
             w(\delta k/\sqrt{d}),
             [w(\delta k/\sqrt{d})]^{-1}
           \big\},
    \,\text{as a function in } \ell,k \in \ZZ^d.
  \]
  Together, this yields the first estimate in \eqref{eq:CoveringWeightEst}.
  The other two estimates in \eqref{eq:CoveringWeightEst} are simple consequences
  of $w$ being $v_0^d$-moderate (and thus $w^{-1}$ is as well)
  and of the identity $\Phi(Q^{\delta}_k) = \delta \cdot B_1(k/\sqrt{d})$.
  Note that \eqref{eq:CoveringWeightEst} implies \eqref{eq:ContinuousCoveringWeightCondition}
  with the stated choice of $w^c_{\CalV^\delta_{\Phi}}$. 

  To prove that $(w_{\CalV_\Phi^\delta})_{\ell,k}/(w_{\CalV_\Phi^\delta})_{\ell_0,k_0} \lesssim 1$
  if $V_{\ell,k}^\delta \cap V_{\ell_0,k_0}^\delta \neq \emptyset$, first note that
  since $w$ is $v_0^d$-moderate and $v_0$ is radially increasing (and hence radial).
  Note that taking reciprocal values, as well as pointwise minima/maxima
  preserve moderateness relations, see Remark~\ref{rem:InverseMaxMinOfNoderateWeights}, such that 
  \[
    \frac{\min \bigl\{ w(\delta k/\sqrt{d}), [w(\delta k/\sqrt{d})]^{-1} \bigr\}}
         {\min \bigl\{ w(\delta k_0/\sqrt{d}), [w(\delta k_0/\sqrt{d})]^{-1} \bigr\}}
    \leq \bigl[v_0(\delta (k - k_0) / \sqrt{d})\bigr]^d
    \qquad \forall \, k, k_0 \in \Z^d .
  \]
  Furthermore, part~\ref{enu:admissibility} of the proposition shows that if 
  $V_{\ell,k}^\delta \cap V_{\ell_0,k_0}^\delta \neq \emptyset$, then $|k - k_0| \leq 2 \sqrt{d}$.
  Combining these observations with Equation~\eqref{eq:CoveringWeightEst}
  and with the fact that $v_0$ is radially increasing, we see \nicki{
  \(
    (w_{\CalV_\Phi^\delta})_{\ell,k} / (w_{\CalV_\Phi^\delta})_{\ell_0,k_0}
    \lesssim v_0(2\delta)^d
    \lesssim 1,
  \)
  where the implied constant depends (only) on $d$, $\delta$, and $v_0$.}
\end{proof}

The next lemma is concerned with the sets
$\bd{V}_\lambda = \bigcup_{i \in I \text{ s.t. } \lambda \in V_i} V_i$
defining the oscillation $\oscVG$, see Definition~\ref{def:genosckern}.
For the induced coverings $\CalV^\delta_\Phi$, the set $\bd{V}_\lambda^\delta$ can once more
be estimated by a convenient product set.
Moreover, the lemma implies that if $\lambda = (z,\eta) \in \bd{V}_{\lambda_0}^\delta$,
with $\lambda_0 = (y,\omega)$, then
\[
  |A^{T}(\Phi(\omega))\langle z - y \rangle| \leq C_\delta
  \qquad \text{and} \qquad
  |\Phi(\eta) - \Phi(\omega)| \leq C_\delta,
  \quad \text{with} \quad
  C_\delta > 0 \text{ independent of } \lambda,\lambda_0.
\]
In particular, this holds if there exists $(\ell,k) \in \ZZ^{2d}$
such that $\lambda,\lambda_0 \in V_{\ell,k}^\delta$.
These estimates will be crucial for estimating $\|\oscVGd\|_{\BBm}$.

\begin{lemma}\label{lem:coverings}
  Let $\Phi$ be a warping function,
  and let $\CalV^\delta_\Phi$ be the $\Phi$-induced $\delta$-fine covering.
  For all $(y,\omega)\in \Lambda$ and all $\delta > 0$, we have
  \begin{equation}
              \bd{V}_{(y,\omega)}^\delta
    =         \bigcup_{\substack{(\ell,k) \text{ s.t.}\\
                                 (y,\omega) \in V_{\ell,k}^{\delta}}}
                  V_{\ell,k}^{\delta}
    \subset (y+\bd P_\omega^\delta) \times \bd Q_\omega^\delta,
    \vspace{-0.3cm}
  \end{equation}
  where
  \begin{equation}
    \bd{Q}_\omega^\delta := \Phi^{-1}\bigl(\Phi(\omega)+B_{2\delta}(0)\bigr)
    \qquad \text{ and } \qquad
    \bd{P}_\omega^\delta
    := v_0 (\delta) \cdot A^{-T} (\Phi(\omega))
                            \left\langle B_{2\delta}(0)\right\rangle.
  \end{equation}
\end{lemma}

\begin{proof}
  Let $(\ell,k) \in \ZZ^{d} \times \Z^d$ be such that $(y,\omega) \in V_{\ell,k}^{\delta}$.
  Then $\delta k/\sqrt{d} \in \Phi(\omega)+\delta \cdot B_{1}(0)$
  and by extension of that argument,
  $Q_{k}^{\delta} \subset \Phi^{-1}(\Phi(\omega)+2\delta B_{1}(0)) = \bd{Q}_{\omega}^\delta$,
  which proves the first part of the claim.

  Next, for $(x, \xi) \in V_{\ell, k}^{\delta}$, we have
  $|A^{T}(\delta k /\sqrt{d}) \left\langle x - y\right\rangle| < 2\delta$, since
  $x, y \in A^{-T}(\delta k / \sqrt{d}) \left\langle \delta B_1(\ell/\sqrt{d})\right\rangle$.
  Hence,
  \begin{align*}
          \left| A^{T}(\Phi(\omega))\left\langle x-y\right\rangle \right|
    &=    \left|
            A^{T}(\Phi(\omega))
            A^{-T}(\delta k /\sqrt{d})
            A^{T}(\delta k / \sqrt{d})
              \left\langle x - y\right\rangle
          \right| \\
    &<    2\delta \cdot \big\| A^{T}(\Phi(\omega)) A^{-T}(\delta k / \sqrt{d}) \big\|
     =    2 \delta
          \cdot \big\|
                  \phi_{\delta k / \sqrt{d}} \bigl(\Phi(\omega) - \delta k / \sqrt{d}\bigr)
                \big\| \\
    ({\scriptstyle \text{cf.~Eq.~\eqref{eq:PhiHigherDerivativeEstimate}}})
    & \leq 2 \delta \cdot v_0 (\Phi(\omega) - \delta k /\sqrt{d})
      \leq 2 \delta \cdot v_0 (\delta),
  \end{align*}
  which shows
  $x-y \in A^{-T}(\Phi(\omega))
           \left\langle 2\delta v_0(\delta) B_1(0) \right\rangle$,
  and thus $x \in y + \bd P_\omega^{\delta}$, as desired.
\end{proof}

\begin{proposition}\label{prop:GreatSimplification}
  Let $\Phi$ be a $0$-admissible warping function with control weight $v_0$.
  Let further  $m_0 : \PhSpace \times \PhSpace \to \R^+$ be continuous and symmetric,
  with $1\leq m_0(\PhVar,\PhVarA) \leq C^{(0)} \cdot m_0(\PhVar,\PhVarC) \cdot m_0(\PhVarC,\PhVarA)$, 
  for all $\PhVar,\PhVarA,\PhVarC\in \PhSpace$ and some $C^{(0)}\geq 1$, satisfy
   \[
     m_0((y,\xi),(z,\eta)) \leq (1 + |y - z|)^p \cdot \zeta_1\bigl(\Phi(\xi) - \Phi(\eta)\bigr)
     \quad \forall \, (y,\xi),(z,\eta) \in \PhSpace. 
   \]
  Here, $p=0$ if $R_\Phi = \sup_{\xi \in D} \| \mathrm{D}\Phi (\xi) \| = \infty$
  and $p\in\NN_0$ otherwise, and $\zeta_1 : \R^d \to \R^+$ is a continuous function
  with $ \zeta_1(-\tau) = \zeta_1(\tau)$ for all $\tau\in\RR^d$.
  Define, for some arbitrary, fixed $\PhVarC\in\PhSpace$, 
   \[
    \begin{split}
    u : \quad
    \PhSpace \to \R^+,  & \quad
    \PhVar \mapsto m_0(\PhVar,\PhVarC)\qquad \text{and}\\
    v : \quad
    \PhSpace \to \R^+,  & \quad
    (y,\xi) \mapsto u(y,\xi) \cdot \max \big\{ w(\Phi(\xi)), [w(\Phi(\xi))]^{-1} \big\} 
    \end{split}
   \]
   and let $m_v$ be as in Equation~\eqref{eq:vDerivedWeight}.
   Then $u$ is $\CalV_\Phi^\delta$-moderate, for any $\delta>0$, and 
   $m_0$ and $m_v$ are $\Phi$-convolution-dominated by $(1 + |\bullet|)^p \cdot \zeta_1(\bullet)$
   and $m_v^\Phi:=(1 + |\bullet|)^p \cdot \zeta_2(\bullet)$, where $\zeta_2 = v_0^d\cdot \zeta_1$.
   In particular, items (1)-(3) of Assumption \ref{assu:CoorbitAssumptions1} are satisfied.
\end{proposition}

\begin{proof}
  \Cref{pro:inducedcoverProps} provides product-admissibility of $\CalV^\delta_\Phi$,
  such that item (1)  of Assumption~\ref{assu:CoorbitAssumptions1} is satisfied.
  Item (3) is a direct consequence of the symmetry of $m_0$:
  \[
    m_0(\PhVar,\PhVarA)
    \leq C^{(0)} \cdot m_0(\PhVar,\PhVarC) \cdot m_0(\PhVarC,\PhVarA)
    = C^{(0)} \cdot u(\PhVar) \cdot u(\PhVarA)
    .
  \]
  To show $\CalV^\delta_\Phi$-moderateness of $u$
  (which coincides with item (2) of Assumption~\ref{assu:CoorbitAssumptions1}), observe that 
  \begin{equation}\label{eq:EstimateUbyM0}
    \frac{u(\PhVar)}{u(\PhVarA)}
    \leq C^{(0)}\frac{m_0(\PhVar,\PhVarA)m_0(\PhVarA,\PhVarC)}{m_0(\PhVarA,\PhVarC)}
    = C^{(0)} m_0(\PhVar,\PhVarA).
  \end{equation}
  If $\PhVar = (y,\xi)$ and $\PhVarA = (z,\eta)$ are both contained in $V^\delta_{\ell,k}$,
  for some $\ell,k\in\ZZ^d$, then $|\Phi(\xi)-\Phi(\eta)| < \delta$, and
  \[
    |y-z|
    \leq \delta\cdot \|A^{-T}(\delta k/\sqrt{d})\|
    \leq \delta \cdot R_\Phi,
    \quad \text{if }  R_\Phi <\infty. 
  \]
  Hence, and
  \(
    \frac{u(y,\xi)}{u(z,\eta)}
    \leq C^{(0)} m_0((y,\xi),(z,\eta))
    \leq C^{(0)}(1 + \delta R_\Phi)^p \cdot \zeta_1(\delta),
  \)
  independent of $\ell,k\in\ZZ^d$. 
  If $R_\Phi = \infty$, then
  \(
    \frac{u(y,\xi)}{u(z,\eta)}
    \leq C^{(0)}m_0((y,\xi),(z,\eta))
    \leq C^{(0)}\zeta_1(\delta)
  \)
  instead.

  That $m_0$ is $\Phi$-convolution-dominated by $(1 + |\bullet|)^p \cdot \zeta_1(\bullet)$ is immediate.
  To prove that $m_v$ is $\Phi$-convolution-dominated by $m_v^\Phi$, observe 
  \[
    \frac{\max \big\{ w(\tau), [w(\tau)]^{-1} \big\}}{\max \big\{ w(\upsilon), [w(\upsilon)]^{-1} \big\}}
    \leq \max\left\{\frac{w(\tau)}{w(\upsilon)},\frac{w(\upsilon)}{w(\tau)}\right\}
    \leq v_0^{d}(\tau-\upsilon),  \text{ for all } \tau,\upsilon\in\RR^d. 
  \]
  Combine the above with \eqref{eq:EstimateUbyM0}, such that 
  \[
    \frac{v(y,\Phi(\xi))}{v(z,\Phi(\eta))}
    \leq C^{(0)}\cdot m_0((y,\Phi(\xi)),(z,\Phi(\eta)))\cdot v_0^{d}(\Phi(\xi)-\Phi(\eta)).
    \qedhere
  \]
\end{proof}


\section{Controlling the \texorpdfstring{$\BBm$}{𝓑ₘ}-norm of the oscillation}
\label{sec:discretewarped}

In this section, we employ the $\Phi$-induced $\delta$-fine phase-space coverings $\CalV^{\delta}_\Phi$,
constructed in the previous section, to derive conditions concerning the prototype function
$\theta$ which ensure that $\|\oscVFGd\|_{\BBm} < \infty$
with $\|\oscVFGd\|_{\BBm} \rightarrow 0$ as $\delta \rightarrow 0$. We will obtain the following result.

\begin{theorem}\label{thm:main2_discreteframes}
 Let $\Phi$ be a $(d+p+1)$-admissible warping function with control weight $v_0$,
 where $p=0$ if $R_\Phi = \sup_{\xi\in D} \|\mathrm{D}\Phi(\xi)\|=\infty$ and $p \in \NN_0$ otherwise.
 Let furthermore $m : \PhSpace \times \PhSpace \to \R^+$ be a symmetric weight that satisfies
 \begin{equation}\label{eq:m_weight_estimate3}
       m((y, \xi), (z, \eta))
       \leq 
       (1 + |y-z|)^p \cdot v_1 (\Phi(\xi) - \Phi(\eta)),
       \forall \, y,z\in\RR^d \text{ and } \xi,\eta\in D,
 \end{equation}
 for some continuous and submultiplicative weight $v_1 : \RR^\dimension \to \R^+$
 satisfying $v_1(\upsilon)=v_1(-\upsilon)$ for all $\upsilon\in\RR^d$.

 Finally, with
 \[
      w_2 : \RR^d \to \RR^+,
            \upsilon \mapsto (1+|\upsilon|)^{d+1} \cdot v_1(\upsilon) \cdot [v_0 (\upsilon)]^{9d/2+3p+3},
 \]
 assume that $\theta \in \mathcal C^{d+p+1}(\RR^d)$ and
 \[
   v_0^n \cdot \frac{\partial^{(d+p+1)-n}}{\partial \upsilon_j^{(d+p+1)-n}}\theta\in \bd L^2_{w_2}(\RR^d),
   \qquad \text{ for all } i \in \underline{d},
   \,\,\, 0 \leq n \leq d+p+1.
 \]

 Then, with
 \(
   \Gamma:
   \PhSpace \times \PhSpace \to \CC,
   \bigl((y,\omega),(z,\eta)\bigr) \mapsto e^{-2\pi i\langle y-z,\omega\rangle}
   ,
 \)
 and $\CalV_{\Phi}^{\delta} = (V_{\ell,k}^\delta)_{\ell, k \in \ZZ^d}$
 the $\Phi$-induced $\delta$-fine covering:
 \begin{equation}\label{eq:osc_kern_in_BBm}
   \| \oscVFGd\|_{\BBm} < \infty \quad \text{for all } \delta > 0
   \qquad \text{ and }\qquad
   \| \oscVFGd\|_{\BBm}
   \overset{\delta\rightarrow 0}{\rightarrow} 0.
 \end{equation}
\end{theorem}

\begin{remark}\label{rem:Main2InpliesMain1}
  The conditions of \Cref{thm:main2_discreteframes} are largely the same
  as those for \Cref{thm:MR1_kernel_is_in_AAm}.
  The only difference is the appearance of an additional factor $v_0^n$, for certain $n\in\NN_0$,
  in the conditions on $\theta$.
  Since $v_0 \geq v_0(0)$, the conditions of \Cref{thm:main2_discreteframes}
  imply those of \Cref{thm:MR1_kernel_is_in_AAm}.
\end{remark}

To prove \Cref{thm:main2_discreteframes}, we study the second component of the oscillation, i.e.,
\(
  g_{\PhVar} - \Gamma(\PhVar,\PhVarA)g_{\PhVarA},
\)
for $\PhVarA\in\CalV^\delta_{\PhVar}$.
If we can bound certain weighted $\bd L^2$-norms of this difference and its derivatives
uniformly in $\PhVar\in\PhSpace$ and $\PhVarA\in\CalV^\delta_{\PhVar}$,
then we can show that $\oscVFGd\in\BBm$ by a slight variation on \Cref{thm:MR1_kernel_is_in_AAm}.
In fact, the estimates we obtain converge to $0$ for $\delta\to 0$,
such that we naturally obtain the second part of \Cref{eq:osc_kern_in_BBm} as well.

\subsection{Local behavior of the oscillating component}
\label{ssec:localOscComp}

  In order to rely on the machinery we already developed in Section \ref{sec:coorbits},
  it will be useful to rewrite
  \(
    g_{\PhVar} - \Gamma(\PhVar,\PhVarA)g_{\PhVarA}
  \)
  as the warping of a function $\tilde{\theta}_{\PhVar, \PhVarA} \in \Ltv$
  (dependent on $\PhVar,\PhVarA\in\PhSpace$) derived from the prototype $\theta$.

\begin{proposition}\label{pro:difference_as_warping}
  For $D \subset \RR^d$ open, let $\PhSpace = \RR^d \times D$, and define
  the phase function $\Gamma$ via
  \begin{equation}
    \Gamma : \PhSpace \times \PhSpace \to \CC,
             ((y,\omega),(z,\eta)) \mapsto e^{-2\pi i\langle y-z,\omega\rangle}.
    \label{eq:PhaseFunctionDefinition}
  \end{equation}
  Let $\Phi : D \to \RR^d$ be a warping function, assume $\theta\in\Ltv (\RR^d)$
  and denote $(g_{y,\omega})_{(y,\omega)\in\PhSpace} = \mathcal G(\theta,\Phi)$ as usual.
  Then the identity
  \begin{equation}\label{eq:difference_as_warping}
    \widehat{g_{y,\omega}} - \Gamma((y,\omega),(z,\eta)) \widehat{g_{z,\eta}}
    = e^{-2\pi i\langle y,\cdot\rangle} \cdot
      \left(
          w(\Phi(\omega))^{-1/2} \cdot
          \left(\bd T_{\Phi(\omega)} \tilde{\theta}_{(y,\omega),(z,\eta)}\right)\circ \Phi
      \right),
  \end{equation}
  holds for all $(y,\omega),(z,\eta)\in\PhSpace$, with
  \begin{equation}\label{eq:tilde_theta}
    \tilde{\theta}_{(y,\omega),(z,\eta)}
    := \left(
          \theta
          - \sqrt{
              \frac{w(\Phi(\omega))}{w(\Phi(\eta))}
            }
            \cdot
            \mathbf{\operatorname{E}}_{
                                        \Phi(\omega),
                                        A^{T}(\Phi(\omega)) \langle y-z \rangle
                                      }
            \left(
              \bd T_{\Phi(\eta)-\Phi(\omega)}\theta
            \right)
       \right)
    \in \Ltv.
  \end{equation}
  The operator $\Eye$ in Equation \ref{eq:tilde_theta} is a multiplication operator defined by
  \begin{equation}\label{eq:defOfEye}
       \Eye f
       := e^{2\pi i
             \langle
                A^{-T}(\tau)\langle \eps\rangle,
                \Phi^{-1}(\cdot+\tau)-\Phi^{-1}(\tau)
             \rangle
            }
          \cdot f
       \quad \text{ for all } \quad
       f:\RR^d \to \CC
       \text{ and }
       \tau,\eps \in \RR^d.
  \end{equation}
\end{proposition}

\begin{proof}
  To see that $\tilde{\theta}_{(y,\omega),(z,\eta)}\in\Ltv$, note that
  $\bd T_{\Phi(\eta)-\Phi(\omega)}$ and $\Eye$ are bounded operators on $\Ltv$ and that
  $\sqrt{w(\Phi(\omega))/w(\Phi(\eta))}$ is finite for all $\omega,\eta\in D$.
  Here, boundedness of $\translation_{x}$ on $\Ltv$ is a consequence of
  \eqref{eq:StandardTrafoAndTranslationEstimate}, since $w_0$ is submultiplicative.
  To prove \eqref{eq:difference_as_warping}, note that, by definition,
  \[
  \begin{split}
   \left(\widehat{g_{y,\omega}} - \Gamma((y,\omega),(z,\eta)) \widehat{g_{z,\eta}}\right)(\xi)
   & = e^{-2\pi i\langle y,\xi\rangle}
       g_\omega (\xi)
       - e^{-2\pi i\langle y-z,\omega\rangle}
         e^{-2\pi i\langle z,\xi\rangle}
         g_\eta (\xi)\\
   & = e^{-2\pi i\langle y,\xi\rangle}
       \left(
          g_\omega - e^{-2\pi i\langle y-z,\omega-\cdot\rangle}g_\eta
       \right)(\xi),
   \end{split}
  \]
  and furthermore
   \begin{align*}
   & g_\omega - e^{-2\pi i\langle y-z,\omega-\cdot\rangle}g_\eta \\
   & = w(\Phi(\omega))^{-1/2}\left(\bd T_{\Phi(\omega)}\theta\right)\circ \Phi
       - e^{-2\pi i\langle y-z,\omega-\cdot\rangle}
         w(\Phi(\eta))^{-1/2}
         \left(\bd T_{\Phi(\eta)}\theta\right) \circ \Phi \\
   & = w(\Phi(\omega))^{-1/2}
       \left(
          \bd T_{\Phi(\omega)}\theta
          - \sqrt{\frac{w(\Phi(\omega))}{w(\Phi(\eta))}}
            \cdot e^{2\pi i\langle y-z,\Phi^{-1}(\cdot)-\omega\rangle}
                  \bd T_{\Phi(\eta)}\theta
       \right)
       \circ \Phi\\
   & = w(\Phi(\omega))^{-1/2}
       \left(
          \bd T_{\Phi(\omega)}\left(
                                  \theta
                                  - \sqrt{\frac{w(\Phi(\omega))}{w(\Phi(\eta))}}
                                    \cdot e^{2\pi i
                                             \langle
                                               y-z,
                                               \Phi^{-1}(\cdot+\Phi(\omega)) - \omega
                                             \rangle}
                                    \bd T_{\Phi(\eta)-\Phi(\omega)}\theta
                              \right)
       \right)\circ \Phi\\
   & = w(\Phi(\omega))^{-1/2}
       \left(
          \bd T_{\Phi(\omega)}\tilde{\theta}_{(y,\omega),(z,\eta)}
       \right)
       \circ \Phi. \qedhere
   \end{align*}
\end{proof}

Now that we can express
\(
  g_{\PhVar} - \Gamma(\PhVar,\PhVarA)g_{\PhVarA}
\)
through $\tilde{\theta}(\PhVar,\PhVarA)$, we aim to derive conditions on $\theta$, such that
Lemma~\ref{lem:NiceCrossGramianEstimate} can be applied
with $\theta_1=\theta,\theta_2=\tilde{\theta}(\PhVar,\PhVarA)$.
In particular, we investigate the (uniform) continuity of the map $(\tau, \eps) \mapsto \Eye$,
in the next lemma.
Here, $\Eye$ is considered as an operator on $\lebesgue^q_{\tilde{w}} (\RR^d)$, for suitable weights $\tilde{w}$.

\begin{lemma}\label{lem:continuousdiffeye}
  Let $q \in [1,\infty)$ and 
  let $\tilde{w} : \RR^d \to \RR^+$ be a continuous weight function.
  Furthermore, assume that $\Phi$ is a $k$-admissible warping function with control weight $v_0$.

  The operator $\Eye ~:~ \lebesgue^q_{\tilde{w}} (\RR^d) \to \lebesgue^q_{\tilde{w}} (\RR^d)$,
  $\tau,\eps\in\RR^d$, given by \eqref{eq:defOfEye}, is well-defined and has the following properties:
  \begin{itemize}
   \item[(1)]   If $\vartheta \in \lebesgue^q_{\tilde{w}} (\RR^d)$
                with $\supp(\vartheta) \subset \overline{B_{\delta}}(0)$ for some $\delta > 0$, then
                \begin{equation}\label{eq:eyeestimate}
                  \|\vartheta-\Eye \vartheta\|_{\lebesgue^q_{\tilde{w}}}
                  \leq \sqrt{2[1-\cos\left(\pi\cdot \min\{1,2|\eps| \delta v_0(\delta)\}\right)]}
                       \cdot \|\vartheta\|_{\lebesgue^q_{\tilde{w}}}.
                \end{equation}

   \item[(2)] If $\vartheta\in \lebesgue^q_{\tilde{w}} (\RR^d)$, then
              \(
                \sup\limits_{\tau\in\RR^d}
                  \|\vartheta-\Eye \vartheta\|_{\lebesgue^q_{\tilde{w}}}
                \overset{\eps \rightarrow 0}{\rightarrow} 0
                .
              \)

   \item[(3)] If $\vartheta \in \mathcal C^{m}(\RR^d)$ for some $0 \leq m \leq k+1$,
              and if $j\in\underline{d}$ with
              \begin{equation}
                v_0^n
                \cdot \frac{\partial^{m-n}}
                           {\partial \upsilon_j^{m-n}}
                      \vartheta
                \in \lebesgue^q_{\tilde{w}} (\RR^d)
                \text{ for all } 0\leq n \leq m,
                \label{eq:EyeDerivativeAssumption}
              \end{equation}
              then
              \(
                \frac{\partial^\ell}{\partial \upsilon_j^\ell} \vartheta
                \in \lebesgue_{\tilde{w}}^q (\RR^d)
              \)
              for $0 \leq \ell \leq m$,
              $\frac{\partial^m}{\partial \upsilon_j^m}(\Eye \vartheta)\in \lebesgue^q_{\tilde{w}} (\RR^d)$
              for all $\tau,\eps\in\RR^d$, and
              \[
                \sup\limits_{\tau\in\RR^d}
                   \left\|
                       \frac{\partial^m}{\partial \upsilon_j^m}(\vartheta-\Eye \vartheta)
                   \right\|_{\lebesgue^q_{\tilde{w}}}
                \overset{\eps \rightarrow 0}{\rightarrow} 0.
              \]
              Furthermore, for each $\eps_0 > 0$, there is a constant $C_{m,\eps_0} > 0$
              satisfying for all $|\eps| \leq \eps_0$ that
              \begin{equation}
                  \sup_{\tau \in \RR^d}
                      \left\|
                         \frac{\partial^m}{\partial \upsilon_j^m} (\Eye \vartheta)
                      \right\|_{\lebesgue^q_{\tilde{w}}}
                  \leq \left\|
                         \frac{\partial^m}{\partial \upsilon_j^m} \vartheta
                       \right\|_{\lebesgue^q_{\tilde{w}}}
                       + C_{m,\eps_0}
                         \cdot |\eps|
                         \cdot \sum_{m=1}^m
                                 \left\|
                                    v_0^n
                                    \cdot \frac{\partial^{m-n}}{\partial \upsilon_j^{m-n}} \vartheta
                                 \right\|_{\lebesgue^q_{\tilde{w}}}
                  < \infty.
                  \label{eq:EDerivativeQuantitativeEstimate}
              \end{equation}
 \end{itemize}
 \end{lemma}

\begin{proof}
   Assumption \eqref{eq:EyeDerivativeAssumption} implies
   $\frac{\partial^\ell}{\partial \upsilon_j^\ell} \vartheta \in \lebesgue^q_{\tilde{w}} (\RR^d)$
   for all $0 \leq \ell \leq n$, since $v_0$ is radially increasing.
   Now, to prove (1), note for arbitrary $\upsilon \in \RR^d$ that
  \begin{equation*}
      |\vartheta(\upsilon) - (\Eye \vartheta)(\upsilon)|
      = \left|
          1 - e^{2\pi i\langle \eps,A^{-1}(\tau)(\Phi^{-1}(\upsilon+\tau)-\Phi^{-1}(\tau))\rangle}
        \right|
        \cdot |\vartheta(\upsilon)|,
  \end{equation*}
  where $\supp \vartheta \subset \overline{B_{\delta}}(0)$, such that it suffices to estimate this expression for
  $|\upsilon| \leq \delta$.
  We begin by expressing the difference $\Phi^{-1}(\upsilon+\tau)-\Phi^{-1}(\tau)$
  through the Jacobian $A = \mathrm{D}\Phi^{-1}$ of $\Phi^{-1}$ by using the directional derivative.
  This furnishes the following estimate:
  \[
   \begin{split}
     |A^{-1}(\tau)(\Phi^{-1}(\upsilon+\tau)-\Phi^{-1}(\tau))|
     &  =    \left|\int_0^1 A^{-1}(\tau)A(\tau+r\upsilon) \langle \upsilon \rangle ~dr\right|\\
     &  \leq |\upsilon|
             \cdot \max\limits_{r\in [0,1]}
                      \|A^{-1}(\tau)A(\tau+r\upsilon)\|
     \overset{\upsilon\in \overline{B_{\delta}}(0)}{\leq}
             \delta \cdot v_0(\delta),
   \end{split}
  \]
  where we used \eqref{eq:PhiHigherDerivativeEstimate} in the last step.
  Therefore,
  \(
    |\langle \eps,A^{-1}(\tau)(\Phi^{-1}(\upsilon+\tau)-\Phi^{-1}(\tau))\rangle|
    \leq |\eps| \cdot \delta \cdot v_0(\delta).
  \)

  Next, a simple calculation shows that $|1-e^{\pi i r}| = \sqrt{2[1-\cos(\pi r)]}$,
  which is an even function that is increasing on $[0,1]$ and converges to $0$ for $r\rightarrow 0$.
  Thus, we obtain
  \[
   |(\vartheta-\Eye \vartheta)(\upsilon)|
   \leq \sqrt{2 \left[ 1 - \cos(\pi \cdot 2|\eps|\delta v_0(\delta)) \right]} \cdot |\vartheta(\upsilon)|
  \]
  for all $0 \leq |\eps| \leq \frac{1}{2\delta v_0(\delta)}$.
  For $|\eps|>(2\delta v_0(\delta))^{-1}$ apply the trivial estimate
  $|1-e^{\pi i r}|\leq 2 = \sqrt{2[1-\cos(\pi)]}$ instead.
  This easily yields \eqref{eq:eyeestimate}, in fact for
  any solid Banach space $X$, and not only for $\lebesgue^q_{\tilde{w}}$.

  \medskip{}

  To prove (2), note that for a given $\vartheta\in \lebesgue^q_{\tilde{w}} (\RR^d)$,
  we have $\| \vartheta - \vartheta_n \|_{\lebesgue^q_{\tilde{w}}} \to 0$ as $n \to \infty$
  for the sequence $\vartheta_n = \vartheta \cdot \Indicator_{\overline{B_n}(0)}$,
  by the dominated convergence theorem.
%
%
  Furthermore, for every $n\in\NN$,
  \begin{equation*}
   \begin{split}
     \sup\limits_{\tau\in\RR^d}
         \|\vartheta-\Eye \vartheta\|_{\lebesgue^q_{\tilde{w}}}
     & \leq \|\vartheta-\vartheta_n\|_{\lebesgue^q_{\tilde{w}}}
       +    \sup\limits_{\tau\in\RR^d}
              \left(
                  \|\vartheta_n - \Eye \vartheta_n\|_{\lebesgue^q_{\tilde{w}}}
                  + \|\Eye \vartheta_n - \Eye \vartheta\|_{\lebesgue^q_{\tilde{w}}}
              \right) \\
     & = 2 \|\vartheta - \vartheta_n\|_{\lebesgue^q_{\tilde{w}}}
         + \sup\limits_{\tau\in\RR^d}
             \|\vartheta_n-\Eye \vartheta_n\|_{\lebesgue^q_{\tilde{w}}}.
   \end{split}
  \end{equation*}
  For any $n \in \NN$ and any $\eps_0 > 0$, we can choose $\eps_n > 0$ such that
  \[
    3 \|\vartheta\|_{\lebesgue^q_{\tilde{w}}}
    \cdot \sqrt{2[1-\cos\left(\pi\cdot \min\{1,2|\eps_n|n v_0(n)\right)\}}
    < \eps_0.
  \]
  \nicki{Noting that $\supp(\vartheta_n)\subset \overline{B_n}(0)$ by definition,
    we can now apply \eqref{eq:eyeestimate} with $\delta = n$
    and any $\eps \in \overline{B_{\eps_n}}(0)$ to obtain
    $\|\vartheta_n-\mathbf{\operatorname{E}}_{\tau,\eps} \vartheta_n\|_{\lebesgue^q_{\tilde{w}}} < \eps_0/3$,
  for all $\tau\in\RR^d$.}
  If additionally, $n\in\NN$ is such that $\|\vartheta-\vartheta_n\|_{\lebesgue^q_{\tilde{w}}} < \eps_0/3$,
  then $\|\vartheta-\Eye \vartheta\|_{\lebesgue^q_{\tilde{w}}}<\eps_0$.
  Since $\eps_0 > 0$ was arbitrary, we obtain
  \begin{equation*}
    \forall\ \eps_0 > 0
      \ \exists\ n\in \NN \text{ and } \eps_n > 0, \text{ such that }
          \eps \in \overline{B_{\eps_n}}(0)
          \text{ implies }
          \sup_{\tau\in\RR^d}
            \|\vartheta-\mathbf{\operatorname{E}}_{\tau,\eps} \vartheta\|_{\lebesgue^q_{\tilde{w}}} < \eps_0.
%
  \end{equation*}

  \medskip{}

  To prove (3), we first note that for $m=0$, all claims in this part are easy consequences
  of the definitions and of item (2).
  Therefore, we can assume $m \in \underline{k+1}$.
  Apply Leibniz's rule to obtain
  \begin{equation}
    \frac{\partial^m}{\partial \upsilon_j^m} (\Eye \vartheta)(\upsilon)
    = \sum_{n=0}^m
        \left(
           \binom{m}{n}
           \frac{\partial^n}{\partial \upsilon_j^n}
               \left(
                 e^{2\pi i \langle A^{-T}(\tau)\langle\eps\rangle, \Phi^{-1}(\cdot+\tau)-\Phi^{-1}(\tau)\rangle}
               \right)(\upsilon)
           \cdot \frac{\partial^{m-n}}{\partial \upsilon_j^{m-n}} \vartheta(\upsilon)
        \right).
    \label{eq:EyeDerivativeLeibnizApplication}
  \end{equation}
  Moreover, Faa Di Bruno's formula~\cite[Corollary~2.10]{FaaDiBrunoMultidimensional}---%
  a form of the chain rule for higher derivatives---yields
  for $n\in\underline{m}$ that
  \[
    \frac{\partial^n}{\partial \upsilon_j^n}
       \left(
         e^{2\pi i \langle A^{-T}(\tau)\langle\eps\rangle, \Phi^{-1}(\cdot+\tau)-\Phi^{-1}(\tau)\rangle}
       \right)(\upsilon)
    = e^{2\pi i \langle A^{-T}(\tau)\langle\eps\rangle, \Phi^{-1}(\upsilon+\tau)-\Phi^{-1}(\tau)\rangle}
      \cdot P_{n,\tau,\eps}(\upsilon)
    = \Eye P_{n,\tau,\eps}(\upsilon),
  \]
  where
 \begin{equation}\label{eq:PMTEest}
  \begin{split}
    P_{n,\tau,\eps}(\upsilon)
    & = \sum_{\ell=1}^n
            \left(
                (2\pi i)^\ell \cdot
                \sum_{\sigma\in (\underline{n - \ell + 1})^\ell}
                \left(
                   C_\sigma\cdot
                   \prod_{i=1}^{\ell}
                      \frac{\partial^{\sigma_i}}{\partial \upsilon_j^{\sigma_i}}
                      \left\langle
                        A^{-T}(\tau)\langle\eps\rangle,
                        \Phi^{-1}(\upsilon+\tau)-\Phi^{-1}(\tau)
                      \right\rangle
                \right)
            \right)\\
    & = \sum_{\ell=1}^n
           \left(
               (2\pi i)^\ell \cdot
               \sum_{\sigma\in (\underline{n - \ell + 1})^\ell}
               \left(
                   C_\sigma\cdot
                   \prod_{i=1}^{\ell}
                       \frac{\partial^{\sigma_i}}{\partial \upsilon_j^{\sigma_i}}
                       \left\langle
                         \eps,
                         A^{-1}(\tau)\langle \Phi^{-1}(\upsilon+\tau)\rangle
                       \right\rangle
               \right)
           \right),
  \end{split}
\end{equation}
 for suitable constants $C_\sigma \geq 0$.
 For the second equality, note that $\sigma_i\geq 1$ for all $i$,
 so that the term $\langle A^{-T}(\tau) \langle \eps \rangle ,\Phi^{-1}(\tau)\rangle$%
 ---which is constant with respect to $\upsilon$---can be ignored.
 In fact, the main statement of Faa Di Bruno's formula is exactly which
 $C_\sigma$ are nonzero and what value they attain, see also Lemma~\ref{lem:FaaDiBruno},
 but these details are not required here.
 Similar to \eqref{eq:PsiDerivativeIsPhi}, we have that
 \[
   \begin{split}
    \frac{\partial^{\sigma_i}}{\partial \upsilon_j^{\sigma_i}}
      \left\langle
          \eps,
          A^{-1}(\tau)\langle\Phi^{-1}(\upsilon+\tau)\rangle
      \right\rangle
   &= \frac{\partial^{\sigma_i - 1}}{\partial \upsilon_j^{\sigma_i - 1}}
        \left\langle
          \eps,
          A^{-1} (\tau) A(\upsilon + \tau) \langle e_i \rangle
        \right\rangle \\
   &= \frac{\partial^{\sigma_i - 1}}{\partial \upsilon_j^{\sigma_i - 1}}
        \left(
            \left[
                A^{-1} (\tau) A(\upsilon + \tau)
            \right]^{T}
            \eps
        \right)_{i}
   = \left(
        \frac{\partial^{\sigma_i-1}}{\partial \upsilon_j^{\sigma_i-1}}
           \phi_{\tau}(\upsilon) \langle \eps \rangle
      \right)_i \,\, ,
 \end{split}
 \]
 where $\phi_\tau = \left[ A^{-1} (\tau) A(\cdot + \tau) \right]^{T}$ is as in \eqref{eq:PhiDefinition}.
 By \eqref{eq:PhiHigherDerivativeEstimate}, we can estimate 
 \[
   \left|
    \left(
      \frac{\partial^{\sigma_i-1}}
           {\partial \upsilon_j^{\sigma_i-1}}
        \phi_{\tau}(\upsilon) \langle \eps \rangle
    \right)_i
  \right|
  \leq \left\|
         \frac{\partial^{\sigma_i-1}}{\partial \upsilon_j^{\sigma_i-1}} \phi_{\tau}(\upsilon)
        \right\|
        \cdot |\eps|
  \leq v_0 (\upsilon) \cdot|\eps|
  \quad \text{ and inserting this into \eqref{eq:PMTEest},}
 \]
 \[
  |P_{n,\tau,\eps}(\upsilon)|
  \leq \sum_{\ell=1}^n
          \left(
              (2\pi \cdot v_0(\upsilon)\cdot|\eps|)^\ell
              \cdot \sum_{\sigma\in (\underline{n - \ell + 1})^\ell}
                          C_\sigma
          \right)
  \leq \tilde{C} \cdot |\eps|\cdot \sum_{\ell=1}^n\left(v_0(\upsilon)^\ell\cdot|\eps|^{\ell-1}\right),
 \]
 for a suitably large $\tilde{C} = \tilde{C}(n) > 0$.
 Since we only consider $n\in\underline{m}$, we can in fact choose the same constant $\tilde{C}$
 for all values of $n$.
 Moreover, $1\leq v_0^\ell \leq v_0^n$ for all $\ell\leq n$.

 By assembling all the pieces and by separating the term $n = 0$
 in \eqref{eq:EyeDerivativeLeibnizApplication}, we thus get
 \[
  \begin{split}
  \left|
      \frac{\partial^m}{\partial \upsilon_j^m} (\Eye \vartheta)(\upsilon)
      - \Eye \left( \frac{\partial^m}{\partial \upsilon_j^m} \vartheta \right) (\upsilon)
  \right|
  & \leq \sum_{n=1}^m
            \binom{m}{n}
            \left|
              (\Eye P_{n,\tau,\eps})(\upsilon)
              \cdot \left( \frac{\partial^{m-n}}{\partial \upsilon_j^{m-n}} \vartheta \right) (\upsilon)
            \right|\\
  & = \sum_{n=1}^m
          \binom{m}{n}
          \left|
              P_{n,\tau,\eps}(\upsilon)
              \cdot \left(\frac{\partial^{m-n}}{\partial \upsilon_j^{m-n}} \vartheta\right)(\upsilon)
          \right|\\
  & \leq |\eps|
         \cdot \sum_{n=1}^m
                    \left(
                        \left|
                            \frac{\partial^{m-n}}{\partial \upsilon_j^{m-n}} \vartheta(\upsilon)
                        \right|
                        \cdot \sum_{\ell=1}^n
                                   \left(
                                       \tilde{C}
                                       \binom{m}{n}
                                       v_0(\upsilon)^\ell\cdot |\eps|^{\ell-1}
                                   \right)
                    \right)\\
  & \leq |\eps| \cdot
         \sum_{n=1}^m
           \left(
              \left|
                  \left(
                      v_0^n\cdot \frac{\partial^{m-n}}{\partial \upsilon_j^{m-n}} \vartheta(\upsilon)
                  \right)
              \right|
              \cdot \sum_{\ell=1}^n
                      \left(
                          \tilde{C}
                          \binom{m}{n}
                          \cdot|\eps|^{\ell-1}
                      \right)
           \right).
  \end{split}
 \]
 Let
  \(
    0 \leq C_{m,\eps}
      :=   \max_{n \in \underline{m}}
           \left(
             \sum_{\ell=1}^n
               \binom{m}{n}
               \cdot \tilde{C}
               |\eps|^{\ell-1}
           \right)
      <    \infty
  \)
 to obtain the estimate
 \[
  \left|
      \frac{\partial^m}{\partial \upsilon_j^m} (\Eye \vartheta)(\upsilon)
      - \Eye \left(  \frac{\partial^m}{\partial \upsilon_j^m} \vartheta \right) (\upsilon)
  \right|
  \leq C_{m,\eps} \cdot |\eps|
       \cdot\sum_{n=1}^m
          \left|
              v_0^n \cdot \frac{\partial^{m-n}}{\partial \upsilon_j^{m-n}} \vartheta(\upsilon)
          \right|.
 \]
 Since $\lebesgue^q_{\tilde{w}}$ is solid,
 we conclude
 \begin{equation}
    \left\|
        \frac{\partial^m}{\partial \upsilon_j^m} (\Eye \vartheta)
        - \Eye \left( \frac{\partial^m}{\partial \upsilon_j^m} \vartheta \right)
    \right\|_{\lebesgue^q_{\tilde{w}}}
    \leq C_{m,\eps} \cdot |\eps|
         \cdot \sum_{n=1}^m
                 \left\|
                    v_0^n\cdot \frac{\partial^{m-n}}{\partial \upsilon_j^{m-n}} \vartheta
                 \right\|_{\lebesgue^q_{\tilde{w}}}
    <    \infty.
   \label{eq:EyeDerivativeMainEstimate}
 \end{equation}
 Finally, with \nicki{$C_{m,\eps} \leq C_{m,\eps_0\cdot e_1} =: C_{m,\eps_0}$} for $|\eps| \leq \eps_0$, we obtain
 \begin{align*}
   & \sup_{\tau \in \RR^d}
        \left\|
            \frac{\partial^m}{\partial \upsilon_j^m}
              (\vartheta - \Eye \vartheta)
        \right\|_{\lebesgue^q_{\tilde{w}}} \\
   & \leq \sup_{\tau \in \RR^d}\left(
              \left\|
                  \frac{\partial^m}{\partial \upsilon_j^m} \vartheta
                  - \Eye \left( \frac{\partial^m}{\partial \upsilon_j^m} \vartheta \right)
              \right\|_{\lebesgue^q_{\tilde{w}}}
          + \left\|
                  \Eye \left( \frac{\partial^m}{\partial \upsilon_j^m} \vartheta \right)
                  - \frac{\partial^m}{\partial \upsilon_j^m}
                      (\Eye \vartheta)
              \right\|_{\lebesgue^q_{\tilde{w}}}\right)
     \xrightarrow[|\eps| \rightarrow 0]{} 0,
 \end{align*}
 as a consequence of part (2),
 and \eqref{eq:EyeDerivativeMainEstimate}.

 To prove \eqref{eq:EDerivativeQuantitativeEstimate}
 (and thus also $\frac{\partial^m}{\partial \upsilon_j^m} (\Eye \vartheta) \in \lebesgue^q_{\tilde{w}}$),
 observe $\| \Eye f \|_{\lebesgue^q_{\tilde{w}}} = \| f \|_{\lebesgue^q_{\tilde{w}}}$ for all
 $f \in \lebesgue^q_{\tilde{w}}(\RR^d)$.
 By Equation \eqref{eq:EyeDerivativeMainEstimate} the triangle inequality for norms yields
 \begin{align*}
     \sup_{\tau \in \RR^d}
         \left\|
             \frac{\partial^m}{\partial \upsilon_j^m} (\Eye \vartheta)
         \right\|_{\lebesgue^q_{\tilde{w}}}
     &\leq \sup_{\tau \in \RR^d} \left\|
              \Eye \left( \frac{\partial^m}{\partial \upsilon_j^m} \vartheta \right)
          \right\|_{\lebesgue^q_{\tilde{w}}}
          + C_{m,\eps} \cdot |\eps|
            \cdot \sum_{n=1}^m
                       \left\|
                          v_0^n\cdot \frac{\partial^{m-n}}{\partial \upsilon_j^{m-n}} \vartheta
                       \right\|_{\lebesgue^q_{\tilde{w}}} \\
     &= \left\|
          \frac{\partial^m}{\partial \upsilon_j^m} \vartheta
       \right\|_{\lebesgue^q_{\tilde{w}}}
       + C_{m,\eps} \cdot |\eps|
         \cdot \sum_{n=1}^m
                    \left\|
                       v_0^n\cdot \frac{\partial^{m-n}}{\partial \upsilon_j^{m-n}} \vartheta
                    \right\|_{\lebesgue^q_{\tilde{w}}}.
 \end{align*}
 This proves \eqref{eq:EDerivativeQuantitativeEstimate}, since $C_{m,\eps} \leq C_{m,\eps_0}$ as noted before.
\end{proof}

We now show that $\tilde{\theta}_{(y,\omega),(z,\eta)}$ uniformly converges to $0$ as $\delta \!\to\!\! 0$,
for $(y,\omega) \!\in\! \PhSpace$ and $(z,\eta) \!\in\! (y+\bd P^\delta_\omega)\times \bd Q^\delta_\omega$.
Recall that  $(y+\bd P^\delta_\omega) \times \bd Q^\delta_\omega$ was introduced in Lemma~\ref{lem:coverings}
as a simple superset to $\CalV^\delta_{(y,\omega)} = \bigcup_{V_{\ell,k} \ni (y,\omega)} V_{\ell,k}$,
appearing in the oscillation. The considered notion of convergence is in terms of the
$\lebesgue^q_{\tilde{w}}$-norm of certain derivatives of $\tilde{\theta}_{(y,\omega),(z,\eta)}$.
With Lemma \ref{lem:continuousdiffeye}, obtaining the desired estimates for $\tilde{\theta}_{(y,\omega),(z,\eta)}$
amounts to little more than an application of the triangle inequality and a somewhat elaborate three-$\eps$-argument.

 \begin{lemma}\label{lem:convergence_of_tildetheta}
   Let $q \in [1,\infty)$ and 
   let $\tilde{w} : \RR^d \to \RR^+$ be a continuous, submultiplicative weight function.
   Furthermore, assume that $\Phi$ is a $k$-admissible warping function with control weight $v_0$.
   If
   \[
     \theta \in \mathcal C^{m}(\RR^d) \text{ for some } 0 \leq m \leq k+1 ,
     \quad \text{ and }\quad
     v_0^n \cdot \frac{\partial^{m-n}}{\partial \upsilon_j^{m-n}}\theta
     \in  \lebesgue^q_{\tilde{w}} (\RR^d)
     \text{ for all } 0 \leq n \leq m,\ j\in\underline{d},
   \]
   then
   \begin{equation}\label{eq:tildetheta_in_X}
       \frac{\partial^{m}}{\partial \upsilon_j^m}
          \tilde{\theta}_{(y,\omega),(z,\eta)}
     = \frac{\partial^{m}}{\partial \upsilon_j^m}
          \left(
            \theta
            - \sqrt{\frac{w(\Phi(\omega))}{w(\Phi(\eta))}}
              \mathbf{\operatorname{E}}_{\Phi(\omega),
                                         A^{T}(\Phi(\omega))\langle y-z\rangle}
                 \left(
                    \bd T_{\Phi(\eta)-\Phi(\omega)}\theta
                \right)
          \right)
     \in \lebesgue^q_{\tilde{w}} (\RR^d)
   \end{equation}
   for all $(y,\omega),(z,\eta)\in\PhSpace$, and $j\in\underline{d}$.
   Furthermore, with
   \begin{equation}
       F_{j,m} (\delta; \theta, q, \tilde{w})
       := \sup_{(y,\omega)\in\PhSpace} \,\,
              \sup_{z\in (y+ \bd P_\omega^\delta),\eta\in \bd Q_\omega^\delta}
                \left\|
                    \frac{\partial^{m}}{\partial \upsilon_j^m} \tilde{\theta}_{(y,\omega),(z,\eta)}
                \right\|_{\lebesgue^q_{\tilde{w}}},
       \label{eq:TildeThetaFDefinition}
   \end{equation}
   where $\bd Q_\omega^\delta$  and $\bd P_\omega^\delta$ are as in Lemma~\ref{lem:coverings},
   we have
   \begin{equation}\label{eq:tildetheta_uniformconvergence}
      F_{j,m} (\delta; \theta, q, \tilde{w})
       < \infty \quad \text{ for all } \delta > 0,
       \qquad \text{ and } \quad F_{j,m} (\delta; \theta, q, \tilde{w})
       \xrightarrow[\delta \to 0]{} 0.
   \end{equation}
 \end{lemma}
 \begin{proof}
  Since $v_0$ and $\tilde{w}$ are submultiplicative, so is
  $v_0^n \tilde{w}$, and $\lebesgue^q_{v_0^n \tilde{w}} (\RR^d)$ is translation-invariant,
  see \eqref{eq:StandardTrafoAndTranslationEstimate}.
  Hence, since
  \(
    \frac{\partial^{m-n}}{\partial \upsilon_j^{m-n}} \theta
    \in \lebesgue^q_{v_0^n \tilde{w}} (\RR^d)
    ,
  \)
  $0 \leq n \leq m$ and $i \in \underline{d}$, the same holds for arbitrary translates.
  Thus, \Cref{lem:continuousdiffeye}(3) shows
  \(
    \frac{\partial^m}{\partial \upsilon_j^m}
      \theta, \frac{\partial^m}{\partial \upsilon_j^m} \Eye (\translation_{\tau_0} \theta)
    \in \lebesgue^q_{\tilde{w}} (\RR^d)
  \)
  for all $\tau_0, \tau, \eps \in \RR^d$. This establishes \eqref{eq:tildetheta_in_X}, since
  $\frac{w(\Phi(\omega))}{w(\Phi(\eta))}<\infty$.

  Fix $\delta> 0$ and $(y,\omega)\in\PhSpace$
  and $(z,\eta) \in (y+ \bd P_\omega^\delta)\times \bd Q_\omega^\delta$.
  For brevity, set $\tau := \Phi(\omega) - \Phi(\eta)$ and
  $\eps := A^{T}(\Phi(\omega)) \langle y-z \rangle$, noting that $\tau\in B_{2\delta}(0)$ and
  \(
    \eps
    \in A^{T}(\Phi(\omega)) \langle\bd P_\omega^\delta\rangle
    = B_{2\delta v_0(\delta)}(0)
    =:  B_{\eps_\delta}(0)
    .
  \)
  In particular, $\eps_\delta \leq \eps_{\delta_0}$, for all $\delta \leq \delta_0$,
  and $\eps_\delta\rightarrow 0$ as $\delta\rightarrow 0$.
  Recall the definition of $\tilde{\theta}_{(y,\omega), (z,\eta)}$ (given in \eqref{eq:tilde_theta}),
  and apply the triangle inequality twice to obtain the estimate
  \begin{equation}
      \begin{split}
           \left\|
              \frac{\partial^{m}}{\partial \upsilon_j^m}
                \tilde{\theta}_{(y,\omega),(z,\eta)}
           \right\|_{\lebesgue^q_{\tilde{w}}} 
           & \leq \!
              \left|
                  1 \! - \! \sqrt{\frac{w(\Phi(\omega))}{w(\Phi(\eta))}}
              \right|
              \cdot
              \left\|
                  \frac{\partial^{m}}{\partial \upsilon_j^m}\theta
              \right\|_{\lebesgue^q_{\tilde{w}}} 
             + \sqrt{\frac{w(\Phi(\omega))}{w(\Phi(\eta))}} \cdot \!
                  \left\|
                      \frac{\partial^{m}}{\partial \upsilon_j^m} \!
                      \left(
                        \theta \! - \! \mathbf{\operatorname{E}}_{\Phi(\omega),\eps}\theta
                      \right)
                  \right\|_{\lebesgue^q_{\tilde{w}}}\\ 
             & \quad + \sqrt{\frac{w(\Phi(\omega))}{w(\Phi(\eta))}} \cdot \!
              \left\|
                        \frac{\partial^m}{\partial \upsilon_j^m}
                            \mathbf{\operatorname{E}}_{\Phi(\omega), \eps}
                                (\theta \! - \! \translation_{\Phi(\eta) - \Phi(\omega)} \theta)
                    \right\|_{\lebesgue^q_{\tilde{w}}}
               \! . 
      \end{split}
      \label{eq:thetatilde_intermediate_est}
  \end{equation}

  Next, Lemma~\ref{lem:continuousdiffeye}(3) yields 
  \begin{align}
      E_\delta & := \sup_{|\eps| \leq \eps_\delta} \,\, \sup_{\omega \in D}
        \left\|
            \frac{\partial^{m}}{\partial \upsilon_j^m} \!
                \left(
                    \theta - \mathbf{\operatorname{E}}_{\Phi(\omega),\eps}\theta
                \right)
        \right\|_{\lebesgue^q_{\tilde{w}}} \leq \infty,\ \text{for all}\ \delta>0, \text{ with }
      E_\delta \rightarrow 0 \ \text{ as } \delta \rightarrow 0, \text{ and}\label{eq:FirstTermEst}\\
      F_\delta & := \sup_{|\eps| \leq \eps_\delta} \,\, \sup_{\omega \in D}
         \left\|
             \frac{\partial^m}{\partial \upsilon_j^m}
                 \mathbf{\operatorname{E}}_{\Phi(\omega), \eps}
                     (\theta \! - \! \translation_{-\tau} \theta)
         \right\|_{\lebesgue^q_{\tilde{w}}} \nonumber\\
    &\leq \left\|
            \frac{\partial^m}{\partial \upsilon_j^m}
              \theta
            - \translation_{-\tau}\left(\frac{\partial^m}{\partial \upsilon_j^m} \theta\right)
          \right\|_{\lebesgue^q_{\tilde{w}}}
          + C_{m,\eps_{\delta_0}} \cdot \eps_\delta \cdot
            \sum_{n=1}^m
                \left\|
                    v_0^n \cdot \frac{\partial^{m-n}}{\partial \upsilon_j^{m-n}}
                        (\theta \! - \! \translation_{-\tau}\theta)
                \right\|_{\lebesgue^q_{\tilde{w}}}.\label{eq:SecTermEst}
  \end{align}
  Note that the first term of the right-hand side of \eqref{eq:SecTermEst} converges to $0$
  for $\delta\rightarrow 0$, since $|\tau|\leq 2\delta$ and translation is continuous
  in $\bd L^q_{\tilde{w}}$, since $\tilde{w}$ is continuous and submultiplicative.
  Furthermore, the sum over $n$ in the right-hand side of \eqref{eq:SecTermEst} is finite,
  since $\bd L^q_{\tilde{w}}$ is translation-invariant and hence,
  all summands are finite by assumption.
  Therefore, $F_\delta$ vanishes for $\delta\rightarrow 0$.
  In fact, since $|\eps|\leq \eps_\delta$ and $w$ is $v_0^d$-moderate
  with radially increasing $v_0$ (cf.\ \Cref{lem:assume_conclude}),
  $\frac{w(\Phi(\omega))}{w(\Phi(\eta))} \leq v_0^d(\eps_\delta)$,
  which settles the desired convergence of the second and third term
  in \eqref{eq:thetatilde_intermediate_est}.

  To settle convergence of the first term, we need to show that
  $\frac{w(\Phi(\omega))}{w(\Phi(\eta))}\overset{\delta\rightarrow 0}{\rightarrow} 1$,
  uniformly with respect to $\omega\in D, \eta\in \bd Q_\omega^\delta$.
  To this end, note that
  \[
      \frac{w(\Phi(\omega))}{w(\Phi(\eta))}
      = \frac{w(\Phi(\eta)) + \int_0^1 \frac{d}{dt}\big|_{t=s}
                    \left[w(\Phi(\eta)+s\tau)\right]~ds}{w(\Phi(\eta))}
      \leq 1 +\frac{\sup_{\upsilon\in B_{2\delta}(\Phi(\eta))}\nabla_{\tau}w(\upsilon)}{w(\Phi(\eta))},
  \]
  where $\nabla_{\tau}$ denotes the derivative in direction $\tau\in\RR^d$.
  We now use Jacobi's formula
  \[
    \frac{d}{dt} \det A(t) = \det A(t) \cdot \mathop{\operatorname{trace}} ([A(t)]^{-1} \cdot A'(t)),
  \]
  valid for the derivative of the determinant of any differentiable function
  $M : I \subset \RR \to \GL (\RR^d)$
  (see \cite[Section 8.3, Equation (2)]{MatrixDifferentialCalculus}),
  to obtain
  \[
    \begin{split}
    \nabla_{\tau}w(\upsilon)
    & = \sum_{j\in\underline{d}}
          \tau_j \frac{\partial}{\partial \upsilon_j} \det(A(\upsilon))
      = \det(A(\upsilon))
        \cdot \sum_{j\in\underline{d}}
                 \tau_j \cdot
                 \mathop{\operatorname{trace}}
                 \left(
                    A^{-1}(\upsilon)
                    \frac{\partial}{\partial \upsilon_j} A(\upsilon)
                 \right)\\
    & = w(\upsilon) 
        \cdot \sum_{j\in\underline{d}}
                 \tau_j \cdot
                 \mathop{\operatorname{trace}}
                 \left(
                     \left(
                     \frac{\partial}{\partial \eta_i} \bigg|_{\eta = 0}
                        A^{-1}(\upsilon)
                        A(\upsilon + \eta)
                     \right)^T
                 \right) \\
    & = w(\upsilon) 
        \cdot \sum_{j\in\underline{d}}
                 \tau_j \cdot
                 \mathop{\operatorname{trace}}
                 \left(
                     \frac{\partial}{\partial \eta_i} \bigg|_{\eta = 0} \phi_{\upsilon}(\eta)
                 \right),
    \end{split}
  \]
  with $\phi_{\upsilon}$ as in \eqref{eq:PhiDefinition}.
  Note that $\phi_{\upsilon}(0) = \mathop{\operatorname{id}}$ for all
  $\upsilon\in\RR^d$, so that \eqref{eq:PhiHigherDerivativeEstimate} yields
  $\| (\partial_i \phi_{\upsilon}) (0) \| \leq v_0(0)$.
  Additionally, the trace of a matrix $M \in \RR^{d\times d}$ can be (coarsely) estimated by
  $| \mathop{\operatorname{trace}}(M) | \leq d\|M\|$, such that
  \[
    \left|\nabla_{\tau}w(\upsilon) \right|
    \leq d \cdot w(\upsilon) \cdot \sum_{j\in\underline{d}} v_0(0) \cdot |\tau_j|
    \leq d \cdot w(\upsilon) \cdot \|\tau\|_1\cdot v_0(0)
    \leq d^{3/2} \cdot w(\upsilon) \cdot |\tau| \cdot v_0(0).
  \]
  Therefore, with $|\tau|\leq 2\delta$ and $v_0^d$-moderateness of $w$,
  \begin{equation}
      \left|
          1 - \frac{w(\Phi(\omega))}{w(\Phi(\eta))}
      \right|
      \leq |\tau| \cdot
           d^{3/2} \cdot
           v_0(0)
           \cdot \max_{r\in[0,1]}
                   \frac{w(\Phi(\eta)+r\tau)}{w(\Phi(\eta))}
      \leq 2\delta \cdot d^{3/2} \cdot v_0^d (2\delta) \cdot v_0(0) =: C^\delta < \infty.
      \label{eq:WeightQuotientConvergence}
  \end{equation}
  The final estimate is independent of $\omega\in D$, and of $\eta\in \bd Q_\omega^\delta$,
  and $C^\delta \rightarrow 0$ as $\delta \rightarrow 0$.
 \end{proof}

 We are now ready to prove \Cref{thm:main2_discreteframes}.

\subsection{Proof of Theorem~\ref{thm:main2_discreteframes}}

Recall that, by \Cref{rem:OscKernelContinuity}, $\oscVFGd$ is continuous.
Using \Cref{pro:difference_as_warping} and Parseval's formula,
we can rewrite the oscillation at $((y,\omega), (z,\eta))\in\PhSpace\times\PhSpace$, as follows:
 \begin{align}
     \oscVFGd ( (y,\omega), (z,\eta))
     &= \sup_{(z_0, \eta_0) \in \CalV_{(z,\eta)}^\delta}
          \left|
              \left\langle
                  \widehat{g_{y,\omega}},
                  \widehat{g_{z,\eta}} - \Gamma ( (z,\eta), (z_0, \eta_0)) \cdot \widehat{g_{z_0, \eta_0}}
              \right\rangle
          \right| \nonumber \\
     &= \sup_{(z_0, \eta_0) \in \CalV_{(z,\eta)}^\delta}
          \left|
              K_{\theta, \tilde{\theta}_{(z,\eta), (z_0,\eta_0)}, \Phi} ( (y,\omega), (z,\eta))
          \right| .
          \label{eq:OscillationRewritteNormalOrdering}
 \end{align}

 Based on \eqref{eq:OscillationRewritteNormalOrdering},
 \Cref{pro:kern_in_theta} provides
  \[
   \begin{split}
    \lefteqn{\left|
         K_{\theta, \tilde{\theta}_{(z,\eta), (z_0,\eta_0)}, \Phi} ( (y,\omega), (z,\eta))
    \right|}\\
    & = \sqrt{\frac{w(\Phi(\eta))}{w(\Phi(\omega))}}
      \cdot \nicki{
              \left|
                L_{\Phi(\eta)} [\theta, \tilde{\theta}_{(z,\eta), (z_0, \eta_0)}]
                (A^{T}(\Phi(\eta)) \langle z-y \rangle,\Phi(\omega)-\Phi(\eta))
              \right|
            },
    \end{split}
  \]
  where $L_{\Phi(\eta)}$ is as in \eqref{eq:defOfL}.
  If we define $\mathcal L_{\tau_0}:\RR^d\times\RR^d \rightarrow \RR^+_0$, $\tau_0\in\RR^d$, by
  \begin{equation}\label{eq:OscSupKern1}
    \mathcal L_{\tau_0}(x,\tau)
    := \sup_{z\in\RR^d} \,\,
         \sup_{(z_0,\eta_0)\in \CalV^\delta_{z,\Phi^{-1}(\tau_0)}}
          \nicki{\left|L_{\tau_0} [\theta, \tilde{\theta}_{(z,\Phi^{-1}(\tau_0)), (z_0, \eta_0)}](x,\tau)\right|},
  \end{equation}
  then, for all $(y,\omega), (z,\eta)\in\PhSpace$,
  \[
    \oscVFGd ( (y,\omega), (z,\eta))
    \leq \sqrt{\frac{w(\Phi(\eta))}{w(\Phi(\omega))}}
          \mathcal{L}_{\Phi(\eta)}(A^{T}(\Phi(\eta)) \langle z-y \rangle,\Phi(\omega)-\Phi(\eta)).
  \]

  Via a tedious, but straightforward derivation involving several changes of variable
  in a manner similar to the proof of \Cref{pro:kern_in_theta}, we obtain in particular that
  \begin{equation}\label{eqEstOscByScriptL}
    \|\oscVFGd\|_{\BBm}
    \leq \|\oscVFGd\|_{\BB_{m^\natural}}
    \leq \esssup_{\tau_0\in\RR^d}
           \int_{\RR^d}\int_{\RR^d} M(x,\tau)\mathcal L_{\tau_0}(x,\tau) ~dx~d\tau,
  \end{equation}
  where $M$ is defined as in \eqref{eq:BigWeightDefinition} and we used that
  $m$ is $\Phi$-compatible with the (symmetric) dominating weight
  $m^\Phi(x,\tau) = (1+|x|)^p\cdot v_1(\tau)$.

  By \Cref{lem:convergence_of_tildetheta}, with $\tilde{w} \!=\! w_2$,
  all functions $\tilde{\theta}_{(z,\eta),(z_0,\eta_0)}$ with $(z,\eta) \!\!\in\! \PhSpace$ and
  $(z_0,\eta_0) \!\!\in\! \CalV^\delta_{z,\eta} \!\subset\! (z + \bd P^\delta_{\eta})\times \bd Q^\delta_{\eta}$
  satisfy the conditions of \Cref{lem:NiceCrossGramianEstimate}, as does $\theta$.
  Hence, for any $z,\tau_0\in\RR^d$ and $(z_0,\eta_0)\in \CalV^\delta_{z,\Phi^{-1}(\tau_0)}$,
  \Cref{lem:NiceCrossGramianEstimate} yields
  \begin{equation}\label{eq:EstOfL}
    \nicki{|L_{\tau_0} [\theta, \tilde{\theta}_{(z,\Phi^{-1}(\tau_0)), (z_0, \eta_0)}](x,\tau)|}
    \leq C \cdot C_{\textrm{max}}
           \cdot (1+|x|)^{-(d+p+1)}
           \nicki{\cdot v_0^{4d+3p+3}(\tau)}
           \cdot [w_2(\tau)]^{-1},
  \end{equation}
  \nicki{where $C>0$ depends only on $d$, $k$ and the control weight $v_0$, and furthermore }
  \[
    \begin{split}
    C_{\textrm{max}}
    = &\,\, C_{\textrm{max}}\left(d+p+1,\theta,\tilde{\theta}_{(z,\Phi^{-1}(\tau_0)), (z_0, \eta_0)}\right)\\
    = & \max_{\substack{j \in \underline{d} \\ 0 \leq m \leq d+p+1}}
                        \bigg\|
                            \frac{\partial^m}{\partial \upsilon_j^m} \theta
                        \bigg\|_{\lebesgue^2_{w_2}(\RR^d)} \cdot
                        \max_{\substack{j \in \underline{d} \\ 0 \leq m \leq d+p+1}}
                        \bigg\|
                            \frac{\partial^m}{\partial \upsilon_j^m}
                            \tilde{\theta}_{(z,\Phi^{-1}(\tau_0)), (z_0, \eta_0)}
                        \bigg\|_{\lebesgue^2_{w_2}(\RR^d)}\\
      ({\scriptstyle \text{Lem. \ref{lem:convergence_of_tildetheta}}})
      \, \leq & \max_{\substack{j \in \underline{d} \\ 0 \leq m \leq d+p+1}}
                        \bigg\|
                            \frac{\partial^m}{\partial \upsilon_j^m} \theta
                        \bigg\|_{\lebesgue^2_{w_2}(\RR^d)}
                        \cdot \max_{\substack{j \in \underline{d} \\ 0 \leq m \leq d+p+1}}
                                F_{j,m}(\delta;\theta,2,w_2) =: D_{\textrm{max}}^\delta
    < \infty.
    \end{split}
  \]
  Note that the estimate $D_{\textrm{max}}^\delta$ is independent of $\tau_0\in\RR^D$,
  $z\in\RR^d$, and $(z_0,\eta_0)\in \CalV^\delta_{z,\Phi^{-1}(\tau_0)}$,
  such that taking $D_{\textrm{max}}^\delta$ instead of $C_{\textrm{max}}$
  in \eqref{eq:EstOfL} produces a valid upper estimate for $\mathcal L_{\tau_0}(x,\tau)$.
  Moreover, note that \Cref{lem:convergence_of_tildetheta} implies
  $D_{\textrm{max}}^\delta\rightarrow 0$ as $\delta\rightarrow 0$.

  Proving $\|\oscVFGd\|_{\BBm}<\infty$ is now analogous
  to the proof of \Cref{thm:MR1_kernel_is_in_AAm}, and $\|\oscVFGd\|_{\BB_{m}}\rightarrow 0$
  as $\delta\rightarrow 0$ follows directly from $D_{\textrm{max}}^\delta\rightarrow 0$.
\hfill\qed


 \section{Coorbit space theory of warped time-frequency systems}
\label{sec:warpedcoorbits}

We have now developed explicit sufficient conditions that ensure $K_{\theta,\Phi},
\oscVFGd\in\BBm$ and hence, by Eq.~\eqref{eq:MaxKernOscEstimate},
$\mathrm{M}_{\CalV_\Phi^\delta} K_{\theta,\Phi}\in\BBm$, since $\BBm$ is solid.
These are the crucial ingredients for applying coorbit theory
in the setting of warped time-frequency representations.

\begin{theorem}\label{cor:warped_disc_frames}
Let $\Phi$ be a $(d+p+1)$-admissible warping function with control weight $v_0$,
where $p=0$ if $R_\Phi = \sup_{\xi\in D} \|\mathrm{D}\Phi(\xi)\|=\infty$ and $p \in \NN_0$ otherwise.
Let furthermore $m_0 : \PhSpace \times \PhSpace \to \R^+$ be a symmetric weight
that satisfies $1\leq m_0(\PhVar,\PhVarA)\leq C^{(0)} m_0(\PhVar,\PhVarC)m_0(\PhVarC,\PhVarA)$
for all $\PhVar,\PhVarA,\PhVarC\in\Lambda$ and
\begin{equation}
   m_0((y, \xi), (z, \eta))
   \leq (1 + |y-z|)^p \cdot v_1 (\Phi(\xi) - \Phi(\eta)),
   \text{ for all } y,z\in\RR^d \text{ and } \xi,\eta\in D,\ \tau,\upsilon\in\RR^\dimension,
\end{equation}
for some continuous and submultiplicative weight $v_1 : \RR^\dimension \to \R^+$
with $v_1(\upsilon)=v_1(-\upsilon)$ for all $\upsilon\in\RR^d$.

Then there exist nonzero $\theta\in \lebesgue^2_{v_0^{d/2}} (\RR^d)$, such that for any rich,
solid Banach space $\BanachOne \hookrightarrow \lebesgue_{\mathrm{loc}}^1 (\PhSpace)$
with $\mathcal{B}_{m_0}(\BanachOne) \hookrightarrow \BanachOne$,
\begin{enumerate}
   \item $\Co(\mathcal G(\theta,\Phi),\BanachOne)$ is a well-defined Banach function space.

   \item There is a $\delta_0 = \delta_0(\theta,\Phi,m_0) > 0$ independent of $\BanachOne$, such that
         \[
           (g_{y_{\ell,k},\omega_{\ell,k}})_{\ell,k\in\ZZ^d}
           \subset \mathcal G(\theta,\Phi)
         \]
         is a Banach frame decomposition for $\Co(\mathcal G(\theta,\Phi),\BanachOne)$, whenever
         the points $\left((y_{\ell,k},\omega_{\ell,k})\right)_{\ell,k\in\ZZ^d}\subset \PhSpace$
         satisfy $(y_{\ell,k},\omega_{\ell,k})\in V^\delta_{\ell,k}$, where
         $\CalV_\Phi^\delta = (V_{\ell,k}^\delta )_{\ell,k\in\ZZ^d}$
         is the $\Phi$-induced $\delta$-fine covering and $\delta \leq \delta_0$.
\end{enumerate}
In particular, items (1) and (2) above hold for $\BanachOne = \lebesgue^{p,q}_{\kappa}(\PhSpace)$,
with $1\leq p,q\leq\infty$ and any weight $\kappa: \PhSpace \rightarrow [1,\infty)$ that satisfies
$m_{\kappa} \lesssim m_0$.
\end{theorem}

\begin{proof}
  By \Cref{pro:inducedcoverProps,prop:GreatSimplification}, the $\Phi$-induced
  $\delta$-fine covering $\CalV_\Phi^\delta$ is a topologically admissible,
  product-admissible covering that satisfies items (1)-(3) of \Cref{assu:CoorbitAssumptions1}
  and item (1) of \Cref{assu:CoorbitAssumptions2}.
  Moreover, item (6) of \Cref{assu:CoorbitAssumptions1} is satisfied,
  by the assumptions of this theorem.

  Next, choose $\theta\in\Ltv(\RR^d)$, such that $\|\theta\|_{\lebesgue^2(\RR^d)} = 1$
  and the assumptions of \Cref{thm:main2_discreteframes} are satisfied with $m=m_v$ defined by
  \begin{align*}
    & m_v((y,\omega),(z,\eta))
      = \max\left\{\frac{v((y,\omega))}{v((z,\eta))},\frac{v((z,\eta))}{v((y,\omega))}\right\}, \\
    \text{with}\quad
    & v((y,\omega)) := m_0((y,\omega),(x,\xi))
                       \cdot \max\{w(\Phi(\omega)),[w(\Phi(\omega))]^{-1}\},
  \end{align*}
  for all $(y,\omega),(z,\eta)\in\PhSpace$ and some fixed, arbitrary $(x,\xi)\in\PhSpace$.
  This is always possible, since any function
  $\theta\in\mathcal{C}^\infty_c(\RR^d)\subset \lebesgue^2(\RR^d)$ with unit $\lebesgue^2$-norm
  satisfies these assumptions.
  In particular, the assumptions of \Cref{thm:main2_discreteframes}
  are also satisfied for $m=m_0\leq m_v$.
  By \Cref{prop:WarpedSystemWeaklyContinuous}, the map $(y,\omega)\mapsto g_{y,\omega}$
  is continuous and by \Cref{cor:WarpingInversion}, the warped time-frequency system
  $\mathcal G(\theta,\Phi)$ is a tight Parseval frame,
  such that item (4) of \Cref{assu:CoorbitAssumptions1} is satisfied.
  In particular, by Eq.~\eqref{eq:FrameElementsBounded},
  $\sup_{(y,\omega)\in\PhSpace} \|g_{y,\omega}\|_2 \leq \|\theta\|_{\Ltv}  < \infty$.
  Hence, with $w^c_{\CalV_\Phi^\delta}=\max\{w(\Phi(\omega)),[w(\Phi(\omega))]^{-1}\}$
  as in \Cref{pro:inducedcoverProps} and $u(\PhVar) := m_0(\PhVar,(x,\xi))$
  with the same choice of $(x,\xi)\in\PhSpace$ as above, item (5) of \Cref{assu:CoorbitAssumptions1}
  is satisfied as well.
 
  Moreover, by choice of $\theta$, and with $\Gamma$ as in \Cref{thm:main2_discreteframes}, we have
  \[
    \|K_{\theta,\Phi}\|_{\mathcal{B}_{m_v}}
    < \infty
    \quad \text{and}\quad
    \|\mathrm{M}_{\CalV_\Phi^\delta} K_{\theta,\Phi}\|_{\mathcal{B}_{m_0}}
    \leq \|K_{\theta,\Phi}\|_{\mathcal{B}_{m_0}} + \|\oscVFGd\|_{\mathcal{B}_{m_0}}
    < \infty,
  \]
  showing that the final item (7) of \Cref{assu:CoorbitAssumptions1} is satisfied.
  Hence, \Cref{assu:CoorbitAssumptions1} is fully satisfied and we can apply \Cref{thm:coorbits}
  to show that $\Co(\mathcal G(\theta,\Phi),\BanachOne)$ is a well-defined Banach function space.

  Finally, note that $\Gamma$ as in \Cref{thm:main2_discreteframes} is continuous,
  to verify that item (2) of \Cref{assu:CoorbitAssumptions2} is satisfied.
  By the same theorem, we can choose $\delta_0 > 0$, such that
  \[
    \| \oscVFGd \|_{\mathcal{B}_{m_v}}
    \cdot (2 \| K_\Psi \|_{\mathcal{B}_{m_v}} + \| \oscVFGd \|_{\mathcal{B}_{m_v}})
    < 1
  \]
  for all $\delta\leq\delta_0$, proving the second assertion.
  The proof is completed by observing that the statement about weighted, 
  mixed-norm Lebesgue spaces is a direct consequence of \eqref{eq:BBmBoundedOnLPQ}.
\end{proof}

By definition, the coorbit space $\Co (\mathcal{G}(\theta,\Phi),\BanachOne)$ depends
on both the prototype function $\theta$ and the warping function $\Phi$.
The dependence on the warping function $\Phi$ is an essential consequence
of (sufficiently) different warping functions inducing time-frequency representations
with vastly different properties. Relations between coorbit spaces associated
to different warping functions are studied in the framework of decomposition spaces~\cite{fegr85,boni07,voigtlaender2016embeddings}
in a follow-up contribution.
Here, we will show that the dependence on the generating prototype $\theta$ can be weakened,
i.e., under certain conditions on $\theta_1,\theta_2$,
the coorbit spaces $\Co (\mathcal{G}(\theta_1,\Phi),\BanachOne)$
and $\Co (\mathcal{G}(\theta_2,\Phi),\BanachOne)$ are equal,
similar to modulation spaces for the STFT.
Before we do so, however, we show that the mixed kernel associated with
two warped time-frequency systems inherits the membership in $\BBm$ (or $\AAm$)
from the kernels of the individual systems.

\begin{lemma}\label{pro:protoindep}
  Let $X\in\{\AAm,\BBm\}$, with a symmetric weight $m$
  satisfying $m(\PhVar,\PhVarA)\leq C^{(0)}m(\PhVar,\PhVarC)m(\PhVarC,\PhVarA)$,
  for some $C^{(0)}$ and all $\PhVar,\PhVarA,\PhVarC\in\PhSpace$.
  If $\theta_1,\theta_2 \in \Ltv\cap \LtRd$ are nonzero and such that
  $K_{\theta_1,\Phi},K_{\theta_2,\Phi}\in X$, then
  \begin{equation}
    K_{\theta_1,\theta_2,\Phi} := K_{\mathcal{G}(\theta_1,\Phi),\mathcal{G}(\theta_2,\Phi)} \in X.
    \label{eq:mixed_kernel_nice}
  \end{equation}
%
\end{lemma}

\begin{proof}
  We first consider the case $\langle \theta_1,\theta_2\rangle \neq 0$.
  In that case, the orthogonality relations, Theorem~\ref{thm:orthrel},
  applied to the kernel $K_{\theta_1,\Phi} \cdot K_{\theta_2,\Phi}$ yield,
  for all $(y,\omega),(z,\eta)\in\Lambda$,
  \[
    \begin{split}
    K_{\theta_1,\Phi} \cdot K_{\theta_2,\Phi}((y,\omega),(z,\eta))
    & = \int_\Lambda
          K_{\theta_1,\Phi}((y,\omega),(x,\xi))
          K_{\theta_2,\Phi}((x,\xi),(z,\eta))
        ~d(x,\xi) \\
    & = \int_\Lambda
          \overline{\langle g^{(1)}_{y,\omega}, g^{(1)}_{x,\xi}\rangle}
          \langle g^{(2)}_{z,\eta}, g^{(2)}_{x,\xi}\rangle
        ~d(x,\xi)\\
    ({\scriptstyle{\text{Def.\ of } V_{\bullet,\Phi}}})
    & = \int_\Lambda
          V_{\theta_2,\Phi}
          g^{(2)}_{z,\eta}(x,\xi)
          \overline{V_{\theta_1,\Phi}g^{(1)}_{y,\omega}(x,\xi)}
        ~d(x,\xi)
      = \langle
          V_{\theta_2,\Phi}g^{(2)}_{z,\eta},
          V_{\theta_1,\Phi}g^{(1)}_{y,\omega}
        \rangle \\
    ({\scriptstyle{\text{orth.~rel.}}})
    & = \langle g^{(2)}_{z,\eta},g^{(1)}_{y,\omega}\rangle
        \langle \theta_1,\theta_2\rangle
      = \langle \theta_1,\theta_2\rangle
        \cdot K_{\theta_1,\theta_2,\Phi}((y,\omega),(z,\eta)).
    \end{split}
  \]
  Since, under the conditions on $m$, $\AAm, \BBm$ are algebrae, this establishes \eqref{eq:mixed_kernel_nice}.

  If $\langle \theta_1,\theta_2\rangle = 0$,
  then we need an auxiliary function $\theta_3$,
  which may be any function  in $\Ltv\cap \LtRd$ such that
  $K_{\theta_3,\Phi}\in X$ and that is neither orthogonal to
  $\theta_1$ nor to $\theta_2$.
  For example, $\theta_3$ could satisfy the conditions of
  Theorem.~\ref{thm:MR1_kernel_is_in_AAm}.
  By the first part of the proof, we obtain
  \[
    (K_{\theta_1,\Phi} \cdot K_{\theta_3,\Phi})
    \cdot (K_{\theta_3,\Phi} \cdot K_{\theta_2,\Phi})
    = \langle \theta_1,\theta_3\rangle
      \overline{\langle \theta_2,\theta_3\rangle}
      \cdot \ K_{\theta_1,\theta_3,\Phi}
      \cdot K_{\theta_3,\theta_2,\Phi}\in X.
  \]
  Now, apply the argument in the first part of the proof again to obtain that
  \[
    K_{\theta_1,\theta_2,\Phi}
    = C^{-1}
      (K_{\theta_1,\Phi} \cdot K_{\theta_3,\Phi})
      \cdot (K_{\theta_3,\Phi}\cdot K_{\theta_2,\Phi})
    \qquad \text{with} \qquad
    C
    = \|\theta_3\|^2
      \langle \theta_1,\theta_3\rangle
      \overline{\langle \theta_2,\theta_3\rangle}
    .
    \qedhere
  \]
\end{proof}

\begin{remark}
  If $\theta_1, \theta_2$ satisfy the conditions of Theorem~\ref{thm:MR1_kernel_is_in_AAm},
  then the assumptions of Lemma~\ref{pro:protoindep} can be verified by applying that theorem.
  However, since Theorem~\ref{thm:MR1_kernel_is_in_AAm} only provides \emph{sufficient} conditions,
  there might be $\theta_1, \theta_2$ with $K_{\theta_1,\Phi}, K_{\theta_2,\Phi} \in \BBm$
  that do not satisfy those conditions, for which Lemma~\ref{pro:protoindep} remains valid.
\end{remark}

\begin{theorem}\label{pro:WarpedCoorbitIndependence2}
  Assume that $\Phi$, $m_0$ and both $\theta_1\in\Ltv$ and $\theta_2\in\Ltv$ jointly satisfy
  the conditions of \Cref{thm:main2_discreteframes}.

  Then, for any rich, solid Banach space
  $\BanachOne\hookrightarrow \lebesgue_{\mathrm{loc}}^1 (\PhSpace)$
  with $\mathcal{B}_{m_0}(\BanachOne)\hookrightarrow \BanachOne$, we have
  \[ \Co (\mathcal{G}(\theta_1,\Phi),\BanachOne) = \Co (\mathcal{G}(\theta_2,\Phi),\BanachOne).\]
  In particular, the statement holds for $\BanachOne = \lebesgue^{p,q}_\kappa(\mu)$,
  with $1\leq p,q\leq\infty$ and any weight $\kappa: \PhSpace \rightarrow [1,\infty)$
  that satisfies $m_\kappa \lesssim m_0$.
\end{theorem}

\begin{proof}
  The same derivations as in the proof of \Cref{cor:warped_disc_frames} show that
  \Cref{assu:CoorbitAssumptions1,assu:CoorbitAssumptions2} are fully satisfied and
  consequently, by \Cref{thm:coorbits}, $\Co (\mathcal{G}(\theta_1,\Phi),\BanachOne)$ and
  $\Co (\mathcal{G}(\theta_2,\Phi),\BanachOne)$ are well-defined Banach spaces.
  By \Cref{pro:protoindep}, the mixed kernel $K_{\theta_1,\theta_2}$
  is contained in $\mathcal B_{m_v}\subset \mathcal B_{m_0}$,
  with $v$ as in the proof of \Cref{cor:warped_disc_frames}.
  Hence, we can apply \Cref{pro:mixedkern} to obtain the desired result.
  The statement about weighted, mixed-norm Lebesgue spaces is, once more,
  a direct consequence of \eqref{eq:BBmBoundedOnLPQ}.
\end{proof}


 \section{Radial warping}
\label{sec:RadialWarping}

In this section, we consider warped time-frequency representations
for which the warping of frequency space depends only
on the modulus in the frequency domain, i.e., we study maps of the form
\[
  \Phi_\vrho :
  \R^d \to \R^d,
  \xi \mapsto \xi/|\xi| \cdot \vrho(|\xi|)\, ,
\]
which we call the \textbf{radial warping function} associated
to the \textbf{radial component} $\vrho : [0,\infty) \to [0,\infty)$.
More precisely, we will provide conditions on the radial component $\vrho$
which ensure that $\Phi_{\vrho}$ is a ($k$-admissible) warping function, as
introduced in Definitions \ref{def:warpfun} and
\ref{assume:DiffeomorphismAssumptions}.
In particular, we will see that if $\vrho$ is a strictly increasing
$\mathcal C^{k+1}$ diffeomorphism which is also linear on a neighborhood of the origin,
then $\Phi_\vrho$ is a $\mathcal C^{k+1}$ diffeomorphism, with inverse
$\Phi_\vrho^{-1} = \Phi_{\vrho^{-1}}$.
Finally, under additional ``moderateness assumptions'' on the derivatives
of $\vrho^{-1}$, we will show that the diffeomorphism $\Phi_\vrho$
is a $k$-admissible warping function.
These claims will be established in
Section \ref{sub:RadialWarpingGeneralConstruction}.

\Cref{sub:SlowStartConstruction} is concerned with circumventing the somewhat unnatural restriction
that $\vrho$ is linear in a neighborhood of the origin.
Using the so-called \textbf{slow-start construction},
one can associate to a ``sufficiently well-behaved'' function $\vsig : [0,\infty) \to [0,\infty)$
a $k$-admissible radial component $\vrho : [0,\infty) \to [0,\infty)$,
which equals $\vsig$ outside an arbitrarily small neighborhood of the origin.

Finally, we discuss several examples of radial warping functions
in Section \ref{sub:RadialWarpingExamples}.

\subsection{General properties of radial warping functions}
\label{sub:RadialWarpingGeneralConstruction}

To enable a more compact notation, we will from now on denote by
$\vrhoinv:=\vrho^{-1}$ the inverse of a bijection $\vrho : \RR \to \RR$.

\begin{definition}\label{def:AdmissibleRho}
  Let $k \in \NN_0$. A function $\vrho : \RR \to \RR$ is called a
  \textbf{$k$-admissible radial component with control weight
  $v : \RR \to \RR^+$},
  if the following hold:
  \begin{enumerate}[leftmargin=0.7cm]
    \item $\vrho$ is a strictly increasing $\mathcal C^{k+1}$-diffeomorphism with inverse
          $\vrhoinv = \vrho^{-1}$.
    \item $\vrho$ is antisymmetric, that is, $\vrho(-\xi) = -\vrho(\xi)$ for all $\xi \in \RR$.
          In particular, $\vrho(0)=0$.

    \item There are $\eps > 0$ and $c > 0$ with $\vrho(\xi) = c \cdot \xi$
          for all $\xi \in (-\eps, \eps)$.

    \item The weight $v$ is continuous, submultiplicative,
          and radially increasing.
          Additionally, $\vrhoinv'$ and
          \begin{equation}
            \widetilde{\vrhoinv} : \RR \to \RR^+, 
            \quad \text{defined by} \quad
            \widetilde{\vrhoinv}(\xi) := \vrhoinv(\xi)/\xi,
            \quad \text{for} \quad \xi \neq 0,
            \quad \text{and} \quad \widetilde{\vrhoinv}(0) := c^{-1}
            \label{eq:PsiTildeDefinition}
          \end{equation}
          are $v$-moderate.

    \item There are constants $C_0, C_1 > 0$ with
          \begin{equation}
            C_0 \cdot \widetilde{\vrhoinv}(\xi)
            \leq \vrhoinv'(\xi)
            \leq C_1 \cdot (1+\xi) \cdot \widetilde{\vrhoinv}(\xi)
            \qquad \forall \, \xi \in \RR^+.
            \label{eq:AdmissibilityFirstDerivative}
          \end{equation}

    \item We have
          \begin{equation}
            |\vrhoinv^{(\ell)} (\xi)| \leq v (\xi-\eta) \cdot \vrhoinv' (\eta)
            \qquad \forall \, \eta,\xi \in [0,\infty)
                   \text{ and } \ell \in \underline{k + 1} \, .
            \label{eq:AdmissibilityHigherDerivative}
          \end{equation}
    \end{enumerate}
\end{definition}

Note that the property \eqref{eq:AdmissibilityHigherDerivative} can equivalently be exchanged
by the simpler $|\vrhoinv^{(\ell)}| \leq C \vrhoinv'$, for all $\ell \in \underline{k + 1}$
(using that $\vrho_\ast '$ is $v$-moderate and $v$ is submultiplicative),
at the cost of introducing a multiplicative constant $Cv(0)$
on the right-hand side of \eqref{eq:AdmissibilityHigherDerivative}.

\nicki{
\begin{remark}\label{rem:NotesOnRadialComponents}
  The reader may wonder why Definition~\ref{def:AdmissibleRho} prescribes properties of $\vrho$
  on the negative half-axis at all.
  These requirements are not strictly necessary, but neither are they an actual restriction:
  The existence of an odd extension of regularity $\mathcal C^{k+1}(\RR)$
  of a function $\vrho_0 \colon [0,\infty) \rightarrow [0,\infty)$
  is, in fact, necessary for the radial warping function
  $\Phi_{\vrho_0}$ induced by $\vrho_0$ to be in $\mathcal C^{k+1}(\RR^d)$.
  This is easily seen by considering the case $d=1$. 
  
  On the other hand, the third condition in Definition~\ref{def:AdmissibleRho}
  could indeed be slightly weakened, as long as $\widetilde{\vrhoinv}(\xi) = \vrhoinv(\xi)/\xi$
  has a positive, finite limit for $\xi\rightarrow 0$
  (and sufficiently many of its derivatives have a finite limit at $0$),
  and none of the other conditions are violated.
  However, the behavior of $\varrho$ in a small neighborhood of zero has comparably little effect
  on the induced warped time-frequency system.
  The slow-start construction discussed in Section~\ref{sub:SlowStartConstruction} provides a method
  to modify functions satisfying a weaker variant of Definition~\ref{def:AdmissibleRho}
  in a small neighborhood of zero, resulting in a $k$-admissible radial component.
  Concerning the (lack of) impact of the slow-start construction on the resulting coorbit spaces,
  cf.\ Remark~\ref{rem:WhySlowStart}.
\end{remark}
}

\begin{remark}\label{rem:RadialAdmissibleProperties}
  (1) An important consequence of these assumptions is that there exists a
      constant \nicki{$C_2 = C_2 (\vrho, v) > 0$} with
      \begin{equation}
        | \vrhoinv^{(\ell)} (\xi)|
        \leq C_2 \cdot (1+\xi) \cdot \widetilde{\vrhoinv}(\xi)
        \qquad \forall \, \xi \in \RR^+
                       \text{ and } \ell \in \{0\} \cup \underline{k + 1} \,.
        \label{eq:AdmissibilityHigherDerivativeConsequence}
      \end{equation}

      Indeed, for $\ell=0$
      \eqref{eq:AdmissibilityHigherDerivativeConsequence}
      is always satisfied as long as $C_2 \geq 1$, since $\vrhoinv$ is
      increasing with $\vrhoinv(0)=0$, whence
      $|\vrhoinv^{(0)} (\xi)| = \vrhoinv(\xi) = \xi \cdot \widetilde{\vrhoinv}(\xi)
       \leq (1+\xi) \cdot \widetilde{\vrhoinv}(\xi)$ for $\xi \in \RR^+$.
      Thus, it remains to verify Equation
      \eqref{eq:AdmissibilityHigherDerivativeConsequence} for
      $\ell \in \underline{k + 1}$.
      But for this case, applying
      \eqref{eq:AdmissibilityHigherDerivative} with $\eta = \xi$, we see
      that
      \[
        |\vrhoinv^{(\ell)}(\xi)| \leq v(\xi-\xi) \cdot \vrhoinv'(\xi)
        = v(0) \cdot \vrhoinv '(\xi) \, ,
      \]
      so that \eqref{eq:AdmissibilityFirstDerivative} yields
      $|\vrhoinv^{(\ell)}(\xi)| \leq v(0) \cdot \vrhoinv '(\xi)
      \leq C_1 \cdot v(0) \cdot (1+\xi) \cdot \widetilde{\vrhoinv}(\xi)$.
      Setting $C_2 := \max \{1, C_1 \cdot v(0)\}$ \nicki{and $C_1$ only depends on the radial component $\vrho$, } we have thus
      established \eqref{eq:AdmissibilityHigherDerivativeConsequence}.

  \medskip{}

  (2) To indicate that being an admissible radial component is a nontrivial
      restriction on $\vrho$, we observe that
      condition \eqref{eq:AdmissibilityFirstDerivative} entails certain
      growth restrictions on the function $\vrhoinv = \vrho^{-1}$.
      Indeed, for arbitrary $\eps > 0$ and $\xi \geq 1/\eps$, Equation
      \eqref{eq:AdmissibilityFirstDerivative} shows
      $\vrhoinv'(\xi) \leq C_1 \cdot (1+\xi) \cdot \vrhoinv(\xi) / \xi
      \leq (1+\eps)C_1 \cdot \vrhoinv(\xi)$. This implies
      \begin{align*}
        \frac{d}{d\xi} \left( e^{-(1+\eps)C_1 \xi} \cdot \vrhoinv (\xi) \right)
        & = -(1+\eps)C_1 \cdot e^{-(1+\eps)C_1 \xi} \cdot \vrhoinv(\xi)
            + e^{-(1+\eps)C_1 \xi} \cdot \vrhoinv'(\xi) \\
        & \leq -(1+\eps)C_1 \cdot e^{-(1+\eps)C_1 \xi} \cdot \vrhoinv(\xi)
                + (1+\eps)C_1 \cdot e^{-(1+\eps)C_1 \xi} \cdot \vrhoinv(\xi)
          = 0
      \end{align*}
      for all $\xi \geq 1/\eps$. For any $\xi \geq a \geq 1/\eps$, this implies
      $e^{-(1+\eps) C_1 \xi} \cdot \vrhoinv(\xi)
       \leq e^{-(1+\eps)C_1 a} \cdot \vrhoinv(a)$,
      and hence
      \begin{equation}
        \vrhoinv (\xi) \leq \frac{\vrhoinv(a)}{e^{(1+\eps)C_1 a}}
                          \cdot e^{(1+\eps)C_1 \xi}
        \qquad \forall \, \xi \geq a \geq \eps^{-1} \, ,
                          \text { for any } \eps > 0 \, .
        \label{eq:AdmissibleGrowthRestrictionUpperBound}
      \end{equation}

      Likewise, the lower bound in \eqref{eq:AdmissibilityFirstDerivative}
      implies
      \[
        \frac{d}{d\xi} \left( \xi^{-C_0} \cdot \vrhoinv(\xi) \right)
        = (-C_0) \xi^{-C_0 - 1} \cdot \vrhoinv (\xi) + \xi^{-C_0} \cdot \vrhoinv' (\xi)
        \geq (-C_0) \xi^{-C_0} \cdot \frac{\vrhoinv(\xi)}{\xi}
             + C_0 \cdot \xi^{-C_0} \cdot \frac{\vrhoinv(\xi)}{\xi}
        = 0
      \]
      for all $\xi \in \RR^+$.  Thus, for $\xi \geq a > 0$, we get
      $\xi^{-C_0} \cdot \vrhoinv(\xi) \geq a^{-C_0} \cdot \vrhoinv(a)$, and thus
      \begin{equation}
        \vrhoinv(\xi) \geq \frac{\vrhoinv(a)}{a^{C_0}} \cdot \xi^{C_0}
        \qquad \forall \, \xi \geq a > 0 \, .
        \label{eq:AdmissibleGrowthRestrictionLowerBound}
      \end{equation}

      In words, Equations \eqref{eq:AdmissibleGrowthRestrictionUpperBound}
      and \eqref{eq:AdmissibleGrowthRestrictionLowerBound} show that
      \emph{the inverse of an admissible radial component $\vrho$
      can grow at most exponentially, and has to grow at least
      like a positive (not necessarily integer) power of $\xi$.}
\end{remark}

We define (for a larger class of radial components)
the radial warping function associated with $\vrho$.

\begin{definition}\label{def:RadialWarpingFunction}
  For a diffeomorphism $\vrho : \RR \to \RR$ with $\vrho(\xi) = c \xi$ for all
  $\xi \in (-\eps , \eps)$ and suitable $\eps,c > 0$, the
  \textbf{associated radial warping function} is given by
  \begin{equation}
    \Phi_\vrho
    : \RR^d \to \RR^d,
    \xi \mapsto \widetilde{\vrho}(|\xi|) \cdot \xi,
    \quad \text{with} \quad
    \widetilde{\vrho}(t) := \vrho(t)/t
    \quad \text{for } t \in \RR \setminus \{0\} \, ,
    \quad \text{and} \quad \widetilde{\vrho}(0) := c \,.
    \label{eq:RadialWarping}
  \end{equation}
\end{definition}

Clearly, if $\vrho\in\mathcal C^k(\RR)$, then $\widetilde{\vrho}\in\mathcal C^k(\RR)$.
Our goal in this section is to show that $\Phi_\vrho$ is a $k$-admissible
warping function as per Definition \ref{assume:DiffeomorphismAssumptions},
provided that $\vrho$ is a $k$-admissible radial component.
To this end, we first show that the inverse $\Phi_\vrho^{-1}$ of $\Phi_\vrho$
is given by $\Phi_\vrho^{-1} = \Phi_{\vrho^{-1}}$, and provide a convenient expression
of the Jacobian $\mathrm{D}\Phi_\vrho^{-1}$.
The following notation
will be helpful for that purpose: For $\xi \in \RR^d \setminus \{0\}$, we define
\begin{equation}
  \xi_\circ := \xi / |\xi|,
  \qquad
  \pi_\xi : \RR^d \to \RR^d,
            \tau \mapsto \langle \tau, \xi_\circ \rangle \cdot \xi_\circ \, ,
  \qquad \text{and} \qquad
  \pi_\xi^{\perp} := \identity_{\RR^d} - \pi_\xi \, ,
  \label{eq:SpecialProjections}
\end{equation}
so that $\pi_\xi$ is the orthogonal projection on the space spanned by $\xi$,
while $\pi_\xi^{\perp}$ is the orthogonal projection on the orthogonal
complement of this space.
With these notations, the derivative of $\Phi_\vrho$ and $\Phi_\vrho^{-1}$
can be described as follows:

\begin{lemma}\label{lem:RadialWarpingExplicitInverse}
  Let $\vrho : \RR \to \RR$ be a $\mathcal C^k$-diffeomorphism with $\vrho(t) = ct$
  for all $t \in (-\eps , \eps)$ and suitable $\eps,c > 0$.
  Then $\Phi_\vrho$ is $\mathcal C^k$, and for $\xi \in \RR^d \setminus \{0\}$, we have
  \begin{equation}
    \mathrm{D}\Phi_\vrho (\xi)
    = \widetilde{\vrho} (|\xi|) \cdot \pi_{\xi}^{\perp}
      + \vrho'(|\xi|) \cdot \pi_{\xi} \, ,
    \quad \text{and} \quad
    [\mathrm{D}\Phi_\vrho (\xi)]^{-1}
    = [\widetilde{\vrho} (|\xi|)]^{-1} \cdot \pi_{\xi}^{\perp}
      + [\vrho'(|\xi|)]^{-1} \cdot \pi_{\xi} \, .
    \label{eq:RadialWarpingDerivativeExplicit}
  \end{equation}
  Furthermore, $\Phi_\vrho$ is a $\mathcal C^k$-diffeomorphism, with inverse
  $\Phi_{\vrho}^{-1} = \Phi_{\vrhoinv}$ and
  satisfies $\vrhoinv (t) = t / c$ for $t \in (-c\eps, c\eps)$.

  Finally, if $\vrho$ is a $0$-admissible radial component, then we have
  \begin{equation}
    \| [\mathrm{D}\Phi_{\vrhoinv} (\xi)]^{-1} \|
    \lesssim 1 / \widetilde{\vrhoinv} (|\xi|)
    \qquad \forall \, \xi \in \RR^d ,
    \quad \text{with } \widetilde{\vrhoinv}
    \text{ as in } \eqref{eq:PsiTildeDefinition},
    \label{eq:InverseJacobianNormEstimate}
  \end{equation}
  where the implied constant only depends on the constant
  $C_0$ in \eqref{eq:AdmissibilityFirstDerivative}.
\end{lemma}

\begin{proof}
   Recall that $\widetilde{\vrho}\in\mathcal C^k(\RR)$,
  with $\widetilde{\vrho} \equiv c$ on $(-\eps,\eps)$, and hence
  $\Phi_{\vrho}\in \mathcal C^k(\RR^d)$.

  Now, a direct computation using the identity
  $\partial_j |\xi| = \xi_j / |\xi|$ shows for $\xi \in \RR^d \setminus \{0\}$
  that
  \begin{align*}
    \partial_j (\Phi_{\vrho})_i (\xi)
      = \partial_j \left( \xi_i \cdot \frac{\vrho(|\xi|)}{|\xi|} \right)
    & = \widetilde{\vrho}(|\xi|) \cdot \delta_{i,j}
        + \frac{\vrho'(|\xi|) - \widetilde{\vrho}(|\xi|)}{|\xi|^2}
          \cdot \xi_i \xi_j \, .
  \end{align*}
  In vector notation, and with $\xi_\circ = \xi/| \xi |$ as in
  \eqref{eq:SpecialProjections}, this means
  \[
    \mathrm{D}\Phi_{\vrho} (\xi)
    = \widetilde{\vrho}(|\xi|) \cdot \identity
      + \left(
          \vrho'(|\xi|) - \widetilde{\vrho}(|\xi|)
        \right)
        \cdot \xi_\circ \xi_\circ^T \, .
  \]
  Now, recall that $\xi_\circ \xi_\circ^T$ is the matrix representing the linear map
  $\pi_{\xi}$, and that $\identity = \pi_{\xi} + \pi_{\xi}^{\perp}$.
  Inserting these identities into the previous displayed equation establishes
  the claimed formula for $\mathrm{D}\Phi_{\vrho} (\xi)$. In particular, each $\eta \in \RR^d$
  with $\eta \perp \xi$ is mapped to $\widetilde{\vrho}(|\xi|) \cdot \eta$ by
  $\mathrm{D}\Phi_{\vrho} (\xi)$, while each $\eta \in \mathrm{span}(\xi)$ is mapped to
  $\vrho' (|\xi|) \cdot \eta$. Since $\RR^d = \xi^{\perp} \oplus \mathrm{span}(\xi)$,
  the stated formula for $[\mathrm{D}\Phi_{\vrho}(\xi)]^{-1}$ follows.

  \medskip{}

  Linearity of $\vrhoinv(t) = t/c$ for $t\in(-c\eps, c\eps)$ is clear, such that
  $\Phi_{\vrhoinv}$ is a radial warping function as per Definition \ref{def:RadialWarpingFunction}.
%
%
  Note $|\Phi_\vrho(\xi)| = \vrho(|\xi|)$ for $\xi \in \RR^d \setminus \{0\}$,
  such that $\vrhoinv(|\Phi_\vrho(\xi)|) = \vrhoinv(\vrho(|\xi|)) = |\xi|$
  and $\Phi_\vrho(\xi)/|\Phi_\vrho(\xi)|= \xi/|\xi|$.
  Together, this implies
  \(
    \Phi_{\vrhoinv} (\Phi_\vrho (\xi))
    = \xi \, ,
  \)
  for all $\xi\in\RR^d\setminus\{0\}$ and thus, by continuity, for $\xi = 0$ as well.
  Repeating this argument after interchanging $\vrhoinv$ and $\vrho$
  yields $\Phi_\vrho \circ \Phi_{\vrhoinv} = \identity$.

  \medskip{}

  To prove \eqref{eq:InverseJacobianNormEstimate}, consider
  $\xi \!\in\! \RR^d \setminus \{0\}\vphantom{\sum_j}$ and observe that
  $\| [\mathrm{D}\Phi_{\vrhoinv} (\xi)]^{-1} \|
  = \max \big\{
           [\widetilde{\vrhoinv}(|\xi|)]^{-1} , [\vrhoinv' (|\xi|)]^{-1}
         \big\}$, by \eqref{eq:RadialWarpingDerivativeExplicit}.
  Applying the lower inequality in \eqref{eq:AdmissibilityFirstDerivative}, we get
  \[
    \| [\mathrm{D}\Phi_{\vrhoinv} (\xi)]^{-1} \|
    \leq \max \{1, C_0^{-1}\} \cdot 1 / \widetilde{\vrhoinv}(|\xi|) \, .
  \]
For $\xi = 0$ the result follows by continuity.
\end{proof}

To verify Property~\eqref{eq:PhiHigherDerivativeEstimate} of
Definition~\ref{assume:DiffeomorphismAssumptions}, i.e.,
\(
  \left\Vert
    \partial^{\alpha}\phi_{\tau}\left(\upsilon\right)
  \right\Vert
  \leq v_0 (\upsilon)
\),
for all  $\tau, \upsilon \in \mathbb{R}^{d}$
and all $\alpha \in \mathbb{N}_{0}^{d},\  \left|\alpha\right|\leq k$, we need to control certain
derivatives of the (matrix-valued) function
\begin{equation}
   \phi_{\tau}\left(\upsilon\right)
  =\left(A^{-1}(\tau) \cdot A(\upsilon+\tau)\right)^T
  \quad \text{with} \quad
  A(\tau) = \mathrm{D}\Phi_\vrho^{-1} (\tau)
  \label{eq:PhiTauReminder}
\end{equation}
from \eqref{eq:PhiDefinition}. To this end, we will frequently
use \emph{Faa di Bruno's formula}, a chain rule for higher derivatives. Precisely, we will use the
following form of the formula, which is a slightly simplified (but less precise)
version of \cite[Corollary 2.10]{FaaDiBrunoMultidimensional}.
Note that, for a nonnegative multiindex $\alpha$, i.e., $\alpha \in \NN_0^{d}$, we denote the
sum of its components by $|\alpha|\geq 0$ and by $\alpha = 0$ we refer to the unique
multiindex with $|\alpha| = 0$.

\begin{lemma}\label{lem:FaaDiBruno}
  For $\alpha \in \NN_0^{d} \setminus \{0\}$ and $n \in \underline{|\alpha|}$,
  set
  \[
    \Gamma_{\alpha, n}
    := \left\{
          \gamma = (\gamma_1, \dots, \gamma_n)
          \in \Big[\NN_0^{d} \setminus \{0\}\Big]^n
          \with
          \smash{\sum_{j=1}^n}\vphantom{\sum} \gamma_j = \alpha
       \right\} \, .
  \]
  Furthermore, set
  $\Gamma := \bigcup_{\alpha \in \NN_0^d \setminus \{0\}}
               \bigcup_{n=1}^{|\alpha|}
                 \Gamma_{\alpha,n}$.

  Then, for each $\gamma \in \Gamma$, there is a constant $D_\gamma \in \RR$
  such that for any open sets $U \subset \RR^d$ and $V \subset \RR$, and any
  $\mathcal C^k$ functions $f : V \to \RR$ and $g : U \to V$, the following holds for
  any $\alpha \in \NN_0^d$ with $|\alpha| \in \underline{k}$:
  \[
    \partial^\alpha (f \circ g)(x)
    = \sum_{n=1}^{|\alpha|}
        \left[
          f^{(n)}(g(x))
          \cdot \sum_{\gamma \in \Gamma_{\alpha,n}}
                  \bigg(
                    D_{\gamma} \cdot
                    \prod_{j=1}^n
                      (\partial^{\gamma_j} g) (x)
                  \bigg)
        \right]
        \qquad \forall \, x \in U \, ,
  \]
  where $f^{(n)}$ denotes the $n$-th derivative of $f$.
\end{lemma}

\begin{rem*}
  From the statement of \cite[Corollary 2.10]{FaaDiBrunoMultidimensional},
  it might appear that the constants $D_\gamma$ also depend on $\alpha,n,d$,
  in addition to $\gamma$. But these parameters are determined by $\gamma$:
  On the one hand, we have $\gamma \in [\NN_0^d]^n$, which uniquely determines
  $n$ and $d$. On the other hand, $\alpha = \sum_{j=1}^n \gamma_j$ for
  $\gamma \in \Gamma_{\alpha,n}$.
\end{rem*}

With these preparations, we can now prove that the radial warping function
$\Phi_{\vrho}$ associated to a $k$-admissible radial component $\vrho$ is indeed
a $k$-admissible warping function.
Most significantly, the following proposition proves that Property~\eqref{eq:PhiHigherDerivativeEstimate},
cf.\ Definition~\ref{assume:DiffeomorphismAssumptions} or the discussion preceding the above lemma,
is satisfied.

\begin{proposition}\label{prop:RadialWarpingFundamental}
  Let $\vrho : \RR \to \RR$ be a $k$-admissible radial
  component with control weight $v : \RR \!\to\! \RR^+$.

  Then there is a constant $C \geq 1$, dependent on $\vrho$, $v$, $d$, and $k$, 
  such that with
  \[
    v_0 : \RR^d \to \RR^+,
          \tau \mapsto C \cdot (1+|\tau|) \cdot v(|\tau|),
  \]
  the function $\Phi_{\vrho}$ satisfies
  \eqref{eq:PhiHigherDerivativeEstimate} 
  for all $\alpha \in \NN_0^d$ with $|\alpha| \leq k$.
\end{proposition}

\begin{proof}
  It is easy to see that $v_0$ is submultiplicative and radially increasing
  as the product of submultiplicative and radially increasing weights
  $C$, $(1+|\bullet|)$ and $v(|\bullet|)$.

 The proof is divided into five steps.
  As a preparation for these,
  recall from \Cref{lem:RadialWarpingExplicitInverse} that
  $\Phi_{\vrho}^{-1} = \Phi_{\vrhoinv} = (\bullet) \cdot \widetilde{\vrhoinv}(|\bullet|)$,
  with $\vrhoinv = \vrho^{-1}$ and $\widetilde{\vrhoinv}$ as defined in
  \eqref{eq:PsiTildeDefinition}.
  By \Cref{lem:RadialWarpingExplicitInverse}, $\widetilde{\vrhoinv}\in \mathcal C^{k+1}(\RR)$. 
  Our main goal is to estimate the derivatives of $\Phi_{\vrhoinv}$.

  \medskip{}

  \textbf{Step 1 - Estimate the derivatives of $\widetilde{\vrhoinv}$:}
  A trivial induction shows
  $\frac{d^\ell}{d t^\ell} t^{-1}
   = (-1)^\ell \cdot \ell! \cdot t^{-(1+\ell)}$. 
  With this, Leibniz's rule shows for any $n \in \underline{k+1}$
  and any $t \in [c\eps, \infty)$ that
  \begin{align}
      \left|
        \widetilde{\vrhoinv}^{(n)} (t) - \frac{\vrhoinv^{(n)}(t)}{t}
      \right|
      = \left|
            \sum_{\ell=1}^{n}
              \binom{n}{\ell}
              \cdot \frac{d^\ell t^{-1}}{d t^\ell}
              \cdot \vrhoinv^{(n-\ell)} (t)
        \right| 
      & \leq C(k) \cdot \sum_{\ell=1}^n
                              t^{-(1+\ell)} \cdot |\vrhoinv^{(n-\ell)} (t)|
        \nonumber \\
      & \overset{\eqref{eq:AdmissibilityHigherDerivativeConsequence}}{\leq} C(k)C_2 \cdot
             \sum_{\ell=1}^n
               t^{-\ell + 1}
               \cdot \left(\frac{1+t}{t}\right)^2
               \cdot \frac{\widetilde{\vrhoinv}(t)}{1+t}\nonumber \\
      ({\scriptstyle{t^{-1} \leq (c\eps)^{-1} \text{ and } (1+t)/t
      =    t^{-1}+1
      \leq (c\eps)^{-1}+1}})
      & \leq C^{(1)} \cdot \widetilde{\vrhoinv}(t)/(1+t) \, , 
    \label{eq:PsiDerivativeHardCase}
  \end{align}
  where the constant $C(k) > 0$ in the first inequality only depends on $k$, $C_2$ is as in
  \eqref{eq:AdmissibilityHigherDerivativeConsequence},
  and $C^{(1)}$ is given by \nicki{$C^{(1)} = C(k) \cdot C_2 ((c\eps)^{-1}+1)^2 \cdot (k+1) \cdot \max\{1,(c\eps)^{-k}\}$} .

  In particular,
  \(
    |\widetilde{\vrhoinv}^{(n)} (t)|
    \leq C^{(1)} \cdot \widetilde{\vrhoinv}(t)/(1+t)
         + |\vrhoinv^{(n)}(t)/t|
    ,
  \)
  such that
 \begin{equation}
    |\widetilde{\vrhoinv}^{(n)} (t)|
    \leq C^{(1)} \cdot \frac{\widetilde{\vrhoinv}(t)}{1+t}
         + \left|\frac{\vrhoinv^{(n)}(t)}{t}\right|
    \overset{\eqref{eq:AdmissibilityFirstDerivative}, \eqref{eq:AdmissibilityHigherDerivative}}{\leq}
         (C^{(1)}C_0^{-1} + v(0)) \cdot \nicki{\frac{\vrho_\ast'(t)}{t}.}
    \label{eq:PsiTildeDerivativeEstimateEasyCase2}
  \end{equation}
  Furthermore, the same estimate yields, with $(1+t)/t \leq (c\eps)^{-1}+1$
  and \eqref{eq:AdmissibilityHigherDerivativeConsequence}, 
  \begin{equation}
    |\widetilde{\vrhoinv}^{(n)} (t)|
    \leq C^{(1)} \cdot \widetilde{\vrhoinv}(t) + C_2 \cdot ((c\eps)^{-1}+1) \widetilde{\vrhoinv}(t)
    \leq 2C^{(1)} \cdot \widetilde{\vrhoinv}(t),
    \label{eq:PsiTildeDerivativeEstimateEasyCase}
  \end{equation}
  where both \eqref{eq:PsiTildeDerivativeEstimateEasyCase2}
  and \eqref{eq:PsiTildeDerivativeEstimateEasyCase} hold for all $t \in [c\eps, \infty)$
  and $n \in \underline{k+1}$.

  \medskip{}
 \nicki{\textbf{Step 2 - Estimate the partial derivatives of $\tau \mapsto |\tau|$ for $\tau\neq 0\, $:}
  It is well known that the derivative of order $n\in\NN$ of the square root function $t \mapsto t^{1/2}$
  has the form $t\mapsto c_n \cdot t^{-n + 1/2}$, for all $t\in\RR^+$ and some constant $c_n \neq 0$.
  Further, noting that $\partial^{\alpha} |\tau|^2 = 0$ unless $\alpha = i e_j$
  for $i \in \{0,1,2\}$ and $j \in \underline{d}$, it is easy to see that
  $\big| \, \partial^{\alpha} |\tau|^2 \, \big| \leq 2 \cdot |\tau|^{2-|\alpha|}$
  for all $\tau \in \RR^d \setminus \{0\}$ and $\alpha \in \NN_0^d$.

 Since $\tau \mapsto |\tau|$ equals the composition $|\bullet| = (\bullet)^{1/2}\circ |\bullet|^2$,
 Faa di Bruno's formula (see Lemma~\ref{lem:FaaDiBruno}) yields  
  \[
    \big| \partial^{\alpha} |\tau| \big|
    = \left|
        \sum_{n=1}^{|\alpha|}
          c_n |\tau|^{1-2n} \cdot
          \sum_{\gamma \in \Gamma_{\alpha,n}}
            \bigg(
              D_\gamma
              \cdot \prod_{j=1}^n
                      (\partial^{\gamma_j} |\bullet|^2) (\tau)
            \bigg)
      \right| \, ,
  \] 
  for $\alpha \in \NN_0^d \setminus \{0\}$ and $\tau \in \RR^d\setminus\{0\}$.
  But we have $\sum_{j=1}^n \gamma_j = \alpha$ for
  $\gamma \in \Gamma_{\alpha,n}$, and hence, using $n\leq |\alpha|$, we have \( \left|
      \smash{\prod_{j=1}^n} \vphantom{\prod}
        (\partial^{\gamma_j} |\bullet|^2) (\tau)
    \right|    
    \leq 2^{|\alpha|} \cdot |\tau|^{2n - |\alpha|}.
  \)
  Overall, we obtain 
  \begin{equation}
   \big| \partial^{\alpha} |\tau| \big|
    \leq C_\alpha \cdot |\tau|^{1 - |\alpha|}
    \qquad \forall \, \tau \in \RR^d \setminus \{0\}
    \text{ and } \alpha \in \NN_0^d,
    \label{eq:EuclideanNormHigherDerivatives}
  \end{equation}
  for some constants $C_\alpha$ that may also depend on $d$.
  The estimate is trivial in case of $\alpha=0$.}

  \medskip{}

  \textbf{Step 3 - Estimate the partial derivatives of
  $\zeta : \RR^d \setminus\{0\}\to\RR,
           \tau \mapsto \widetilde{\vrhoinv}(|\tau|)$:}
  Note that this map is just the composition of $\widetilde{\vrhoinv}$ with
  the map $\tau\mapsto |\tau|$ analyzed in the preceding step.
  Thus, Faa di Bruno's formula
  (see Lemma~\ref{lem:FaaDiBruno}) shows for any
  $\alpha \in \NN_0^d \setminus \{0\}$ with $|\alpha| \leq k+1$
  and $\tau \in \RR^d \setminus \{0\}$ that
  \begin{equation}
    \partial^\alpha \zeta (\tau)
    = \sum_{n=1}^{|\alpha|}
        \left[
          \widetilde{\vrhoinv}^{(n)}(|\tau|)
          \cdot \sum_{\gamma \in \Gamma_{\alpha, n}}
                \left(
                  D_\gamma
                  \cdot \smash{\prod_{j=1}^n}\vphantom{\prod} \,
                          \partial^{\gamma_j} |\tau|
                \right)
        \right] \, .
    \label{eq:PsiTildeEuclideanNormFaaDiBruno}
  \end{equation}
  In the previous step, we saw $|\partial^{\gamma_j} |\tau||
  \leq C_{\gamma_j} \cdot |\tau|^{1-|\gamma_j|}$.
  Since $\sum_{j=1}^n \gamma_j = \alpha$ for
  $\gamma = (\gamma_1, \dots, \gamma_n) \in \Gamma_{\alpha,n}$, there thus exists a
  \nicki{constant $C_\gamma>0$ that may additionally depend on $d$, such that}
  \begin{equation}
    \left| \prod_{j=1}^n \partial^{\gamma_j} |\tau| \right|
    \leq C_\gamma \cdot |\tau|^{n-|\alpha|}
    \qquad \forall \, \tau \in \RR^d \setminus \{0\}, \quad
                      \alpha \in \NN_0^d \setminus \{0\}, \quad
                      n \in \underline{| \alpha |}, \quad
                      \text{and} \quad \gamma \in \Gamma_{\alpha,n} \, .
    \label{eq:InnerFaaDiBrunoProductEstimate}
  \end{equation}

  Now, let us focus on the case $|\tau| \geq c\eps$.
  Then, if $n \in \underline{|\alpha|-1}$, the estimate
  \eqref{eq:PsiTildeDerivativeEstimateEasyCase} yields with
  \(
    |\tau|^{n-|\alpha|}
   \leq (c\eps)^{n+1-|\alpha|} \cdot |\tau|^{-1}
   \leq (c\eps)^{n+1-|\alpha|}(1 + (c\eps)^{-1}) / (1+|\tau|)
  \)
  that \nicki{
  \begin{equation}
    \left|
      \widetilde{\vrhoinv}^{(n)}(|\tau|) \cdot
      \sum_{\gamma \in \Gamma_{\alpha, n}} \!\!\!
      \left(
        D_\gamma
        \cdot \smash{\prod_{j=1}^n}\vphantom{\prod} \,
                \partial^{\gamma_j} |\tau|
      \right)
    \right|
    \lesssim \frac{\widetilde{\vrhoinv}(|\tau|)}{1 + |\tau|}
    \quad \forall \, \tau \in \RR^d\setminus\overline{B_{c\eps}}(0)
                   \text{ and } n \in \underline{| \alpha | - 1} \, .
    \label{eq:PsiTildeEuclideanNormEasyCase}
  \end{equation}
  Here, $\alpha \in \NN_0^d$ with $|\alpha| \in \underline{k+1}$. 

  Overall, by combining \eqref{eq:PsiTildeEuclideanNormFaaDiBruno}--%
  \eqref{eq:PsiTildeEuclideanNormEasyCase}
  (and noting that the case $n=|\alpha|$ is not covered by \eqref{eq:PsiTildeEuclideanNormEasyCase}),
  we get
  \begin{align}
    | \partial^\alpha \zeta (\tau) |
    = \Big| \partial^\alpha \big(\widetilde{\vrhoinv}(|\tau|)\big) \Big|
    & \lesssim 
                            \Big|\widetilde{\vrhoinv}^{(|\alpha|)}(|\tau|)\Big|
                            + \frac{\widetilde{\vrhoinv}(|\tau|)}{1+|\tau|}
                         \label{eq:PsiTildeEuclideanTotalCase1}\\
    & \overset{\eqref{eq:PsiTildeDerivativeEstimateEasyCase}}{\lesssim}
      \widetilde{\vrhoinv} (|\tau|)
    \quad
    \forall \, \alpha \in \NN_0^d \, \text{ with } \, |\alpha| \in\underline{k + 1} 
    \, \text{ and } \, \tau \in \RR^d \setminus B_{c\eps} (0) \, .
  \label{eq:PsiTildeEuclideanTotalCase2}
  \end{align}
  We note that the total implied constant between the left and right hand sides
  of Equation~\eqref{eq:PsiTildeEuclideanTotalCase2} depends on $\vrho,\ v$, $d$, and $k$.
  Note that the quantities $c,\ \eps,\ C_0,\ C_1$ and $C_2$ that are more explicitly present
  in the dependencies are themselves directly derived from $\vrho$ and $v$.
  Finally, the case $\alpha=0$ is trivial.}

  \medskip{}

  \textbf{Step 4 - Estimate
  $\partial^{\alpha} \phi_\tau (\upsilon)$ for $0 \leq |\alpha| \leq k$ and
  $\upsilon \in \RR^d \setminus B_{c\eps} (-\tau)$:}
  Recall from \eqref{eq:PhiTauReminder} the definition of
  $\phi_\tau (\upsilon) = \left( A^{-1}(\tau) \cdot A(\tau+\upsilon) \right)^T$.
  Since $\|M\| = \|M^T\|$ for all $M \in \RR^{d \times d}$,
  and since $\partial^\alpha [M(\tau)]^T = [\partial^\alpha M(\tau)]^T$ for any
  sufficiently smooth matrix-valued function $M : \RR^d \to \RR^{d \times d}$,
  it is sufficient to estimate $\| \partial^\alpha \varphi_\tau (\upsilon) \|$
  with $\varphi_\tau (\upsilon) := A^{-1}(\tau) \cdot A(\tau+\upsilon)$, where
  $A(\tau) = \mathrm{D}\Phi_{\vrhoinv} (\tau)$.

  Furthermore, $\|M\| \leq \sum_{j=1}^d | M_{\bullet,j} |$ for all $M\in\RR^{d\times d}$,
  such that it is sufficient to estimate the columns of $M$ individually.
  In the following, we denote, for $\alpha\in\NN_0^d$ and $\upsilon\in\RR^d$,
  \(
    \partial^\alpha_\upsilon
    = \frac{\partial^{|\alpha|}}
           {\partial^{\alpha_1}_{\upsilon_1}\, \cdots\,  \partial^{\alpha_d}_{\upsilon_d}}
    .
  \)
  Let us fix $\tau \in \RR^d \setminus \{0\}$.
  Then, $\partial^\alpha \varphi_\tau (\upsilon) = A^{-1}(\tau) \cdot (\partial^\alpha A) (\tau+\upsilon)$.
  We see that the $j$-th column of
  $\partial^\alpha \varphi_\tau (\upsilon)$ is simply
  \begin{equation}
     \left[ \partial^\alpha \varphi_\tau (\upsilon) \right]_{\bullet,j}
    = A^{-1}(\tau) \cdot \partial_\upsilon^\alpha [A (\tau+\upsilon)]_{\bullet,j}
    = A^{-1}(\tau) \cdot (\partial_\upsilon^{\alpha+e_j} \Phi_{\vrhoinv}) (\tau+\upsilon)
    \, .
    \label{eq:RadialWarpingProductRulePreparation}
  \end{equation}
  Now fix $j \in \underline{d}$, and set $\sigma := \alpha + e_j$ for brevity.
  Note $\sigma \in \NN_0^d \setminus \{0\}$ with $|\sigma| \in \underline{k + 1}$.

  By definition of $\Phi_{\vrhoinv}$, the $i$-th entry of
  $\Phi_{\vrhoinv} (\tau)$ is
  $[\Phi_{\vrhoinv} (\tau)]_i = \tau_i \cdot \widetilde{\vrhoinv}(|\tau|)$.
  Let $\alpha_+$, for $\alpha\in\ZZ_0^d$, be the elementwise positive part,
  i.e., $(\alpha_+)_i = \max\{0,\alpha_i\}$, $i\in\underline{d}$.
  The Leibniz rule, with $\partial^\beta \tau_i = 0$ for
  $\beta \notin \{0,e_i\}$ and $\partial_i \tau_i = 1$, yields
  \[
    [\partial^{\sigma} \Phi_{\vrhoinv} (\tau)]_i
    = \tau_i \cdot \big[\partial^{\sigma} (\widetilde{\vrhoinv}(|\tau|))\big]
      + \sigma_i \cdot \partial_\tau^{(\sigma - e_i)_+} \big(
                                               \widetilde{\vrhoinv}(|\tau|)
                                             \big)
    \qquad \forall \, i \in \underline{d} \, ,
  \]
  or in other words,
  \begin{equation}
    \partial^\sigma \Phi_{\vrhoinv} (\tau)
    = \big(
        \partial_\tau^\sigma \big[\widetilde{\vrhoinv}(|\tau|)\big]
      \big) \cdot \tau
      + v_{\sigma,\tau} \, ,
    \quad \text{with} \quad
    v_{\sigma,\tau} := \left[
                         \sigma_i \cdot \partial_\tau^{(\sigma - e_i)_+}
                               \big(\widetilde{\vrhoinv} (|\tau|)\big)
                      \right]_{i = 1,\dots,d} \, .
    \label{eq:PhiProductRuleApplication}
  \end{equation}

  Now, by \eqref{eq:PsiTildeEuclideanTotalCase2}, we have
  \nicki{$| v_{\sigma,\tau} | \lesssim \widetilde{\vrhoinv}(|\tau|)$,
  for all $\tau \in \RR^d \setminus B_{c\eps} (0)$.}
  Furthermore, Lemma \ref{lem:RadialWarpingExplicitInverse} provides the estimate
  $\|A^{-1}(\tau)\| = \|[\mathrm{D}\Phi_{\vrhoinv}(\tau)]^{-1}\|
   \leq \max\{1,C_0^{-1}\}/\widetilde{\vrhoinv}(|\tau|)$
  with $C_0$ as in \eqref{eq:AdmissibilityFirstDerivative}.
  Note that we inserted the explicit constant derived
  in the proof of Lemma~\ref{lem:RadialWarpingExplicitInverse} above. 
  Since $\widetilde{\vrhoinv}$ is $v$-moderate and $v$ is radially
  increasing, this implies
  \nicki{
    \begin{equation}
      | A^{-1}(\tau) \cdot v_{\sigma, \tau + \upsilon} |
      \lesssim \max\{1,C_0^{-1}\}\cdot \frac{\widetilde{\vrhoinv} (| \tau + \upsilon |)}
                        {\widetilde{\vrhoinv}(| \tau |)}
      \leq  \max\{1,C_0^{-1}\}
      \cdot v(|\upsilon|)
      = 
      \max\{1,C_0^{-1}\}
        \cdot v_{0}(\upsilon)
    \label{eq:RadialWarpingAdmissibleMainStepTerm1}
  \end{equation}
  }%
  for all $\upsilon \in \RR^d \setminus B_{c\eps} (-\tau)$. 
  Thus, in view of \eqref{eq:PhiProductRuleApplication},
  it remains to estimate
  \(
    \big(
                \partial_{(\tau + \upsilon)}^\sigma
                \big[\widetilde{\vrhoinv}(|\tau + \upsilon|)\big]
              \big)
        \cdot A^{-1}(\tau) \langle \tau + \upsilon\rangle
  \)
  for $\tau + \upsilon \in \RR^d \setminus B_{c\eps} (0)$.

  \medskip{}

  Lemma \ref{lem:RadialWarpingExplicitInverse} implies
  \begin{equation}
    A^{-1}(\tau) = [\mathrm{D}\Phi_{\vrhoinv} (\tau)]^{-1}
        = [\widetilde{\vrhoinv} (|\tau|)]^{-1} \cdot \pi_{\tau}^{\perp}
          + [\vrhoinv ' (|\tau|)]^{-1} \cdot \pi_{\tau} \, .
    \label{eq:RadialWarpingAdmissibleAZeroRewritten}
  \end{equation}
  \nicki{Now, we apply \eqref{eq:PsiTildeEuclideanTotalCase2},
  and $v$-moderateness of $\widetilde{\vrhoinv}$ for radially increasing $v$, to derive 
  \nicki{
  \begin{equation}
    \begin{split}
      & \big| \partial_{\upsilon}^\sigma \big(\widetilde{\vrhoinv} (|\tau + \upsilon|) \big) \big|
        \cdot \big[ \widetilde{\vrhoinv} (|\tau|) \big]^{-1}
        \cdot | \pi_{\tau}^\perp (\tau + \upsilon) | \\
      & \lesssim
             \big[ \widetilde{\vrhoinv} (|\tau|) \big]^{-1}
             \cdot \widetilde{\vrhoinv}(|\tau + \upsilon |)
             \cdot | \upsilon | \\
      & \leq v(|\upsilon|) \cdot |\upsilon| \\
      ({\scriptstyle{v_{0} \geq (1+|\bullet|)\cdot v(|\bullet|)}})
       & \leq v_{0} (\upsilon) \, .
    \end{split}
    \label{eq:RadialWarpingAdmissibleMainStepTerm2}
  \end{equation}
  }
  Here, we additionally used the straightforward estimate
  \(
    | \pi_{\tau}^\perp (\tau + \upsilon) |
    = | \pi_{\tau}^\perp (\tau) + \pi_{\tau}^\perp (\upsilon) |
    = | \pi_{\tau}^\perp (\upsilon) |
    \leq |\upsilon|
  \).}

  Finally, with the elementary estimate $| \pi_{\tau} (\upsilon+\tau)| \leq |\upsilon+\tau|$, we get
   \begin{equation}
    \begin{split}
      & \big| \partial_{\upsilon}^\sigma \big( \widetilde{\vrhoinv} (|\upsilon+\tau|) \big) \big|
        \cdot \left[ \vrhoinv' (|\tau|) \right]^{-1}
        \cdot | \pi_{\tau} (\upsilon+\tau) | \\
     ({\scriptstyle{\eqref{eq:PsiTildeEuclideanTotalCase1}, ~ \eqref{eq:AdmissibilityFirstDerivative},
                    \text{ and } \eqref{eq:PsiTildeDerivativeEstimateEasyCase2}}})~
     & \lesssim | \upsilon + \tau|
            \cdot \frac{\vrhoinv'(|\upsilon + \tau|)}
                       {|\upsilon + \tau|\cdot \vrhoinv' (|\tau|)} \\
     ({\scriptstyle{\eqref{eq:AdmissibilityHigherDerivative}}})~
     & \lesssim v(|\upsilon|) \leq C^{(11)} \cdot v_{0}(\upsilon) \, .
    \end{split}
    \label{eq:RadialWarpingAdmissibleMainStepTerm3}
  \end{equation}
  Overall, combining \eqref{eq:RadialWarpingProductRulePreparation}--%
  \eqref{eq:RadialWarpingAdmissibleMainStepTerm3}, we finally see
  \begin{align*}
     \| \partial^\alpha \phi_{\tau} (\upsilon) \|
    = \| \partial^\alpha \varphi_{\tau} (\upsilon) \|
    & \leq d \cdot \max_{j \in \underline{d}}
                      \big|
                        \big[\partial^\alpha \varphi_{\tau} (\upsilon)\big]_{\bullet,j}
                      \big| \\
   ({\scriptstyle{\eqref{eq:RadialWarpingProductRulePreparation}}})~
    & \leq d \cdot
           \max_{j \in \underline{d}}
             \big|
               A^{-1}(\tau)
               \cdot (\partial^{\alpha+e_j} \Phi_{\vrhoinv})(\tau + \upsilon)
             \big| \\
    ({\scriptstyle{\eqref{eq:PhiProductRuleApplication}}})~
    & \leq d \cdot \max_{j \in \underline{d}}
                     \left(
                       \big|
                         \nicki{\partial_{\upsilon}^{\alpha+e_j}}
                         \big( \widetilde{\vrhoinv}(|\upsilon+\tau|) \big)
                       \big|
                       \cdot | A^{-1}(\tau) \, (\upsilon+\tau) |
                       + | A^{-1}(\tau) \, v_{\alpha+e_j , \tau + \upsilon} |
                     \right) \\
    ({\scriptstyle{\eqref{eq:RadialWarpingAdmissibleMainStepTerm1}\text{--}%
                   \eqref{eq:RadialWarpingAdmissibleMainStepTerm3}}})~
    & \nicki{\lesssim v_{0} (\upsilon)}
      \quad \text{for all } \upsilon \in \RR^d \setminus B_{c\eps}(-\tau)
            \text{ and } |\alpha| \leq k \, ,
  \end{align*}
  where the implied constant between left and right hand side depends on $\vrho$, $v$, $d$, and $k$.

  \medskip{}

  \textbf{Step 5 - Estimate $\partial^\alpha \phi_\tau (\upsilon)$
  for $0 \leq |\alpha| \leq k$ and $\upsilon \in B_{c\eps}(-\tau)$:}
  By Lemma \ref{lem:RadialWarpingExplicitInverse}, 
  $\vrhoinv(t) = t / c$, and thus $\widetilde{\vrhoinv}(t) = c^{-1}$ for
  $t \in (-c\eps,c\eps)$.
  Hence, $\Phi_{\vrhoinv} (\tau) = c^{-1} \cdot \tau$ for all
  $\tau \in B_{c\eps}(0)$, so that
  $A(\tau) = \mathrm{D}\Phi_{\vrhoinv} (\tau) = c^{-1} \cdot \identity_{\RR^d}$ for
  $\tau \in B_{c\eps} (0)$.

  Hence, $\phi_\tau (\upsilon) = A^T (\tau + \upsilon) \cdot A^{-T} (\tau)
  = c^{-1} \cdot A^{-T}(\tau)$, whence
  $\| \partial^\alpha \phi_\tau (\upsilon) \| = 0 \leq v_{0}(\upsilon)$ for
  $\upsilon \in B_{c\eps} (-\tau)$ and $\alpha \in \NN_0^d$ with
  $|\alpha| \in \underline{k}$.
  For $\alpha =0$, Eq.~\eqref{eq:InverseJacobianNormEstimate}
  in \Cref{lem:RadialWarpingExplicitInverse} shows
  \nicki{
  \begin{align*}
    \| \phi_\tau (\upsilon)\|
    = c^{-1} \cdot \| A^{-T} (\tau) \|
    = c^{-1} \cdot \| [\mathrm{D}\Phi_{\vrhoinv} (\tau)]^{-1} \|
    \leq \max\{1,C_0^{-1}\} \cdot c^{-1} / \, \widetilde{\vrhoinv} (|\tau|) \, .
  \end{align*}
  }
  But since $\widetilde{\vrhoinv}$ is $v$-moderate, we have
  \(
    c^{-1}
    \!=\! \widetilde{\vrhoinv}(0)
    \leq \widetilde{\vrhoinv}(|\tau|) \cdot v(|\tau|) \, ,
  \)
  and finally $| \tau | \leq c\eps + | \upsilon |$, such that
  $v(| \tau |) \leq v(c\eps) \cdot v(|\upsilon|)$.
  \nicki{Altogether, $\| \phi_\tau (\upsilon)\| \leq 
  \max\{1,C_0^{-1}\} \cdot v(c\eps) \cdot v(|\upsilon|) \lesssim v_0(\upsilon)$, 
  for all $\tau \in \RR^d$
  and $\upsilon \in B_{c\eps}(-\tau)$.}
\end{proof}

That every radial warping function associated to a
$k$-admissible radial component $\vrho$ is indeed a $k$-admissible warping
function is now a straightforward corollary.

\begin{corollary}\label{cor:RadialWarpingIsWarping}
  Let $\vrho : \RR \to \RR$ be a $k$-admissible radial component \nicki{with 
  control weight $v$, for some
  $k \in \NN$ with $k \geq d+1$. 
  Then there is a constant $C \geq 1$, dependent on $\vrho$, $v$, $d$, and $k$, 
  such that with
  \[
    v_0 : \quad
    \RR^d \to \RR^+, \quad
    \tau \mapsto C \cdot (1+|\tau|) \cdot v(|\tau|),
  \]
  the} associated radial warping function
  $\Phi_\vrho : \RR^d \to \RR^d$ is a $k$-admissible warping function,
  with control weight $v_0$. Furthermore, the weight $w = \det(\mathrm{D}\Phi_\vrho^{-1})$ is given by
  \begin{equation}
    w(\tau)
    = \vrhoinv ' (|\tau|) \cdot [\widetilde{\vrhoinv} (|\tau|)]^{d-1}
    \, .
    \label{eq:RadialWarpingWeightExplicit}
  \end{equation}
\end{corollary}

\begin{proof}
  Lemma \ref{lem:RadialWarpingExplicitInverse} shows that
  $\Phi_\vrho : \RR^d \to \RR^d$ is a $\mathcal C^{k+1}$ diffeomorphism
  with $\Phi_\vrho^{-1} = \Phi_{\vrhoinv}$,
  and \eqref{eq:RadialWarpingDerivativeExplicit} implies that
  \(
    w(\tau)
    = \det \mathrm{D}\Phi_{\vrhoinv} (\tau)
    = \vrhoinv' (|\tau|) \cdot [\widetilde{\vrhoinv} (|\tau|)]^{d-1} > 0,
  \) for all $\tau \in \RR^d \setminus \{0\}$.
  By continuity, and since $\vrhoinv ' (0) = \widetilde{\vrhoinv} (0) = c^{-1}$
  is positive, the above formula remains true for $\tau=0$.
  The remaining properties required in Definition~\ref{def:AdmissibleRho}
  follow from Proposition~\ref{prop:RadialWarpingFundamental}.
\end{proof}

\subsection{The slow start construction for radial components}
\label{sub:SlowStartConstruction}

So far, see Definition \ref{def:AdmissibleRho}, we assumed that
a $k$-admissible radial component $\vrho$ has to be linear on a neighborhood
of the origin.
Our goal in this section is to show that if a given function $\vsig$ satisfies
(slightly modified versions of) all the other conditions from
Definition \ref{def:AdmissibleRho}, then one can modify $\vsig$
in a neighborhood of the origin so that it becomes linear there, but
all other properties are retained.
We call this the \textbf{slow start construction}.

\begin{definition}\label{def:SlowStartVariant}
 Fix some $\eps > 0$, and let $\vsig : [0,\infty) \to [0,\infty)$ be continuous
 and strictly increasing with $\vsig (0) = 0$.
 Furthermore, fix 
 $c \in \big(0, \vsig (\eps) / (2\eps)\big)$, 
 \nicki{and an even function $\Omega \in C_c^\infty(\R)$ that satisfies
   $\Omega(\xi) = 1$ for $x\in B_\eps(0)$,  $\Omega(\xi) = 0$ for $x\not\in B_{2\eps}(0)$,
   and $\Omega'(\xi) \leq 0$ for $\xi\in[0,\infty)$. 
 Then the} function
 \begin{equation}
  \vrho : \quad \RR \to \RR, \quad
  \xi \mapsto
  \begin{cases}
    c \xi \cdot \Omega(\xi)
    + \sgn (\xi) \cdot (1 - \Omega(\xi)) \cdot \vsig (|\xi|),
    & \text{if } \xi \neq 0 \, , \\
    0, & \text{if } \xi = 0
  \end{cases}
  \label{eq:SlowStartDefinition}
\end{equation}
is called a \textbf{slow start version} of $\vsig$.
\end{definition}

\nicki{
\begin{remark}\label{rem:WhySlowStart}
  The intent of the slow start construction is to establish a $k$-admissible warping function
  that only differs from a radial function derived directly from $\vsig$
  in a small neighborhood of zero.
  This raises the question whether different slow start versions of $\vsig$,
  obtained, e.g., by choosing different values of $\eps$ in \Cref{def:SlowStartVariant},
  are equivalent in the sense that they generate the same coorbit spaces.
  Although we suspect that this can be shown directly
  by verifying the conditions of Proposition~\ref{pro:mixedkern},
  instead, under fairly general conditions, we will obtain this equivalence
  as a consequence of identifying the respective coorbit spaces
  with certain decomposition spaces~\cite{fegr85,boni07,VoigtlaenderPhDThesis,voigtlaender2016embeddings}
  in a follow-up contribution.
\end{remark}
}

The following lemma summarizes the main \emph{elementary} properties of this
construction.

\begin{lemma}\label{lem:SlowStartElementary}
  Let $\vsig : [0,\infty) \to [0,\infty)$ be continuous and strictly
  increasing with $\vsig (0) = 0$.
  Let $\eps > 0$ be arbitrary, and $c \in \big(0, \vsig(\eps) / (2\eps) \big)$.
  Then, the function $\vrho$ defined in \eqref{eq:SlowStartDefinition}
  has the following properties:
  \begin{enumerate}
    \item \label{enu:SlowStartExtends}We have $\vrho(\xi) = \vsig (\xi)$
          for all $\xi \in [2\eps,\infty)$.

    \item \label{enu:SlowStartAntisymmetric}$\vrho$ is antisymmetric.

    \item \label{enu:SlowStartLinearAtStart}$\vrho(\xi) = c\xi$ for all
          $\xi \in (-\eps,\eps)$.

    \item \label{enu:SlowStartSmooth}If $\vsig |_{\RR^+}$ is $\mathcal C^k$
          for some $k \in \NN_0$, then $\vrho$ is $\mathcal C^k$.

    \item \label{enu:SlowStartIncreasing}If $\vsig|_{\RR^+}$ is $C^1$
          with $\vsig ' (\xi) > 0$ for all $\xi \in (\eps, \infty)$,
          then $\vrho' (\xi) > 0$ for all $\xi \in \RR$.

    \item \label{enu:SlowStartDiffeomorphism1}
          If $\vsig|_{\RR^+}$ is $\mathcal C^k$ with $\vsig'(\xi) > 0$ for all
          $\xi \in (\eps,\infty)$, and if furthermore
          $\vsig(\xi) \to \infty$ as $\xi \to \infty$,
          then $\vrho : \RR \to \RR$ is a $\mathcal C^k$-diffeomorphism and
          $\vsig : [0,\infty) \to [0,\infty)$ is a homeomorphism.
          Finally, we have
          \[
            \vrho^{-1} (\xi) = \vsig^{-1} (\xi)
            \qquad \forall \, \xi \in [\vsig(2\eps),\infty) \, .
          \]
  \end{enumerate}
\end{lemma}
\begin{rem*}
  Item (\ref{enu:SlowStartDiffeomorphism1}) above is particularly interesting, since it is
  often more important to know the properties of the
  \emph{inverse} of the warping function ($\Phi_\vrho^{-1} = \Phi_{\vrho^{-1}}$ by
  Lemma \ref{lem:RadialWarpingExplicitInverse}) than those of the warping function itself.
\end{rem*}

\begin{proof}

  \textbf{Ad (\ref{enu:SlowStartExtends}):} For $\xi \in [2\eps,\infty)$,
  we have $\Omega (\xi) = 0$.
  Therefore, $\vrho (\xi) = \sgn(\xi) \cdot \vsig (|\xi|) = \vsig (\xi)$.

  \medskip{}

  \textbf{Ad (\ref{enu:SlowStartAntisymmetric}):}
  $\Omega$ is symmetric, i.e., $\Omega (-\xi) = \Omega (\xi)$ for all $\xi \in \RR$.
  For $\xi \neq 0$, this implies
  \begin{align*}
    \vrho (-\xi)
    & = c \cdot (-\xi) \cdot \Omega(-\xi)
        + \sgn (-\xi) \cdot (1-\Omega(-\xi)) \cdot \vsig (|-\xi|) \\
    & = - \Big(
            c \xi \cdot \Omega (\xi)
            + \sgn(\xi) \cdot (1-\Omega(\xi)) \cdot \vsig (|\xi|)
          \Big)
      = - \vrho (\xi) \, .
  \end{align*}
  For $\xi=0$, we trivially have $\vrho(-\xi) = 0 = - \vrho(\xi)$.

  \medskip{}

  \textbf{Ad (\ref{enu:SlowStartLinearAtStart}):}
  By choice of $\Omega$, we have $\Omega (\xi) = 1$ for $\xi \in (-\eps,\eps)$.
  For $\xi \neq 0$, this immediately yields
  $\vrho(\xi) = c\xi$, which clearly also holds for $\xi=0$.

  \medskip{}

  \textbf{Ad (\ref{enu:SlowStartSmooth}):} Since $\Omega$ is smooth, and since
  the functions $\xi \mapsto \sgn(\xi)$ and $\xi \mapsto |\xi|$ are smooth on
  $\RR \setminus \{0\}$, it is clear that $\vrho$ is $\mathcal C^k$ on
  $\RR \setminus \{0\}$. But in the preceding point we saw that $\vrho$ is
  linear (and hence smooth) in a neighborhood of zero.
  Hence, $\vrho$ is $\mathcal C^k$.

  \medskip{}

  \textbf{Ad (\ref{enu:SlowStartIncreasing}):} On $(-\eps,\eps)$, we have
  $\vrho (\xi) \!=\! c\xi$, and thus $\vrho' (\xi) \!=\! c \!>\! 0$ on $[-\eps,\eps]$.
  Also, on $(-\infty,-2\eps) \cup (2\eps,\infty)$, we have $\Omega (\xi) = 0$,
  and hence $\vrho(\xi) = \sgn(\xi) \cdot \vsig (|\xi|)$.
  Since $\xi \mapsto |\xi|$ is smooth away from zero, with
  $\frac{d}{d\xi}|\xi| = \sgn(\xi)$, this implies
  $\vrho' (\xi) = (\sgn(\xi))^2 \cdot \vsig ' (|\xi|) > 0$ for $\xi \in \RR$ with
  $|\xi| \geq 2\eps$.

  For $\xi\in(\eps,2\eps)$, we have
  \(
    \vrho'(\xi)
    = \left[
        \Omega(\xi) \cdot c + (1 - \Omega(\xi)) \cdot \vsig'(\xi)
      \right]
      + (-\Omega '(\xi)) \cdot (\vsig (\xi) - c\xi)
    > 0
    ,
  \)
  \nicki{since all three terms are nonnegative and they cannot vanish simultaneously. 
  To see this, note that }$\Omega'(\xi) \leq 0$ for $\xi \in [0,\infty)$, $\vsig'(\xi) > 0$ for $\xi \in (\eps,\infty)$,
  and $\vsig(\xi) \geq \vsig(\eps) > 2 c \eps > c\xi$ for $\xi \in (\eps, 2\eps)$.
  For the last inequality, recall $c \in (0, \vsig(\eps) / (2\eps))$. 
  Positivity of $\vrho'$ on $(-2\eps, -\eps)$ follows from $\vrho$ being antisymmetric.

  \medskip{}

  \textbf{Ad (\ref{enu:SlowStartDiffeomorphism1}):}
  We have $\vrho(0) = 0$ and $\vrho(\xi) = \vsig (\xi)$ for $\xi \geq 2\eps$, such that
  $\vrho([0,\infty)) \supset [0,\infty)$ by the intermediate value theorem.
  Hence, $\vrho$ is surjective by (\ref{enu:SlowStartAntisymmetric}) and with $\vrho' > 0$ by
  (\ref{enu:SlowStartIncreasing}) even bijective.
  As a strictly increasing bijective $\mathcal C^k$ map with positive derivative,
  $\vrho$ is a $\mathcal C^k$-diffeomorphism by the inverse function theorem.

  Similar arguments show that $\vsig$ is a homeomorphism.
  The remaining property $\vrhoinv(\xi) = \vsiginv(\xi)$ for all $\xi \in [\vsig(2\eps), \infty)$
  is now a straightforward consequence of $\vrho(\xi) = \vsig(\xi)$ for all $\xi \in [2\eps, \infty)$.
\end{proof}

Our final goal in this subsection is to state convenient criteria on $\vsig$
which ensure that $\vrho$ is a $k$-admissible radial component.
For this, the following general lemma will be helpful.

\begin{lemma}\label{lem:ModeratenessHalfLine}
  Let $\delta > 0$, and let $\theta_1, \theta_2 : [\delta,\infty) \to [0,\infty)$
  and $u : [0,\infty)\to \RR^+$ be continuous and increasing
  with $u(\xi+\eta) \leq u(\xi)\cdot u(\eta)$ for all $\xi,\eta\in [0,\infty)$.
  Furthermore, assume that there is some $D > 0$ such that
  \begin{equation}
    D\leq \theta_2 (\eta) \cdot u(\eta)
    \qquad \text{and} \qquad
    \theta_1 (\xi) \leq \theta_2 (\eta) \cdot u(| \xi-\eta |)
    \qquad \forall \, \xi,\eta \in [\delta,\infty)\, .
   \label{eq:ModeratenessHalfLineAssumption}
  \end{equation}

  If $\beta_1 : \RR \to [0,\infty)$ and $\beta_2 : \RR \to \RR^+$ are
  continuous with $\beta_j (\xi) = \theta_j (|\xi|)$ for all $\xi \in \RR$ with
  $| \xi | \geq \delta$ and all $j \in \{1,2\}$, then there is a constant
  $C \geq 1$ with
  \[
    \beta_1 (\xi) \leq C \cdot \beta_2 (\eta) \cdot u(| \xi-\eta |)
    \qquad \forall \, \xi,\eta \in \RR \, .
  \]
\end{lemma}

\begin{proof}
  By continuity of $\beta_1 : \RR \to [0,\infty)$ and $\beta_2 : \RR \to \RR^+$,
  there are constants $c_1, c_2 > 0$ with $\beta_1 (\xi) \leq c_1$ and
  $\beta_2(\xi) \geq c_2$ for all $\xi \in [-\delta,\delta]$.
  Further, note that the conditions on $u$ imply $u(0)\geq 1$
  and that $u(|\bullet|)$ is submultiplicative and radially increasing.
  We distinguish four cases:

  \medskip{}

  \textbf{Case 1 ($|\xi| < \delta$ and $|\eta| < \delta$):} 
  \(
    \beta_1 (\xi)
    \leq c_1
    \leq \frac{c_1}{c_2 \cdot u(0)} \cdot \beta_2 (\eta) \cdot u(| \xi-\eta |) \, .
  \)

   \medskip{}

  \textbf{Case 2 ($|\xi| \geq \delta$ and $|\eta| \geq \delta$):}
  \(
    \beta_1 (\xi)
    = \theta_1 (|\xi|)
    \leq \theta_2 (|\eta|) \cdot u\big( \big| \, |\xi| - |\eta| \, \big| \big)
    \leq \beta_2 (\eta) \cdot u(|\xi-\eta|) \, .
  \)

  \medskip{}

  \textbf{Case 3 ($|\xi| < \delta$ and $|\eta| \geq \delta$):} We have
  \(
    D
    \leq \theta_2 (|\eta|) \cdot u(|\eta|)
    \leq \theta_2 (|\eta|) \cdot u(|\eta-\xi|) \cdot u(\delta)\, ,
  \)
  since $u(|\xi|) \leq u(\delta)$.
  Hence,
  \(
    \beta_1 (\xi)
    \leq c_1
    \leq \frac{c_1 \cdot u(\delta)}{D} \cdot \theta_2 (| \eta |)\cdot u(|\eta-\xi|)
    \leq \frac{c_1 \cdot u(\delta)}{D} \cdot \beta_2 ( \eta )\cdot u(|\eta-\xi|)\, .
  \)

  \medskip{}

  \textbf{Case 4 ($| \xi | \geq \delta$ and $|\eta| < \delta$):} We have
  $\big| \, |\xi| - \delta \, \big| \leq |\xi| \leq |\xi-\eta| + |\eta| < |\xi-\eta| + \delta$.
  Hence,
  \(
    \beta_1 (\xi)
    = \theta_1 (|\xi|)\leq c_2^{-1} \cdot \beta_2(\eta) \cdot \theta_1(|\xi|)
    \leq \frac{\theta_2 (\delta)
          \cdot u(\delta)}{c_2} \cdot \beta_2(\eta) \cdot u(| \xi-\eta |)\, .
  \)

  \medskip{}

  Altogether, we have shown $\beta_1 (\xi) \leq C \beta_2(\eta) \cdot u(| \xi-\eta |)$
  for all $\xi,\eta\in\RR$, with
  \[
    C := \max \left\{
                1, \quad
                \frac{c_1}{c_2 \cdot u(0)}, \quad
                \frac{c_1 \cdot u(\delta)}{D}, \quad \frac{\theta_2(\delta)\cdot u(\delta)}{c_2}
              \right\} \, . \qedhere
  \]
\end{proof}

We now formally introduce a class of functions
$\vsig : [0,\infty) \to [0,\infty)$ for which the slow-start construction
produces a $k$-admissible radial component.
This will be proven in Proposition~\ref{prop:SlowStartFullCriterion} below.

\begin{definition}\label{def:WeaklyAdmissibleRadialComponent}
  Let $k \in \NN_0$.
  A continuous function $\vsig : [0,\infty) \to [0,\infty)$
  is called a \textbf{weakly $k$-admissible radial component with control weight
  $u : [0,\infty) \to \RR^+$}, if it satisfies the following conditions:
  \begin{enumerate}

    \item $\vsig$ is $\mathcal C^{k+1}$ on $\RR^+$, with $\vsig'(\xi) > 0$
          for all $\xi \in \RR^+$.

    \item $\vsig (0) = 0$ and $\vsig (\xi) \to \infty$ as $\xi \to \infty$.

    \item The control weight $u$ is continuous and increasing with
          $u(\xi+\eta) \leq u(\xi)\cdot u(\eta)$ for all $\xi,\eta\in [0,\infty)$.
          Furthermore, there are $\delta > 0$ and $C_0, C_1 > 0$ with the following properties:
          \begin{align}
            C_0 \cdot \frac{\vsiginv(\xi)}{\xi}
            \leq \vsiginv ' (\xi)
            \leq C_1 \cdot \vsiginv (\xi)
            & \qquad \forall \, \xi \in [\delta,\infty) \, ,
            \label{eq:SlowStartDerivativeLarge} \\
            \frac{\vsiginv (\xi)}{\xi} \leq \frac{\vsiginv (\eta)}{\eta} \cdot u(| \xi-\eta |)
            & \qquad \forall \, \xi,\eta \in [\delta,\infty) \, ,
            \label{eq:SlowStartPsiTildeModerate} \\
            |\vsiginv^{(m)} (\xi)| \leq \vsiginv'(\eta) \cdot u(| \xi-\eta |)
            & \qquad \forall \, \xi,\eta \in [\delta,\infty)
                                \text{ and } m \in \underline{k + 1} \, .
            \label{eq:SlowStartHigherDerivativesBound}
          \end{align}
  \end{enumerate}
\end{definition}

\begin{remark}\label{rem:thatoneremark}
  Properties (1) and (2) imply that $\vsig : [0,\infty) \to [0,\infty)$ is a homeomorphism, with
  inverse $\vsiginv := \vsig^{-1}$.

  In many cases, one even has the stronger condition $\vsiginv ' (\xi) \asymp \vsiginv (\xi) / \xi$
  for all $\xi \in [\delta,\infty)$ instead of \eqref{eq:SlowStartDerivativeLarge}.
  In this case, it is not necessary
  to verify condition \eqref{eq:SlowStartPsiTildeModerate}, since---after
  possibly replacing $u$ by $C \cdot u$ for some $C \geq 1$---this condition
  is implied by \eqref{eq:SlowStartHigherDerivativesBound} for $m=1$.
  Indeed, if \eqref{eq:SlowStartHigherDerivativesBound} holds, then
  \[
    \frac{\vsiginv(\xi)}{\xi} \asymp \vsiginv '(\xi)
    \leq \vsiginv ' (\eta) \cdot u(|\xi-\eta|)
    \lesssim \frac{\vsiginv (\eta)}{\eta} \cdot u(|\xi-\eta|)
    \qquad \text{for } \xi ,\eta \in [\delta, \infty ) \, .
  \]

  \emph{Overall, if $\vsiginv'(\xi) \asymp \vsiginv(\xi) / \xi$
  for $\xi \in [\delta,\infty)$, then 
  $\vsig$ is a weakly $k$-admissible radial component, if 
  $\vsig$ is $\mathcal C^{k+1}$ with $\vsig' (\xi) > 0$,
  $\vsig (0) = 0$ and with $\vsig(\xi) \to \infty$ as $\xi \to \infty$ 
  and $\vsig$ satisfies \eqref{eq:SlowStartHigherDerivativesBound}.}
\end{remark}

Our final result in this subsection shows that the slow-start construction,
applied to a weakly $k$-admissible radial component, yields a $k$-admissible
radial component.

\begin{proposition}\label{prop:SlowStartFullCriterion}
  Let $k \in \NN_0$, and let $\vsig : [0,\infty) \to [0,\infty)$ be a
  weakly $k$-admissible radial component with control weight
  $u : [0,\infty) \to \RR^+$.
  Furthermore, let $\vrho$ be a ``slow-start version'' of $\vsig$
  as in \eqref{eq:SlowStartDefinition}.
  Then there exists a constant $C := C(k) \geq 1$, such that $\vrho$ is a $k$-admissible radial component
  with control weight
  \[
    v : \RR \to \RR^+, \xi \mapsto C \cdot u(|\xi|) \, .
  \]
\end{proposition}
\begin{proof}
  Lemma~\ref{lem:SlowStartElementary} shows that $\vrho : \RR \to \RR$
  satisfies conditions (1)--(3) of Definition~\ref{def:AdmissibleRho}.
  As already observed in the proof of Lemma~\ref{lem:SlowStartElementary},
  the conditions on $u$ imply that $u(|\bullet|)$ is submultiplicative,
  such that the same holds for $v$, since $C\geq 1$.
  Note furthermore, that $\vrhoinv(\xi) = \vsiginv(\xi)$
  for all $\xi \geq \delta' := \max\{\vsig(2\eps),\delta\}$,
  with $\eps>0$ as in Lemma~\ref{lem:SlowStartElementary}
  and $\delta>0$ as in Definition~\ref{def:WeaklyAdmissibleRadialComponent}.

  \medskip{}

  We proceed to prove condition (5) of Definition~\ref{def:AdmissibleRho}:
  For $|\xi| \geq \delta'$, the inequality \eqref{eq:AdmissibilityFirstDerivative}
  (with some constants $\tilde{C}_1,\tilde{C}_2$ in place of $C_1,C_2$)
  is a direct consequence of \eqref{eq:SlowStartDerivativeLarge} and \eqref{eq:SlowStartDefinition}.

  For $|\xi|\leq \vsig(\eps)$, $\widetilde{\vrhoinv}(\xi) = \vsiginv(\xi)/\xi = c^{-1}$,
  such that $\widetilde{\vrhoinv}$ is continuous and there are $c_1,c_2,c_3,c_4>0$,
  such that for all $\xi\in[-\delta',\delta']$, $c_1 \leq \widetilde{\vrhoinv}(\xi) \leq c_2$
  and $c_3 \leq \vrhoinv'(\xi) \leq c_4$.
  Thus, with $C_1 = \min\{\tilde{C}_1,c_3/c_2\}$
  and $C_2 = \max\{\tilde{C}_2,c_4/c_1\}$, \eqref{eq:AdmissibilityFirstDerivative} is satisfied
  for all $\xi\in\RR$.
%

  \medskip{}

  To prove condition (6) of Definition~\ref{def:AdmissibleRho}, consider the following:
  For $|\xi| \geq\delta'$, the antisymmetry of $\vrho$ implies that
  $\vrhoinv(\xi) = \sgn(\xi) \cdot \vsiginv(\sgn(\xi) \cdot \xi)$.
  A straightforward induction therefore shows
  \[
    |\vrhoinv^{(m)}(\xi)|
    = \big|\vsiginv^{(m)}(\sgn(\xi) \cdot \xi)\big|
    = \big| \vsiginv^{(m)}(|\xi|) \big|
    \quad \text{for all} \quad
    m \in \underline{k+1}
    \quad \text{and} \quad
    |\xi| \geq \delta'.
  \]
  Furthermore, note that \eqref{eq:SlowStartHigherDerivativesBound} with $m=1$
  and $\xi = \delta$ and $\vsig'(\xi) > 0$ for all $\xi\in (\eps,\infty)$
  implies $0<\vsiginv'(\delta)/u(\delta) \leq \vsiginv'(\eta)u(\eta)$, since $u$ is increasing.

  Fix some $\ell \in \underline{k+1}$.
  In view of \eqref{eq:SlowStartHigherDerivativesBound},
  we can apply Lemma~\ref{lem:ModeratenessHalfLine} (with $\delta'$ instead of
  $\delta$), with $\theta_1 = \big| \vsiginv^{(\ell)}|_{[\delta',\infty)} \, \big|$,
  $\theta_2 = \vsiginv' |_{[\delta',\infty)}$,
  and with $\beta_1 = |\vrhoinv^{(\ell)}|$, $\beta_2 = \vrhoinv'$.
  Consequently, there is a constant $G_\ell \geq 1$ such that
  \begin{equation}
    | \vrhoinv^{(\ell)}(\xi) |
    = \beta_1 (\xi)
    \leq G_\ell \cdot \beta_2 (\eta) \cdot u(| \eta-\xi |)
    =    G_\ell \cdot \vrhoinv'(\eta) \cdot u(| \eta-\xi |)
    \qquad \forall \, \eta,\xi \in \RR \, .
    \label{eq:SlowStartHigherDerivativeVerification}
  \end{equation}
  Since $\ell \in \underline{k+1}$ was arbitrary, \eqref{eq:AdmissibilityHigherDerivative} is satisfied
  with $C \geq \max \{G_1,\dots,G_{k+1}\}$.

  \medskip{}

  In particular, if we set $\ell = 1$, then \eqref{eq:SlowStartHigherDerivativeVerification} implies
  that $\vrhoinv'$ is $v$-moderate with $v = C u(|\bullet|)$ and any $C\geq G_1$.
  Hence, for condition (4) in Definition \ref{def:AdmissibleRho} it only remains
  to prove that $\widetilde{\vrhoinv}$ is $v$-moderate.

  With $\theta_1 = \theta_2 = \vsiginv/|\bullet|$ (and $\delta'$ instead of $\delta$)
  the inequality \eqref{eq:ModeratenessHalfLineAssumption}
  is implied by \eqref{eq:SlowStartPsiTildeModerate}.
  Therefore, we can invoke Lemma~\ref{lem:ModeratenessHalfLine}
  with this choice of $\theta_1$, $\theta_2$ and $\beta_1 = \beta_2 = \widetilde{\vrhoinv}$.
  Note that $\beta_j (\xi) = \theta_j (|\xi|)$ for $|\xi| \geq \delta'$.
  We obtain a constant $G \geq 1$, such that $\widetilde{\vrhoinv}$ is $v$-moderate
  with $v = C u(|\bullet|)$ and any $C\geq G$.
  Altogether, condition (4) in Definition~\ref{def:AdmissibleRho} is satisfied
  with $v = C u(|\bullet|)$, for any $C\geq \max\{G_1,G\}$.
\end{proof}

\subsection{Examples of radial warping functions}
\label{sub:RadialWarpingExamples}

We now present two examples of radial
components $\vsig : [0,\infty) \to [0,\infty)$.
We show that they are \emph{weakly $k$-admissible}
as per Definition \ref{def:WeaklyAdmissibleRadialComponent}.
By Proposition \ref{prop:SlowStartFullCriterion}
and Corollary \ref{cor:RadialWarpingIsWarping},
any slow start version $\vrho$ of $\vsig$ yields a radial,
$k$-admissible warping function $\Phi_{\vrho}$.
Additionally, we provide in each case a control weight $v_0$ for $\Phi_{\vrho}$.

\begin{example}\label{exa:AlphaModulationWarpingFunctionExample}
  Let \nicki{$p \geq 1$}, and consider the function
  \[
    \vsig : [0,\infty) \to [0,\infty), \xi \mapsto (1+\xi)^{1/p} - 1 \, .
  \]
  Conditions (1)--(2) of Definition \ref{def:WeaklyAdmissibleRadialComponent} are clear. \nicki{For $p=1$, Condition (3) is easily verified with $u\equiv 1$.
  To verify Condition (3) for $p>1$}, we first show that
  \(
     \vsiginv' \asymp \vsiginv / (\bullet).
  \)
  By Remark \ref{rem:thatoneremark}, it is then sufficient to verify only
  \eqref{eq:SlowStartHigherDerivativesBound}.

  Note that $\vsiginv (\xi) = (1+\xi)^p - 1$.
  For $\xi> \delta := 1$, it is easy to see that $(1+\xi)^r - 1 \asymp (1+\xi)^r$, for any $r>0$.
  In particular, with $r=p-1$, we obtain
  \[
    \vsiginv'(\xi)
    = p \cdot (1+\xi)^{p-1}
    \asymp \frac{(1+\xi)^p}{1+\xi}
    \asymp \frac{(1+\xi)^p - 1}{\xi}
    = \frac{\vsiginv (\xi)}{\xi}
    \quad \text{for} \quad \xi \geq 1 \, .
  \]

  Note the inequality
  $1 + \xi \leq 1 + \eta + |\xi-\eta| \leq (1+\eta) \cdot (1+|\xi-\eta|)$,
  which holds for $\eta,\xi \geq 0$.
  As a direct consequence, we obtain for all $\eta,\xi \geq 0$ and $\alpha,\beta\in\RR$
%
  with $\alpha\leq \beta$ that
  \begin{equation}
    (1+\xi)^\alpha
    \leq (1+\xi)^\beta \leq (1+\eta)^\beta \cdot (1+|\eta-\xi|)^{|\beta|}\, .
    \label{eq:PolynomialWeightModerateness}
  \end{equation}

  Define $\tilde{u} = (1+(\bullet))^{|p-1|}$ and note that
  $\vsiginv^{(m)} (\xi) = C_m \cdot (1+\xi)^{p-m}$ for all $m \in \underline{k+1}$,
  for suitable constants $C_m = C_m (m,p) \in \RR$, in particular, $C_1 = p > 0$.
  Therefore,
  \[
    | \vsiginv^{(m)}(\xi) |
    \leq |C_m| \cdot (1+\xi)^{p-m}
    \overset{\eqref{eq:PolynomialWeightModerateness}}{\leq}
         |C_m| \cdot (1+\eta)^{p-1} \tilde{u}(|\xi-\eta|)
    =    \frac{|C_m|}{p} \cdot \vsiginv'(\eta) \cdot \tilde{u}(|\xi-\eta|) \, ,
  \]
  for all $\eta,\xi \geq 1$.
  This proves \eqref{eq:SlowStartHigherDerivativesBound}
  with $u = \max_{m\in\underline{k+1}}\{|C_m|/p\}\cdot \tilde{u}$.

  Hence, $\vsig$ is a weakly $k$-admissible radial component with control weight
  $u : [0,\infty) \to \RR^+$, $u(\xi) = C \cdot (1+\xi)^{|p-1|}$,
  for \nicki{any $k \in \NN_0$ and some appropriate constant $C \geq 1$, depending on $p$ and $k$. }
  By Proposition~\ref{prop:SlowStartFullCriterion} any ``slow start'' version $\vrho$
  of $\vsig$ is $k$-admissible, with control weight $v = C'\cdot u(|\bullet|)$, for some
  $C' \geq 1$. 
  Therefore, Corollary~\ref{cor:RadialWarpingIsWarping} shows that the associated
  radial warping function $\Phi_\vrho$ is indeed a $k$-admissible warping function
  with control weight $v_0 = C''\cdot (1+|\bullet|)\cdot u(|\bullet|) = C''(1+|\bullet|)^{1+| p-1 |}$,
  for constant $C''\geq 1$. 

  At this point, we conjecture that the coorbit spaces
  $\Co (\mathcal{G}(\theta, \Phi_\vrho) , \lebesgue^{p,q}_\kappa)$
  that are associated to the warping function $\Phi_\vrho$ constructed here
  coincide with certain \textbf{$\alpha$-modulation spaces}%
  \nicki{, specifically with $\alpha = p^{-1}(p-1)\in [0,1)$,
  for a proper choice of the weight $\kappa$. }
  In future work, we will verify this by identifying
  $\Co (\mathcal{G}(\theta, \Phi_\vrho) , \lebesgue^{p,q}_\kappa)$ with
  certain decomposition spaces, cf.\ \cite{boni07,fegr85}, and considering embeddings
  between the resulting decomposition spaces and
  \textbf{$\alpha$-modulation spaces}~\cite{gr92-2,fefo06,han2014alpha,daforastte08}
  using the theory developed in~\cite{VoigtlaenderPhDThesis,voigtlaender2016embeddings}.
\end{example}

\begin{example}\label{exa:LogarithmicWarping}
  Consider the function $\vsig : [0,\infty) \to [0,\infty), \xi \mapsto \ln(1+\xi)$.
  It is easy to see that conditions (1)--(2) of Definition \ref{def:WeaklyAdmissibleRadialComponent}
  are satisfied and that $\vsiginv(\xi) = \vsig^{-1}(\xi) = e^{\xi}-1$.

  We now verify condition (3) of Definition~\ref{def:WeaklyAdmissibleRadialComponent}
  by proving that the inequalities \eqref{eq:SlowStartDerivativeLarge}--
  \eqref{eq:SlowStartHigherDerivativesBound} hold with $\delta = 1$ and
  $u : [0, \infty) \to [1, \infty), \xi \mapsto e^\xi$.
  Note that $\vsiginv^{(\ell)} = u$ for all $\ell\in\NN$,
  such that \eqref{eq:SlowStartHigherDerivativesBound} clearly holds,
  even for all $\xi\in \RR^+$.

  \medskip{}

  \textbf{Ad \eqref{eq:SlowStartDerivativeLarge}:} For $\xi \geq \delta = 1$, we
  have $1 \leq e^{\xi} / e$, and thus
  $\vsiginv (\xi) = e^\xi - 1 \geq e^\xi \cdot (1-e^{-1})$. Therefore,
  \[ \frac{e^\xi - 1}{\xi}
    \leq e^\xi \leq (1-e^{-1})^{-1} \cdot \vsiginv (\xi) \, ,
  \]
  so that \eqref{eq:SlowStartDerivativeLarge} is fulfilled with
  $C_0 = 1$ and $C_1 = (1-e^{-1})^{-1} > 0$.

  \medskip{}

  \textbf{Ad \eqref{eq:SlowStartPsiTildeModerate}:} Let
  $\widetilde{\vsiginv}(\xi) := \frac{\vsiginv (\xi)}{\xi} = \frac{e^\xi - 1}{\xi}$ for
  $\xi \in \RR^+$, and note that $\widetilde{\vsiginv}$ has the power series
  expansion
  \[
    \widetilde{\vsiginv}(\xi)
    = \frac{1}{\xi} \cdot \left( \sum_{n=0}^\infty \frac{\xi^n}{n!} - 1 \right)
    = \sum_{n=1}^\infty \frac{\xi^{n-1}}{n!}
    = \sum_{\ell=0}^\infty \frac{\xi^\ell}{(\ell+1)!},
  \]
  which shows that $\widetilde{\vsiginv}$ is increasing, since each term
  of the series is increasing on $\RR^+$. Therefore, $\xi\leq \eta$ implies
  $\vsiginv(\xi) \leq \vsiginv(\eta)\leq \vsiginv(\eta)e^{|\xi-\eta|}$.

  If $0 < \eta < \xi$, then
  \[
   \frac{e^\eta-1}{\eta}\cdot e^{|\xi-\eta|}
   = \frac{e^\eta-1}{\eta}\cdot e^{\xi-\eta}
   = \frac{1-e^{-\eta}}{\eta}\cdot e^{\xi}
   \geq \frac{e^\xi-1}{\xi} \, .
  \]
  Here, the final inequality uses that $\xi\mapsto \tfrac{1-e^{-\xi}}{\xi}$ is decreasing on $\RR^+$. 
  Therefore, \eqref{eq:SlowStartPsiTildeModerate} even holds for all $\xi,\eta\in \RR^+$.
  \medskip{}

  In other words, $\vsig$ is a weakly
  $k$-admissible radial component with control weight
  $u : [0,\infty) \rightarrow \RR^+$, $u(\xi) = e^{\xi}$ (for any $k \in \N_0$).
  By \Cref{prop:SlowStartFullCriterion}, any ``slow start version'' $\vrho$
  of $\vsig$ as per \eqref{eq:SlowStartDefinition}, is a $k$-admissible
  radial component with control weight $v : \RR \to \RR^+, \xi \mapsto C \cdot e^{|\xi|}$,
  for some $C\geq 1$. 
  By \Cref{cor:RadialWarpingIsWarping}, the associated radial warping function $\Phi_\vrho$
  is a $k$-admissible warping function with control weight
  $v_0 : \RR^d\to \RR^+, \tau\mapsto C' \cdot (1+|\tau|) \cdot e^{|\tau|}$,
  for a suitable $C' \geq 1$. 

  It is likely that the coorbit spaces
  $\Co (\mathcal{G}(\theta, \Phi_\vrho) , \lebesgue^{p,q}_\kappa)$ associated
  with the warping function $\Phi_\vrho$ constructed can be embedded into certain
  \textbf{inhomogeneous Besov spaces}~\cite{TriebelTheoryOfFunctionSpaces,tr06,tr88},
  if the weight $\kappa$ is chosen properly.
  If such an embedding exists, we expect the converse to be true as well,
  possibly with a different weight $\tilde{\kappa}$ instead of $\kappa$.
  Similar to the previous examples, the interpretation of
  $\Co (\mathcal{G}(\theta, \Phi_\vrho) , \lebesgue^{p,q}_\kappa)$ as decomposition
  space will be the first step towards verifying such embeddings.
\end{example}
  

  \section{Conclusion}
  \label{sec:conclusion}

  We developed a theory of warped time-frequency systems for functions of arbitrary dimensionality.
  These systems, defined by a prototype function $\theta$ and a diffeomorphism $\Phi$,
  form tight continuous frames and admit the construction of coorbit spaces $\Co_\Phi(Y)$,
  which we have shown to be well-defined Banach spaces, provided that
  $\Phi$ is a $k$-admissible warping function and $Y$ is a suitable, solid Banach space.
  We have further shown that stable discretization, in the sense of Banach frame decompositions,
  of the continuous system $\mathcal{G}(\theta, \Phi)$ is achieved across said coorbit spaces,
  simply by sampling densely enough.
  In all cases, the results are realized by choosing the prototype $\theta$ from a class of smooth,
  localized functions that includes $\mathcal C^\infty_c(\RR^d)$.
  Moreover, the results can be invoked simultaneously for a large class
  of spaces $Y$ including, but not limited to, weighted mixed-norm Lebesgue spaces
  $\lebesgue^{p,q}_\kappa$, $1\leq p,q\leq \infty$.
  Finally, we considered radial warping functions as an important special case,
  showed how they can be constructed from (weakly) admissible radial components,
  and provided examples of radial warping functions for which we expect a relation
  to well-known smoothness spaces.
  Altogether, we have demonstrated that warped time-frequency systems,
  a vast class of translation-invariant time-frequency systems that enable the adaptation to a
  specific frequency-bandwidth relationship, can be analyzed with a unified,
  and surprisingly deep mathematical theory.

  There is an abundance of opportunities for further generalization,
  of which we mention only two:
  (1) That the weight $m$ may only depend on the time variable
  if $\sup_{\xi\in D} \|D\Phi(\xi)\| < \infty$
  (in \Cref{thm:MR1_kernel_is_in_AAm,thm:main2_discreteframes})
  remains an irritating and somewhat unnatural condition,
  but cannot be dropped if $m$ is to be majorized by the product of a time-dependent
  and another frequency-dependent weight.
  If the latter requirement is relaxed and a more general weight is considered,
  it may be possible to consider time-dependent weights if $\|D\Phi(\xi)\|$ is unbounded.
  (2) The construction analyzed in this work does not accommodate frames
  with arbitrary directional sensitivity.
  In particular, the degree of anisotropy is determined directly by the warping function
  and cannot be chosen freely.
  For example, without further modification,
  it cannot currently mimic \nicki{popular directional frames like curvelets or shearlets,
  or even isotropic wavelets; see below for the last point.}
  
  While the present article shows that the coorbit spaces $\Co_{\Phi}(Y)$
  are well-defined Banach spaces admitting a rich discretization theory,
  it does not answer all open questions regarding the structure of $\Co_{\Phi}(Y)$
  as \emph{smoothness spaces}.
  These questions concern, e.g., the description of $\Co_{\Phi}(Y)$ purely in terms
  of Fourier analysis, as well as the existence of embeddings between the spaces $\Co_{\Phi}(Y)$
  for different choices of the warping function $\Phi$ and the space $Y$,
  or between $\Co_{\Phi}(Y)$ and established smoothness spaces, such as Besov spaces,
  Sobolev spaces, $\alpha$-modulation spaces, or spaces of dominating mixed smoothness
  \cite{nikol1962boundary,nikol1963boundary,vybiral2006function}.
  In a follow-up article, we will study these questions
  in the context of decomposition spaces,
  a common generalization of Besov- and modulation spaces.
  Specifically, we will show that the spaces $\Co_\Phi(Y)$ are special decomposition spaces,
  so that the rich theory of these spaces can be employed to answer the questions posed above. 
 \nicki{In that work, we will confirm the conjectured relation to $\alpha$-modulation spaces
 (see Example \ref{exa:AlphaModulationWarpingFunctionExample}) and prove that equality
  between (inhomogeneous) Besov spaces and the coorbit spaces related to warped time-frequency systems
  can only be achieved in the one-dimensional case,
  thereby making the statement about isotropic wavelets in the previous paragraph formal.}


\appendix

\section{Formal details for making sense of the intersection \texorpdfstring{$B \cap B'$}{B ∩ B'}}
\label{sec:GelfandTripleAppendix}

\nicki{
  To make sense of the intersection $B \cap B'$ that appears in \Cref{def:BFD},
  we assume that \emph{$B$ is compatible with a suitable \textbf{Gelfand triple}}
  in the following sense:
  We assume that there exists a topological vector space $V$ of "test functions"
  which satisfies $V \hookrightarrow \Hil$, with dense image.
  For instance, in the case $\Hil = \lebesgue^2 (\R^d)$ one could choose $V = C_c^\infty (\R^d)$
  or $V = \Schwartz(\R^d)$.
  One can then identify each $h \in \Hil$ with the anti-linear functional
  (or "generalized distribution")
  \[
    \varphi_h : \quad
    V \to \CC, \quad
    v \mapsto \langle h, v \rangle_{\Hil}
    ,
  \]
  since it is easy to see that the map $\Hil \to V^\urcorner, h \mapsto \varphi_h$
  is linear and injective.
  Since $V \hookrightarrow \Hil$, we can thus consider $V$ as a subset of $V^{\urcorner}$,
  by virtue of the dual pairing coming from $\Hil$.

  Then, we say that a Banach space $B$ is compatible with the Gelfand triple $(V,\Hil,V^\urcorner)$,
  if $B$ satisfies the following properties:
  \begin{enumerate}[label=(\roman*)]
    \item $(B, \| \cdot \|_B)$ is a Banach space,

    \item $B \subset V^\urcorner$ as sets (here, one potentially has to make some (canonical) identifications,
          such as considering $\lebesgue^p (\R^d)$ as a subset of $\Schwartz^\urcorner(\R^d)$),

    \item the inclusion $B \hookrightarrow V^\urcorner$ is continuous
          (with respect to the weak-$\ast$-topology on $V^\urcorner$),

    \item $V \subseteq B$ is dense
          (this rules out spaces such as $B = \lebesgue^\infty(\R^d)$ for $\Hil = \lebesgue^2(\R^d)$
          and $V = \Schwartz(\R^d)$, but one can then instead use the closure of $V$ in $B$,
          which in this case would be equal to $C_0 (\R^d)$,
          the space of continuous functions "vanishing at infinity").
  \end{enumerate}

  For such a compatible Banach space, we then say that $b \in B \subset V^\urcorner$
  satisfies $b \in B \cap B'$, if there exists a constant $C > 0$ such that
  \[
    |b(v)| \leq C \cdot \| v \|_B \qquad \forall \, v \in V 
    .
  \]
  Since $V \subset B$ is dense, this implies that $\overline{b} \in V'$
  (given by $\langle \overline{b}, v \rangle_{V,V'} = \overline{b(v)}$)
  uniquely extends to a continuous linear functional on $B$;
  we then identify $b$ with this functional.

  Note that if $h \in B \cap \Hil$, then since we are identifying $h$ with the functional $\varphi_h$,
  we have for $v \in V$ that
  \[
    \overline{h} (v)
    = \overline{\varphi_h (v)}
    = \langle v,h \rangle_{\Hil}
    ,
  \]
  so that this interpretation of elements of $B$ as elements of $B'$ is again consistent with the
  duality pairing coming from $\Hil$.
}

\begin{Backmatter}

\paragraph{Acknowledgments}
N.H.\ is grateful for the hospitality and support of the Katholische Universität Eichstätt-Ingolstadt during his visit.
F.V.\ would like to thank the Acoustics Research Institute for the hospitality during several visits,
which were partially supported by the Austrian Science Fund (FWF): 31225--N32.

\paragraph{Funding statement}
N.H.\ was supported by the Austrian Science Fund (FWF): I\,3067--N30;
F.V.\ acknowledges support by the German Science Foundation (DFG) in the context of the Emmy Noether junior research group: VO 2594/1--1.

\paragraph{Competing interests}
None

\paragraph{Ethical standards}
The research meets all ethical guidelines, including adherence to the legal requirements of the study country.

\paragraph{Author contributions}
N.H.\ and F.V.\ collaborated on all aspects of the presented work.
N.H.\ and F.V.\ jointly drafted the manuscript,
and  contributed to and approved the submitted version.
We declare that both authors contributed equally to this work.

\paragraph{Supplementary material}
None


\end{Backmatter}

\end{document}